\documentclass[10pt, leqno]{article}

\PassOptionsToPackage{linktocpage}{hyperref}

\usepackage[french, english]{babel}
\usepackage{a4wide}
\usepackage[utf8,latin1]{inputenc}
\usepackage{times}

\usepackage{float}
\usepackage[cyr]{aeguill}
\usepackage{amsmath,amscd}
\usepackage{amssymb}
\usepackage{mathrsfs}
\usepackage{amsthm}
\usepackage{amsfonts}
\usepackage{graphicx}
\usepackage[colorinlistoftodos]{todonotes}
\usepackage{bbm}
\usepackage{url}
\usepackage{upgreek}

\usepackage{hyperref} 

\usepackage[all]{xy}
\usepackage[T1]{fontenc}

\usepackage[nottoc]{tocbibind}

\usepackage[all]{xy}
\usepackage{hyperref}
\usepackage{array}
\usepackage{tabularx}
\usepackage{yfonts}
\usepackage{ulem}
\usepackage{setspace}  % \linespread{1.1}
\usepackage{abstract}

\usepackage[margin=3cm]{geometry}

\usepackage[nomain,sort=def,toc]{glossaries}
\usepackage{glossary-mcols}

\setglossarystyle{mcolindex}

\newglossary[op-glg]{Symbols}{op-gls}{op-glo}{Index des symboles}

\makenoidxglossaries

\newcommand*{\newoperator}[3][]{%
  \newglossaryentry{#2}{type=Symbols,%
    name={$#3$},text={#3},description={},#1}%
}

\usepackage{subfig}

\makeatletter
\renewcommand\subsubsection{\@startsection{subsubsection}{3}{\z@}%
                                     {-3.25ex\@plus -1ex \@minus -.2ex}%
                                     {-1em}% <---changed
                                     {\normalfont\normalsize\bfseries}}
\makeatother

\newcounter{ite}
\addtocounter{ite}{2} 
\stepcounter{ite}
\newtheorem{theorem}{Théorème}[subsection]
\newtheorem{lemm}[theorem]{Lemme}
\newtheorem{definition}[theorem]{Définition}
\newtheorem{remark}[theorem]{Remarque}

\newtheorem{coro}[theorem]{Corollaire}
\newtheorem{prop}[theorem]{Proposition}
\newtheorem{lemma}{Lemme}

\newtheorem{thm}{Théorème}[section]
\newtheorem{lem}[thm]{Lemme}
\newtheorem{cor}[thm]{Corollaire}
\newtheorem*{rmk}{Remarque}

\newtheorem*{ex}{Exemple}

\numberwithin{equation}{subsection}

\newcommand{\Hom}{\mathrm{Hom}}

\renewcommand{\bar}{\overline}
\renewcommand{\d}{\mathrm{d}}

\newcommand{\leftquotient}[2]{{\raisebox{-.2em}{$#1$}}\backslash\raisebox{.1em}{$#2$}}

\renewcommand{\Re}{\mathrm{Re}}
\renewcommand{\Im}{\mathrm{Im}}
\newcommand{\Fix}{\mathrm{Fix}}
\renewcommand{\hat}{\widehat}
\newcommand{\vol}{\mathrm{vol}}
\newcommand{\bbb}{\mathbb}

\newcommand{\Aut}{\mathrm{Aut}}
\newcommand{\rg}{\mathrm{rg } }
\newcommand{\Gal}{\mathrm{Gal}}
\newcommand{\Ind}{\mathrm{Ind}}

\newcommand{\M}{\mathbf{M}}
\newcommand{\E}{\mathbf{E}}
\newcommand{\val}{v}

\newcommand{\HE}{\mathrm{H}}

\DeclareMathOperator{\Id}{Id}
\DeclareMathOperator{\stab}{\mathfrak{stab}}

\newcommand{\ZZZ}{\mathbb{Z}}
\newcommand{\AAA}{\mathbb{A}}

\newcommand{\ago}{\mathfrak{a}}

\renewcommand{\leq}{\leqslant}
\renewcommand{\geq}{\geqslant}
\newcommand{\htau}{\hat \tau}
\newcommand{\Fr}{\mathbf{F}}
\makeatletter
\newcommand*{\rom}[1]{\expandafter\@slowromancap\romannumeral #1@}
\makeatother
\renewcommand{\cal}{\mathcal}

\newoperator{AAAk}{\AAA_k, \AAA}
\newoperator{aBp}{\ago_{B}^{+}}
\newoperator{baBp}{\bar{\ago_{B}^{+}}}
\newoperator{ago}{\ago_M, \ago_M^{*}, \ago_P, \ago_P^{*}  }

\newoperator{aL}{\ago_{L,\mathbb{Z}}, \ago_{L, \mathbb{C}}^{*} }
\newoperator{APi}{\mathcal{A}_{P,{\pi}}, \mathcal{A}_{Q,{\pi}}}
\newoperator{AnX}{\mathcal{A}_{n}(X_{1})}
\newoperator{Ane}{A_{n,e}(X_1)}

\newoperator{aut}{\mathrm{aut}_{X_1}(z)}
\newoperator{B}{B}
\newoperator{cM}{c_{M}}
\newoperator{CN}{C_{n}(X_{k})}
\newoperator{cQR}{c_{R}^{Q}(\lambda)}
\newoperator{deg}{\deg}
\newoperator{DeltaP}{\Delta_{P}}
\newoperator{DeltaPQ}{\Delta_P^Q}
\newoperator{Deltavee}{(\Delta_{P}^{Q})^{\vee}}
\newoperator{Deltahat}{\hat{\Delta}_P^Q}

\newoperator{delta}{\delta({\pi}, L)}
\newoperator{Delta}{\Delta({\pi}, w)}

\newoperator{detM}{\mathrm{det}_{M_{P}, i}}
\newoperator{Dnd}{D_{n}(d)}
\newoperator{dT}{d(T)}

\newoperator{eta}{\eta}
\newoperator{etas}{\eta_{{s}, X_l}(z)}
\newoperator{EQsig}{ \E^{Q}_{\sigma}(\varphi,\lambda)}

\newoperator{FF}{F, F_k}
\newoperator{F}{{\mathbb{F}}}
\newoperator{FG}{F^{{G}}(\cdot, T)}
\newoperator{Fix}{\Fix({\pi})}
\newoperator{Fr}{\Fr_{X}}
\newoperator{G}{G_n}
\newoperator{GLnAAA}{{GL_n(\AAA)^{e}}, G(\AAA)^{e}}
\newoperator{GnX}{\mathcal{G}_n(X_1)}
\newoperator{GAe}{G_n(\AAA)^e}
\newoperator{GammaP}{\Gamma_{I}(H,T)}
\newoperator{GammaPhp}{\hat{ \Gamma}_{I,\mathfrak{h}}(T)}

\newoperator{he}{ {\mathfrak{h}_{(\underline{e}  )}^{e}   }  }
\newoperator{HP}{H_P}
\newoperator{HQe}{ \widetilde{\HE}_{Q}^{e}}
\newoperator{tHQe}{{\HE}_{Q_s}^{e}}
\newoperator{HM}{\mathcal{H}_M}
\newoperator{HMx}{\mathcal{H}_{M, x}}
\newoperator{ilambda}{i(\lambda)}

\newoperator{ImX}{\Im X^{G_n}_M, \Re X^{G_n}_M}
\newoperator{ImXo}{(\Im X_{M_{P}}^{L^w}){^{\circ}}}
\newoperator{JeT}{J_e^T}

\newoperator{kappaA}{\kappa(A)}
\newoperator{kPxy}{k_{P}(x,y)}
\newoperator{Kx}{K_x}
\newoperator{K}{K}

\newoperator{Lw}{L^{w}}
\newoperator{LaL}{\widetilde{\Lambda}_{L}}
\newoperator{LM}{\mathcal{L}(M)}
\newoperator{LambdaH}{\lambda^{H}, \lambda^{\alpha^{\vee}}}
\newoperator{LH}{\langle \lambda, H   \rangle}

\newoperator{Mw}{\M(w,\lambda), \M_{Q|P}(\lambda), \M_{P|Q}(w,\lambda)}
\newoperator{MQ}{\mathcal{M}_Q(\lambda, P; \mu)}
\newoperator{ne}{(n,e)}
\newoperator{mu}{\mu}
\newoperator{muE}{\mu(\mathcal{E})}
\newoperator{n,m}{(n,m)}
\newoperator{Nf}{\mathrm{N}(f)}
\newoperator{nbeta}{n_{\beta}({\pi},  z)}
\newoperator{Npi}{{{{N}}}}
\newoperator{npi}{n_{{\pi}}(w,\lambda)}

\newoperator{O}{\mathcal{O}, \mathcal{O}_x}

\newoperator{bP}{\bar{P}}
\newoperator{P}{\mathcal{P}}
\newoperator{PB}{\mathcal{P}(B)}
\newoperator{PLM}{\mathcal{P}^{L}(M), \mathcal{P}(M)}
\newoperator{PRM}{\mathcal{P}^{R}(M)}
\newoperator{Pf}{\mathrm{P}(f)}
\newoperator{p.g.c.d.}{\mathrm{p.g.c.d.}}
\newoperator{Pne}{\mathcal{P}_{n}^{e}(X_1)}
\newoperator{PhiP}{\Phi_P, \Phi(Z_{M}, G_n)}
\newoperator{piboxnu}{\pi\boxtimes\nu}
\newoperator{pi1Xo}{\pi_{1}(X, o)}
\newoperator{psiM}{e_{M_P,i}}
\newoperator{Ppi}{(P,{\pi})}
\newoperator{QL}{Q_L}
\newoperator{Qs}{Q_s}
\newoperator{stab}{\stab(P, {\pi})}
\newoperator{jaj}{(j^{a_j})\vdash n}
\newoperator{RdvX}{R_{d, \nu }(X_{k})}
\newoperator{rpi}{\mathrm{r}(\pi)}
\newoperator{ReX}{\Re X^{L}_{M}}
\newoperator{rhoP}{\rho_{P}}
\newoperator{Silambda}{S_i(\lambda)}
\newoperator{sM}{w(L)}

\newoperator{tauP}{\tau_P}
\newoperator{htauP}{\hat{\tau}_P}
\newoperator{thetaQ}{\theta_{Q}}

\newoperator{vp}{v_p}

\newoperator{WaPaQ}{W(\mathfrak{a}_{P}, \mathfrak{a}_{Q}))}
\newoperator{WaM}{W(\mathfrak{a}_{M})}
\newoperator{WM}{W, W^M, W^P}
\newoperator{WX1o}{W(X_{1},o)}
\newoperator{w}{|w|}

\newoperator{x}{[x]}
\newoperator{XM}{X^{*}(M)}
\newoperator{XML}{X_M^L}
\newoperator{XPG}{X_P^{G_n}}
\newoperator{X1p}{|X_{1}|}
\newoperator{Xk}{X_k}
\newoperator{XiM}{\Xi_M, \Xi_{G_n}}
\newoperator{xn}{\binom{x}{n}}
\newoperator{ZM}{Z_{M}}
\newoperator{zv}{[z^{v}]f(z)}
\newoperator{zeta}{\eta_{s,X_1}(z)}
\newoperator{1Q}{\hat{\mathbbm{1}}_{Q}}
\newoperator{1Qe}{\hat{\mathbbm{1}}^{e}_{Q}}
\newoperator{vdashxi}{\vdash_{\xi}}
\newoperator{dline}{|\cdot|}
\newoperator{ZSg}{\mathbb{Z}[z_1,\cdots, z_g, tz_1^{-1},\ldots, tz_g^{-1}]^{\mathcal{S}_g\ltimes(\mathcal{S}_2)^{g}}}

\usepackage{ulem}
\usepackage{xcolor}

\begin{document}

\selectlanguage{french}

\setcounter{secnumdepth}{3}
\setcounter{tocdepth}{2}

\hypersetup{							% Information sur le document
pdfauthor = {HONGJIE YU},			% Auteurs
pdftitle = {counting local system},			% Titre du document
%pdfkeywords = {Tag1, Tag2, Tag3, ...},	% Mots-clefs
}					
\title{COMPTAGE  DES SYSTÈMES LOCAUX $\ell$-ADIQUES SUR UNE COURBE}
\author{HONGJIE YU \footnote{Université Paris Diderot-Paris 7, Institut de Mathématiques de Jussieu - Paris Rive Gauche, Paris, France; Current address: hongjie.yu@weizmann.ac.il,  Weizmann Institute of Science, Herzl St 234, Rehovot, Israel}
}
\date{\vspace{-2ex}}
\maketitle

        \begin{abstract}
        Soit $X_{1}$  une courbe projective lisse et géométriquement connexe sur un corps fini $\mathbb{F}_{q}$ avec $q=p^{n}$ éléments où $p$ est un nombre premier. Soit $X$ le changement de base de $X_{1}$ à  une clôture algébrique  de $\mathbb{F}_{q}$. 
Nous donnons une formule pour le nombre des systèmes locaux $\ell$-adiques ($\ell\neq p$) irréductibles de rang donné sur $X$ fixé par l'endomorphisme de Frobenius. Nous montrons que ce nombre est semblable à une formule de point fixe de Lefschetz pour une variété sur $\mathbb{F}_q$, ce qui généralise un résultat de Drinfeld en rang 2 et prouve une conjecture de Deligne. Pour ce faire, nous passerons du côté automorphe, utiliserons la formule des traces d'Arthur non-invariante, et relierons le nombre cherché avec le nombre $\mathbb{F}_q$-points de l'espace des modules des fibrés de Higgs stables. 

        \end{abstract}

        \renewcommand{\abstractname}{Abstract}
\begin{abstract}
Let $X_{1}$ be a projective, smooth and geometrically connected curve over $\mathbb{F}_{q}$ with $q=p^{n}$ elements where $p$ is a prime number, and let $X$ be its base change to an algebraic closure of  $\mathbb{F}_{q}$. We give a formula for the number of irreducible $\ell$-adic local systems ($\ell\neq p$) with a fixed rank over $X$ fixed by the Frobenius endomorphism. We prove that this number behaves like a Lefschetz fixed point formula for a variety over $\mathbb{F}_q$, which generalises a result of Drinfeld in rank $2$ and proves a conjecture of Deligne. To do this, we pass to the automorphic side by Langlands correspondance, then use Arthur's non-invariant trace formula and link this number to the number of $\mathbb{F}_q$-points of the moduli space of stable Higgs bundles.

\end{abstract}

\tableofcontents

\section{Introduction}

\subsection{Principaux résultats}
Soit $X_{1}$ une courbe définie sur un corps fini $\mathbb{F}_{q}$ de cardinal $q$, qui est projective, lisse et géométriquement connexe de genre $g$. Soient $\gls{F}$ une clôture algébrique de $\mathbb{F}_{q}$ et ${X}$ le produit fibré   de $X_{1}$ avec $\mathbb{F}$. L'endomorphisme de $X_{1}$ qui est l'identité sur l'espace topologique sous-jacent, et $f\mapsto f^{q}$ sur le faisceau structural, est un endomorphisme de $\mathbb{F}_{q}$-schéma. Nous noterons $\gls{Fr}$ l'endomorphisme du ${\mathbb{F}}$-schéma X qui s'en déduit par extension des scalaires. C'est l'endomorphisme de Frobenius qu'on utilisera. Fixons un nombre premier $\ell\nmid q$, une clôture algébrique $\bar{\mathbb{Q}}_{\ell}$ de ${\mathbb{Q}}_{\ell}$, et soit $E_n^{(\ell)}$ l'ensemble des classes d'isomorphie de $\bar{\mathbb{Q}}_{\ell}$-systèmes locaux  irréductibles  de rang $n$ sur X.

L'image inverse par $\Fr_{X}$ induit une permutation de $E_n^{(\ell)}$, notée par $\Fr_{X}^{*}$. Drinfeld \cite{Drinfeld} a montré à l'aide de la correspondance de Langlands qu'il a lui-même prouvée et la formule des traces pour $GL_{2}$ établie par Jacquet-Langlands \cite{JL}, que l'on a, en genre $g\geq 2$,  
\begin{equation}\label{points}  |(E_2^{(\ell)})^{\Fr_{X}^{*k}}|=q^{k(4g-3)}+\sum_{i=1}^{u}m_{i}\alpha_{i}^{k},\quad \forall\ k \geq 1, \end{equation}
où les $\alpha_{i}$ sont des $q$-entiers de Weil de poids strictement inférieurs à ${8g-6}$, $m_i$ sont des entiers rationnels et $(E_2^{(\ell)})^{\Fr_{X}^{*k}}$ désigne l'ensemble des éléments de $E_2^{\ell}$ fixés par le puissance $k$$^{i\text{è}me}$ de $\Fr_X^{*}$. Les nombres $m_i$ et $\alpha_i$ sont  indépendants de $\ell$.
Rappelons qu'un $q$-entier de Weil $\alpha$ de poids $i\in \mathbb{N}$ est un entier algébrique tel que pour tout plongement $\varsigma: {\mathbb{Q}}(\alpha)\rightarrow \mathbb{C}$, la valeur absolue $|\varsigma(\alpha)|$ est égale à $q^{\frac{i}{2}}$. 
Par la formule des traces de Grothendieck-Lefschetz, le théorème d'intégralité et le théorème de pureté de Deligne (cf. théorème 5.2.2 \cite{Deligne1} et théorème 1 \cite{Weil2}), pour toute variété $V$ (un schéma séparé de type fini) sur $\mathbb{F}_q$, il existe des $q$-entiers de Weil $\alpha_i$ (unique à permutation près) qui sont les valeurs propres de l'endomorphisme de Frobenius   agissant  sur les groupes de cohomologies $\ell$-adiques à support compact et des entiers $m_i$ qui sont les multiplicités à un signe près de $\alpha_i$ tels que $$|V(\mathbb{F}_{q^k})|=\sum_{i}m_i\alpha_i^k, \quad \forall\ k\geq 1. $$
Ainsi, l'expression \eqref{points} issue du résultat de Drinfeld ressemble aux nombres de $\mathbb{F}_{q^k}$-points d'une variété. Nous renvoyons à \cite{Deligne} pour plus de discussion.

Dans l'article \cite{Kon}  de Kontsevich et \cite{Deligne} de Deligne, les auteurs ont fait des observations sur  le résultat de Drinfeld et ont suggéré de le généraliser aux cas de rang supérieur. En particulier,  Deligne a donné des conjectures précises (cf. conjecture 2.15,  conjecture 2.18, conjecture 6.3 et conjecture 6.7 de \cite{Deligne}). 
Concernant ces conjectures, dans l'article \cite{Deligne-Flicker}, Deligne et Flicker se consacrent au cas où toutes les monodromies locales sont unipotentes avec un seul bloc de Jordan et ces monodromies sont fixées en un ensemble fini $S$ de points de $X$ au-dessus d'au moins $2$ places de $X_1$. 
En rang $2$, Flicker \cite{Flicker2} a calculé le nombre désiré quand la monodromie   apparaît  en une place de $X_1$ avec un seul bloc de Jordan. Arinkin a obtenu aussi des résultats intéressants quand les monodromies sont des sommes directes de caractères modérés de position générale (cf. section 3 \cite{Deligne}). 
Pour tous ces cas, des résultats similaires à ceux de Drinfeld sont montrés. Mais il manque encore une généralisation fidèle du résultat de Drinfeld quand le rang est supérieur à  $2$.  
C'est ce que nous nous proposons de faire dans cet article.

Avant d'énoncer nos résultats, nous introduisons encore quelques notations.
L'espace $\mathrm{H}^1(X, \mathbb{Q}_{\ell})$ muni de la dualité de Poincaré est un espace symplectique. L'action de $\Fr_X$ sur $\mathrm{H}^1(X, \mathbb{Q}_{\ell})$ définit alors une classe de conjugaison semi-simple dans $GSp_{2g}(\bar{\mathbb{Q}}_\ell)$. Soient $\sigma_{1},\ldots, \sigma_{2g }$ les $q$-entiers de Weil de la courbe $X_{1}$, c'est-à-dire les valeurs propres de $\Fr_X$. On peut les indexer de telle façon que $\sigma_{i}\sigma_{i+g}=q$ pour tout $g \geq i\geq 1$. Soit  $\mathbb{Z}[t^{\pm1}, z_{1}^{\pm1}, \ldots, z_{g}^{\pm1}]^{\mathcal{S}_g\ltimes(\mathcal{S}_2)^{g}}$ l'anneau des polynômes de Laurent qui sont symétriques en les $z_i$ et invariants par les substitutions $z_i\mapsto tz_i^{-1}$ $(1\leq i\leq g)$. Soit $\gls{ZSg}$ (cf.  \ref{641} pour plus de détails) le sous-anneau formé des polynômes $P$ tels que chaque monôme $t^{m}z_1^{n_1}\cdots z_g^{n_g}$ qui  apparaît dans $P$ vérifie \begin{equation*}\label{posit} m+\sum_{i=1}^g \min\{n_i, 0\}\geq 0 . \end{equation*}
Notons que cette dernière condition est suffisante et nécessaire pour que les nombres de Weil apparaissant dans l'expression de $|(E_{n}^{(\ell)})^{\Fr_{X}^{*k} }|$ du théorème \ref{P} soient entiers algébriques (cf.  \ref{641}).

Voici le principal théorème: 

\begin{thm}\label{P}\label{Final}
\begin{enumerate}
\item
Pour tous entiers $g\geq 2$ et $n\geq 1$, il existe un unique élément $$P_{g, n}(t, z_{1},\ldots, z_{g})\in \mathbb{Z}[z_{1}, \ldots, z_{g}, tz_1^{-1}, \ldots, tz_g^{-1}]^{\mathcal{S}_g\ltimes(\mathcal{S}_2)^{g}}$$  tel que  pour tout $k\geq 1$, tout corps fini $\mathbb{F}_q$ de cardinal $q$, toute courbe $X_1$ projective lisse et géométriquement connexe sur $\mathbb{F}_q$ de genre $g$ et tout nombre premier $\ell\nmid q$, on ait
 $$|(E_{n}^{(\ell)})^{\Fr_{X}^{*k} }|=P_{g, n}(q^k,\sigma_{1}^k, \ldots, \sigma_{g}^k).$$

De plus, si on pose $\deg t=2$ et $\deg z_i=1$, le terme de poids dominant de $P_{g, n}$ est $t^{(g-1)n^2+1 }$, c'est-à-dire $\deg P_{g, n} =2((g-1)n^2+1)$ et $\deg(P_{g, n} -t^{(g-1)n^2+1 })< 2((g-1)n^2+1)$. 

\item
Il existe un élément \[Q_{g,n}(t, z_1, \cdots, z_g)\in \mathbb{Z}[z_{1}, \ldots, z_{g}, tz_1^{-1}, \ldots, tz_g^{-1}]^{\mathcal{S}_g\ltimes(\mathcal{S}_2)^{g}}\] tel que 
$$P_{g,n}(t, z_1, \ldots, z_g)=Q_{g,n}(t,z_1,\ldots, z_g)\prod_{i=1}^{g}\left((1-z_i)(1-tz_i^{-1})\right).$$
En particulier, pour tout $k\geq 1$, $|\mathrm{Pic}_{X_1}^0(\mathbb{F}_{q^k})|$ divise $|(E_n^{(\ell)})^{\Fr_X^{*k}}|$.
De plus, quand $g\geq 2$, on a $$Q_{g,n}(1,1,\ldots,1)=\sum_{l\mid n}\mu(l)\mu(n/l)l^{2g-3}, $$
où $\gls{mu}$ est la fonction de M\"obius. 
\end{enumerate}

\end{thm}
 
\begin{rmk}
\textup{
 Par la théorie du corps de classes, on a $\forall\ k\geq 1$,  $|(E_{1}^{(\ell)})^{\Fr_{X}^{*k} }|=| \mathrm{Pic}_{X_1}^0(\mathbb{F}_{q^k})|, $  
de sorte qu'on a $Q_{g,1}=1$. Comme le groupe fondamental étale de l'espace projectif sur $\mathbb{F}$ est trivial et celui d'une courbe elliptique sur $\mathbb{F}$ est abélien,   on a aussi $Q_{g,n} = 0$ pour $n\geq 2$ et $g = 0, 1$.} 
\end{rmk}

Ce théorème résout positivement la conjecture 2.15  de Deligne dans \cite{Deligne}, et la conjecture 6.3 de $loc.$ $cit.$ dans le cas partout non-ramifié. Pour une interprétation conjecturale de la valeur de $Q_{g,n}(1, . . . , 1)$ on renvoie à la conjecture 6.7 de $loc.$ $cit.$.

\subsection{Approche utilisée}\label{automorphe}
Soit $X_k:=X_1\otimes_{\mathbb{F}_q} \mathbb{F}_{q^k}$. 
Soit $\glslink{AAAk}{\AAA_k}$ l'anneau des adèles de $F_k := \mathbb{F}_{q^k}(X_k)$ et soit $C_n(X_k)$ le nombre de classes d'équivalence inertielle (cf.  définition \ref{inert}) de représentations automorphes irréductibles, absolument cuspidales et partout non-ramifiées de $GL_n(\AAA_k)$. Une représentation automorphe irréductible et cuspidale $\pi$ de $GL_n(\AAA_k)$ est appelée absolument cuspidale si elle reste irréductible cuspidale après tout changement de base, au sens de la fonctorialité de Langlands, du corps de base $F_k$  à $F_{kd}$,  $\forall d\geq 1$.

On montre dans le corollaire \ref{PASS} et la proposition \ref{GEO}  à l'aide de la correspondance de Langlands qu'on a \begin{equation}\gls{CN}= |(E_n^{(\ell)})^{\Fr_{X}^{*k}}|.\end{equation} 
En particulier, $|(E_n^{(\ell)})^{\Fr_{X}^{*k}}|$ est finie, puisque par un  théorème de Harder (\cite[1.2.1]{Harder}), les supports de toutes les fonctions automorphes cuspidales partout non ramifiées sont contenus dans un sous-ensemble commun de $G(F)\backslash GL_n(\AAA)$ qui est compact modulo le centre.

Posons $\glslink{FF}{F}=F_1$ et $\glslink{AAAk}{\AAA}=\AAA_1$. 
Soit $\glslink{O}{\mathcal{O}}$ le sous-anneau égal au produit des anneaux locaux $\glslink{O}{\mathcal{O}_{x}}$ qui sont les complétés des anneaux locaux du faisceau structural de $X_1$ pour tous $x\in |X_1|$. 
Pour $e\in\mathbb{Z}$, soit $\glslink{GLnAAA}{GL_n(\AAA)^{e}}$ le sous-ensemble de $GL_n(\AAA)$ des éléments dont le degré du déterminant est $e$.
Soit $\mathbbm{1}_{GL_n(\mathcal{O})}$ la fonction caractéristique de ${GL_n(\mathcal{O})}$ et soit $\mathbf{R}$ l'opérateur qui agit sur l'espace
\begin{equation}\label{L2e}L^2(GL_n(F)\backslash GL_n (\AAA)^e)\end{equation}
par la convolution à droite par la fonction $\mathbbm{1}_{GL_n(\mathcal{O})}$. Cet opérateur n'est pas traçable si $n\geq 2$. Cependant Arthur a défini une trace tronquée que l'on note $J_e$ dans la suite. Il en donne une première expression reposant sur la décomposition spectrale de Langlands de l'espace (\ref{L2e}). Cette décomposition exprime tout en terme de briques élémentaires qui sont les représentations automorphes cuspidales des groupes généraux linéaires de rang $m\leq n$. De même, dans l'article, nous calculons toute la trace tronquée $J_e$ en terme des quantités $C_m(X_k)$ ($m\leq n$). 

Pour énoncer notre théorème, nous introduisons la série formelle pour $s\in \mathbb{C}$ 
 $$\mathrm{aut}_{X_{1}}(z)^s=\mathrm{exp}(s\sum_{m\geq 1}\sum_{k\geq 1}\frac{mC_{m}(X_{k})}{k}z^{mk})\  . $$
Plus généralement, pour $l\geq 1$,  on définit $\mathrm{aut}_{X_l}(z)^s$ en  remplaçant $X_k$ par $X_{kl}$ dans la formule ci-dessus.
Pour tout $v\geq 1$, et $f(z)\in \mathbb{C}[[z]]$, soit $\gls{zv}$ le coefficient de $f(z)$ de degré $v$.  

\begin{thm}\label{b}
Supposons $g\neq 1$. Soit   $e\in \mathbb{Z}$ tel que $(e,n)=1$, on a
\begin{equation}\label{J_e}
J_{e}= \sum_{l| n}  \frac{ {\mu(l)} }{n(2g-2)}\sum_{\lambda\vdash n}\frac{1}{\sum_{{j}}a_{j}}  \prod_{j\geq 1 } [z^{a_{j}}]\mathrm{aut}_{X_{l}} (z^{l})^{\frac{(2g-2)S_{j}(\lambda)}{l}}  ,       \end{equation}
où $\sum_{\lambda\vdash n}$ porte sur les partitions non-ordonnées $\lambda=(1^{a_1}, 2^{a_2}, \ldots)$ de $n$ et  $S_{j}(\lambda):=  \sum_{\nu\geq 1}  a_{\nu}  \mathrm{min}\{\nu, j\}$. 
\end{thm}
\begin{rmk}
Dans le produit sur $j$, seul un nombre fini de facteur est distinct de 1. 
En outre, si $l\nmid a_j$ pour un certain $j$, on a 
$[z^{a_j}]\mathrm{aut}_{X_{l}} (z^{l})^{\frac{(2g-2)S_{j}(\lambda)}{l}}    =0$. 
Dans (\ref{J_e}), les termes ne contiennent que des 
expressions $C_m(X_k)$ pour $mk\leq n$. Le terme $C_n(X_k)$ n'apparaît que si $k = 1$ et il n'apparaît que pour $l = 1$ et la partition $(1^n)$. En conséquence, (\ref{J_e}) donne une formule de récurrence pour calculer $C_n(X_1)$ en terme de $J_e$ et des $C_m(X_k)$ pour $m < n$ et $k\leq \frac{n}{m}$.
\end{rmk}

Pour que le théorème ci-dessus soit utile, il nous faut disposer d'une autre expression pour $J_e$. Celle-ci va être fournie par le développement géométrique de la trace tronquée $J_e$. On montre le résultat suivant :
\begin{thm}[cf. théorème \ref{Maing}]\label{a}
Soit $e\in \mathbb{Z}$, on a $$J^{ }_e=\mathcal{P}_n^e(X_1),$$
où $\gls{Pne}$ est le nombre de classes d'isomorphie de fibrés vectoriels isoclines de degré $e$ et de rang $n$ (pour la notion d'isocline voir la définition \ref{isocline}).
\end{thm}

\begin{ex}\textup{
En conséquence, on a par exemple \begin{equation*}
\mathcal{P}_{1}^{1}(X_1)=C_1(X_1),
\end{equation*}
\begin{equation*} \mathcal{P}_{2}^{1}(X_1)=C_2(X_1) +(g-1)C_1(X_1)^2+C_1(X_1),
\end{equation*}
et \begin{multline*}\mathcal{P}_{3}^{1}(X_1)=C_3(X_1)+4(g-1)C_1(X_1)C_2(X_1)+ \\ (g-1) C_1(X_1)C_1(X_2)+2(g-1)^2 C_1(X_1)^3+2(g-1)C_1(X_1)^2+C_1(X_1). \end{multline*}}
\end{ex}

\sloppy
Quand on a $(n,e)=1$, le nombre $J_e$ est le nombre de classes d'isomorphie de fibrés vectoriels géométriquement indécomposables de degré $e$ et de rang $n$ sur $X_1$. En utilisant des résultats de Mellit (\cite{Mellit}) et Schiffmann (\cite{Schiffmann})  on sait que ce nombre est donné par un polynôme dans $\mathbb{Z}[z_{1}, \ldots, z_{g}, tz_1^{-1}, \ldots, tz_g^{-1}]^{\mathcal{S}_g\ltimes(\mathcal{S}_2)^{g}}$. En résolvant la récurrence ci-dessus, on montre l'assertion 1 du théorème \ref{Final} pour un polynôme à coefficient rationnel. Avec un peu plus de travail on montre qu'il est bien dans $\mathbb{Z}[z_{1}, \ldots, z_{g}, tz_1^{-1}, \ldots, tz_g^{-1}]^{\mathcal{S}_g\ltimes(\mathcal{S}_2)^{g}}$. Notons que notre méthode donne une nouvelle démonstration du résultat suivant (dû à Mellit) :

\begin{thm}[Groechenig, Wyss, Ziegler (quand $(n,e)=1$) \cite{GWZ}; Mellit \cite{Mellit}]
Le nombre $\mathcal{P}_n^e(X_1)$ ne dépend que de l'ordre de $e$ dans $\mathbb{Z}/n\mathbb{Z}$. En particulier, quand $(n,e)=1$, le nombre de classes d'isomorphie de fibrés vectoriels géométriquement indécomposables de degré $e$ et de rang $n$ sur $X_1$ est indépendant de $e$. 
\end{thm}
Combinant avec les études sur la formule des traces d'algèbres de Lie de Chaudouard sur un corps de fonctions (\cite{Chau}), nous donnons aussi une nouvelle preuve du fait (dû à Schiffmann) que sous la même condition $(n,e)=1$ on a
 $$A_{n,e}(X_1)=q^{-n^2(g-1)-1}|\mathrm{Higgs}_{n,e}^{st}(X_1)(\mathbb{F}_q)| . $$
où  $\gls{Ane}$ est le nombre des fibrés vectoriels géométriquement indécomposables de degré $e$ et de rang $n$ sur $X_1$, et $\mathrm{Higgs}_{n,e}^{st}(X_1)$ est le schéma de modules des fibrés de Higgs stables de rang $n$ et de degré $e$ (cf.   \ref{Higgs}).

\subsection{Remerciements }
Ce travail est le produit de mon projet de thèse préparée à l'Université Denis-Diderot, Paris 7. 
Il s'est accompli sous la direction de Pierre-Henri Chaudouard. Je voudrais le remercier profondément pour les nombreuses discussions, ainsi que pour ses encouragements, et pour son aide généreuse à de multiples reprise.

Je remercie Pierre Deligne pour ses commentaires très utiles et pour m'avoir signalé une erreur dans une preuve. Je remercie également Anton Mellit pour m'avoir communiqué un résultat sur le nombre de fibrés de Higgs qui me permet de corriger l'erreur. Je remercie Bertrand Lemaire et Wee Teck Gan pour leurs pré-rapports sur ma thèse. Je remercie les rapporteurs anonymes pour les suggestions très utiles sur la rédaction et pour m'avoir aidé à corriger des fautes de français.

\section{Réduction du problème}\label{auto-gal}
Dans ce chapitre, on explique comment ramener via la correspondance de Langlands le comptage de systèmes locaux de rang $n$ à un comptage de certaines représentations automorphes du groupe $GL_n$ (cf. le corollaire \ref{PASS}).

\subsection{Systèmes locaux \texorpdfstring{$\ell$}{l}-adiques et représentations $\ell$-adiques des groupes de Weil}\label{SYSTEM}
\label{syslo}

\subsubsection{} 

Dans l'introduction, on a introduit une courbe projective, lisse et géométriquement irréductible $X_1$ défini sur un corps fini $\mathbb{F}_q$ et son corps de fonctions $F=F_1$. On a aussi défini $X_d=X_1\otimes\mathbb{F}_{q^d}$ pour $d\in \mathbb{N}^*$ et $X$ comme le produit fibré de $X_1$ avec la clôture algébrique $\mathbb{F}$ de $\mathbb{F}_q$.

\subsubsection{}
On suit l'article \cite{Deligne-Flicker} pour la présentation.  
Soit $\bar{F}$ une clôture  séparable  de $F$. On fixe un point géométrique $o: \mathrm{spec}(\bar{F})\rightarrow X$ au-dessus du point générique de $X$. 
Par abus de notation, on notera encore $o$ tout le composé   $\mathrm{spec}(\bar{F})\rightarrow X\rightarrow X_{d}$ pour $d\geq 1$.

Soit $\gls{pi1Xo}$ (resp. $\pi_{1}(X_{1},o)$) le groupe fondamental étale  de $X$ (resp.  $X_{1}$) en $o$. On a une suite exacte courte $$ 0\longrightarrow \pi_{1}(X, o)\longrightarrow \pi_{1}(X_{1},o)\longrightarrow \Gal(\mathbb{F}|\mathbb{F}_{q})\longrightarrow 0 .$$
Soit $W(\mathbb{F}|\mathbb{F}_{q})$ le sous-groupe de $\Gal(\mathbb{F}|\mathbb{F}_{q})$ engendré algébriquement par l'élément de Frobenius arithmétique. Ce groupe est isomorphe à $\mathbb{Z}$ par l'isomorphisme qui envoie l'inverse de l'élément de Frobenius arithmétique, appelé élément de Frobenius géométrique, sur $1$. 
L'image inverse de $W(\mathbb{F}|\mathbb{F}_{q})$ par le morphisme $\pi_{1}(X_{1},o)\longrightarrow \Gal(\mathbb{F}|\mathbb{F}_{q})$ est appelé le groupe de Weil, qu'on note $\gls{WX1o}$. On a $W(X_{1},o)\cong\pi_1(X, o)\rtimes \mathbb{Z}$. On munit $W(X_{1},o)$ de la topologie produit de $\pi_1(X, o)$ et $\mathbb{Z}$,  où la topologie sur $\mathbb{Z}$ est la topologie discrète. 

Une représentation $\ell$-adique de rang $n\geq 1$ d'un groupe localement profini $H$ est un $\bar{\mathbb{Q}}_{\ell}$-espace linéaire de dimension $n\geq 1$ avec une action linéaire continue de $H$. Un système local $\ell$-adique de rang $n$ sur  $X_1$ (resp. $X$)  est un faisceau lisse $\ell$-adique de rang $n$ sur $X_1$ (resp. $X$). On identifie les systèmes locaux $\ell$-adiques sur $X_1$ (resp. $X$) et les représentations $\ell$-adiques de $\pi_1(X_1, o)$ (resp. $\pi_1(X, o)$).

Nous noterons $\gls{GnX}$ l'ensemble des classes d'isomorphie des représentations irréductibles $\ell$-adiques de rang $n$ de $W(X_{1},o)$.
Un caractère $\ell$-adique est une représentation $\ell$-adique de dimension $1$. 
On dit qu'un caractère $\ell$-adique de $W(X_{1}, o)$ est inertiel s'il se factorise à travers $W(\mathbb{F}|\mathbb{F}_{q})$. Les caractères $\ell$-adiques inertiels forment un groupe abélien. 
Pour une représentation $\sigma\in \mathcal{G}_{n}(X_{1})$, on note $\Fix(\sigma)$ le sous-groupe formé par des caractères $\ell$-adiques inertiels ${\lambda}$ tels que $$\sigma\otimes \lambda\cong \sigma.$$
Le groupe $\Fix(\sigma)$ est cyclique d'ordre divisant $n$. On définit une relation d'équivalence  sur $\mathcal{G}_{n}(X_{1})$:

\begin{definition}\label{OK!} Pour $\sigma_{1}, \sigma_{2} \in \mathcal{G}_{n}(X_{1})$, on dit que $\sigma_{1}$ et $\sigma_{2}$ sont inertiellement équivalentes s'il existe un caractère $\ell$-adique inertiel  $\lambda$  tel  que $\sigma_{1}\otimes \lambda\cong \sigma_{2}$.

On dit qu'une représentation $\ell$-adique de $W(X_{1}, o)$ est absolument irréductible si sa restriction à $\pi_{1}(X,o)$ est irréductible. 
\end{definition}

On fixe une section $j: \mathbb{Z}\rightarrow W(X_{1}, o)$. Soit $\sigma$ une représentation $\ell$-adique de $\pi_1(X,o)$. On définit une représentation $\sigma^{j(1)}$ qui agit sur l'espace sous-jacent de $\sigma$ par $$\sigma^{j(1)}(g)=\sigma( j(1)g j(1)^{-1})\quad \forall g\in \pi_1(X,o), $$ 
clairement la classe d'isomorphie de $\sigma^{j(1)}$ ne dépend pas de la section choisie. La conjugaison par $j(1)$ sur l'ensemble des classes d'isomorphie de représentations $\ell$-adiques de $\pi_{1}(X,o)$ coïncide avec l'action de $\Fr_{X}^{*}$ sur $E_n^{(\ell)}$, l'ensemble des systèmes locaux $\ell$-adiques sur $X$ (Lemma 1.3 et le paragraphe avant la section 1.5 de \cite{Deligne-Flicker}). Par Lemma 1.6 et Lemma 1.9 (ii) de $loc.$ $cit.$, on a le lemme suivant.

\begin{lemm}\label{Descente} %\label{Ok}
Soit $\sigma$ une représentation $\ell$-adique irréductible de dimension $n$ de $\pi_{1}(X, o)$. Elle peut s'étendre à une représentation de $W(X_{1},o)$ si et seulement si la représentation  $\sigma^{j(1)}$ est isomorphe à $\sigma$. Toutes les extensions sont inertiellement équivalentes. 

Le foncteur de restriction de la catégorie des représentations $\ell$-adiques de $W(X_1, o)$ vers la catégorie des représentations $\ell$-adiques de  $\pi_1(X,o)$ induit  une bijection entre l'ensemble des classes d'équivalence inertielle des représentations absolument irréductibles  de  rang $n$ de $W(X_1, o)$ et l'ensemble $(E_n^{(\ell)})^{\Fr_X^{*}}$.
\end{lemm}

\begin{prop}\label{GEO}\sloppy
Soit $\sigma: W(X_{1},o)\rightarrow GL_{n}(\bar{\mathbb{Q}}_{\ell})$  une représentation $\ell$-adique de dimension $n$ et $d\mid n$. Alors on a l'équivalence entre:
\begin{enumerate}
\item $\sigma$ est irréductible et $|\Fix(\sigma)|=d$.
\item Il existe une représentation $\ell$-adique absolument irréductible $\sigma'$ de $W(X_{d},o )$ de rang $n/d$ telle que %$|\Fix(\sigma')|=1$ et 
$\sigma\cong\Ind_{W(X_{d},o )}^{W(X_{1},o )}(\sigma')$. De plus, les conjugués de $\sigma'|_{\pi_{1}(X, o)}$ par $j(0), j(1), \ldots, j(d-1)$ sont  deux à deux non-isomorphes.   
\end{enumerate}
Les représentations $\sigma'$ satisfaisant  la condition 2 se déduisent les unes des autres par conjugaison par $j(i)$, $i\in \mathbb{Z}$. Il y en a exactement $d$. 

En particulier, $\sigma$ est absolument irréductible si et seulement si $|\Fix(\sigma)|=1$.  
\end{prop}
\begin{proof}
On démontre tout d'abord qu'une représentation $\ell$-adique irréductible $\sigma$ de $W(X_1, o)$ est absolument irréductible si $|\Fix(\sigma)|=1$. 

Comme $\sigma$ est irréductible, pour tout  $d\geq 1$,  $\sigma|_{W(X_{d}, o)}$ est semi-simple, donc on peut supposer 
$$ \sigma|_{W(X_{d}  , o)}= \bigoplus \sigma_{i}\ , $$
où les représentations $\sigma_{i}$ de $W(X_{d}, o)$ sont irréductibles. Par la réciprocité de Frobenius, $\sigma$ est un quotient irréductible de $\Ind^{W(X_{1},o)}_{W(X_{d}, o)}(\sigma_{i})$, représentation induite  par $\sigma_{i}$. 
De même pour tout caractère inertiel $\chi$ trivial sur $W(X_d,o)$: $\sigma\otimes \chi$ est un quotient irréductible de $\Ind^{W(X_{1},o)}_{W(X_{d}, o)}(\sigma_{i})$. Par l'hypothèse $\Ind^{W(X_{1},o)}_{W(X_{d}, o)}(\sigma_{i})$ admet $d$ quotients non-isomorphes de dimension $\dim \sigma$. Donc la dimension de $\Ind^{W(X_{1},o)}_{W(X_{d}, o)}(\sigma_{i})$, qui est égale à $d\dim \sigma_{i}$, est supérieure  à $d \dim \sigma$. Ce n'est possible que si $\sigma_{i}=\sigma$.

Il suffit de montrer que si $\sigma|_{W(X_{d}, o)}$ est irréductible pour tout $d\geq 1$ alors $\sigma$ est absolument irréductible.

Soit $G_{g\text{é}om}$ le groupe de monodromie géométrique, c'est-à-dire la clôture de Zariski de $\sigma(\pi_1(X,o))$ dans $GL_{n}(\bar{\mathbb{Q}}_{\ell})$. Comme $\sigma$ est irréductible, sa restriction à $\pi_1(X, o)$ est semi-simple.  Cela montre que $G_{g\text{é}om}$ est un groupe réductif défini sur $\bar{\mathbb{Q}}_{\ell}$. De plus, l'élément $\sigma(j(1))$ normalise $G_{g\text{é}om}$. On utilise un théorème de Grothendieck (cf. théorème 3.3(2) de \cite{KW}) selon lequel il existe un élément $u\in G_{g\text{é}om}(\bar{\mathbb{Q}}_{\ell})$ et un entier $d_{0}\geq 1$ tel que $u\sigma(j(d_0))$ commute avec $G_{g\text{é}om}$ et $\sigma(j(1))$.

Supposons que $\sigma$ est $W(X_{d_{0}},o)$-irréductible. L'action de $u\sigma(j(d_0))$ sur $\bar{\mathbb{Q}}_{\ell}^{n}$ commute avec $W(X_{d_{0}},o)$ donc  induit un $W(X_{d_{0}},o)$-isomorphisme de $\sigma$. Par le lemme de Schur, $u\sigma(j(d_0))$ est un scalaire. Cela implique alors que tout sous-espace de $\sigma$ qui est $G_{g\text{é}om}$ stable est  aussi $W(X_{d_{0}}, o)$ stable. Donc la restriction de $\sigma$ à $\pi_{1}(X, o)$ est irréductible.

$(1.\implies 2.) $ Soit $\sigma$ une représentation $\ell$-adique de $W(X_1, o)$ irréductible tel que $|\Fix(\sigma)|=d$. Soit $\chi$ un caractère inertiel d'ordre $d$ tel qu'il existe une matrice $\Psi\in GL_n(\bar{\mathbb{Q}}_{\ell})$, $$ \chi(g)\sigma(g)\Psi=\Psi\sigma(g)\quad \forall\ g\in W(X_{1},o).$$
Les espaces propres de $\Psi$ sont $W(X_d, o)$-invariant et
leur somme est $W(X_1, o)$-invariante donc est égale à l'espace sous-jacent de $\sigma$. Les valeurs propres de $\Psi$ sont permutées par le groupe des racines $d$-ièmes d'unité de $\bar{\mathbb{Q}}_{\ell}$ donc il y en a au moins $d$. On sait alors que chaque espace propre est irréductible de dimension $n/d$, et $\sigma\cong\Ind_{W(X_{d},o)}^{W(X_{1},o)}(\sigma') $  où $\sigma'$ est l'un des espaces propres. 
L'irréductibilité de $\sigma$ implique celle de $\sigma'$. 
Clairement $|\Fix(\sigma)|=d|\Fix(\sigma')|$, donc $|\Fix(\sigma')|=1$.

\sloppy
Si pour un entier $1\leq d_0\leq d-1$, ${\sigma^{\prime j(d_0)}|_{\pi_1(X,o)}}\cong \sigma'|_{\pi_1(X,o)}$, alors par le lemme \ref{Descente} on peut supposer qu'il existe un caractère inertiel 
 $\chi_0$ de $W(X_d, o)$ tel que $$\sigma'^{j(d_0)}\cong \sigma'\otimes\chi_0 .$$
Par le théorème de Clifford, l'irréductibilité de $\sigma$ implique que $\sigma', \sigma'^{j(1)}, \ldots, \sigma'^{j(d-1)}$  sont deux à deux différentes donc $\chi_0$ est non-trivial.
\color{black}
On le prolonge en un caractère $\tilde{\chi}$ de $W(X_{1}, o)$ qui est d'ordre strictement plus grand que $d$ et similairement  $$\sigma\otimes \tilde{\chi}\cong \Ind_{W(X_{d},o)}^{W(X_{1},o)}(\sigma'\otimes \chi_{0})\cong \Ind_{W(X_{d},o)}^{W(X_{1},o)}(\sigma'^{j(d_0)})\cong\sigma .$$ 
Cela implique $\tilde{\chi}\in \Fix(\sigma)$ donc d'ordre inférieur à $d$ mais c'est une contradiction. Donc les conjugués de $\sigma'|_{\pi_1(X,o)}$ par $j(i)$, $i=0, 1, \ldots, d-1$,   sont deux à deux non-isomorphes.

$(2.\implies 1.)$
Par la théorie de Clifford, $\sigma$ est irréductible. De plus, $|\Fix(\sigma)|=d|\Fix(\sigma')|=d$.

Par la formule de Mackey, les représentations $\sigma'$ satisfaisant  la condition 2 se déduisent des unes des autres par conjugaison par $j(i)$, $i\in \mathbb{Z}$. 
\end{proof}

\subsection{Notations pour les formes automorphes}

\subsubsection{}

	Les valuations discrètes normalisées de $F$ s'identifient aux points fermés de $X_{1}$. Soit $\gls{X1p}$ leur ensemble. %Soit $|\cdot|_{\AAA^{\times}}$ la ``valeur absolue" adélique normalisée de $\mathbb{A}^{\times}$. 
	Pour $x\in |X_1|$, soit $\mathcal{O}_x$ le complété d'anneau local du faisceau structural de $X_1$ en $x$ et $F_x$ son corps des fractions.  On définit pour $g=(g_x)_{x\in |X_{1}|} \in \AAA^{\times}$: $$\gls{deg} g=-\sum_{x\in|X_1|}[\kappa(x):\mathbb{F}_{q}] x(g_x)$$ où  $\kappa(x)$ est le corps résiduel de $\mathcal{O}_{x}$ et $x(g_x)$ est la valuation normalisée de $g_x$.

\subsubsection{Groupes algébriques}\label{partition}

Soit $\gls{G}=GL_n$ le groupe linéaire général de rang $n$ défini sur $\mathbb{Z}$. Si le contexte indique clairement quel est le rang de $G_{n}$, nous le désignerons par $G$ simplement. Soit $\gls{B}$ le sous-groupe de Borel des matrices triangulaires supérieures défini sur $\mathbb{Z}$. Soit $T$ le sous-tore maximal déployé sur $\mathbb{Z}$ de $G_{n}$ formé des matrices diagonales.

 On désigne par $\gls{P}$ l'ensemble des sous-groupes paraboliques semi-standards (c'est-à-dire ceux qui contiennent $T$) définis sur $F$ de $G$, et $\gls{PB}\subseteq \mathcal{P}$ le sous-ensemble des sous-groupes paraboliques standards.  Pour chaque $P\in\mathcal{P}$, soit $N_{P}$ le radical unipotent de $P$. Il existe un unique sous-groupe de Levi $M_{P}$ de $P$ qui contient $T$. Soit $M$ un sous-groupe de Levi (d'un sous-groupe parabolique)  de $G$ défini sur $F$, on désigne par $\gls{LM}$ l'ensemble des sous-groupes de Levi (des sous-groupe paraboliques de $G$)   contenant $M$. 
Pour deux sous-groupes de Levi $M\subseteq L$ semi-standards, on désigne par $\glslink{PLM}{\mathcal{P}^L(M)}$ l'ensemble des sous-groupes paraboliques $Q$ de $L$ tels que $M_{Q}=M$. Quand $L={G}$ on le note  simplement $\glslink{PLM}{\mathcal{P}(M)}$.

\subsubsection{Groupes topologiques}

Pour tout $e\in\mathbb{Z}$, soit  
 $$\gls{GAe}=\{ g\in G(\mathbb{A}) | \deg \det g =e    \}.$$

Soit $M$ un sous-groupe de Levi de $G$ sur $F$, on a $M\cong G_{n_1}\times\cdots\times G_{n_r}$ avec $\sum_{i}n_i=n$. Soit $\gls{ZM}$ le centre de $M$ et $\gls{XM}$ le groupe des caractères rationnels de $M$ définis sur $F$. On a $Z_M\cong \mathbb{G}_m^r$ et on a l'identification $X^{*}(M)\cong \mathbb{Z}^r$ donnée par le caractère déterminant de chaque facteur $G_{n_i}$.  
Soit 
         $$M(\AAA)^0=\bigcap_{\chi \in X^*(M)} \ker \deg\circ\chi.$$
Alors $M(\AAA)^{0}$ est un sous-groupe distingué de $M(\AAA)$ qui s'identifie à $G_{n_1}(\AAA)^{0}\times\cdots\times G_{n_r}(\AAA)^{0}$. 
   
On fixe un idèle $a$ de degré $1$.   
Soit  $\glslink{XiM}{\Xi_M}$ le sous-groupe discret  de $Z_{M}(\mathbb{A})$ dont les éléments sont de la forme $(a^{j_{1}},\ldots, a^{j_{r}})$ avec $j_{1}, \ldots, j_{r}\in \mathbb{Z}$, où $a$ est vue comme une matrice scalaire dans $G_{n_i}(\AAA)$.

\label{XPG}
Soit $X_M$ le groupe des homomorphismes du groupe $M(\mathbb{A})$ vers $\mathbb{C}^{\times}$ triviaux sur $M(\mathbb{A})^{0}$. Dans le texte, on l'appelle \textit{ le groupe des caractères inertiels} de $M(\AAA)$. 
On a l'identification $X_M\cong (\mathbb{C}^{\times})^r$. 	 
Pour tout sous-groupe de Levi $L$ contenant $M$, soit $X_{M}^{L}\subseteq X_M$ le groupe des homomorphismes triviaux sur $Z_{L}(\mathbb{A})$. On a $Z_{L}(\mathbb{A})M(\mathbb{A})^{0}=\Xi_L M(\mathbb{A})^{0}$. On a  donc 
	  $$\gls{XML}=\Hom(M(\bbb{A})^{0}\backslash M(\bbb{A})/\Xi_{L}, \mathbb{C}^{\times}). $$
Par définition, $X_M^L$ est une variété affine complexe (pas connexe en général).

Soit $\gls{Kx}:=G(\mathcal{O}_{x})$ pour $x\in |X_1|$, et  $\gls{K}=G(\mathcal{O})=\prod_{x\in |X_1|} K_x$.  Le groupe $K$  est un sous-groupe compact maximal de $G(\mathbb{A})$.

\subsubsection{Mesures de Haar}	
Soit $P$ un sous-groupe parabolique sur $F$ avec sous-groupe de Levi $M_P$ et le radical nilpotent $N_P$. 
On fixe des mesures de Haar (unimodulaires) sur les groupes $\Xi_M$, $M_P(F)$, $N_{P}(F)$, $G(F)$, $N_P(\AAA)$, $M_P(\AAA)$ et $G(\AAA)$ de sorte que les conditions suivantes sont satisfaites:\\
-  les mesures sur $\Xi_M, M_P(F), N_{P}(F)$, $G(F)$ soient les mesures de comptage; \\
- $\vol(M_P(\mathcal{O}))=\vol(G(\mathcal{O}))=1$;\\
- $\vol(N_P(F)\backslash N_P(\AAA))=1$.

\subsubsection{Algèbres de Hecke}

Soient $M$ un sous-groupe de Levi de $G$ défini sur $F$ et $x\in |X_1|$.  L'espace $C_c^\infty(M(\mathcal{O}_x)\backslash M(F_x)/M(\mathcal{O}_x))$ des fonctions sur $M(F_x)$ à valeurs dans $\mathbb{C}$ localement constantes, à support compacts, et $M(\mathcal{O}_x)$ bi-invariantes est muni  d'une multiplication définie comme la convolution des fonctions (avec la mesure de Haar locale telle que $\vol(M(\mathcal{O}_x))=1$). La fonction caractéristique de $\mathbbm{1}_{M(\mathcal{O}_x)}$ est l'élément neutre. On l'appelle algèbre de Hecke sphérique locale en $x$, notée par $\mathcal{H}_{M, x}$. Le produit tensoriel restreint des $\gls{HMx}$ est appelé algèbre de Hecke sphérique globale, noté par $\gls{HM}$.

\subsection{Passage du côté automorphe}\label{CLanglands}

\subsubsection{}
On dit qu'une fonction $\varphi$ définie sur $G(F)\backslash G(\AAA)$ à valeurs dans $\mathbb{C}$ est cuspidale partout non-ramifiée si elle est invariante par translation à droite par $K$ et si pour tout sous-groupe parabolique standard $P\in\mathcal{P}(B)$, on a $$\int_{N_P(F)\backslash N_P(\AAA)} \varphi(ng)\d n=0 \quad \forall\ g\in G(\AAA)/K.$$
C'est un théorème de Harder (\cite[1.2.1]{Harder}) que pour toute telle $\varphi$,  son support est compact modulo le centre $Z_G(\AAA)$. 

Soit  $$\theta\in\Hom(Z_{G}(F)\backslash Z_{G}(\mathbb{A})/Z_{G}(\mathcal{O}), \mathbb{C}^{\times})$$ un caractère de $Z_{G}(F)\backslash Z_{G}(\mathbb{A})/Z_{G}(\mathcal{O})$.
Soit ${\mathcal{A}}_{G, \theta, cusp}$ l'espace des fonctions automorphes cuspidales  (à valeur dans $\mathbb{C}$) partout non-ramifiées de caractère central $\theta$, i.e. celles satisfaisant $\varphi(zm)=\theta(z)\varphi(m)$, pour tout $z\in Z_{M_P}(\AAA)$ et $m\in M_P(\AAA)$. 

La convolution à droite fait de l'espace ${\mathcal{A}}_{G, \theta,cusp}$  un $\mathcal{H}_G$-module qui est semi-simple (voir par example \cite[Theorem 9.2.14]{Laumon2}). De plus, le théorème de multiplicité un de Jacquet, Langlands, Piatetski-Shapiro et Shalika (voir \cite[Theorem 5.5]{Shalika})  indique que ces facteurs simples sont deux à deux non-isomorphes.  
Pour tout caractère $\theta$, un sous-$\mathcal{H}_{G}$-module simple de $ {\mathcal{A}}_{G,\theta,cusp}$  est  appelé \textit{une représentation cuspidale automorphe irréductible partout non-ramifiée} de $G(\AAA)$.

Soit $\gls{AnX}$ l'ensemble des représentations cuspidales irréductibles partout non-ramifiées de $G(\AAA)$.

Pour une représentation $\pi\in \mathcal{A}_{n}(X_{1})$, soit $\Fix(\pi)\subseteq X_G$ le sous-groupe formé par des caractères inertiels ${\lambda}$ tels que $$\pi\otimes \lambda= \pi,$$
où $\pi\otimes \lambda$ est le $\mathcal{H}_{G}$-module formé des fonctions $\varphi_\lambda: g\in G(\AAA)\mapsto \varphi(g)\lambda(g)$ avec $\varphi\in \pi$. 
 En regardant l'identité entre le caractère central de $\pi\otimes \lambda$ et celui de $\pi$, on trouve que le groupe $\Fix(\pi)$ est cyclique et son ordre divise $n$.

\begin{remark}\label{multi1}
Comme $\lambda$ est un caractère de $G(\AAA)$, il engendre un $\mathcal{H}_{G}$-module. 
La structure du $\mathcal{H}_G$-module sur l'espace $\pi\otimes\lambda$ est isomorphe au produit tensoriel des deux $\mathcal{H}_{G}$-modules $\pi$ et $\lambda$. 
Si comme $\mathcal{H}_G$-module, $\pi\otimes \lambda$ est isomorphe à $\pi$, les espaces sous-jacents coïncident par le théorème de multiplicité un de Jacquet, Langlands, Piatetski-Shapiro, et Shalika. On peut même vérifier que les  deux structures de $\mathcal{H}_G$-modules $\pi$ et $\pi\otimes \lambda$ sont les mêmes (pas seulement isomorphes). Donc la définition de $\Fix(\pi)$ coïncide avec la définition usuelle (cf. \cite[II.1, p.71]{Wald-Moe}). Néanmoins, il faut comparer avec la remarque  \ref{multi2}.
 \end{remark}

On définit une relation d'équivalence  sur $\mathcal{A}_{n}(X_{1})$:
\begin{definition}\label{inert} Soient $\pi_{1}, \pi_{2} \in \mathcal{A}_{n}(X_{1})$. On dit que $\pi_{1}$ et $\pi_{2}$ sont inertiellement équivalentes s'il existe un caractère inertiel  $\lambda$  tel  que $\pi_{1}\otimes \lambda= \pi_{2}$.
\end{definition}

\subsubsection{}\label{PASSS}
La correspondance de Langlands prouvée par L.Lafforgue  (\cite[Théorème VI.9]{Lafforgue}, voir aussi Section IV.3.5 de \cite{Hen-Le}) implique qu'il existe une bijection canonique (après qu'on fixe un isomorphisme entre $\mathbb{C}$ et $\bar{\mathbb{Q}}_{\ell}$) entre $\mathcal{A}_{n}(X_{1})$ et $\mathcal{G}_{n}(X_{1})$ qui envoie les classes inertielles sur les classes inertielles et qui préserve l'ordre des sous-groupes fixateurs.

 Par le lemme \ref{Descente} et la proposition \ref{GEO}, l'ensemble $(E_n^{(\ell)})^{\Fr_X^{*k}}$ est en bijection avec l'ensemble des classes d'équivalence inertielle des $\sigma\in \mathcal{G}_{n}(X_k)$ telles que $|\Fix(\sigma)|=1$. Donc on a le corollaire suivant. 

\begin{coro}[de la correspondance de Langlands]\label{PASS}
L'ensemble $(E_n^{(\ell)})^{\Fr_X^{*k}}$ est en bijection avec l'ensemble des classes d'équivalence inertielle des $\pi\in \mathcal{A}_{n}(X_k)$ telles que $|\Fix(\pi)|=1$. 
 
\end{coro}

On est donc ramené à compter les classes d'équivalence inertielle des représentations automorphes cuspidales de $G(\mathbb{A}_{F\otimes\mathbb{F}_{q^{k}}})$. 

\subsubsection{Paires discrètes} \label{234}
Cette partie, ainsi que \ref{Pa2} ne seront utilisées que dans la section \ref{S5} et \ref{technique}. 
Le développement spectral de la formule des traces d'Arthur-Lafforgue contient non seulement les termes associés aux représentations automorphes cuspidales de $G$ mais aussi à des représentations automorphes discrètes (représentations automorphes cuspidales et représentations automorphes résiduelles) sur les sous-groupes de Levi de $G$.

Soit $P$ un sous-groupe parabolique de $G$. 
Soit  $\theta$ un caractère de  \[ Z_{M_P}(F) \backslash Z_{M_P}(\mathbb{A})/Z_{M_P}(\mathcal{O}).\]  On considère l'espace des fonctions  $\varphi$ définies sur $M_P(\AAA)$ à valeurs dans $\mathbb{C}$ qui sont  invariantes par translation à droite par $M_P(\mathcal{O})=M_P(\AAA)\cap K$  et  par translation  à gauche par $M_P(F)$, telles que:\\ $(1)$ $\varphi(zm)=\theta(z)\varphi(m)$, $\forall\ z\in Z_{M_P}(\AAA)$ et $m\in M_P(\AAA)$; \\ $(2)$ pour $|\theta|: M_P(F)\backslash M_P(\AAA)\rightarrow \mathbb{R}^\times_+$ l'unique caractère continu  qui prolonge le module de $\theta$,  $$\int_{M_{P}(F)\backslash M_P(\mathbb{A})/Z_{M_P}(\mathbb{A})}\frac{ |\varphi(m)|^2}{|\theta(m)|^2 }\d m<\infty .$$ 
On désigne par $L^2(M_P(F)\backslash M_P(\AAA))_\theta^K$ cet espace. C'est  un $ \mathcal{H}_{M_P}$-module. 

On introduit les représentations automorphes discrètes (irréductibles) partout non-ramifiées ci-dessous.  Pour plus de détails, nous renvoyons le lecteur à \cite[p.25-p.26]{Laumon2}. 

\begin{definition}\label{PD}
Soit $P$ un sous-groupe  parabolique. Une représentation automorphe discrète  (irréductible) partout non-ramifiée $\pi$ de $M_P(\AAA)$ est un sous-$\mathcal{H}_{M_P}$-module simple de $L^2(M_P(F)\backslash M_P(\AAA))_\theta^K$. Le caractère $\theta$ déterminé par $\pi$ est appelé le caractère central de $\pi$. \end{definition}

\begin{rmk}
\textup{1. Soit $\pi$ une représentation automorphe discrète partout non-ramifiée de $G(\AAA)$. Par la décomposition à la Flath, $\pi\cong\otimes_{v\in |X_1|} \pi_v$, où $\pi_v$ est un $\mathcal{H}_{G, v}$-module simple. La commutativité ainsi que la finitude du nombre de générateurs de $\mathcal{H}_{G, v}$ (par l'isomorphisme de Satake) implique que toute $\pi_v$, donc $\pi$, est de dimension $1$. 
\\
2. Tout sous-$\mathcal{H}_{M_P}$-module simple $\pi$ de $L^2(M_P(F)\backslash M_P(\AAA))_\theta^K$ est un facteur simple, parce que le produit scalaire nous permet de définir le supplémentaire orthogonal $\pi^{\perp}$ et $\pi$ est fermé.   
 \\
 3. Par le théorème de Harder qu'on a mentionné (\cite[Theorem 1.2.1]{Harder}) qui implique que toute fonction automorphe cuspidale sur $M_P(\mathbb{A})$ est à support compact modulo le centre $Z_{M_P}(\AAA)$, donc toute représentation automorphe cuspidale est discrète. Mais il existe des représentations discrètes qui ne sont pas cuspidales, par exemple, la représentation engendrée par une fonction constante. 
\\
4. Pour les groupes généraux linéaires, les représentations automorphes discrètes 
sont classées par Moeglin-Waldspurger en termes d'induction parabolique de représentations automorphes cuspidales. Cette classification est un résultat essentiel pour notre calcul.
On rappellera leur théorème plus tard (le théorème \ref{MWres}). }\end{rmk}

Soient $P, P'$ deux sous-groupes paraboliques semi-standards. Soit $W^G$ le groupe de Weyl de $G$ par rapport à $T$. 
Soient $w\in W^G$ tel que $wM_Pw^{-1}=M_{P'}$, et ${\pi} $ une représentation discrète irréductible partout non-ramifiée de $M_P(\AAA)$. On définit $w(\pi)$ par $$w(\pi)=\{\varphi(w^{-1}[\cdot ]w )|\ \varphi\in \pi\}$$
alors $w(\pi)$ est un sous-$\mathcal{H}_{M_{P'}}$-module simple de $L^2(M_{P'}(F)\backslash M_{P'}(\AAA))_{w(\theta)}^K$. Notons que $w(\pi)$  ne dépend que de la double classe représentée par $w$ dans $W^{P'}\backslash W/W^P$ où $\glslink{WM}{W^P}$ est le groupe de Weyl de $M_P$ par rapport à $T$. 

\begin{definition}\label{pairedi}
Une paire discrète partout non-ramifiée $\gls{Ppi}$ de $G$ est la donnée d'un sous-groupe parabolique standard $P\in\mathcal{P}(B)$, et d'une représentation discrète irréductible partout non-ramifiée ${\pi}$ de $M_P(\AAA)$ dont le caractère central $\chi_{{\pi}}$ est  trivial sur $\Xi_{M}$. On dit que deux paires discrètes $(P,{\pi})$ et $(P',{\pi}')$ sont inertiellement  équivalentes s'il existe $w\in W^G$ tel que $wM_Pw^{-1}=M_{P'}$ et $\lambda\in X_{M_P}^{{G}}$ tels que 
$$\pi'= w(\pi\otimes\lambda) .$$
\end{definition}

Soit ${\pi}$ une représentation discrète partout non-ramifiée de $M(\AAA)$. On désigne par $$\gls{Fix}$$ le sous-groupe de $X_M$ des $\lambda\in X_M$ tels que ${\pi}\otimes\lambda={\pi}$. 
En comparant les caractères centraux de ${\pi}$ et ${\pi}\otimes\lambda$, on voit que le groupe $\Fix({\pi})$  est un sous-groupe de $X_M^M$, donc il est fini. 
Soit $(P,{\pi})$ une paire discrète.
On désigne par
$$\gls{stab}$$ l'ensemble des couples $(w,\lambda)$  telles que $w$ appartient dans $W^{P}\backslash W^G/W^{P}$ et $\lambda\in X_P^{G}$ pour lesquels  $wM_P w^{-1}=M_P$ et $w({\pi}\otimes\lambda) = {\pi}$. 
Le cardinal de $\Fix({\pi})$ et celui  de $\stab(P, {\pi})$ ne dépendent que de la classe d'équivalence inertielle de $(P,{\pi})$.

Notons que pour une paire discrète $(P,\pi)$, après avoir pris l'équivalence inertielle, on peut demander que les facteurs de $\pi$ qui sont inertiellement équivalentes sont isomorphes, et on peut prendre des facteurs isomorphes les uns à côté des autres. C'est-à-dire dans toute classe d'équivalence inertielle de paires discrètes partout non-ramifiées, il existe un représentant  $(P, {\pi})$ tel que \[ M_P\cong \underbrace{ G_{n_1}\times\cdots\times  G_{n_1}}_{m_{1}}\times \underbrace{ G_{n_2}\times\cdots\times  G_{n_2}}_{m_{2}}\times\cdots \times\underbrace{ G_{n_k}\times\cdots\times G_{n_k}}_{m_{k}}, \] (on n'exige pas que $n_1, \cdots, n_k$ soient distincts) et
 comme $\mathcal{H}_{M_P}$-module 
\[ {\pi}=\underbrace{ \Pi_1\otimes\cdots\otimes  \Pi_{1}}_{m_{1}}\otimes \underbrace{ \Pi_{2}\otimes\cdots\otimes  \Pi_2 }_{m_{2}}\otimes\cdots \otimes\underbrace{ \Pi_{k}\otimes\cdots\otimes \Pi_{k}}_{m_{k}}, \]
où chaque $\Pi_i$ est une  représentation automorphe discrète irréductible partout non-ramifiée de $G_{n_i}(\AAA)$ 
et les $\Pi_i$ sont deux-à-deux non  inertiellement équivalentes. On dit qu'un tel représentant est \textit{un bon représentant}. L'intérêt d'introduire un bon représentant qui va faciliter les calculs est le fait suivant: 
pour tout bon représentant $(P, {\pi})$, un couple $(w, \lambda)\in \stab(P,{\pi})$ si et seulement si $(w,1)\in \stab(P, \pi)$ et $\lambda\in \Fix({\pi})$.

\subsubsection{Classification du spectre discret}\label{Pa2} 
Nous rappelons la structure du spectre discret dans cette sous-section. Brièvement, le théorème de Moeglin-Waldspurger détermine le support cuspidal de toute représentation discrète automorphe de $G(\AAA)$.

Soient $P$ un sous-groupe parabolique standard de $G$  et $x\in |X_1|$ un point fermé. Pour tout $f\in \mathcal{H}_{{G},x}$, soit $t_P(f)$ le terme constant de $f$ le long de  $P$ qui est défini par  
$$ t_P(f)(m)=\rho_P(m)\int_{N_P(F_x)} f(mn)\d n  \ ,$$  pour tout $m\in M_P(F_x)$, où $\rho_P$ est la 
racine carrée positive du caractère modulaire de $M_P(\AAA)$ par lequel $M_P(\AAA)$ agit sur les mesures de Haar $\d n$ de $N_P(\AAA):$ $ m.\d n.m^{-1} = \rho_P^{2}(m) \d n$, $\forall m\in M_P(\AAA)$. 
 Alors $t_P(f)\in \mathcal{H}_{M_P,x}$. Cela définit en effet un morphisme d'algèbres $t_P: \mathcal{H}_{{G},x}\rightarrow \mathcal{H}_{M_P,x}$, qui munit $\mathcal{H}_{M_P, x}$ d'une structure de $\mathcal{H}_{{G},x}$-algèbre. 

Soit $|\cdot|: {G}(\AAA)\rightarrow \mathbb{C}^{\times}$ le caractère défini par $|g|=q^{\deg(\det g)}$.
Le théorème suivant caractérise toute représentation discrète irréductible partout non-ramifiée en terme de représentation cuspidale  irréductible partout non-ramifiée.

\begin{theorem} \cite[Théorème p. 606, Moeglin-Waldspurger]{Wald-Moe2}\label{MWres}
Soit $ \Pi$ une  représentation automorphe discrète irréductible partout non-ramifiée de $G_{n}(\mathbb{A})$. Alors il existe un entier $d\mid n$, un sous-groupe parabolique standard $P$ de $G$ dont le sous-groupe de Levi $M_P$ s'identifie à $ G_{d}\times\cdots\times G_{d}$
 et une représentation automorphe cuspidale irréductible partout non-ramifiée 
$\pi$ de $G_{d}(\mathbb{A})$  
tels que si on note $\tilde{\pi}=\pi|\cdot|^{\frac{{n}/{d}-1}{2}} \otimes  \pi|\cdot|^{\frac{{n}/{d}-3}{2}} \cdots \otimes\pi|\cdot|^{-\frac{{n}/{d}-1}{2}}  $ la représentation cuspidale irréductible  de $M_P(\AAA)$, on a
$$\Pi\cong \tilde{\pi} ,$$
comme $\mathcal{H}_{{G}}$-module via le morphisme $t_P$. De plus la donnée $(\pi, \frac{n}{d})$ associée à $\Pi$ est unique.

$Vice$-$versa$, si $\pi$ est une représentation cuspidale irréductible partout non-ramifiée de ${G}_{d}(\mathbb{A})$, $\nu$ un entier positif alors la représentation cuspidale irréductible de $M_P(\AAA)$ $$  \pi|\cdot|^{\frac{\nu-1}{2}} \cdots \otimes\pi|\cdot|^{-\frac{\nu-1}{2}}, $$ vue comme $\mathcal{H}_{{G}}$-module est isomorphe au $\mathcal{H}_{G}$-module sous-jacent  à une représentation discrète irréductible de partout non-ramifiée de $G_{\nu d}(\AAA)$. 

 Soit  $\pi$ une représentation cuspidale irréductible partout non-ramifiée de $G_{d}(\AAA)$ et $\nu\in \mathbb{N}^*$. 
On va noter  $\gls{piboxnu}$ la représentation automorphe discrète unique de $G_{d\nu}(\AAA)$ déterminée par la paire $(\pi, \nu)$.
\end{theorem}

Soient $\nu, d\in \mathbb{N}^*$ tels que $\nu d=n$.  Soit $M$ le sous-groupe de Levi standard de $G$ correspondant à la partition $( \underbrace{d, \ldots, d}_{\nu\ \text{fois}})$. L'application $(m_1, \cdots, m_{\nu})\mapsto m_1\cdots m_\nu$ de $M$ dans $G_d$ induit une inclusion $X_{G_d}^{G_d} \hookrightarrow X_M^M\subseteq X_M^{G}$. De plus l'image de $X_{G_d}^{G_d}$ est contenu dans le sous-groupe  $X_{G}^{G}$ de $X_M^{G}$, cela nous permet d'identifier $X_{G_d}^{G_d}$ avec un sous-groupe de  $X_{G}^{G}$. 

\begin{prop}\label{res e}
 Soit $\pi\boxtimes\nu$  
  une représentation discrète automorphe partout non-ramifiée de $G_{n}(\AAA)$. 
Avec l'identification $X_{G_d}^{G_d}\subseteq X_{G}^{G}$ ci-dessus, on a 
$$\Fix(\pi)=\Fix(\pi\boxtimes \nu)$$ 
et $({G},\pi\boxtimes\nu)$ est   inertiellement équivalente à $({G},\pi'\boxtimes\nu')$ si et seulement $\nu=\nu'$ et si $\pi$ et $\pi'$ sont inertiellement équivalentes. 
\end{prop}
\begin{proof}\sloppy
Clairement, l'inclusion $X_{G_d}^{G_d}\subseteq X_{G}^{G}$ induit une inclusion $\Fix(\pi) \subseteq \Fix(\pi\boxtimes \nu)$. Soit $\lambda\in \Fix(\pi\boxtimes \nu)$, on a $\lambda\in X_{G}^{G}$.
%Soit $\lambda=|\cdot|^{s}$.  
 Par l'inclusion $X_{G}^{G}\subseteq X_M^{G}$ il définit un caractère de $M(\AAA)$. Par la partie d'unicité du théorème \ref{MWres}, on a l'autre inclusion, d'où le dernier énoncé. 
\end{proof}

\section{La partie géométrique de la formule des traces}\label{geometric}

On introduit en \ref{3.1} l'intégrale adélique d'un certain noyau tronqué défini par Arthur. On montre que cette intégrale exprime le nombre de classes d'isomorphie de fibrés isoclines de rang et degré fixés sur $X_1$ (cf. théorème \ref{Maing}). En utilisant des résultats de Schiffmann et Mellit, on en déduit que sous une hypothèse de coprimalité l'intégrale est donnée par un polynôme universel en les valeurs propres du Frobenius agissant sur le $H^1(X, \mathbb{Q}_{\ell})$ (cf. corollaire \ref{geom}).

\subsection{Notations}
\subsubsection{}\label{coracine}

Soit $M$ un sous-groupe de Levi semi-standard défini sur $F$ de $G$. 
Soit $\glslink{ago}{\ago_M^{*}}=X^*(M)\otimes_{\mathbb{Z}} \mathbb{R}$. Soit $\glslink{ago}{\ago_M}=\Hom_\mathbb{Z}(X^*(M), \mathbb{R})$ son espace dual. Pour tout sous-groupe parabolique $P\in\mathcal{P}(M)$, on utilise souvent $\glslink{ago}{\ago_P}$ (resp.  $\glslink{ago}{\ago_P^*}$) au lieu de $\ago_{M}$ (resp.  $\ago_{M}^*$).

L'application de restriction de $M$ à $Z_{M}$ donne un morphisme injectif de $X^{*}(M)$ à $X^{*}(Z_{M})$ dont le conoyau est fini. Cela fournit un isomorphisme $\ago_{M}^{*}=X^{*}(M)\otimes_{\mathbb{Z}}\mathbb{R}\cong  X^{*}(Z_M)\otimes_{\mathbb{Z}}\mathbb{R}$. 
 Soient $P\subseteq Q$ deux sous-groupes paraboliques définis sur $F$. On a un morphisme $X^{*}(Z_{M_{P}})\rightarrow X^{*}(Z_{M_{Q}})$ induit par l'inclusion $Z_{M_{Q}}\subseteq Z_{M_{P}}$ et un morphisme $X^{*}(M_{Q})\rightarrow X^{*}(M_{P})$ induit par l'inclusion $M_{P}\subseteq M_{Q}$. Ils induisent  une surjection $\ago_{P}^{*}\rightarrow \ago_{Q}^{*}$ et une section  $\ago_{Q}^{*}\rightarrow \ago_{P}^{*}$ de la surjection. Cela permet d'identifier $\ago_{Q}^{*}$ à un sous-espace de $\ago_{P}^{*}$ et on a une décomposition $$\ago_{P}^{*}=\ago_{Q}^{*}\oplus\ago_{P}^{Q,*},$$
où $\ago_{P}^{Q,*}$ est le noyau de la projection de $\ago_{P}^{*}\rightarrow\ago_{Q}^{*}$.   
Duellement, on a une décomposition $$\ago_{P}=\ago_{Q}\oplus\ago_{P}^{Q}.  $$

Il y a une dualité canonique à valeurs dans $\mathbb{Z}$ entre le groupe des caractères rationnels de $T$ et celui du groupe à un paramètre $X_{*}(T):=\Hom(\mathbb{G}_{m}, T)$. Cela nous permet d'identifier $\ago_{B}$ avec $X_{*}(T)\otimes_{\mathbb{Z}} \mathbb{R}$.

\subsubsection{}\label{raciness}
Soit $M$ un sous-groupe de Levi semi-standard défini sur $F$  de $G$. 
Soit $\glslink{PhiP}{\Phi(Z_M, {G})}$ l'ensemble des racines de $G$ relatives à $Z_{M}$. Il engendre l'espace $\ago_M^{{G},*}$. Tout sous-groupe parabolique $P\in \mathcal{P}(M)$ détermine un sous-ensemble des racines $\glslink{PhiP}{\Phi_P}$, qui sont les poids de l'action adjointe de $Z_M$ sur l'algèbre de Lie de radical nilpotent $N_P$ de $P$, 
 et un sous-ensemble de racines simples $\gls{DeltaP}$. 
L'ensemble $\Delta_{P}$ est caractérisé par le fait qu'il est une base de $\ago_P^{{G},*}$ telle que tout l'élément de $\Phi_{P}$ est une combinaison à coefficients entiers positifs des éléments de $\Delta_{P}$. Il consiste des éléments non-nuls de la projection de $\Delta_{B'}$ à $\ago_P^{*}$ pour un quelconque sous-groupe de Borel $B'$ semi-standard défini sur $F$ contenu dans $P$.

Pour les sous-groupes paraboliques $P\subseteq Q$ définis sur $F$, on dispose de l'ensemble $\Delta_P^Q\subseteq \Delta_P$ où $\gls{DeltaPQ}=\Delta_{P}\cap \Phi(Z_{M_{P}}, M_{Q})$. 
On dispose  de la base des coracines $(\Delta_{B'}^{P})^{\vee}$ dans $\ago_{B'}^{P}$ pour tout sous-groupe de Borel $B'$ et tout sous-groupe parabolique $P$ tels que  $B'\subseteq P$. {Pour deux paraboliques $P\subseteq Q$ soit $\gls{Deltavee}$ la projection de $(\Delta_{B'}^{Q})^{\vee}-(\Delta_{B'}^{P})^{\vee}$ sur $\ago_{P}^{Q}$ pour un sous-groupe de Borel $B'\subseteq P$, cela ne dépend pas de $B'$. } 
On note $\gls{Deltahat}$ la base de $\ago_{P}^{Q, *}$ qui est duale de $(\Delta_{P}^{Q})^{\vee}$, et  $(\hat{\Delta}_{P}^{Q})^{\vee}$  la base de $\ago_{P}^{Q}$ qui est duale de $\Delta_{P}^{Q}$.

\subsubsection{}\label{hattau}
Soit $\gls{aBp}$ le cône dans $\ago_B$ défini par $\{H\in \ago_B|\ \alpha(H)>0, \ \forall\ \alpha \in  \Delta_{B}  \}$.
Soit $\gls{baBp}$ l'adhérence  du cône positif $\ago_{B}^{+}$.

Soit $\gls{tauP}$ et $\gls{htauP}$ les fonctions caractéristiques respectives  des cônes
           $$\{H\in \ago_{B} \mid \alpha(H) > 0, \  \forall\  \alpha \in  \Delta_{P}\},$$
et
            $$\{H\in \ago_{B} \mid \varpi(H) > 0, \  \forall\  \varpi \in  \hat{\Delta}_P \}.$$

\subsubsection{}\label{HC}

 Pour tout sous-groupe parabolique $P$ de $G$, on dispose de l'application de Harish-Chandra
            $$\gls{HP} : G(\AAA)\to \ago_P $$
qui vérifie
            $$\chi(H_{P}(g))=\deg \chi (p)$$
pour tout $p\in P(\AAA)$, $g\in pK$, et $\chi\in X^*(P)$. On peut voir \ref{Pa3} pour une présentation explicite.

\subsubsection{Bases spécifiques}\label{syst}\label{coor}
Soit $P$ un sous-groupe parabolique standard tel que $M_{P}\cong G_{n_{1}}\times \cdots\times G_{n_{r}} $, les caractères 
\begin{align*}
\gls{detM}:\quad  M_{P}\quad  &\longrightarrow \mathbb{G}_{m} \\
(m_{1},\ldots, m_{r})&\longmapsto \det m_{i}
\end{align*}
forment une base de $X^*(M_P)$ donc de $\ago_{P}^{*}$. Les éléments de $ \Phi(Z_{M_{P}}, {G})$ sont des morphismes \begin{align*}
\alpha_{ij}:  \quad Z_{M_{P}}\quad&\longrightarrow \mathbb{G}_{m}  \\
(t_{1},\ldots, t_{r}) &\longmapsto t_{i}t_{j}^{-1} \end{align*}
pour  $i\neq j$, ces éléments peuvent s'écrire dans la base $(\textrm{det}_{M_{P}, i})$  comme $$\alpha_{ij}= \frac{1}{n_{i}}\textrm{det}_{M_{P},i}-\frac{1}{n_{j}}\textrm{det}_{M_{P}, j}  .$$
On utilise les $\textrm{det}_{M_{P}, i}$ pour identifier $\ago_{P}^{*}$ avec $\mathbb{R}^{r}$, et on munit $\ago_{P}$ de la base duale que nous noterons $\gls{psiM}$.

\subsubsection{}\label{Pa3}
Cette sous-section n'est pas indispensable  pour ce qui suit. Pour la commodité du lecteur,  on explicite certaines constructions précédentes.  

Soit $P$ un sous-groupe parabolique dont le sous-groupe de Levi $M_{P}\cong G_{n_{1}}\times \cdots\times G_{n_{r}} $.
Le morphisme de Harish-Chandra $\gls{HP}$ est donné par $$H_P(nmk)=(\deg\det m_1, \cdots, \deg\det m_r),$$ pour $m=(m_1, \ldots, m_r)\in M_P(\AAA)$, $n\in N_P(\AAA)$ et $k\in K=G(\mathcal{O})$. La projection $\ago_{B}^{*}\longrightarrow \ago_{P}^{*}$ est donnée par 
$$(x_{1},\ldots, x_{n}) \longmapsto 
 (\frac{x_{1}+\cdots+x_{n_{1}}}{n_{1}}, \ldots , \frac{x_{n_{1}+\cdots+n_{r-1}+1}+\cdots+x_{n}}{n_{r}}  ). $$
La projection $\ago_{B}\longrightarrow \ago_{P}$ est donnée par 
$$ (x_{1},\ldots, x_{n})\longmapsto (x_{1}+ \cdots+ x_{n_{1}}, \ldots, x_{n_{1}+\cdots+n_{r-1}+1}+\cdots+x_{n} ).$$
L'inclusion $\ago^{*}_{P}\rightarrow \ago_{B}^{*}$ est $$(x_{1},\ldots, x_{r})\longmapsto ( \underbrace{x_{1}, \ldots, x_{1}}_{n_{1} \text{ fois}}, \underbrace{x_{2},  \ldots }_{n_{2} \text{ fois}}, \ldots, \underbrace{\ldots, x_{r}}_{n_{r} \text{ fois}}), $$ et l'inclusion
$\ago_{P}\rightarrow \ago_{B}$ est $$(x_{1},\ldots, x_{r})\longmapsto ( \underbrace{\frac{x_{1}}{n_{1}}, \ldots, \frac{x_{1}}{n_{1}}}_{n_{1} \text{ fois}}, \underbrace{\frac{x_{2}}{n_{2}},  \ldots }_{n_{2} \text{ fois}}, \ldots, \underbrace{\ldots, \frac{x_{r}}{n_{r}}}_{n_{r} \text{ fois}}).$$
Pour l'espaces $\ago_{P}^{{G},*}$ et $\ago_{P}^{{G}}$, on a 
$$  \ago_{P}^{{G},*}= \{ (x_{1},\ldots, x_{r})\in \ago_{P}^{*} |\ n_{1}x_{1}+\cdots+n_{r}x_{r}=0  \}  ,       $$
$$   \ago_{P}^{{G}}= \{ (x_{1},\ldots, x_{r})\in \ago_{P} |\ x_{1}+\cdots+ x_{r}=0  \}  .$$
On a aussi
$$\Delta_{P}^{\vee}=\{ \alpha_{i}^{\vee}=e_{M_{P}, i}-e_{M_{P}, i+1}\mid 1\leq i\leq r-1     \},$$
et  $$\hat{\Delta}_{P}=\{ \varpi_{i} |  1\leq i\leq r-1  \},$$ 
où $\varpi_{i} $ est  $$\frac{n_{i+1}+\cdots+n_{r}}{n}(\mathrm{det}_{M_{P}, 1}+\cdots+\mathrm{det}_{M_{P}, i}) -   \frac{n_{1}+\cdots+n_{i}}{n}(\mathrm{det}_{M_{P}, i+1}+\cdots+\mathrm{det}_{M_{P}, r}) .    $$

\subsection{Intégrale d'Arthur et fibrés indécomposables}\label{3.1}

\subsubsection{} \label{tracet}
Soit $P$ un sous-groupe parabolique standard. L'action de $\mathbbm{1}_{K}\in \mathcal{H}_{G}$ sur 
$$L^2(M_P(F)N_P(\AAA)\backslash G(\mathbb{A})/\Xi_GK)$$ est un opérateur intégral de noyau $$\gls{kPxy}=  \sum_{a\in \Xi_{{G}}}\sum_{\gamma\in M_{P}(F)}  \int_{N_{P}(\mathbb{A})}  \mathbbm{1}_{K}(y^{-1}\gamma n xa)\d n, $$
où $x, y\in M_P(F)N_P(\AAA) \backslash G(\AAA)/\Xi_G$. 
Lorsque $x=y$, seul $a=1$ contribue, donc on a $$k_{P}(x,x)=  \sum_{\gamma\in M_{P}(F)}  \int_{N_{P}(\mathbb{A})}  \mathbbm{1}_{K}(x^{-1}\gamma n x)\d n  . $$
Arthur a défini un noyau tronqué, pour 
$x\in G(F)\backslash G(\mathbb{A})/\Xi_{G} K$, et $T\in \ago_B$ par
\begin{equation}\label{noyauf}
k^{T}(x,x)= \sum_{P\in \mathcal{P}(B)}\sum_{\delta\in P(F) \backslash G(F)}(-1)^{\dim\ago_P^G}\hat{\tau}_{P}(H_{B}(\delta x)-T)k_P(\delta x,\delta x)  . \end{equation}
Notons que la somme sur $\delta$ peut être prise sur un ensemble fini (qui dépend de $x$ et $T$, cf. \cite[Lemma 5.1]{A1}). 

On pose
\begin{equation}\label{J e}
\gls{JeT}:=\int_{G(F)\backslash G(\mathbb{A})^{e}} k^{T}(x,x)\d x .
\end{equation}
et $$J_e= J_e^{T=0}. $$
L'intégrale dans (\ref{J e}) est absolument convergente (proposition 11 page 227, de \cite{Laff}).

\subsubsection{}
Pour nous, un fibré vectoriel sur $X_{1}$ est un $\mathcal{O}_{X_{1}}$-module localement libre de rang fini. Soit $\mathcal{E}$ un fibré vectoriel, on note  $\deg(\mathcal{E})$ pour le degré de $\mathcal{E}$, $\rg(\mathcal{E})$ pour le rang et $\gls{muE}$ pour la pente $$\mu(\mathcal{E})=\frac{\deg(\mathcal{E})}{\rg(\mathcal{E}})\quad \text{si $\mathcal{E}\neq 0$}.$$

Un fibré est \textit{indécomposable} s'il ne peut pas être écrit comme une somme directe de deux sous-fibrés vectoriels. On a le résultat suivant:
\begin{lemm} \label{ind}
Tout fibré vectoriel est une somme directe de fibrés indécomposables, et cette décomposition est unique à isomorphisme près.
 Un fibré vectoriel indécomposable de degré $e$ premier avec son rang $n$ est géométriquement indécomposable au sens où son produit tensoriel avec $\mathbb{F}$ est indécomposable comme fibré vectoriel sur $X$. 
\end{lemm}
\begin{proof}
On peut voir la preuve du lemme 2.6 de \cite{Schiffmann}.
\end{proof}

\begin{definition}\label{isocline}
Soit $\mathcal{E}$ un fibré vectoriel  sur $X_1$. On dit que $\mathcal{E}$ est isocline si pour toute décomposition $\mathcal{E}\cong \mathcal{F}\oplus \mathcal{G}$, on a  $\mu(\mathcal{F})=\mu(\mathcal{E})$. 
\end{definition}
\begin{remark}\label{rmk-iso}
Il résulte du lemme \ref{ind} que tout fibré vectoriel se décompose, uniquement à isomorphisme près, comme une somme directe de fibrés isoclines de pentes distinctes.
 \end{remark}

Soit $\mathcal{P}_{n}^{e}(X_1)$ le nombre des classes d'isomorphie de fibrés vectoriels isoclines de rang $n$ et de degré $e$. On peut énoncer:

\begin{theorem}\label{Maing}
Soit $e\in \mathbb{Z}$, on a $$J_e=\mathcal{P}_n^e(X_1). $$
\end{theorem}

La démonstration se trouve en la section \ref{KKK}.

\subsubsection{Conséquence du théorème \ref{Maing}}\label{Higgs}
Soit $\textbf{Higgs}_{n,e}(X_{1})$ le $\mathbb{F}_q$-champ des fibrés de Higgs sur $X_{1}$, qui paramètre les couples $(\mathcal{E}, \theta)$ où \\
$\bullet$ $ \mathcal{E}$ est un fibré vectoriel  sur $X_1$ de rang $n$ de degré $e$;\\
$\bullet$ $\theta: \mathcal{E}\rightarrow \mathcal{E}\otimes \omega_{X_{1}}$ est un morphisme de $\mathcal{O}_{X_{1}}$-module, où $\omega_{X_{1}}$ est le fibré en droite canonique. \\
On sait que $\textbf{Higgs}_{n,e}(X_{1})$ est un champ algébrique localement de type fini (on peut le prouver en considérant le morphisme d'oubli de $\textbf{Higgs}_{n,e}(X_{1})$ à $\textbf{Bun}_{n}^{e}$ qui est représentable de type fini, où $\textbf{Bun}_{n}^{e}$ est le champ algébrique des fibrés vectoriels de rang $n$ et de degré $e$ sur $X_1$; pour plus de détails, voir \cite[Theorem 7.18]{CMW}). 

Un couple $(\mathcal{E},\theta)$ est appelé stable, si tout sous-fibré $0\neq \mathcal{F}\subsetneq \mathcal{E}$ qui est $\theta$-stable (i.e. $\theta(\mathcal{F})\subseteq \mathcal{F}\otimes \omega_{X_{1}}$)  satisfait $$\mu(\mathcal{F})< \mu(\mathcal{E}). $$
Soit $\mathbf{Higgs}_{n,e}^{st}(X_{1})$ le sous-champ ouvert de $\mathbf{Higgs}_{n,e}(X_{1})$ des fibrés de Higgs stables. Soit $\mathrm{Higgs}_{n,e}^{st}(X_{1})$ le schéma des modules grossiers des fibrés de Higgs stables. On sait que $\mathrm{Higgs}^{st}_{n,e}(X_1)$ est un schéma lisse, quasi-projectif et de dimension $2(g-1)n^{2}+2$ (cf. proposition 7.4. \cite{Nitsure}).

Dans l'introduction, on a introduit $\sigma_{1},\ldots, \sigma_{2g }$ les $q$-entiers de Weil de la courbe $X_{1}$ indexés de telle façon que $\sigma_{i}\sigma_{i+g}=q$ pour tout $i$ tel que $g \geq i\geq 1$.

\begin{theorem}[Schiffmann, Mellit]\label{Mel}
Soit $e$ un nombre premier avec $n$, on a
\begin{equation}\label{Number}\mathcal{P}_{n}^{e}(X_{1})=q^{-1-(g-1)n^{2}}| \mathrm{Higgs}^{st}_{n,e}(X_{1})(\mathbb{F}_{q})|. \end{equation}
De plus il existe un polynôme de Laurent $A_{g,n}$ dans $\mathbb{Z}[q, z_{1}^{\pm 1}, \ldots, z_{g}^{\pm 1}]$ (ici $q$ est une variable)  qui  ne dépend que des nombres $g$ et $n$ tel que 
$$\mathcal{P}_{n}^{e}(X_{1})=A_{g,n}(q, \sigma_{1}, \cdots, \sigma_{g}  ). $$ 
De plus, pour chaque monôme  $q^{m}z_1^{a_1}\cdots z_g^{a_g} $ dans $ A_{g,n} $ on a $$m+\sum_{i=1}^g\min\{a_i, 0\} \geq 0 .  $$
 \end{theorem}
 
 L'hypothèse que $e$ est premier avec $n$ implique qu'un fibré vectoriel est isocline de rang $n$ et de degré $e$ si et seulement s'il est géométriquement indécomposable.  
L'égalité (\ref{Number}) ci-dessus est montrée par Mozgovoy et Schiffmann (cf. \cite[Theorem 1.2]{Schiffmann} \cite[Theorem 1.2]{MS}). 
Dans l'appendice $B$, on observe que la trace tronquée est liée à une ``trace" tronquée sur l'algèbre de Lie de $G$. Par des résultats de Chaudouard \cite{Chau}, cela donne une autre preuve de cette égalité.  
 
Pour l'existence du polynôme de Laurent $A_{g,n}$, on renvoie au théorème 1.1 de \cite{Mellit}.  Notons que l'existence d'une fraction rationnelle qui joue le même rôle était auparavant connue par les travaux de Garcia-Prada, Heinloth et Schmitt (cf. corollaire 2.2 \cite{GPHS} et le théorème B \cite{GPH}).

Pour la commodité du lecteur, nous rappelons le résultat de Mellit ci-dessous. Mellit a introduit une série $\Omega_g\in \mathbb{Q}(q,z)[ z_1^{\pm 1}, \ldots, z_{g}^{\pm 1}][[T]]$, qui est définie par \[    \Omega_g=\sum_{\mu\in \mathcal{P}  } T^{|\mu|}\prod_{\square\in \mu} \frac{\prod_{i=1}^{g}  (z^{a(\square)+1   }     - z_i  q^{l(\square)} ) (   z^{a(\square)   }    -z_i^{ -1 } q^{l(\square) +1 } ) }{  (z^{a(\square)+1   } -q^{l(\square)}    ) (z^{a(\square)}   -q^{l(\square)  +1 } )  } ,  \]
où $\mathcal{P}=\coprod_{k}\mathcal{P}_k$ et $\mathcal{P}_k$ est l'ensemble des partitions (non ordonnées) de $k$ ou, du point de vue de cette formule,  $\mathcal{P}_k$ est aussi l'ensemble des diagrammes de Young de taille $k$; soit $\mu\in \mathcal{P}_k$ un diagramme de Young, on a $|\mu|=k$; soit $\square\in \mu$, $a(\square)$ et $l(\square)$ sont respectivement la longueur de ``arm" et de ``leg" de $\square$. Rappelons que $a(\square), l(\square)\in \mathbb{N}$ et pour tout $\mu\in \mathcal{P}$ il existe une case $\square\in \mu$ telle que $a(\square)=l(\square)=0$. 

Le polynôme de Laurent $H_{g,n}\in \mathbb{Q}[q^{\pm 1},z^{\pm 1}, z_1^{\pm 1}, \ldots, z_{g}^{\pm 1}]$ de Mellit est donné par la formule: \[  -(1-g)(1-z) \mathrm{Log}   \Omega_g = \sum_{n=1}^{\infty} H_{g,n}T^n,  \]
où $\mathrm{Log}$ est la fonction logarithmique pléthystique définie par 
\[ \mathrm{Log}(f) = \sum_{k\geq 1} \frac{\mu(k)}{k} \psi_k(log f) , \]
avec $\mu(\cdot)$ la fonction de Möbius, $\psi_k$ l'opérateur d'Adam qui envoie respectivement toute variable $T$, $q$, $z$, $z_1$, $\ldots,$ $z_g$ en $T^k$, $q^k$, $z^k$, $z_1^k$, $\ldots,$ $z_g^k$, et $log$ est la fonction logarithmique usuelle. Le polynôme de Laurent $A_{g,n}$ est donné par 
\[A_{g,n}(q, z_1, \ldots, z_g)=H_{g,n}(q,1,z_1, \ldots, z_g).\]

La dernière assertion du théorème \ref{Mel} est un corollaire du théorème \ref{ordinary}, elle peut aussi se déduire directement  de la formule de Mellit. 
En effet, en utilisant la formulation de Mellit, on vérifie aisément que, \textit{a priori}, il existe un polynôme de Laurent $L_{g,n}$ dans $\mathbb{Q}[q, z_{1}^{\pm 1}, \ldots, z_{g}^{\pm 1}]$ tel que pour chaque monôme  $q^{m}z_1^{a_1}\cdots z_g^{a_g} $ dans $L_{g,n} $ on a $$m+\sum_{i=1}^g\min\{a_i, 0\} \geq 0,  $$
et un polynôme $R(q)\in \mathbb{Q}[q]$ tel que  $$A_{g,n}(q, z_1, \ldots, z_g)=\frac{L_{g,n}(q, z_1, \ldots, z_g)}{1+qR(q)}. $$
Si $A_{g,n}$ ne satisfaisait  pas l'hypothèse, on obtiendrait une contradiction en examinant les termes dans $A_{g,n}$ de plus bas degré en $q$ qui ne vérifient  pas l'hypothèse. 

Rappelons que $\mathcal{P}_{n}^{e}(X_1)$ et $| \mathrm{Higgs}^{st}_{n,e}(X_{1})(\mathbb{F}_{q})|$ dépendent a priori de $e$. Il se trouve qu'ils n'en dépendent pas si $e$ est premier avec $n$.  C'est montré par Mellit et indépendamment par Groechenig, Wyss et Ziegler \cite{GWZ}. Nous en donnerons une preuve différente dans l'appendice A.

\begin{coro}\label{geom}
Quand $(e,n)=1$, on a
$$J^{ }_{e}=A_{g,n}(q, \sigma_{1}, \cdots, \sigma_{g}  ) .$$
\end{coro}

\subsection{Preuve du théorème \ref{Maing}}\label{KKK}

\subsubsection{}

Le théorème suivant est démontré par Arthur sur un corps de nombres (cf. théorème 9.1 de \cite{Arthur}). La même méthode marche pour un corps de fonctions aussi, on renvoie le lecteur au théorème 5.2.1 de \cite{Chau} pour une preuve. 

\begin{theorem}\label{QP}
En tant que fonction en $T\in \mathfrak{a}_{B}$, l'application $T\mapsto J_{e}^{T}$ est quasi-polynomiale au sens de la définition \cite[Définition 4.5.3]{Chau}.
\end{theorem}
 D'après Lafforgue (proposition 11, page 227, de \cite{Laff}) pour $T\in \ago_B$ tel que  $\gls{dT}:=  \min_{\alpha\in \Delta_B}\{\alpha(T)\}  \geq \max\{   2g-2,0\} $ on a \begin{equation}\label{LafTr}J_e^{T}=\int_{G(F)\backslash G(\AAA)^{e}} F^{{G}}(g,T)\sum_{\gamma \in G(F)}\mathbbm{1}_{K}(g^{-1}\gamma g) \d g \end{equation}
où $\gls{FG}$ est la fonction caractéristique  des $g\in G(F)\backslash G(\AAA)/K$ tels que pour tout sous-groupe parabolique standard $P$ et pour tout $ \delta \in P(F) \backslash G(F)$, on a
\begin{equation}\label{htauHP}  \hat{\tau}_P(H_P(\delta g)-T)=0. \end{equation}
\begin{remark}\label{332.}
La fonction $F^{G}(\cdot, T)$ est aussi la fonction caractéristique des $g\in G(F)\backslash G(\AAA)/K$ tels que pour tout sous-groupe parabolique standard $P$ avec $M_P\cong G_{r}\times G_{n-r}$ et pour tout $ \delta \in P(F) \backslash G(F)$, la condition \eqref{htauHP} est satisfaite. En effet, soit $g\in G(\AAA)$, si \begin{equation} \hat{\tau}_P(H_P(\delta g)-T)=1, \end{equation} 
pour un certain couple $(P, \delta)$, alors pour tout sous-groupe parabolique $Q$ contenant $P$, en particulier on peut supposer que $Q$ est maximal, on a par définition (comme $\hat{\Delta}_Q\subseteq \hat{\Delta}_{P}$) \begin{equation} \hat{\tau}_Q(H_Q(\delta g)-T)=1.  \end{equation} 
\end{remark}

Rappelons qu'on a une équivalence (cf. Appendice E \cite{Laumon2}, voir aussi \cite[2.6]{Chau}) entre $\mathbf{Bun}_{n}(\mathbb{F}_{q})$ le groupoïde des fibrés vectoriels sur $X_{1}$  et le groupoïde quotient  $$[\leftquotient{G(F)}{\left(G(\mathbb{A})/K\right)}] .$$
Précisément, un élément $g\in G(\mathbb{A})/K$ donne naissance à un fibré vectoriel $\mathcal{E}^{g}$ muni d'une trivialisation générique de $\mathcal{E}_{\eta}^{g}\rightarrow F^{n}$, où $\eta$ est le point générique de $X_1$. Un drapeau $\mathcal{F}_{\bullet}=(0=\mathcal{F}_0\subsetneq\mathcal{F}_1\subsetneq\ldots \subsetneq\mathcal{F}_{k-1}\subsetneq \mathcal{F}_k=\mathcal{E}^g)$ de $\mathcal{E}^g$ correspond à un couple $(P, \delta)$ avec $P$ un sous-groupe parabolique standard de $G$ et $\delta\in  P(F)\backslash G(F)$. 

\sloppy
Identifions $\ago_P$ avec $\mathbb{R}^k$ par la base fixée dans \ref{coor}. Alors $H_P(\delta g)=(\deg(\mathcal{F}_{1}/\mathcal{F}_0  ), \ldots, \deg(\mathcal{F}_{k}/\mathcal{F}_{k-1}  ) )$.
Soit $T=(T_1, \ldots, T_n)$. 
On suppose $\min_{i=1, \ldots, n-1}\{T_i-T_{i+1}\}\geq 0$. 
Selon la définition de Chaudouard \cite[4.1]{Chau}, on dit qu'un fibré vectoriel $\mathcal{E}$ de rang $n$ est \textbf{$T$-semi-stable} si pour tout sous-fibré $\mathcal{F}$ de $\mathcal{E}$ de rang $r$, on a l'inégalité de pente
\begin{equation}\label{tauPs} \frac{\deg(\mathcal{F})}{r}-\frac{T_1+\cdots +T_r}{r}\leq \frac{\deg(\mathcal{E})}{n}-\frac{T_1+\cdots +T_n}{n}. \end{equation}
De plus,  \[\htau_P(H_P(\delta g)-T)=1,\] si et seulement si pour tout $0< i< k$, le sous-fibré vectoriel $\mathcal{F}_i$ ne satisfait pas la condition \eqref{tauPs}. Donc la condition \eqref{tauPs} pour tout sous-fibré de $\mathcal{E}^{g}$ est équivalent à la condition \eqref{htauHP} pour tout couple $(P, \delta)$ tel que $M_P\cong G_{r}\times G_{n-r}$ (i.e. $k=2$) et $\delta\in P(F) \backslash  G(F)$. {En vertu de la remarque \ref{332.}, $F^{G}(g, T)$ est la fonction caractéristique des $g\in G(\AAA)$ tels que $\mathcal{E}^{g}$ est $T$-semi-stable.}

Avec cette interprétation, $\gamma \in G(F)$ satisfait  $\gamma g\in gK$ si et seulement si $\gamma$ induit un automorphisme de $\mathcal{E}^g$. Donc on a
\begin{equation}\begin{cases} \vol(G(F)\backslash G(F)gK)=\vol (\frac{K}{K\cap g^{-{1}}G(F)g})=|\Aut(\mathcal{E}^{g})|^{-1};\\
\sum_{\gamma \in G(F)}\mathbbm{1}_{K}(g^{-1}\gamma g)=|\Aut(\mathcal{E}^{g})|. \end{cases}\end{equation}
Comme la fonction $g\mapsto F^{{G}}(g,T)\sum_{\gamma \in G(F)}\mathbbm{1}_{K}(g^{-1}\gamma g)$ est $K$-invariante à droite, en décomposant l'intégrale dans la formule \ref{LafTr} de Lafforgue comme une somme suivant le double quotient $G(F)\backslash G(\AAA)/K$, 
on obtient alors lorsque $\gls{dT}\geq \max\{0,2g-2\}$, la formule \ref{LafTr} s'interprète comme: 
\begin{align}\label{JeTs}
J_{e}^{T}%& = \sum_{g\in G(F)\backslash G(\AAA)^{e}/K }  \vol (\frac{K}{K\cap g^{-{1}}G(F)g}) F^G(g,T) \\
& = \sum_{\cal{E} \  T\text{-semi-stable de degré} \ e}\frac{|Aut(\mathcal{E})|}{|Aut(\mathcal{E})|}\\
& = \#\text{classes d'isomorphie de fibrés $T$-semi-stables de degré}\ e.  \nonumber  \end{align}

Soit $I=(n_1, \cdots, n_r)$ une partition de $n$. 
Soit $T\in \mathbb{R}^n$, nous introduisons la fonction caractéristique $\gls{GammaP} $ des $H\in \mathbb{R}^r$, 
tels que \\
$\bullet$ $\frac{H_1}{n_1}>\frac{H_2}{n_2}>\ldots>\frac{H_r}{n_r}$;\\
$\bullet$ $\frac{H_1+\cdots +H_i}{n_1+\cdots+n_i} -   \frac{H_1+\cdots +H_r}{n_1+\cdots+n_r}\leq \frac{T_{1}+\cdots+T_{n_{1}+\cdots+n_{i}}}{n_{1}+\cdots+n_i}-\frac{T_{1}+\cdots+T_{n}}{n}$ pour $i=1,2,\ldots, r-1$. \\
\begin{remark}\textup{
Lorsque $r=1$, $\Gamma_{I}(\cdot, T)$ est la fonction constante égale à $1$. Si $T=0$, $\Gamma_{I}(\cdot, 0)\equiv 0$ pour tout $r>1$. }
\end{remark}
Pour tout fibré vectoriel $\mathcal{E}$ de rang $n$, on a une décomposition à isomorphisme près  $$\mathcal{E}=\bigoplus_{i=1}^{k}\mathcal{I}_i\ ,$$ où $\mathcal{I}_{i}$  est un fibré vectoriel isocline de pente $\mu_{i}$, tel que  $\mu_{1}>\cdots>\mu_{k}$. Soit $n_{i}$ le rang de $ \mathcal{I}_i$.
La partition $I_{\mathcal{E}}=(n_1, \ldots, n_r)$ de $n$ ne dépend que de la classe d'isomorphie de $\mathcal{E}$. Soit $H_{\mathcal{E}}=(n_{1}\mu_{1},\ldots,n_{k}\mu_{k})$. 

\begin{prop}\label{key}
Soit $T\in \mathbb{R}^n$ tel que $\gls{dT}= \min_{\alpha\in \Delta_B}\{\alpha(T)\}\geq  \max\{2g-2, 0\}$. Le fibré vectoriel  $\mathcal{E}$ est $T$-semi-stable si et seulement si ${\Gamma}_{I_{\mathcal{E}} }(H_{\mathcal{E}}, T)=1$. 
\end{prop}

On reporte la démonstration jusqu'à \ref{pdp}. Par la remarque  \ref{rmk-iso}, le comptage des fibrés vectoriels $T$-semi-stables peut être fait selon leur décomposition en somme directe de fibrés isoclines, donc quand $d(T)\geq \max\{0,2g-2\}$, par la proposition \ref{key}, l'égalité \ref{JeTs} devient
  \begin{align*}
J^{T}_{e}&= \sum_{\mathcal{E}\  \text{fibrés vectoriels de degré}\  e} {\Gamma}_{I_{\mathcal{E}}}(H_{\mathcal{E}}, T)\\
&= \sum_{n_{1}+\cdots+n_{m}=n} \sum_{ \substack{ \frac{e_{1}}{n_{1}}> \cdots> \frac{e_{m}}{n_{m}}  \\ e_{1}+\cdots+e_{m}=e } }{\Gamma}_{(\underline{n})}((\underline{e}), T)\prod_{i=1}^{m}\mathcal{P}_{n_{i}}^{e_{i}}(X_1) .
\end{align*}
Comme quand ${\Gamma}_{(\underline{n})}((\underline{e}), T)=1$, on a automatiquement $\frac{e_{1}}{n_{1}}> \cdots> \frac{e_{m}}{n_{m}}$, donc
$$J^{T}_{e} =  \sum_{n_{1}+\cdots+n_{m}=n}  \sum_{ e_{1}+\cdots+e_{m}=e  }{\Gamma}_{(\underline{n})}((\underline{e}), T)\prod_{i=1}^{m}\mathcal{P}_{n_{i}}^{e_{i}}(X_1).$$

Puisque le nombre $\mathcal{P}_n^e(X_1)$ ne dépend que de la classe de $e$ dans  $\mathbb{Z}/n\mathbb{Z}$, on a le droit d'écrire $\mathcal{P}_n^{e}(X_1)$ pour $e\in \mathbb{Z}/n\mathbb{Z}$. 
Soit $e\in \mathbb{Z}$ et des éléments $e_{i}\in \mathbb{Z}/n_{i}\mathbb{Z}$ pour $1\leq i \leq r$, 
on associe à ces données l'ensemble $\mathfrak{h}_{(e)}^{e}\subseteq \mathbb{Z}^{r}$ défini par 
\begin{equation*}\glslink{he}{ {\mathfrak{h}_{(\underline{e}  )}^{e}   }   }=\{(d_{1},\ldots, d_{r})\in \mathbb{Z}^{r}| \ d_{1}+\cdots+ d_{r}=e\quad \text{et $\quad d_{i}\equiv e_{i}$ mod $n_{i}$}  \} .\end{equation*}
On peut alors regrouper la somme par les types modulo les rangs: 
\begin{equation*}
J^{T}_{e}=  \sum_{ n_{1}+\cdots+n_{m}=n}\sum_{\substack{ e_{i}\in\bbb{Z}/n_{i}\mathbb{Z}\\ i=1,\ldots,m}}  \sum_{H\in \mathfrak{h}^{e}_{(\underline{e})}  }{\Gamma}_{(\underline{n})}(H, T)\prod_{i=1}^{m}\mathcal{P}_{n_{i}}^{e_{i}}(X_1). \end{equation*}
Par le théorème \ref{QP}, le côté gauche est quasi-polynomial en $T$. Par la proposition 4.5.5 de \cite{Chau}, le côté droit est quasi-polynomial dans la région où $T_1\geq T_2\geq \cdots\geq T_n$. Donc cette égalité est vraie pour tout $T$ tel que $T_1\geq T_2\geq \cdots\geq T_n$. 
Posons $T=0$, seul le terme $(n_{1},\ldots,n_{m})=(n)$ contribuera, on obtient donc $$J_{e}=\mathcal{P}_{n}^{e}(X_1).$$

\subsubsection{Preuve de la proposition \ref{key}}\label{pdp}

Il faut démontrer l'assertion suivante:  
\textit{
Soit $ \mathcal{E}$ un fibré vectoriel de rang $n$ donné par: $$\mathcal{E}=\bigoplus_{i=1}^{k}(\mathcal{I}_{i,1}\oplus\cdots\oplus\mathcal{I}_{i,s_{i}})$$ où $\cal{I}_{i,j}$  sont des fibrés indécomposables de pentes $\mu_{i}$ tels que  $\mu_{1}>\cdots>\mu_{k}$. Soit $n_{i}$ le rang de $\mathcal{I}_{i,1}\oplus\cdots\oplus\mathcal{I}_{i,s_{i}}$. Soit $T\in \mathbb{R}^n$  tel que $\min_{i=1, \ldots, n-1}\{T_i-T_{i+1}\}\geq  \max\{2g-2, 0\}$.
Pour que $\mathcal{E}$ soit $T$-semi-stable, il faut et il suffit qu'on ait 
$$\frac{n_{1}\mu_{1}+\cdots+n_{i}\mu_{i}}{n_{1}+\cdots+n_{i}}-\mu(\mathcal{E})\leq \frac{T_{1}+\cdots+T_{n_{1}+\cdots+n_{i}}}{n_{1}+\cdots+n_i}-\frac{T_{1}+\cdots+T_{n}}{n} , $$
pour $i=1,\ldots, k-1$.}

Dans le cadre de la $T$-semi-stabilité, on a aussi l'existence et l'unicité de la filtration de Harder-Narasimhan (cf. \cite[Théorème 4.1.2]{Chau}).

La nécessité est claire par la définition de la $T$-semi-stabilité. 

Réciproquement, posons $H_{1}=\cdots=H_{n_{1}}=\mu_{1}$, $ \ldots $, $H_{n_{1}+\cdots+ n_{k-1}+1}=\cdots=H_{n}=\mu_{k}$, notons qu'on a $H_{1}\geq H_{2}\cdots\geq H_{n}$. 
On va montrer l'inégalité suivante: 
\begin{equation}\label{ine}
\frac{H_{1}+\cdots+H_{l}}{l}-\mu({\mathcal{E}})\leq  \frac{T_{1}+\cdots+ T_{l}}{l}-\frac{T_{1}+\cdots+T_{n}}{n}  \end{equation}
pour tout $1\leq l\leq n$. En admettant cela, soit $J$ un sous-ensemble de $\{1,\ldots, n\}$ de cardinal $l$, on  a:
\begin{equation}\label{INE}\frac{\sum_{i\in J}H_{i}}{l}\leq \frac{\sum_{i=1}^{l}H_{i}}{l}\leq \mu(\mathcal{E})+\frac{T_{1}+\cdots+ T_{l}}{l}-\frac{T_{1}+\cdots+T_{n}}{n} . 
\end{equation}
Si $\mathcal{E}$ n'est pas $T$-semi-stable, il existe une $T$-filtration de Harder-Narasimhan non-triviale  $0=\mathcal{F}_{0}\subsetneq \mathcal{F}_{1}\subsetneq \cdots \mathcal{F}_{r}=\mathcal{E}$. Puisque $\min_{i=1, \ldots, n-1}\{T_i-T_{i+1}\}\geq \max\{0,2g-2\}$, cette filtration est scindée par le lemme 4.2.1 \cite{Chau} et donc $\mathcal{F}_{1}$ est isomorphe à certaine somme de $\mathcal{I}_{i,j}$ en vertu du  lemme \ref{ind}. Mais cela contredit l'inégalité (\ref{INE}).

Pour prouver l'inégalité (\ref{ine}), on définit les formes linéaires pour $i=1,\ldots, n-1$
\begin{align*}\varpi_{i}:\quad \mathbb{R}^{n}&\longrightarrow \mathbb{R}  ,\\
(T_{1},\ldots, T_{n})&\longmapsto \frac{(n-i)}{n}(T_{1}+\cdots+T_{i})- \frac{i}{n}(T_{i+1}+\cdots+T_{n})
 \end{align*}
et de plus
\begin{align*}\alpha_{i}:\quad \mathbb{R}^{n}&\longrightarrow \mathbb{R}   . \\
(T_{1},\ldots, T_{n})&\longmapsto T_{i}-T_{i+1}
 \end{align*}
On note $H=(H_{1},\ldots,H_{n})\in \mathbb{R}^{n}$, soit $N_{i}:=n_{1}+\cdots+n_{i}\leq l\leq n_{1}+\cdots+n_{i+1}-1:=N_{i+1}-1$,  l'hypothèse  implique l'inégalité:
$$\varpi_{N_{i}}(H-T)\leq 0 ,  $$
pour $1\leq i\leq k$. 
Après un calcul direct, 
\[ \begin{split} & (N_{i+1}-l)\varpi_{N_{i}}+(l-N_{i})\varpi_{N_{i+1}}-n_{i+1}\varpi_{l}  \\ =&-(N_{i+1}-l)\sum_{t=N_i}^{l-1} (t-N_i+1)\alpha_{t}  -(l-N_{i})\sum_{t=l}^{ N_{i+1}-1 }(N_{i+1}-t)\alpha_{t} . \end{split} \]
Mais par la définition $\alpha_{t}(H)=0$ pour $N_{i+1}-1\geq t\geq N_{i}$, et $\alpha_{t}(T)\geq 0$ pour tout $t$. Donc 
$$0\geq (N_{i}-l)\varpi_{N_{i}}(H-T)+(l-N_{i})\varpi_{N_{i+1}}(H-T)\geq n_{i+1}\varpi_{l}(H-T) . $$
Ceci implique l'inégalité (\ref{ine}).

\section{Quelques résultats sur des $(G,M)$-familles d'Arthur}
Cette section peut être lue indépendamment. Les résultats de cette section servent  aux calculs des aspects spectraux de la formule des traces. 
\subsection{Notations}  \label{tracespectrale}

\subsubsection{Groupes de Weyl} Rappelons que $G=G_n=GL_n$. 
Pour tout sous-groupe de Levi $M$,  soit $\glslink{WM}{W^M}$  le groupe de Weyl de $M$. On identifie un élément $w\in W^{M}$ avec une matrice de  permutation dans $M(F)$.  Quand $M=G$, on note  simplement $\glslink{WM}{W_n}$ au lieu de $W^{G}$. Pour un sous-groupe parabolique $P$ de sous-groupe de Levi $M$, on utilise aussi $\glslink{WM}{W^{P}}$ pour $W^{M}$.

\subsubsection{}\label{4.1.2}

Soit $L$ un sous-groupe de Levi semi-standard (i.e. contenant $T$) de $G$. 
L'ensemble des chambres dans $\ago_L^{G}$ est en bijection naturelle  avec $\mathcal{P}(L)$,  l'ensemble des sous-groupes paraboliques dont le sous-groupe de Levi semi-standard est $L$, de telle façon que chaque $P\in\mathcal{P}(L)$ détermine une chambre par $$\{H\in \ago_L^{G}|\ \alpha(H)>0, \ \forall\ \alpha\in \Delta_P\}.$$

Pour $P\in \mathcal{P}(L)$, on désigne par $\gls{bP} \in \mathcal{P}(L)$ le sous-groupe parabolique opposé: $\Phi_{\bar{P}}=-\Phi_P$. 
On dit que deux sous-groupes paraboliques $P,Q\in\mathcal{P}(L)$ sont adjacents si les chambres associées sont adjacentes, i.e. on a $|\Phi_{\bar{P}}\cap \Phi_{Q}|=1$.

 \subsubsection{}\label{Lw}

Soit $M$ un sous-groupe de Levi semi-standard et $w\in W_n$. On notera $\glslink{sM}{w(M)}$ pour $wMw^{-1}$. On a $w(\ago_M)=\ago_{w(M)}$.  
Chaque $w\in W_n$ induit un isomorphisme 
$$w:X_M^{G}\rightarrow X_{w(M)}^{G}, $$
qui est défini par $w(\lambda)(m):=\lambda(w^{-1}m w)$, $\forall\ \lambda\in X_M^{G}$ et $m\in w(M)(\AAA)$.

Soit $w\in W_n$ tel que $w(M)=M$.  Il existe un sous-groupe de Levi minimal $\gls{Lw}$ contenant $M$ et $w$. Ce sous-groupe de Levi est caractérisé par le fait que $\ago_{L^w}=\ker((w-id) |_{\ago_M})$.  En particulier,  $w$ agit trivialement sur $X_{L^w}^{G}$.

Soit $M\subseteq L$ deux sous-groupes de Levi de $G$, l'inclusion $M(\AAA)\subseteq L(\AAA)$ induite par restriction une inclusion $$X_L^G\subseteq X_M^G. $$

\subsubsection{}\label{ReIminc}

Soit $M\subseteq L$ deux sous-groupes de Levi semi-standards définis sur $F$. Soit $\glslink{ImX}{\Im X_M^L}\subseteq X_M^L$ le sous-groupe constitué des caractères unitaires. Soit  $\glslink{ImX}{\Re X_M^L}\subseteq X_M^L$ le sous-groupe constitué des caractères réels.

\subsubsection{} \label{SYS}

Soit $L$ un sous-groupe de Levi semi-standard muni d'un isomorphisme  $L\cong G_{n_1}\times G_{n_2}\times\cdots\times G_{n_r}$.
Un $\lambda\in X_L$ se factorise à travers le morphisme de multi-degré $\deg_L: L(\AAA)\rightarrow \mathbb{Z}^{r}$. Soit $\lambda_i$ l'image par $\lambda$ de l'élément $(0, \ldots, 0, 1, 0, \ldots, 0)$ où l'unique $1$ est placé en $i$-ème position. Alors $X_L$ s'identifie  avec $(\mathbb{C}^{\times})^{r}$. Sous cette identification
$$X_{L}^{{G}}\cong \{(\lambda_{1},\ldots, \lambda_{r})\in  \bbb{C}^{r}  \ {|}\ \lambda_{1}^{n_{1}}\lambda_{2}^{n_{2}}\cdots \lambda_{r}^{n_{r}}=1\},  $$
et  $$\Im X_{L}^{{G}}\cong \{(\lambda_{1}, \ldots, \lambda_{r})\in \bbb{C}^{r} \ {|}\ |\lambda_{1}|=\cdots=|\lambda_{r}|=1;\quad \lambda_{1}^{n_{1}}\lambda_{2}^{n_{2}}\cdots \lambda_{r}^{n_{r}}=1\}  .$$

\subsubsection{  $\lambda^H$ et $\langle\lambda, H\rangle$.}\label{Ctimes}
Un caractère $\alpha\in X^{*}(L)$ définit un élément dans $X_L$ comme le composé $$L(\AAA)\xrightarrow{ \alpha } \mathbb{G}_{m}(\AAA)\xrightarrow{q^{\deg(\cdot)}} \mathbb{C}^{\times}.$$
On a donc un isomorphisme canonique $X_L \cong X^{*}(L)\otimes_{\mathbb{Z}}\mathbb{C}^{\times}$.

Soit $\glslink{aL}{\ago_{L,\mathbb{Z}}}:=\Hom(X^{*}(L), \mathbb{Z})$.
Comme $X_L \cong X^{*}(L)\otimes_{\mathbb{Z}}\mathbb{C}^{\times}$, l'appariement $(\lambda, n)\mapsto \lambda^n$  sur $\mathbb{C}^{\times}\times \mathbb{Z}$ à $\mathbb{C}^{\times}$ induit un appariement  sur $X_L\times \ago_{L,\mathbb{Z}}$ qu'on note par $\glslink{LambdaH}{\lambda^{H}}$ pour $\lambda\in X_L$  et $H\in  \ago_{L,\mathbb{Z}}$. On a $$(\lambda_1, \ldots, \lambda_r)^{(H_1, \cdots, H_r)}=\lambda_1^{H_1}\cdots\lambda_r^{H_r} .$$
%On a aussi $\lambda(m)=\lambda^{H_P(m)}$.

Soit $\glslink{aL}{\ago_{L,\mathbb{C}}^{*}}:=\ago_{L}^{*}\otimes_{\mathbb{R}}\mathbb{C}\cong X^{*}(L)\otimes_{\mathbb{Z}}\mathbb{C}$.  
%Par l'inclusion $\mathbb{C}^{\times}\hookrightarrow \mathbb{C}$, 
On a un plongement $X_L\hookrightarrow \ago_{L,\mathbb{C}}^{*}$ qui envoie $(\lambda_1, \lambda_2, \ldots, \lambda_r)\in X_L$ vers $(\lambda_1, \lambda_2, \ldots, \lambda_r)\in \ago_{L,\mathbb{C}}^{*}$ (sous les systèmes de coordonnées dans \ref{coor} et \ref{SYS}).
La forme bilinéaire canonique sur $\ago_{L,\mathbb{C}}^{*}\times \ago_{L,\mathbb{C}}$ définit une fonction $(\lambda,H)\mapsto\glslink{LH}{\langle\lambda,H\rangle}$ sur $X_L \times \ago_{L,\mathbb{C}}$ par restriction. Pour $\lambda=(\lambda_1, \ldots, \lambda_r)\in X_L$  et $H= {(H_1, \cdots, H_r)} \in  \ago_{L,\mathbb{C}}$, on a 
$$\langle (\lambda_1, \ldots, \lambda_r), {(H_1, \cdots, H_r)}\rangle=\lambda_1 H_1+\cdots+\lambda_r H_r .$$

\subsection{$({G},M)$-familles}\label{spGM}\label{4.4}
Soit $M$ un sous-groupe de Levi semi-standard défini sur $F$. Dans cette section, nous introduisons la notion de $({G},M)$-familles. Un lecteur qui ne s'intéresse qu'aux  principaux résultats de cet article 
est invité à lire seulement \ref{GMfam}, et à revenir sur les résultats plus techniques s'il y est fait référence.  

Les notations ici sont standardes. Ce sont celles utilisées sur un corps de nombres, sauf qu'on considère, suivant Lafforgue, les $({G},M)$-familles comme des fonctions méromorphes sur $X_M^{G}$ ou $X_M$ au lieu de fonctions méromorphes sur $\ago_{M,\mathbb{C}}^{*}$ pour éviter la notion de $({G},M)$-familles périodiques.

\subsubsection{$({G},M)$-familles}\label{GMfam}

Pour tout $Q\in \mathcal{P}(M)$ et $\lambda\in X_{M}$, soit $$\gls{thetaQ}(\lambda)=\prod_{\alpha\in\Delta_{Q}}\langle\lambda,\alpha^{\vee}\rangle,$$
voir \ref{Ctimes} pour la notation $\langle\lambda,\alpha^{\vee}\rangle$. Remarquons que notre $\theta_Q$ 
diffère de celui d'Arthur par un facteur de volume. 

\begin{definition}
Soit $\Omega\subseteq X_{M}^{{G}}$ ou $X_M$ un domaine.  Soit $(c_{P}(\lambda))_{P\in \mathcal{P}(M)}$ une famille  de fonctions holomorphes sur $\Omega$. On dit que c'est  une $({G},M)$-famille si $c_{P}(\lambda)=c_{P'}(\lambda)$ pour toute paire de sous-groupes paraboliques adjacents $P, P'\in\mathcal{P}(M)$, i.e. $|\Phi_{P}\cap\Phi_{\bar{P}'}|=1$ (cf. \ref{4.1.2}),  et pour tout $\lambda\in \Omega$ tel que $\lambda^{\alpha^{\vee}}=1$ où $\alpha$ est l'unique racine dans $\Phi_{P}\cap\Phi_{\bar{P}'}$. \end{definition}

Soit $(c_Q)_{Q\in\mathcal{P}(M)}$ une $({G},M)$-famille sur $\Omega\subseteq X_M^{G}$, on définit une fonction méromorphe sur $\Omega$ par \begin{equation}c_{M}(\lambda)=  \sum_{Q\in\mathcal{P}(M)}{c_{Q}(\lambda)}{\theta_{Q}(\lambda)}^{-1}.   \end{equation} 
%où $\gls{thetaQ}(\lambda)=\prod_{\alpha\in\Delta_{Q}}\langle\lambda,\alpha^{\vee}\rangle$ 

\begin{theorem}[lemme 6.2, \cite{Arthur inv}]
Soit $(c_Q)_{Q\in\mathcal{P}(M)}$ une famille des fonctions méromorphes sur $X_M^{G}$ ou $X_M$.  Si c'est une $({G},M)$-famille sur un voisinage de $\lambda_0\in X_M^{G}$ (ou $\lambda_0\in X_M$), alors $c_M(\lambda)$
 est une fonction régulière en $\lambda_0$.  
\end{theorem}
La preuve de cette théorème est identique à celle d'Arthur: $c_M(\lambda)$ n'a pas de singularité en codimension $1$ et $\lambda^{\alpha^{\vee}}=1$ implique alors $\langle\lambda, \alpha^{\vee}\rangle=0$.

\subsubsection{Des familles associées à une $(G,M)$-famille}\label{GLfa} 

Soit $(c_Q(\lambda))_{Q\in\mathcal{P}(M)}$ une $({G},M)$-famille sur un voisinage de $1\in X_M^{G}$. On notera  $c_{M}$ pour la valeur $c_{M}(1)$, i.e.
\begin{equation}\gls{cM}=\lim_{\lambda\rightarrow 1}\sum_{Q\in\mathcal{P}(M)}{c_{Q}(\lambda)}{\theta_{Q}(\lambda)}^{-1}.\end{equation}
Soit $L\in\mathcal{L}(M)$ un sous-groupe de Levi contenant $M$. Soit $\lambda\in X_M^{G}$, pour tout $R\in\mathcal{P}(L)$ et $Q\in {\mathcal{P}^{L}({M})}$ (sous-groupes paraboliques de $L$ qui admettent $M$ comme sous-groupe de Levi),  
on pose
 \begin{equation}\glslink{cQR}{c_Q^R(\lambda)}=c_{QN_{R}}(\lambda),  \end{equation} 
où $QN_{R}$ est l'unique sous-groupe parabolique dans $\mathcal{P}(M)$ qui est contenu dans $R$ dont l'intersection avec $L$ est $Q$. Pour chaque $R$, cela donne une $(L,M)$-famille sur un voisinage de $1$. 
On pose \begin{equation}c_M^R=\lim_{\lambda\rightarrow 1}\sum_{Q\in \mathcal{P}^L(M)}  c_{Q}^{R}(\lambda) \theta_Q^L(\lambda)^{-1}, \end{equation} 
où $\theta_Q^L(\lambda)=\prod_{\alpha\in \Delta_Q^L}{\langle\lambda, \alpha^{\vee} \rangle}.$
De plus, pour tout $Q\in \mathcal{P}(L)$, et $\lambda\in X_{L}^{{G}}$, on pose \begin{equation}c_{Q}(\lambda)=c_{P}(\lambda),\end{equation} pour n'importe quel groupe $P\in \mathcal{P}(M)$ contenu dans $Q$, cette construction ne dépend pas du choix du groupe $P$, et on obtient une $({G},L)$-famille sur un voisinage de $1$.

\begin{prop}[Corollary 6.5, \cite{Arthur inv}]\label{SCIN}
Soient $(c_{Q})_{Q\in\mathcal{P}(M)}$ et $(d_{Q})_{Q\in\mathcal{P}(M)}$ deux $({G},M)$-familles sur un voisinage de $1$ telles que pour chaque $L\in \mathcal{L}(M)$, les valeurs $c_{M}^{Q}$ soient indépendantes  de $Q\in\mathcal{P}(L)$, on notera $c_{M}^{L}$ cette valeur.
Nous avons la formule de scindage du produit de deux $({G},M)$-familles $$(cd)_{M}=\sum_{L\in\mathcal{L}(M)}c_{M}^{L}d_{L}    .$$ 
\end{prop}
\begin{proof}
C'est une variante de \cite[Corollary 6.5]{Arthur inv}. La preuve d'Arthur ne marche pas directement ici, mais notre résultat peut être obtenu en utilisant le sien. 

On a donné dans \ref{Ctimes} un plongement $X_{M}\hookrightarrow \ago_{M}^{*}\otimes_{\mathbb{R}} \mathbb{C}$. Notons que la restriction à $X_M$ d'une $({G},M)$-famille définie par Arthur est une $({G},M)$-famille suivant notre définition.  
Soient $P\subseteq Q$ deux sous-groupes paraboliques semi-standards. Comme le volume $\vol( \ago_{P}^{Q}/\mathbb{Z}(\hat{\Delta}_{P}^{Q})^{\vee} )$ ne dépend que des sous-groupes de Levi $M_{P}$ et $M_{Q}$, on note
$$v_{M_{P}}^{M_{Q}}:=\vol( \ago_{P}^{Q}/\mathbb{Z}(\hat{\Delta}_{P}^{Q})^{\vee} ). $$

Par le Corollary 6.5 de \cite{Arthur inv}, on a
\begin{align*}
(cd)_{M}&=\sum_{L\in\mathcal{L}(M)}     \frac{v_{M}^{L}v_{L}^{{G}} }{  v_{M}^{{G}}   }  c_{M}^{L}d_{L}. 
\end{align*}
Pour tout $P\subseteq Q$,  l'espace $\ago_{P}^{Q}$ est orthogonal à l'espace $\ago_{Q}^{{G}}$, et l'ensemble $\Delta_{Q}^{\vee}$ est la projection de $\Delta_{P}^{\vee}-(\Delta_{P}^{Q})^{\vee}$ en $\ago_{Q}^{{G}}$, on obtient alors $v_{M}^{L}v_{L}^{{G}}=v_{M}^{{G}}$, et donc 
 $$(cd)_{M}=  \sum_{L\in\mathcal{L}(M)}      c_{M}^{L}d_{L}   . $$
 \end{proof}

\subsubsection{Des résultats concernant des $({G},M)$-familles spéciales.}\label{sp{G}M}
Nous nous intéressons particulièrement aux $(G,M)$-familles d'une forme spéciale qui interviennent des aspects spectraux de la formule des traces. % Bien sûr,  résultat d'Arthur sur un corps de nombres. 

\begin{theorem}[Arthur, Lemma 7.1. \cite{Arthur Eis}] \label{GMM}
Soit $(c_{\beta})_{\beta\in \Phi(Z_{M}, {G})}$ une famille des fonctions méromorphes sur $\mathbb{C}^{\times}$ indexées par les racines avec $c_{\beta}(1)=1$. Soit  
$(c_{Q})_{Q\in\mathcal{P}(M)}$ la famille de fonctions définies par: $$c_{Q}(\lambda)=\prod_{\beta\in\Phi_{Q}}c_{\beta}(\lambda^{\beta^{\vee}}).$$
 
Alors $(c_{Q})_{Q\in \mathcal{P}(M)}$ est une $({G},M)$-famille sur le domaine où chaque $c_Q$ est holomorphe (en  particulier, ce domaine contient $1$). 
On a
$$c_M=\lim_{\mu\rightarrow 1} \sum_{Q\in \mathcal{P}(M)} \theta_{Q}(\mu)^{-1}{c_{Q}( \mu)}=\sum_{F} \prod_{\beta\in F}{c'_{\beta}(1)}, $$
où dans la première somme, $F$ parcourt toutes les parties de $\Phi(Z_M,{G})$ qui forment des bases de $\ago_{M}^{{G},*}$ et $c'_\beta$ est la dérivée de $c_\beta$.
\end{theorem}
\begin{proof}
C'est une variante du  \cite[Lemma 7.1.]{Arthur Eis}. La preuve de \textit{loc. cit.} ne marche pas directement dans notre cas, mais nous suivons sa stratégie.

Tout d'abord, on montre que $(c_{Q})_{Q\in\mathcal{P}(M)}$ est une $({G},M)$-famille sur un domaine contenant $1$.
Soient $Q$ et $Q'$ deux sous-groupes paraboliques adjacents dans $\mathcal{P}(M)$. Alors il existe un seul élément $\alpha$ dans $\Phi_{Q}\cap-\Phi_{{Q'}}$, tel que pour tout $\lambda\in X_{M}^{{G}}$ avec $\lambda^{\alpha^{\vee}}=1$, on ait
\begin{align*}
c_{Q}(\lambda)&=\prod_{\beta\in\Phi_{Q}}c_{\beta}(\lambda^{\beta^\vee})\\
                        &=c_{\alpha}(1)\prod_{ \beta\in\Phi_{Q}\cap \Phi_{Q'}}c_{\beta}(\lambda^{\beta^\vee})  \\
                        &=\prod_{ \beta\in\Phi_{Q}\cap \Phi_{Q'}}c_{\beta}(\lambda^{\beta^\vee})\\
                        &= c_{Q'}(\lambda) .
 \end{align*}

Maintenant on calcule $c_M$.  On suit la méthode d'Arthur (cf. page 37 \cite{Arthur inv} et page 1317 \cite{Arthur Eis}).  On pose $\mu=1+\xi t$,  avec $\xi\in {X_P^{G}}\subseteq (\mathbb{C}^{\times})^{\dim\ago_{P}}$. Quand $t\in\mathbb{R}$ est assez petit, par le développement de Taylor on a: 
$$c_M(1+\xi t)=\sum_{P\in\mathcal{P}(M)}\theta_{P}(\xi)^{-1}\sum_{k\geq 0}    \frac{1}{k!}\biggr((\frac{\d}{\d t})^{k}\prod_{\beta\in\Phi_{P}}c_{\beta}((1+\xi t)^{\beta^{\vee}})\vert_{t=0} \biggr) t^{k-m}, $$ 
où $m=\dim \mathfrak{a}_{M}^{{G}}$. Comme $(c_Q)_{Q\in\mathcal{P}(M)}$ est une $({G}, M)$-famille, $c_M(1+\xi t)$ est holomorphe en $t$ donc
\begin{equation}\label{GMc}\begin{cases}0= \sum_{P\in\mathcal{P}(M)}\theta_{P}(\xi)^{-1}    \frac{1}{k!}\biggr((\frac{\d}{\d t})^{k}\prod_{\beta\in\Phi_{P}}c_{\beta}((1+\xi t)^{\beta^{\vee}}) \vert_{t=0} \biggr),  \quad 1\leq k\leq m-1 ;\\
\\
c_M=\sum_{P\in\mathcal{P}(M)}\theta_{P}(\xi)^{-1}\frac{1}{m!}\biggr((\frac{\d}{\d t})^{m}\prod_{\beta\in\Phi_{P}}c_{\beta}((1+\xi t)^{\beta^{\vee}})\vert_{t=0} \biggr) . \end{cases}\end{equation}

Il s'ensuit que $c_M$ ne dépend que des valeurs $c_{\beta}^{(j)}(1)$ pour tout $\beta\in\Phi(Z_{M}, G)$, $1\leq j\leq m$ et ne dépend pas de $\xi$. On observe de plus que si on suppose $$c_{\beta}(z)=1+\sum_{1\leq j\leq m}\frac{X_{\beta,j}}{j!}(z-1)^{j}\in \mathbb{C}[z, X_{\beta,1},\ldots , X_{\beta,m}],$$ alors $c_M$ sera un polynôme bien défini dans $\mathbb{C}[X_{\beta, j}:1\leq j\leq m, \beta\in \Phi(Z_{M}, {G})]$
et si on évalue ce polynôme en posant $X_{\beta,j}=c^{(j)}_{\beta}(1)$, on retrouvera la valeur désirée. Nous travaillerons donc avec ces polynômes dans la suite de la démonstration.

Si $\beta$ est une racine telle que $\beta^{\vee}=e_{M,u}-e_{M,v}\in \mathfrak{a}_{M}$, $1\leq u\neq v\leq \dim\ago_{{M}}$ (la base $(e_{M, i})$ de $\ago_M$ est définie dans  \ref{syst}):

\begin{align}\label{equation0}
& c_{\beta}((1+\xi t)^{\beta^{\vee}}) \\ =&1+\sum_{1\leq j\leq m}\frac{X_{\beta,j}}{j!}(\frac{1+\xi_{u}t}{1+\xi_{v}t}-1)^{j} \nonumber\\ 
=&1+\sum_{m\geq j\geq 1}\sum_{n\geq 0}\frac{X_{\beta,j}}{j!} (\xi_u-\xi_v)^{j} \binom{-j}{n} \xi_v^n t^{n+j}  \nonumber \\
=&1+\sum_{N\geq 1}  \left(   \sum_{1\leq j\leq \textrm{min}\{m, N\}}\frac{X_{\beta,j}}{j!} \binom{N-1}{j-1}(-1)^{N-j}(\xi_{u}-\xi_{v})^{j}\xi_{v}^{N-j}\right)     t^{N}   \nonumber .
\end{align}

Notons que pour le polynôme $c_{M}\in \mathbb{C}[X_{\beta, j}:1\leq j\leq m, \beta\in \Phi(Z_{M}, {G})]$ , il n'existe que des monômes dont chaque variable est de degré au plus $1$.
Soient $\beta_1, \ldots, \beta_k \in \Phi(Z_M, {G})$ et $m\geq j_1, \ldots, j_k \geq 1$. 
Considérons le coefficient de monôme $X_{\beta_{1},j_{1}}\cdots X_{\beta_{k},j_{k}}$  dans $c_{M}$, pour lequel on peut supposer que $\beta_{i}\neq \beta_{j}$ si $i\neq j$. 

Supposons tout d'abord qu'il existe un $i$ tel que $j_{i}\geq 2$ par exemple c'est $j_{1}\geq 2$.  Posons $$\xi=z\beta_{1}+\xi_{1},$$ avec $\beta_{1}$ orthogonale à $\xi_{1}$. Nous pouvons certainement choisir $\xi_{1}$ tel que $$\{\alpha\in\Phi(Z_M, {G})| \langle\alpha^{\vee},\xi_{1}\rangle=0  \}=\{\beta_{1}, -\beta_{1}\}. $$
Observons que le coefficient de $X_{\beta_{1},j_{1}}$ est d'ordre $j_{1}$ en $z$ quand $z$ approche $0$, i.e. est $O(z^{j_{1}})$. D'autre part, on sait que $$\theta_{Q}(z\beta_{1}+\xi_{1})=O(z^{-1}),$$ pour tout sous-groupe parabolique $Q$. Donc le coefficient considéré approche $0$ quand $z$ approche $0$, mais ce nombre $a\ priori$ ne dépend pas de $z$. Cela implique alors que les coefficients de $X_{\beta_{1},j_{1}}\cdots X_{\beta_{k},j_{k}}$ s'annulent dès qu'il existe un $i$ tel que $j_{i}\geq 2$.  On conclut que $c_M$ est indépendant de $X_{\beta, j}$ pour $j\geq 2$. 

Donc on peut supposer que $X_{\beta,j}=0$ pour $j\geq 2$, 
 et par simplicité on note $X_{\beta,1}=X_{\beta}$, i.e. $c_\beta(z)=1+X_\beta (z-1)$. Soient $\beta_1, \ldots, \beta_k\in \Phi(Z_M, {G})$ deux-à-deux distincts. 
Considérons le coefficient du monôme  $X_{\beta_{1}}\cdots X_{\beta_{k}}$ dans $c_{M}$. %D'après le loi de Leibniz: 
% $$\left(f_1 f_2 \cdots f_l \right)^{(m) }=\sum_{a_1+a_2+\cdots+a_l=m;  k_{i}\geq 1} {m \choose a_1, a_2, \ldots, a_l} \prod_{1\le t\le l}f_{t}^{(a_{t})}$$ 
Il est égal à
$$(\sum_{Q}'\theta_{Q}(\xi)^{-1})\prod_{j=1}^{k}\langle\xi, \beta_{j}^{\vee}\rangle(\sum_{\{ a_{j}\geq 0|\    \sum a_{j}+k=m \}}  (-1)^{m-k}   \prod_{j=1}^{k}  \xi^{a_{j}}_{t(\beta_{j})}),$$
où la somme porte sur tous les $Q\in \mathcal{P}(M)$ tels que $\Phi_{Q}\supseteq \{\beta_{1},\ldots,\beta_{k}\}$ et on suppose $\beta^{\vee}=e_{M, s(\beta)}-e_{M, t(\beta)}$ (cf.   \ref{coor}).  

Si $k>m$, le coefficient est donc $0$. 

Si $m> k$. On considère le coefficient du monôme  $X_{\beta_{1}}\cdots X_{\beta_{k}}$ dans la première équation de (\ref{GMc}) qui est zéro mais peut s'écrire aussi comme 
$$(\sum_{Q}'\theta_{Q}(\xi)^{-1})\prod_{j=1}^{k}\langle\xi, \beta_{j}^{\vee}\rangle ,$$ 
où la somme porte sur tous les $Q\in\mathcal{P}(M)$ tels que $\Phi_{Q}\supseteq \{\beta_{1},\ldots,\beta_{k}\}$. Donc on a $$0=(\sum_{Q}'\theta_{Q}(\xi)^{-1})\prod_{j=1}^{k}\langle\xi, \beta_{j}^{\vee}\rangle(\sum_{\{ a_{j}\geq 0|\    \sum a_{j}+k=m \}}  (-1)^{m-k}   \prod_{j=1}^{k}  \xi^{a_{j}}_{t(\beta_{j})})  .$$

Il ne reste qu'à considérer le cas où $k=m$. Dans ce cas le coefficient de $X_{\beta_{1}}\cdots X_{\beta_{m}}$ est égal à $$(\sum_{Q}'\theta_{Q}(\xi)^{-1})\prod_{j=1}^{k}\langle\xi, \beta_{j}^{\vee}\rangle ,$$
où la somme porte sur tous les $Q\in\mathcal{P}(M)$ tels que $\Phi_{Q}\supseteq \{\beta_{1},\ldots,\beta_{k}\}$. Cela a été considéré par Arthur dans \cite{Arthur Eis} (p.1319-p.1320), où il a montré que ce coefficient est égal à $0$ si les $\beta_{i}$ sont linéairement dépendants, 
et est égal à $1$ si les $\beta_{i}$ forment une base.

En conclusion, $$c_{M}=\sum_{F} \prod_{ \beta\in F} X_{\beta} ,   $$
où $F\subseteq \Phi(M,{G})$ parcourt les bases de $\ago_{P}^{{G},*}$. La preuve est complète si on pose $X_{\beta}=c_{\beta}'(1)$.
\end{proof}

\begin{lemm}\label{Rus}
Soit $\mathbb{S}^{1}$ le cercle unité dans $\mathbb{C}^{\times}$ orienté dans le sens antihoraire et $\d z$ la mesure complexe usuelle. Pour des fonctions $f_{\beta}$ complexes continues définies sur $\mathbb{S}^1$ et une partie $F$  de $\Phi(Z_M,{G})$ qui forme une base de $\ago_{M}^{{G},*}$, 
$$\int_{\Im{X_M^{G}}}\prod_{\beta\in F}f_{\beta}(\lambda^{\beta^{\vee}}){\lambda}^{\beta^{\vee}}\d \lambda=\prod_{\beta\in F}\frac{1}{2\pi i}\int_{\mathbb{S}^{1}}f_{\beta}(z)\d z .$$
\end{lemm}
\begin{proof}
Soit $m=\dim\ago_{M}^{{G}}$. Considérons le morphisme défini par 
\begin{align*}
u_{F}:\Im{X_M^{G}}&\longrightarrow (\mathbb{S}^{1})^{m} . \\
\lambda&\longmapsto (\lambda^{\beta^{\vee}})_{\beta\in F}
\end{align*}
Ce morphisme réalise $(\mathbb{S}^{1})^{m}$ comme un quotient de $\Im X_M^{G}$. La fonction $$\lambda\mapsto\prod_{\beta\in F}f_{\beta}(\lambda^{\beta^{\vee}}){\lambda}^{\beta^{\vee}} $$ est invariante sous l'action de $X_{{G}}^{G}=\ker (u_F)$. 
Soit $\widetilde{\d z}$ la mesure de Lesbegue de $\mathbb{S}^{1}$ telle que le volume de $\mathbb{S}^{1}$ vaut $1$, on a $$\int_{\Im{X_M^{G}}}\prod_{\beta\in F}f_{\beta}(\lambda^{\beta^{\vee}}){\lambda}^{\beta^{\vee}}\d \lambda=\prod_{\beta\in F}\int_{\mathbb{S}^{1}}f_{\beta}(z) z \widetilde{\d z}.$$
Le lemme alors découle de la définition de l'intégrale de contour. 
\end{proof}

\begin{coro}\label{INT}
Soit $(c_{\beta})_{\beta\in \Phi(Z_M, {G})}$ une famille de  fonctions  holomorphes non-nulles  sur un voisinage de $\{z\in\mathbb{C}^\times|\ |z|=1\}$. Soit $(c_{Q})_{Q\in\mathcal{P}(M)}$ une famille des fonctions définies par $$c_{Q}(\lambda)=\prod_{\beta\in\Phi_{Q}}c_{\beta}(\lambda^{\beta^{\vee}}). $$
Alors
$$\int_{\Im{X^{{G}}_{M}}} \lim_{\mu\rightarrow 1} \sum_{Q\in \mathcal{P}(M)} \theta_{Q}(\mu)^{{-1}} \frac{c_{Q}( \lambda\mu)}{c_{Q}(\lambda)}\d \lambda= \sum_{F}\prod_{\beta\in F}(\mathrm{N}(c_{\beta})-\mathrm{P}(c_\beta)) , $$
où dans la somme, $F$ parcourt les parties de $\Phi(Z_M,{G})$ qui forment des bases de $\ago_{P}^{{G},*}$, où $\mathrm{N}(c_{\beta})$ $(resp.  $ $\mathrm{P}(c_{\beta}))$ désigne le nombre des zéros ($resp. $ des pôles) de $c_{\beta}$ dans le disque $|z|<1$.
\end{coro}
\begin{proof}
La famille des fonctions  $$(\mu\mapsto\frac{c_{Q}(\lambda\mu)}{c_{Q}(\lambda)})_{Q\in\mathcal{P}(M)}$$  
satisfait la condition du théorème \ref{GMM}, donc c'est une $({G},M)$-famille et on a 
 $$\lim_{\mu\rightarrow 1} \sum_{Q\in \mathcal{P}(M)} \theta_{Q}(\mu)^{{-1}} \frac{c_{Q}( \lambda\mu)}{c_{Q}(\lambda)}= \sum_{F}\prod_{\beta\in F} \frac{c'_\beta(\lambda^{\beta^{\vee}})\lambda^{\beta^{\vee}} }{c_\beta(\lambda^{\beta^{\vee}})} .$$
Donc par le lemme \ref{Rus} et le théorème des résidus, on obtient le résultat. 
\end{proof}

\begin{lemm}\label{4.2.7}
Soit $(c_{Q})_{Q\in\mathcal{P}(M)}$ une $({G},M)$-famille comme dans le théorème \ref{GMM}. Alors pour tout sous-groupe de Levi  $L\in\mathcal{L}(M)$, les valeurs  $c_{M}^{R}$ (définis dans \ref{GLfa}) sont indépendantes des sous-groupes paraboliques $R\in\mathcal{P}(L)$, on notera cette valeur par $c_{M}^{L}$.
\end{lemm}
\begin{proof}

  Pour tout sous-groupe parabolique $Q\in \mathcal{P}^{L}(M)$, soit $\Phi_{Q}^L:=\Phi(Z_M, N_Q)$, on a 
$$c_{Q}^{R} (\lambda)=c_{ QN_{R}  } (\lambda)=  \prod_{\beta\in \Phi_{Q}^L}c_{\beta}(\lambda^{\beta^{\vee}})   \prod_{\beta\in \Phi_{{R}} } c_{\beta}(\lambda^{\beta^{\vee}}).    $$ 
De plus comme $c_{\beta}(1)=1$, on déduit que
 \begin{align*} & \lim_{\lambda\rightarrow 1} \sum_{Q\in\mathcal{P}^{L} (M) }\theta_{Q}^{L}({\lambda} )^{-1} c_{Q}^{R}(\lambda) \\ = &  \lim_{\lambda\rightarrow 1}  \left( \sum_{Q\in\mathcal{P}^{L} (M) } \theta_{Q}^{L}({\lambda} )^{-1}  \prod_{\beta\in \Phi_{Q}^L}c_{\beta}(\lambda^{\beta^{\vee}})  \right) \prod_{\beta\in \Phi_{{R}} } c_{\beta}(\lambda^{\beta^{\vee}})  \\
= &    \lim_{ \lambda\rightarrow 1}\sum_{Q\in\mathcal{P}^{L} (M) }\theta_{Q}^{L}({\lambda} )^{-1}  \prod_{\beta\in \Phi_{Q}^L}c_{\beta}(\lambda^{\beta^{\vee}}),
\end{align*} 
qui est indépendante de $R\in\mathcal{P}(L)$.
\end{proof}

\begin{lemm}\label{SPE}
Soit $\mu_0\in X_M^{G}$. 
Soit $(c_{Q})_{Q\in\mathcal{P}(M)}$  une $({G},M)$-famille sur un domaine contenant $\mu_0^{\mathbb{Z}}$ telle que pour chaque $L\in \mathcal{L}(M)$, $c_{M}^{R}$ soient indépendants de $R\in \cal{P}(L)$, qu'on note $c_{M}^{L}$. Supposons que $c_{Q}(\lambda\mu_{0})=c_{Q}(\lambda)$ pour tout $\lambda\in X_{M}^{{G}}$ où $c_Q$ est définie. Alors 
$$\lim_{\lambda\rightarrow 1} \sum_{Q\in \mathcal{P}(M)} \theta_{Q}(\lambda\mu_{0})^{-1}{c_{Q}( \lambda\mu_{0})} , $$ 
s'annule sauf si $\mu_{0}\in X_{{G}}^{{G}}$ et dans ce cas la limite est égale à $\mu_{01}^{-\dim\ago_{M}^{{G}}}c_{M}$, où $\mu_{01}$ est la première coordonnée de $\mu_{0}$ (cf.  \ref{SYS}). 
\end{lemm}
\begin{proof}
Tout d'abord, la limite existe car $(c_{Q})_{Q\in\cal{P}(M)}$ est une $({G},M)$-famille sur un voisinage de $\mu_0$.  

Pour tout $Q\in \mathcal{P}(M)$, on définit la fonction rationnelle 
%$$d_{Q}(\lambda^{\beta^{\vee}}):=\theta_{Q}(\lambda) \theta_{Q}(\lambda\mu_{0})^{-1},$$ 
$$d_{Q}(\lambda):=\theta_{Q}(\lambda) \theta_{Q}(\lambda\mu_{0})^{-1}.$$ Notons que
\begin{equation}\label{poiuyy}
d_{Q}(\lambda)=\prod_{\beta\in \Delta_{Q}}d_{\beta}(\lambda), \end{equation}
pour $d_{\beta}(\lambda)= \frac{\langle\lambda, \beta^{\vee}\rangle}{\langle\lambda\mu_{0}, \beta^{\vee}\rangle}$ qui est holomorphe en $\lambda=1$.

Nous montrons tout d'abord que $(d_{Q})_{Q\in\mathcal{P}(M)}$ est une $({G},M)$-famille sur un voisinage de $1$. 
Soient $Q_1, Q_2\in \mathcal{P}(M)$ deux sous-groupes paraboliques adjacents. Leur mur commun dans $X_M^{G}$ est défini par $\lambda^{\alpha^{\vee}}=1$ pour une $\alpha\in \Delta_{Q_1}$ telle que $-\alpha\in \Delta_{Q_2}$. On a deux choses l'une: \\
$-$ Soit $\mu_{0}^{\alpha^{\vee}}\neq 1$. Pour tout $\lambda$ tel que $\lambda^{\alpha^{\vee}}=1$, on a ${\langle \lambda, \alpha^{\vee}\rangle}=0$ et ${\langle\lambda\mu_{0}, \alpha^{\vee}\rangle}\neq 0$, donc      $$\frac{\langle \lambda, \alpha^{\vee}\rangle}{\langle\lambda\mu_{0}, \alpha^{\vee}\rangle}=0 .$$
Dans ce cas, sur le mur    $\lambda^{\alpha^{\vee}}=1$  on a $d_{Q_{1}}(\lambda)=d_{Q_{2}}(\lambda)=0$ puisque $\alpha\in \Delta_{Q_1}$ et $-\alpha\in \Delta_{Q_2}$. \\
$-$ Soit $\mu_{0}^{\alpha^{\vee}}= 1$. Pour tout $\beta_1\in \Delta_{Q_1}-\{\alpha\}$, il existe un unique $\beta_2\in \Delta_{Q_2}-\{-\alpha\}$ tel que pour tout  $\lambda\in X_M^{G}$ satisfaisant $\lambda^{\alpha^{\vee}}=1$, on a $\lambda^{\beta_1^{\vee}}=\lambda^{\beta_2^{\vee}}$, donc $\langle\lambda, \beta_1^\vee \rangle=\langle\lambda, \beta_2^\vee \rangle$. Puisque l'on a $(\lambda\mu_{0})^{\alpha^{\vee}}=1$ aussi. Donc
$$\frac{\langle\lambda, \beta_1^\vee \rangle}{\langle\lambda\mu_0, \beta_1^\vee  \rangle}=\frac{\langle\lambda, \beta_2^\vee \rangle}{\langle\lambda\mu_0, \beta_2^\vee  \rangle}.$$
De plus soit $\alpha^\vee=e_{M,u}-e_{M,v}$ (cf.  \ref{coor}), pour tout $\lambda^{\alpha^{\vee}}\neq 1$,  on a
$$\frac{\langle \lambda, \alpha^{\vee}\rangle}{\langle\lambda\mu_{0}, \alpha^{\vee}\rangle}=\frac{1}{\mu_{0u}}=\frac{\langle \lambda, -\alpha^{\vee}\rangle}{\langle\lambda\mu_{0}, -\alpha^{\vee}\rangle} .$$
Donc on a $$ d_{\alpha}(\lambda)=d_{-\alpha}(\lambda),$$
pour tout $\lambda\in X_M^G$. 

Par \eqref{poiuyy}, on peut déduire que 
 $(d_{Q})_{Q\in \mathcal{P}(M)}$ est une $({G},M)$-famille sur un domaine contenant $1$. 

Soit $L\in \mathcal{L}(M)$, un sous-groupe de Levi contenant $M$, on considère la valeur $d_{L}$. Il y a deux cas:\\
$-$ Soit il existe $\alpha\in \Phi(Z_M,{L})$ tel que $\mu_{0}^{\alpha^{\vee}}\neq 1$, alors pour tout $\lambda\in X_{L}^{{G}}$  (en particulier on a $\lambda^{\alpha^\vee}=1$), par la définition de $d_{Q}$ quand $Q\in \mathcal{P}(L)$ (cf. \ref{GLfa}), on a: $$d_{Q}(\lambda)=0\quad \forall\ Q\in \cal{P}(L), $$
car il existe toujours un sous-groupe parabolique $P\in\mathcal{P}^{Q}(M)$ tel que $\alpha\in \Delta_{P}$. Dans ce cas on a $d_L=0$ pour tout $L\in \mathcal{L}(M)$. 
\\
$-$ Soit $\mu_{0}^{\alpha^{\vee}}= 1$ pour $\forall\ \alpha\in \Phi(Z_M,{L})$, alors $\mu_{0}\in X_{L}^{{G}}$ et pour $\lambda\in X_L^G$ et tout $Q\in\cal{P}(L)$, on a
$$d_{Q}(\lambda)=f_L(\mu_0) \theta_{Q}(\lambda)\theta_{Q}(\lambda{\mu_{0}})^{-1} ,  $$
où $$f_L(\mu_0)= \prod_{\alpha\in \Delta_{P\cap L}^{L} }\frac{\langle\lambda, \alpha^{\vee}\rangle}{\langle \lambda\mu_0, \alpha^{\vee}\rangle}, $$
qui ne dépend pas de $\lambda$, ni de $P$.  
D'où quand $L\neq {G}$
$$d_L=\sum_{Q\in \mathcal{P}({L})}\theta_{Q}(\lambda)^{-1}d_{Q}(\lambda)= f_L(\mu_0)\sum_{Q\in\mathcal{P}(L)} \theta_{Q}(\lambda\mu_{0})^{-1}\equiv 0\quad \forall\ \lambda\in X_{L}^{{G}} .$$
Quand $L={G}$, c'est-à-dire $\mu_{0}\in X_{{G}}^{{G}}$. On a $\theta_{{G}}= 1$ et $d_{{G}}=\mu_{0 1}^{-\dim\ago_{P}^{{G}}}$ (où $\mu_{01}$ est la première coordonnée de $\mu_{0}$ par le plongement $X_M^{G}\hookrightarrow (\mathbb{C}^\times)^{\dim \ago_M}$). 

Par la définition de $d_Q(\lambda)$ et l'hypothèse que $c_{Q}(\lambda\mu_{0})=c_{Q}(\lambda)$, on a
$$ \sum_{Q\in \mathcal{P}(M)} \theta_{Q}(\lambda\mu_{0})^{-1}{c_{Q}( \lambda\mu_{0})}=  \sum_{Q\in \mathcal{P}(M)} \theta_{Q}(\lambda)^{-1}d_Q(\lambda){c_{Q}( \lambda)}. $$ 
Pour conclure, il suffit d'appliquer la formule du produit (la proposition \ref{SCIN}) aux  $({G},M)$-familles $(c_Q(\lambda))_{Q\in\mathcal{P}(M)}$   et $(d_Q(\lambda))_{Q\in \mathcal{P}(M)}$.
Comme $d_L$ est toujours nul si $L\in \mathcal{P}(M)$ mais $L\neq G$, on obtient 
$$\lim_{\lambda\rightarrow 1} \sum_{Q\in \mathcal{P}(M)} \theta_{Q}(\lambda\mu_{0})^{-1}{c_{Q}( \lambda\mu_{0})} = d_G c_M^G. $$ 
Pour que $d_G$ soit non-nul, il faut que $\mu_0\in X_G^G$ et dans ce cas il est égal à $\mu_{01}^{-\dim\ago_{M}^{{G}}}$.  \end{proof}

\section{La partie spectrale de la formule des traces}\label{S5}
 
\subsection{L'énoncé d'une expression de $J_e$  provenant du développement spectral }\label{4.2}
Dans cette section, nous allons énoncer une expression pour $J_e$ (la proposition \ref{expr1}) venant du développement spectral de Lafforgue (le corollaire \ref{Trace}).

\subsubsection{Un choix auxiliaire: $\kappa\in \ago_B$}\label{1theta}
 Outre le sous-groupe de Borel $B$, il faut faire un choix auxiliaire.
On fixe $\kappa\in \ago_B$ dans la chambre positive tel que la projection de $\kappa$ sur $\ago_L$ ne soit pas nulle pour tout sous-groupe de Levi semi-standard $L$. Par ce choix, on a $\alpha(\kappa)\neq 0$ pour tout $\alpha\in \Phi(Z_L, {G})$. L'avantage de choisir ce vecteur est que nous pouvons fixer un sous-groupe parabolique $Q_L$ dans $\mathcal{P}(L)$ de manière à ce que toutes ses racines soient positives, et donc nous pouvons fixer un isomorphisme $L\cong G_{n_1}\times G_{n_2}\times\cdots\times G_{n_r}$ (i.e. en utilisant l'élément de Weyl $w$ qui conjugue $Q_L$ à un sous-groupe parabolique standard, on a $w(L)\cong G_{n_1}\times G_{n_2}\times\cdots\times G_{n_r}$). 

On a besoin d'un tel choix même si on ne considère que des sous-groupes paraboliques standards. À la page 305 de \cite{Laff}, pour une permutation $\sigma$, Lafforgue a introduit un sous-groupe parabolique standard $P_{\sigma}$ et une autre permutation $\tau_{\sigma}$. La paire $(P_{\sigma}, \tau_{\sigma})$ implique un choix.

\subsubsection{Une expression de $J_e$ provenant du développement spectral} 
Soit $(P,{\pi})$ une paire discrète partout non-ramifiée (cf. \ref{234}) avec $M_P\cong G_{n_1}\times\cdots\times G_{n_r}$. Supposons que ${\pi}=\Pi_1\otimes\cdots\otimes\Pi_r$ où $\Pi_i$ est une représentation discrète de $G_{n_i}(\AAA)$.  Soit $\beta\in \Phi(Z_{M_P}, {G})$ le racine $\frac{1}{n_i}\det_{M_P, i}-\frac{1}{n_j}\det_{M_P, j}$ (cf. \ref{coor} pour notation).  On notera pour tout $z\in \mathbb{C}$: 
\begin{equation} \label{nbetapiz}\gls{nbeta}=q^{(1-g)n_in_j} \frac{{L}(\Pi_{i}\times \Pi_{j}^{\vee}, z)}{{L}(\Pi_{i}\times \Pi_{j}^{\vee}, q^{-1} z)}. \end{equation}
où $L(\Pi_i\times \Pi_j^{\vee}, z)$ est la fonction $L$ de paire (cf. par exemple \cite[Appendice B]{Lafforgue} pour une définition). 
De plus, soit $(P,{\pi})$ une paire discrète partout non-ramifiée, $w\in W_n/W^P$ et $\lambda\in X_{M_P}^{G}$,
on pose \begin{equation}\gls{npi}=\prod_{\substack{ \beta\in \Phi(Z_M, {G})\\     \beta>0; w(\beta)<0}} n_{\beta}({\pi},\lambda^{-\beta^\vee})  .\end{equation}

Soit $L$ un sous-groupe de Levi contenant $M_P$.
Pour tout $\alpha\in \Phi(Z_L, {G})$ et $z\in \mathbb{C}$,  on définit la fonction méromorphe $n_{\alpha}({\pi},  z)$ par:  
\begin{equation}\label{Zeros}
n_{\alpha}({\pi},  z)=\prod_{\{ \beta\in \Phi(Z_{M_P}, {G})|\ \beta|_{\ago_{L}^{G} }=\alpha  \} }  n_{\beta}({\pi},  z). \end{equation}

\begin{prop}\label{expr1}
Soit $(e,n)=1$, la trace tronquée $J_{e}=J^{T=0}_e$ est égale à la somme 
$$\sum_{(P,\pi)}$$
portant sur un ensemble de bons représentants $(P,{\pi})$ (voir \ref{234} pour la notion des bons représentants) des classes d'équivalence inertielle des paires discrètes partout non-ramifiées et la somme $$\sum_{(w,1)\in \stab(P,\pi)}$$
portant sur les $w\in W^P\backslash W_n/W^P$ tels que $(w,1)\in  \stab(P,\pi)$ (le groupe $\stab(P,\pi)$ est défini en-dessous de la définition \ref{pairedi})
du produit des trois facteurs $(a)$, $(b)$ et $(c)$ suivants
\begin{enumerate}\item[(a)]
$$\frac{1}{|\stab{(P,{\pi})|}}  \frac{1}{|w|| X_{L}^{L}|} $$
où 
$(1)$ $L=L^w$ est le plus petit sous-groupe de Levi  contenant $M$ tel que $w$ agit trivialement sur $\ago_L$
(cf.   \ref{Lw}); \\ 
$(2)$ $\gls{w}$ est le produit des longueurs des cycles de $w$; \\
$(3)$ $|\mathfrak{stab}(P,{\pi})|$ (resp. $|X_{L}^{L}|$) est le cardinal du groupe $\mathfrak{stab}(P,{\pi})$  (resp. du groupe $X_{L}^{L}$ des caractères inertiels de $L(\AAA)$ triviaux sur $Z_{L}(\AAA)$ ); 
\item[(b)]
$$ \sum_{\lambda_\pi\in \Fix(\pi)}
\sum_{  \substack{      \lambda \in \Im X_{G}^{{G}}, \lambda^{L}\in (\Im X_{M_{P}}^{L})^{{\circ}}       \\        \lambda_{{{\pi}}}=\lambda^{L}\lambda     }   } \sum_{\substack{ \lambda_{w}\in \Im X_{M_P}^{L}   \\ \lambda_{w}/w^{-1}(\lambda_{w})=\lambda^{L}} } n_{{{\pi}}}( w, w^{-1}(\lambda_{w})  ) {\lambda_{1}^{-e} }$$    
où $(1)$ la deuxième somme est prise sur l'ensemble des couples $( \lambda,  \lambda^{L})\in  \Im X_{G}^{{G}}\times  (\Im X_{M_{P}}^{L})^{{\circ}} $ tels que $\lambda_{{{\pi}}}=\lambda^{L} \lambda $, où $(\Im X_{M_{P}}^{L})^{{\circ}} $ est la composante connexe d'identité de $\Im X_{M_{P}}^{L}$; 
\\
$(2)$ la troisième somme est prise sur l'ensemble des éléments $\lambda_w\in \Im X_{M_P}^{L}$ tels que
$\lambda_{w}/w^{-1}(\lambda_{w})=\lambda^{L} $.  \\
%$(3)$ 
$(3)$$\lambda_1=\lambda^{\deg \det x}$ pour n'importe quelle matrice $x\in G(\AAA)$ de degré $1$. 
\item[(c)]
$$ \sum_{F} \prod_{\beta\in F }(\mathrm{N}(n_{\beta}({\pi}, \cdot))- \mathrm{P}(n_{\beta}({\pi}, \cdot))) $$
où $(1)$ la première somme porte sur toutes les parties $F$ de $-\Phi_{Q_{L}}$ telles que $F$ forme une base de $\ago_{L}^{{G},*}$ ($Q_L\in \mathcal{P}(L)$ est un sous-groupe parabolique fixé dans \ref{1theta}).  \\
$(2)$ Les opérateurs $\mathrm{N}$ et $\mathrm{P}$ donnent respectivement le nombre de zéros et de pôles dans la région $|z|<1$.    
\end{enumerate} 
\end{prop}

La preuve se trouve dans \ref{4.3}.

\subsection{Le développement spectral de la formule des traces}\label{spectral}
Nous donnons le développement spectral de Lafforgue. 
Le lecteur trouvera quelques différences dans la formulation de cet article et celle de \cite[Section IV.2]{Laff}. Elles viennent du fait que dans \cite{Laff}, on ne considère que des sous-groupes paraboliques standards et les objets sont indexés par des permutations. Nous choisissons d'utiliser les notations d'Arthur et de changer le point de vue en passant des sous-groupes paraboliques standards aux sous-groupes paraboliques semi-standards. 

\subsubsection{Fonctions $\hat{\mathbbm{1}}_{Q}(\lambda)$ et $\hat{\mathbbm{1}}_{Q}^{e}(\lambda)$} \label{2theta}
\label{a_L}
La fonction  $\hat{\mathbbm{1}}_{Q}(\lambda)$  est définie par Lafforgue dans la preuve de Lemme 5, p.301 de  \cite{Laff} (qu'il a notée $\hat{\mathbbm{1}}_{P, \tau}^{p}(\lambda)$ pour un sous-groupe parabolique standard $P$, une permutation $\tau$ et un polygone $p$ qui doit être $0$ dans notre cas). On rappelle ses définitions dans la suite et on donne dans la proposition \ref{1Qexplicit} une formule explicite que le lecteur peut aussi prendre comme définition.

Soit $M$ un sous-groupe de Levi standard. Soit $Q\in \mathcal{P}(M)$ et $s\in W_n/W^Q$ tel que $ sQs^{-1}$ soit standard. Soit $\phi_Q$ la fonction caractéristique des $H\in \ago_{s(M)}$ tels que pour tout $\alpha\in \Delta_Q$, $$\varpi_{s(\alpha)}(H)\leq 0 \quad \text{si }  \alpha(\kappa)>0,$$ et  $$\varpi_{s(\alpha)}(H)> 0 \quad \text{si }  \alpha(\kappa)<0,$$ où $\varpi_{s(\alpha)}\in \hat{\Delta}_{sQs^{-1}}$ est dual à $s(\alpha)^{\vee}$ (cf. \ref{raciness}). 
Soit $\varepsilon(Q)$ le nombre des  $\alpha\in \Delta_Q$ tels que $\alpha(\kappa)<0$. 
La fonction 
$$\gls{1Q}(\lambda)$$ sur $X_M^G$ est définie par prolongement analytique de la série 
\begin{equation}
(-1)^{\varepsilon(Q)}\sum_{ H \in \ago_{M, \mathbb{Z}}/ X_{*}(Z_G)} \phi_Q(s(H)) \lambda^{-H} , 
\end{equation}
qui converge sur la région des $\lambda\in X_M^G$, tels que $|\lambda^{\alpha^{\vee}}|<1$ pour tout $\alpha\in \Phi(Z_M, G)$ tel que $\alpha(\kappa)>0$ (notons que cette région rencontre toutes les composantes connexes de $X_M^G$),   
où $X_{*}(Z_G)$ est identifié à $\Hom_{\mathbb{Z}}(X^{*}(Z_G), \mathbb{Z})$ qui est vu comme un sous-groupe de $\Hom_{\mathbb{Z}}(X^{*}(Z_M), \mathbb{Z})\subseteq  \ago_{M, \mathbb{Z}}$. 

Soit $L$ un sous-groupe de Levi semi-standard.  Soit $w\in W_n/W^{Q_L}$ l'unique élément tel que $wQ_Lw^{-1}$ soit standard. Pour tout $Q\in \mathcal{P}(L)$, la fonction $\hat{\mathbbm{1}}_Q(\lambda)$ est définie par $$\hat{\mathbbm{1}}_Q(\lambda):= \hat{\mathbbm{1}}_{wQw^{-1}}(w(\lambda)),\quad\forall \lambda\in X_L^G.  $$

On a aussi besoin d'une variante $\hat{\mathbbm{1}}_{Q}^{e}(\lambda)$  (``la partie de degré $e$" de $\hat{\mathbbm{1}}_{Q}(\lambda)$) pour tout $e\in \mathbb{Z}$. Soit $\zeta$ une racine $n$$^{i\text{è}me}$ primitive de l'unité. 
Soit $\eta$ le caractère de $G(\AAA)$ défini par $\eta(g)=\zeta^{\deg\det g}$. C'est un élément de $X_{G}^{G}$. 
Pour chaque $e\in \mathbb{Z}$, soit \begin{equation}\gls{1Qe}(\lambda)=\frac{1}{n} \sum_{k=1}^n  \zeta^{ek} \hat{\mathbbm{1}}_{Q}(\lambda  \eta^{k}  ), \quad \forall \lambda\in X_L^G. \end{equation}
Il s'ensuit que
\begin{equation}\hat{\mathbbm{1}}_{Q}(\lambda)=\sum_{e=0}^{n-1}\hat{\mathbbm{1}}_{Q}^{e}(\lambda)  .\end{equation}

Soit $M\cong G_{n_1}\times\cdots \times G_{n_r}$ un sous-groupe de Levi standard. L'ensemble $\mathcal{P}(M)$ est en bijection avec l'ensemble $\mathfrak{S}_{r}$. En effet, pour tout $P\in \mathcal{P}(M)$, il existe un $w\in W_n$ tel que $wPw^{-1}$ soit standard. Alors $w$ permute les facteurs de $M$ et s'identifie à un élément de $\mathfrak{S}_r$ qui ne dépend pas de choix de $w$.  
Plus généralement, le choix de $\kappa$ dans \ref{1theta} nous permet d'établir une bijection  entre $\mathcal{P}(L)$ et $\mathfrak{S}_r$ pour un sous-groupe de Levi semi-standard à $r$-facteurs.% (on conjugue $L$ à un sous-groupe de Levi standard par le Weyl element qui conjugue $Q_L$ à un sous-groupe parabolique standard).   

\begin{prop}\label{1Qexplicit}
Soit $L$ un sous-groupe de Levi semi-standard défini sur $F$ et $L\cong G_{n_1}\times G_{n_2}\times \cdots \times G_{n_r}$ par le choix de $\kappa$ (cf. \ref{1theta}). 
Soit $Q\in \mathcal{P}(L)$, et $s\in \mathfrak{S}_r$ qui lui correspond. On a
\begin{align}\hat{\mathbbm{1}}^{e}_{Q}(\lambda)&=\lambda^{{\HE}_{Q}^{e}}\prod_{\alpha\in\Delta_{Q}}\frac{1}{1-\lambda^{\alpha^{\vee}}} ,
\end{align} où  $$\glslink{tHQe}{\mathrm{H}^e_{Q}}=s^{-1}([\frac{e}{n}   r_{s}^{0}]-[\frac{e}{n}r_{s}^{1} ],\ldots, [\frac{e}{n}r_{s}^{r-1}]-[\frac{e}{n}r_{s}^{r} ])\in \ago_L^{G},$$
et $r_{s}^{i}=n_{s^{-1}(1)}+\cdots+n_{s^{-1}(i)}$ et $[x]$ pour la partie entière de $x$, i.e. le plus grand entier $m$ tel que $m\leq x$. En particulier, lorsque $e=-1$  on a 
\begin{equation}\label{degree-1}
\hat{\mathbbm{1}}^{-1}_{Q}(\lambda)=\left((-1)^{\dim \mathfrak{a}_{Q}^{{G}}} \prod_{i=1}^{r}\lambda_{i}\right)\theta_{Q}(\lambda)^{-1} .\end{equation}
\end{prop}
\begin{proof}
Par définition, il suffit de considérer le cas où $L$ est standard. 

Soit $Q\in \mathcal{P}(L)$ et $s\in \mathfrak{S}_r$ qui lui est associé. 
Notons que chaque $H\in \ago_{L,\mathbb{Z}}$ peut être écrit uniquement comme
$$H=s^{-1}(e,0,\ldots, 0)+\sum_{\alpha\in\Delta_{Q}}m_{\alpha}\alpha^{\vee},$$
pour $(e,0,\ldots, 0)\in \ago_{s(L), \mathbb{Z}}$. L'assertion résulte d'un calcul direct en utilisant la base $\hat{\Delta}_{sQs^{-1}} $ précisée dans \ref{Pa3} et de ce que pour tout $z\in \mathbb{C}$ et $n\in \mathbb{Z}$, $$\sum_{m>n} z^m = \frac{z^{n+1}}{1-z}, \quad \text{si } |z|<1,  $$
et $$  \sum_{m\leq n} z^m=\frac{z^n}{1-1/z}= -\frac{z^{n+1}}{1-z}, \quad \text{si } |z|>1. $$
\end{proof}

\subsubsection{Opérateurs d'entrelacement}\label{api}
On introduit les opérateurs d'entrelacement pour fixer les notations. On peut consulter \cite[p.284-p.287]{Laff} ou \cite[§5.1 - §5.3]{Labe} pour plus de détails et des références. 

Soit $(P,{\pi})$ une paire discrète de $G$. Soit $M$ le sous-groupe de Levi de $P$. Soit $R\in\mathcal{P}(M)$. Soit $\mathcal{A}_{R}$ l'espace des fonctions complexes sur $M(F)N_{R}(\AAA)\backslash G(\mathbb{A})/K$. 
Soit $\glslink{APi}{\mathcal{A}_{R, {\pi}}}$ l'induite de ${\pi}$ dans $\mathcal{A}_{R}$. C'est-à-dire que l'espace sous-jacent de $\mathcal{A}_{R, {\pi}}$ est constitué des fonctions $\phi$ sur $M(F)N_R(\AAA)\backslash G(\AAA)/K$  telles que la fonction 
$$m\in M(\AAA)\mapsto \rho_{P}(m)^{-1} \phi(m)   $$ est dans l'espace de ${\pi}$, où $\gls{rhoP}$ est la 
racine carrée positive du caractère modulaire de $M(\AAA)$ agissant sur $N_P(\AAA)$. 
% par lequel $M_P(\AAA)$ agit sur les mesures de Haar $\d n$ de $N_P(\AAA):$ $ m.\d n.m_P^{-1} = \rho_P^{2}(m) \d n$, $\forall m\in M_P(\AAA)$. 

Soit $w\in W_n$ et  $M'=wMw^{-1}$. 
Soient $R\in\mathcal{P}(M)$ et $R'\in\mathcal{P}(M')$. On définit l'opérateur d'entrelacement $\glslink{Mw}{\M_{R'|R}(w,\lambda)}:\mathcal{A}_{R, {\pi}}\rightarrow \mathcal{A}_{{R}', w({\pi})}$ par prolongement analytique de l'intégrale suivante qui définit $(\glslink{Mw}{\M_{R'|R}(w,\lambda)}\varphi)(g)$: 
\begin{equation}\label{entre}w(\lambda)^{-H_{R'}(g)}\int_{w N_{R}(\mathbb{A})w^{-1}\cap N_{R'}(\mathbb{A})\backslash N_{R'}(\mathbb{A})}\varphi(w^{-1}ng) \lambda^{H_R(w^{-1}ng)} \d n, \quad \end{equation} 
pour $\varphi\in \mathcal{A}_{R, {\pi}}$ et $\lambda\in X_M^{G}$ dans l'ouvert de la convergence. 

L'opérateur $\M_{R'|R}(1,\lambda)$ sera parfois noté simplement $\glslink{Mw}{\M_{R' | R}(\lambda)}$:
$$\glslink{Mw}{\M_{R' | R}(\lambda)}:=\M_{R'|R}(1,\lambda). $$ 
Quand, $R$ et $R'$ sont standards, et $P$ fixé, la donnée du couple $(w, \lambda)$ suffit à déterminer l'opérateur $\M_{R'|R}(w,\lambda)$ qui sera parfois noté simplement par $\glslink{Mw}{\M(w,\lambda)}$:
$$\glslink{Mw}{\M(w,\lambda)}:=\M_{R'|R}(w,\lambda). $$ 

Pour $\lambda, \mu \in X_M^{G}$, on désignera l'opérateur composé par
\begin{equation} \glslink{MQ}{\mathcal{M}_R(\lambda, P;\mu)}:=\M_{R|P}(\lambda)^{-1}\circ \M_{R|P}(\lambda/\mu).\end{equation}
Soit  $L$ un sous-groupe de Levi contenant $M$ et $Q\in \mathcal{P}(L)$.  Soit $$\glslink{PRM}{\mathcal{P}^{Q}(M)}$$ l'ensemble des sous-groupes paraboliques de $G$ contenus dans $Q$ et admettant $M$ comme un sous-groupe de Levi.
Pour tout $\lambda\in \Im X_{M}^{G}$, 
la famille des opérateurs $$( \mu\mapsto \mathcal{M}_Q(\lambda, P; \mu))_{Q\in \mathcal{P}(M)}$$ est une $({G},M)$-famille à valeurs dans un espace vectoriel topologique sur une région contenant   $\Im X_M^{G}$ (cf. \cite[5.3.2]{Labe}). Donc pour $\mu\in X_L^{G}$,  l'opérateur $\mathcal{M}_R(\lambda, P; \mu)$ ne dépend que de l'ensemble $\glslink{PRM}{\mathcal{P}^{Q}(M)}$ dans lequel $R$ se trouve. Dans ce cas, on notera cet opérateur comme $$\mathcal{M}_Q(\lambda, P; \mu).$$

\subsubsection{L'énoncé de la formule des traces d'Arthur-Lafforgue} 
Soit $\eta\in X_G^G$. Il existe une racine $n^{\text{ième}}$ de l'unité $\zeta$ tel que $\eta(g)=\zeta^{\deg\det g}$ pour tout $g\in G(\AAA)$. 
On va considérer l'intégrale de la fonction $g\mapsto \eta(g)k^{T}(g,g)$ sur $G(F)\backslash G(\mathbb{A})/\Xi_{{G}}$:  \begin{equation}J^{T}_{\eta}:=\int_{G(F)\backslash G(\mathbb{A})/\Xi_{{G}}}\eta(g)k^{T}(g,g)\d g=\sum_{e=1}^{n}\zeta^{e} J^{T}_{e}.\end{equation}
où $k^T(x,y)$ est le noyau tronqué d'Arthur (défini par l'expression (\ref{noyauf})) et 
$$J_e^{T}:=\int_{G(F)\backslash G(\mathbb{A})^{e}} k^{T}(x,x)\d x .
$$

L. Lafforgue a donné une expression spectrale pour $J^{T}_{\eta}$ quand $\eta$ est trivial (\cite[Théorème 11, p.307]{Laff}). Le changement nécessaire lorsque $\eta$ n'est pas trivial est modeste. Prenons garde qu'il y a des petites erreurs dans l'énoncé de ce théorème, qui viennent  de la preuve de \cite[Théorème 11, p.307]{Laff} lorsqu'on applique le développement de Fourier avec les coefficients de Fourier obtenus dans \cite[Lemme 9, p.306]{Laff}. Lafforgue a omis les calculs, donc nous détaillons ces calculs en utilisant le  lemme 9 de $loc.$ $cit.$ pour corriger ces petites erreurs. Les différences de notations sont expliquées dans la preuve. Notre utilisation des notations d'Arthur rend, espérons-le, la formule plus intuitive et les calculs plus faciles.  Des notations  sont expliquées dans la proposition \ref{expr1}.

\begin{theorem}[L. Lafforgue]\label{Laffo}
La trace tronquée tordue $J^{ }_{\eta}$ est égale à la somme portant sur un ensemble de représentants des classes d'équivalence inertielle des paires discrètes $(P,{\pi})$ partout non-ramifiées, et la somme sur les couples $(w,\lambda_{{\pi}}) \in \stab(P,{\pi})$, de l'intégrale suivante
\begin{multline} \frac{1}{|\stab{(P,{\pi})|}}  \int_{\Im X_{L}^{{G}}}
\frac{1}{|w|| X_{L}^{L}|}  
\sum_{  \substack{      \lambda_{L}\in \Im X_{L}^{{G}}, \lambda^{L}\in (\Im X_{M_{P}}^{L})^{{\circ}}       \\        \lambda_{{{\pi}}}=\lambda^{L}\lambda_{L}     }   } \sum_{\substack{ \lambda_{w}\in \Im X_{M_P}^{L}   \\ \lambda_{w}/w^{-1}(\lambda_{w})=\lambda^{L}}          }  \\ \lim_{\mu\rightarrow 1}  \mathbf{Tr}_{\mathcal{A}_{P,{\pi}}}(   
 \sum_{Q\in \cal{P}(L)} \hat{\mathbbm{1}}_{Q}({\mu\lambda_{L}}{\eta}^{-1})  \mathcal{M}_{Q}({\lambda}{ \lambda_w}, P; {\mu}{ \lambda_{L}} ) \circ \M(w,  w^{-1}({\lambda_{w}}))\circ \lambda_{{\pi}}) \d\lambda.
\end{multline} 
Rappelons que $$ \mathcal{M}_{Q}({\lambda}{ \lambda_w}, P; {\mu}{ \lambda_{L}} )= \M_{R|P}(\lambda\lambda_w)^{-1}\circ \M_{R|P}(\lambda\lambda_w/\mu\lambda_L ), $$
pour n'importe quel $R\in\mathcal{P}^Q(M_P)$. 
\end{theorem}
\begin{proof}[Remarques sur la démonstration de Lafforgue]
Pour corriger les  petites erreurs, nous allons commencer par le lemme 9, p. 306 de \cite{Laff} (ainsi que le corollaire 10 qui le suit).  Expliquons tout d'abord le lien entre les notations de Lafforgue et les nôtres. 

Soit $M$ un sous-groupe de Levi standard d'un sous-groupe parabolique standard $P$. Soit $|P|$ le nombre des facteurs de $M$. 
Soit $\sigma\in \mathfrak{S}_{|P|}$,  Lafforgue utilise $P_{\sigma}$ pour le sous-groupe parabolique standard dont le sous-groupe de Levi standard est conjugué à $L^{\sigma}$ (cf. \ref{Lw}) par un élément de Weyl $\tau_\sigma$ (qui dépend d'un choix, cf.  \ref{1theta}). Soit $Q\in \mathcal{P}(M)$ et $\sigma\in \mathfrak{S}_r$ la permutation associée (cf. la discussion avant la proposition \ref{1Qexplicit}), la fonction $\hat{\mathbbm{1}}_{Q}$ coïncide avec $\hat{\mathbbm{1}}_{P, \sigma}$ définie par Lafforgue. Les deux familles de fonctions sur $X_{L^\sigma}^{G}$  $(\hat{\mathbbm{1}}_{Q}(\mu))_{Q\in \mathcal{P}(L^\sigma)}$ utilisée dans cet article et $(\hat{\mathbbm{1}}_{P_\sigma, \tau}(\tau_{\sigma}(\mu)))_{\tau\in \mathfrak{S}_{|P_{\sigma}|}}$ définie par Lafforgue coïncident. De plus, pour une paire discrète $(P,{\pi})$, Lafforgue a utilisé $\Fix({\pi})$ pour notre $\mathfrak{stab}(P,{\pi})$; $\Lambda_P$ pour notre $X_M^{G}$;
$\M_{P,\tau}^{\tau(P)}(\varphi, \lambda)$ pour notre $\M(\tau, \lambda)\varphi $, où $\tau(P)$ est le sous-groupe parabolique standard admettant $\tau M\tau^{-1}$ comme sous-groupe de Levi. 

La présence de $\eta$ ne change pas la preuve de \cite[VI §2]{Laff} jusqu'au début de $(e)$ (page 304). On peut faire un changement de variable $\mu_Q\rightarrow \mu_Q\eta^{-1}$ dans le début de $(e)$ de \cite[VI §2]{Laff}, alors la preuve fonctionne de manière identique jusqu'au \cite[Lemme 9, p.306]{Laff}. On déduit que $J^{ }_{\eta}$ est égal à la somme portant sur toutes classes d'équivalence inertielle des paires discrètes partout non-ramifiées $(P, \pi)$ et l'ensemble des caractères continus $\chi$ de $\Im X_{M_P}^{G}$ de l'expression: 
\begin{multline}   \label{First}
\frac{1}{|\stab(P,{\pi})|} \sum_{(w,\lambda_{{\pi}}) \in \stab(P,{\pi})}   \lim_{\substack{\mu_0\in X_{L^w}^{G} \\   \mu_0\rightarrow 1  }  }         \sum_{Q\in \cal{P}(L^{w})}  \int_{\Im X_{L^w}^{{G}}} \int_{\Im X_{M_P}^{{G}}}   
   \hat{\mathbbm{1}}_{Q}({\mu\mu_{0}}{\eta}^{-1})   \\    \chi(\frac{\lambda  w(\lambda_{{\pi}})\mu_{0}\mu}{w(\lambda)})    \textbf{Tr}_{\mathcal{A}_{P,{\pi}}}( \mathcal{M}_Q(\lambda, P; \mu\mu_0)\circ \M(w^{-1}, {w(\lambda)} ) \circ w(\lambda_{{\pi}})^{-1}  )   \d\lambda   \d\mu
\end{multline} 
où on a utilisé le fait que les opérateurs d'entrelacement $\M_{Q|P}(w,\lambda)$ sont des isométries quand $\lambda\in \Im X_{M_P}^{{G}}$ (cf. \cite[5.2.2.(i)]{Labe}), l'équation fonctionnelle \cite[5.3.4.(1)]{Labe}, et les relations triviales suivantes (cf. \cite[5.3.1]{Labe}): $$w(\lambda_0)^{H_P(\cdot)}\M(w,\lambda)(\varphi)= \M(w,\lambda\lambda_0^{-1})(\varphi\lambda_0^{H_P(\cdot)}); $$ 
$$\M(w,\lambda\mu)=\M(w, \lambda) \quad \forall \mu\in X_{L^w}^{G}. $$

En utilisant  la décomposition $\Im X_{M_{P}}^{{G}}=\Im X_{L^w}^{{G}} \Im X_{M_{P}}^{L^w}$ et $$\int_{\Im X_{M_{P}}^{{G}}}f(\lambda)\d\lambda=\int_{\Im X_{L^w}^{{G}}}\int_{\Im X_{M_{P}}^{L^w}} f(\lambda\lambda')   \d\lambda   \d\lambda'   , $$
et en changeant $(w,\lambda_{{\pi}})\in \stab(P,{\pi})$ par son inverse: $(w^{-1}, w(\lambda_{{\pi}})^{-1})$, l'expression (\ref{First}) est égale à
\begin{multline}   \label{PPP}
\frac{1}{|\stab(P,{\pi})|}    \sum_{(w,\lambda_{{\pi}}) \in \stab(P,{\pi})}   \lim_{\substack{ \mu_0\in X_{L^w}^{G}  \\ \mu_{0}\rightarrow 1}}\sum_{Q\in \cal{P}(L^{w})} \int_{\Im X_{L^w}^{{G}}} \int_{\Im X_{L^w}^{{G}}} \int_{\Im X_{M_P}^{L^w}}  \\
  \hat{\mathbbm{1}}_{Q}({\mu\mu_{0}}{\eta}^{-1})   \chi({\frac{\lambda' \mu_{0}\mu}{w^{-1}(\lambda') \lambda_{{\pi}}}})     \textbf{Tr}_{\mathcal{A}_{P,{\pi}}}( \mathcal{M}_Q(\lambda\lambda', P; \mu\mu_0)\circ \M(w, w^{-1}(\lambda') ) \circ \lambda_{{\pi}} )  \d\lambda'    \d\mu  \d\lambda  .
\end{multline}

On considère la suite exacte courte:
$$0\rightarrow \ker\mu_{w} \rightarrow \Im X_{L^w}^{{G}}\oplus \Im X_{M_{P}}^{L^w}\xrightarrow{\mu_{w}} \Im X_{M_{P}}^{{G}}\rightarrow 0$$
où $\mu_{w}$ est le morphisme qui envoie $(\mu,\lambda')\in \Im X_{L^w}^{{G}}\oplus \Im X_{M_{P}}^{L^w}$ sur $\lambda'\mu/w^{-1}(\lambda')\in \Im X_{M_P}^{{G}}$. La surjectivité de $\mu_w$ résulte du fait que
le morphisme 
 $$\lambda'\longmapsto \frac{\lambda^{\prime}}{w^{-1}(\lambda^{\prime})}\ ,\quad \Im X_{M_{P}}^{L^w}\longrightarrow \Im X_{M_{P}}^{L^w} $$ a pour image $(\Im X_{M_P}^{L^w})^{\circ}$, la composante connexe neutre de $\Im X_{M_P}^{L^w}$, et  le groupe $\Im X_{L^w}^{{G}}$ rencontre chaque composante connexe de $\Im X_{M_{P}}^{L^w}$.

D'après cette suite exacte, les caractères de $\Im X_{M_{P}}^{{G}}$ sont les caractères de $\Im X_{L^w}^{{G}}\oplus \Im X_{M_{P}}^{L^w}$ qui sont triviaux sur $\ker\mu_{w}$. On appliquera l'identité suivante qui découle du développement de Fourier: \begin{equation}\label{Fourier}\sum_{\chi}  \int_{  \Im X_{L^w}^{{G}}\oplus \Im X_{M_{P}}^{L^w}} \chi(\frac{\mu_w(\lambda)}{ \lambda_{{\pi}}}) f(\lambda) \d \lambda =   \frac{1}{|\ker\mu_w|}\sum_{\lambda_0\in \mu_w^{-1}(   \lambda_{{\pi}}   )}   f({\lambda_0} ) , \end{equation}
où $f$ est une fonction  lisse sur $  \Im X_{L^w}^{{G}}\oplus \Im X_{M_{P}}^{L^w} $, et la somme $\sum_{\chi} $ porte sur les caractères continus de $\Im X_{M_P}^{G}$ et $\d \lambda$ est la mesure produit. 
Concernant le cardinal du $\ker\mu_{w}$, observons que le morphisme 
 $$\lambda'\longmapsto \frac{\lambda^{\prime}}{w^{-1}(\lambda^{\prime})}\ ,\quad \Im X_{M_{P}}^{L^w}\longrightarrow (\Im X_{M_{P}}^{L^w}){^{{\circ}}}   . $$ 
a pour noyau $\Im X_{L^w}^{L^w} = X_{L^w}^{L^w}$ et %$\mu_w$ se factorise par 
%$$ \Im X_{L^w}^{{G}}\oplus \Im X_{M_{P}}^{L^w}\xrightarrow{\mathrm{id}\oplus(\lambda\mapsto \lambda w^{-1}(\lambda)^{-1}) } \Im X_{L^w}^{{G}}\oplus(\Im X_{M_{P}}^{L^w})^{{\circ}}\xrightarrow{(\lambda_1, \lambda_2)\mapsto \lambda_1\lambda_2}  \Im X_{M_{P}}^{{G}}  .$$
l'intersection $(\Im X_{M_P}^{L^w})^{{\circ}}\cap \Im X_{L^w}^{{G}}$ a pour cardinal $|w|$, le produit des longueurs des cycles de $w$ en tant qu'élément du groupe des permutations $\mathfrak{S}_{r}$ qui permute les facteurs de $M_P\cong G_{n_1}\times \cdots\times G_{n_r}$. Donc $|\ker\mu_{w}|=|X_{L^w}^{L^w}||w|$.

Par l'égalité (\ref{Fourier}),  la somme des expressions (\ref{PPP}) portant sur tous caractères $\chi$ de $\Im X_{M_{P}}^{{G}}$ est  égale à
\begin{multline}\frac{1}{|\stab{(P,{\pi})|}}  \sum_{(w,\lambda_{{\pi}}) \in \stab(P,{\pi})}     \int_{\Im X_{L^w}^{{G}}} \frac{1}{|w||X_{L^w}^{L^w}|} \sum_{\lambda_{{{\pi}}}=\lambda^{L^w}\lambda_{L^w}} \sum_{\lambda_{w}: \lambda_{w}/w^{-1}(\lambda_{w})=\lambda^{L^w}}    \\  \lim_{\mu\rightarrow 1} \mathbf{Tr}_{\mathcal{A}_{P,{\pi}}}(   
 \sum_{Q\in \cal{P}(L^w)}\hat{\mathbbm{1}}_{Q}({\mu\lambda_{L^w}}{\eta}^{-1})  \mathcal{M}_{Q}({\lambda}{ \lambda_w}, P; {\mu}{ \lambda_{L^w}} ) \circ \M(w,  w^{-1}({\lambda_{w}}))\circ \lambda_{{\pi}}) \d\lambda ,
\end{multline} 
où la somme sur $\sum_{\lambda_{{{\pi}}}=\lambda^{L^w}\lambda_{ L^w}}$ porte sur toutes les écritures telles que $\lambda_{L^w}\in \Im X_{L^w}^{{G}}$ et $\lambda^{L^w}\in (\Im X_{M_{P}}^{L^w})^{{\circ}}$ et la somme $\sum_{\lambda_{w}: \lambda_{w}/w(\lambda_{w})=\lambda^{L^w}}  $ porte sur tous les $\lambda_w\in\Im X_{M_P}^{L^w}$ tels que $ \lambda_{w}/w^{-1}(\lambda_{w})=\lambda^{L^w} $.

\end{proof}

Supposons que $\zeta$ est une racine $n^{\text{ième}}$ primitive de l'unité. 
On a  pour tout $k\in\mathbb{Z}$:
$$J^{T}_{\eta^{k}}=\sum_{e=1}^{n}\zeta^{ek}J^{T}_{e} . $$
Donc pour tout $e\in \mathbb{Z}$:
$$ J^{T}_{e}=\frac{1}{n} \sum_{k=1}^{n}\zeta^{-ek}J^{T}_{\eta^{k}} . $$ 
Par la définition de la fonction $\hat{\mathbbm{1}}_{Q}^{e}$, on obtient le résultat suivant:
\begin{coro}[L. Lafforgue]\label{Trace}
Avec les notations du théorème \ref{Laffo},
pour tout $e\in\mathbb{Z}$, 
$J_{e}$ est égale à la somme portant sur un ensemble de représentants des classes d'équivalence inertielle des paires discrètes $(P,{\pi})$ partout non-ramifiées, et la somme sur les couples $(w,\lambda_{{\pi}}) \in \stab(P,{\pi})$, de l'intégrale suivant
\begin{multline} \frac{1}{|\stab{(P,{\pi})|}}  \int_{\Im X_{L}^{{G}}}
\frac{1}{|w|| X_{L}^{L}|}  
\sum_{  \substack{      \lambda_{L}\in \Im X_{L}^{{G}}, \lambda^{L}\in (\Im X_{M_{P}}^{L})^{{\circ}}       \\        \lambda_{{{\pi}}} = \lambda^{L}\lambda_{L}     }   } \sum_{\substack{ \lambda_{w}\in \Im X_{M_P}^{L}   \\ \lambda_{w}/w^{-1}(\lambda_{w})=\lambda^{L}}       }  \\ \lim_{\mu\rightarrow 1}  \mathbf{Tr}_{\mathcal{A}_{P,{\pi}}}(   
 \sum_{Q\in \cal{P}(L)}\hat{\mathbbm{1}}^{e}_{Q}({\mu\lambda_{L}})  \mathcal{M}_{Q}({\lambda}{ \lambda_w}, P; {\mu}{ \lambda_{L}} ) \circ \M(w,  w^{-1}({\lambda_{w}}))\circ \lambda_{{\pi}}) \d\lambda.
\end{multline}

\end{coro}

\subsection{Preuve de la proposition \ref{expr1}} \label{4.3}
En bref:

On évalue directement les traces dans le développement spectral de la formule des traces d'Arthur-Lafforgue (Corollaire \ref{Trace}).
Soit $\pi$ une représentation automorphe discrète partout non-ramifiée de $G(\AAA)$. Par la décomposition à la Flath, $\pi\cong\otimes_{v\in |X_1|} \pi_v$, où $\pi_v$ est un $\mathcal{H}_{G, v}$-module simple. La commutativité ainsi que la finitude du nombre de générateurs de $\mathcal{H}_{G, v}$ (par l'isomorphisme de Satake) implique que $\pi_v$ est de dimension $1$. Donc $\pi$ est de dimension $1$ et les opérateurs d'entrelacement agissent par des scalaires qui peuvent être exprimés à l'aide du système de racine et de fonctions $L$ (la proposition \ref{cons}). % et le corollaire \ref{const}).  
On démontre dans \ref{identity1} que l'action du $\cdot \lambda_\pi$ sur $\pi$ est l'identité pour tout $\lambda_\pi\in \Fix(\pi)$.

\subsubsection{L'opérateur ``$\cdot \lambda_{{\pi}}$"}\label{identity1}
\begin{prop}\label{Fixer}    
Soit $(P,{\pi})$ une paire discrète partout non-ramifiée. Soit $\varphi\in \mathcal{A}_{P,{\pi}}$.  Alors $\varphi \lambda_{{\pi}}  =\varphi $ pour tout $\lambda_{{\pi}}\in \Fix({\pi})$. 
\end{prop}
\begin{remark}\label{multi2}
Soit $\pi$ une représentation automorphe cuspidale de $G(\AAA)$. Nous l'avons définie comme un $\mathcal{H}_G$-module. Soit $\lambda_\pi\in \Fix(\pi)$, i.e. on a $\pi\otimes\lambda_\pi = \pi$. Prenons garde que $\langle G(\AAA)\pi\rangle$, la représentation régulière engendrée par $\pi$, admette deux structures de représentation de $G(\AAA)$, la représentation régulière $R_G$ et sa torsion par $\lambda_\pi$:  $R_G\otimes \lambda_\pi$. On sait que ces deux représentations sont isomorphes (par la remarque  \ref{multi1}), mais cet isomorphisme n'est pas l'identité sur l'espace sous-jacent sauf si $\lambda_\pi=1$: en effet, si l'isomorphisme est une identité, on a $\varphi(y x)\lambda_\pi(x)=\varphi(y x)$ pour tout $x,y\in G(\AAA)$ et $\varphi\in \pi$ non-nulle, donc $\lambda_\pi=1$. 

C'est le point de de cette proposition.  
\end{remark}

On a besoin d'un lemme tout d'abord. 

\begin{lemm} \label{tensor}
Considérons une représentation automorphe cuspidale partout non-ramifiée $\pi$ de $G$, et $\varphi \in \pi$, alors il existe un élément $x$ de degré $0$ dans $G(\bbb{A})$ tel que $\varphi(x)\neq 0$.
\end{lemm}
\begin{proof}
Une fonction cuspidale peut être retrouvée par ses coefficients de Fourier qui sont des fonctions de Whittaker globales. On va montrer qu'un des coefficient de Fourier de $\varphi$ admet une valeur non-nulle en un élément de degré $0$, ce qui démontrera le résultat. 
Par l'unicité du modèle de Whittaker, une fonction de Whittaker globale est un produit des fonctions de Whittaker locales qui peuvent être précisées par la formule de Shintani. 

Soit $D=\sum_{v\in|X_1|}n_{v}v$ un diviseur canonique. On sait que $$\deg D=\sum_{v\in|X_1|}n_{v}\deg v=2g-2. $$ Choisissons et fixons   une uniformisante pour tout $v\in|X_1|$, on la note $\varpi_{v}$. On définit un idèle $u=(\varpi_{v}^{-n_{v}})_{v\in |X_1|}\in\mathbb{A}^{\times}$, alors par la dualité de Serre et par le fait que $X_1$ est géométriquement connexe, on a (cf. \cite[\rom{2}.5, Proposition 3]{Serre})
$$\mathbb{A}/(F+u\mathcal{O})\cong H^{1}(X_1, \mathcal{O}_{X_1}(D))\cong H^{0}(X_1,\mathcal{O}_{X_1})^{\vee}\cong\mathbb{F}_{q}$$ 
où $H^{0}(X_1,\mathcal{O}_{X_1})^{\vee}$ est le groupe  dual de $H^{0}(X_1,\mathcal{O}_{X})$.
Donc un choix d'un caractère non-trivial de $\mathbb{F}_{q}$ définit un caractère non-trivial $\psi$ de $\mathbb{A}/(F+u\mathcal{O})$ en  utilisant l'isomorphisme ci-dessus. On a une décomposition $\psi=\otimes'_{v}\psi_{v}$.

On voit que le conducteur de $\psi_{v}$ pour chaque $v\in |X_1|$ est exactement $n_v$. 
En effet, soit $v_0\in |X_1|$, il est clair que $\psi$ est trivial sur $\varpi_{v_0}^{-n_{v_0}}\mathcal{O}_{v_0}$. De plus, posons $n_{v}'=n_{v}$ pour $v\neq v_{0}$ et $n'_{v_0}=n_{v_0}+1$,  et notons le faisceau inversible associé au diviseur $v_{0}$ par $\mathcal{L}=\mathcal{O}(\{v_{0}\})$. Alors 
$$\mathbb{A}/(F+(\varpi_{v}^{-n'_{v}})\mathcal{O})\cong H^{1}(X_1, \mathcal{O}_{X_1}(D)\otimes\mathcal{L}  )\cong H^{0}(X_1,\mathcal{L}^{-1})^{\vee}=\{0\}$$ car $\deg \mathcal{L}^{-1}<0$. Comme $\psi$ n'est pas trivial, il ne se factorise pas par un caractère de $\mathbb{A}/(F+(\varpi_{v}^{-n'_{v}})\mathcal{O})$. Donc $\psi_{v_{0}}$ n'est pas trivial sur $\varpi_{v_{0}}^{-n'_{v_{0}}}\mathcal{O}_{v_{0}}$ et alors le conducteur de $\psi_{v_0}$  est  $n_{v_0}$. 

Cela implique que le caractère de $F_v$ défini par $a\mapsto \psi_v(\varpi_{v}^{-n_{v}} a)$ est de conducteur $0$, on le notera par $\varpi_{v}^{-n_{v}}\psi_v$.

On définit un caractère de $N_B(\mathbb{A})$ (resp. $N_B(F_v)$) qu'on notera encore par $\psi$ (resp. $\psi_v$) par $$\psi(n):=\psi(\sum_{i=1}^{n-1} n_{i,i+1}),$$
pour $n=(n_{ij})_{n\times n} \in N_B(\mathbb{A})$ (resp. $n=(n_{ij})_{n\times n} \in N_B(F_v)$). 
Soit $\mathcal{W}(\psi_v)$ (resp. $\mathcal{W}(\psi)$) l'espace des fonctions $$W: G(F_v)/K_v\rightarrow \mathbb{C}$$ $$ \text{(resp. }W: G(\AAA)/K\rightarrow \mathbb{C})$$  telles que  $$W(ng)=\psi(n)W(g)$$ pour tout $n\in N_B(F_v)$ et $g\in G(F_v)/K_v$ (resp. $n\in N_B(\AAA)$ et $g\in G(\AAA)/K$). C'est un $\mathcal{H}_{G,v}$-module %à droite
 (resp. $ \mathcal{H}_{G}$-module %à droite
). Comme $\pi$ est une représentation cuspidale irréductible, on sait que $\pi_v$ est générique. Par le théorème de multiplicité un local, il existe un unique $\mathcal{H}_{G_v}$-sous-module de $\mathcal{W}(\psi_v)$ qui est isomorphe à $\pi_v$. Ce sous-module est appelé le modèle de Whittaker de $\pi_v$, qu'on notera  $\mathcal{W}(\pi_{v},\psi_{v})$. Comme $\pi_v$ est de dimension $1$, $\mathcal{W} (\pi_{v},\psi_{v})$ l'est aussi.  En remplaçant $\psi_v$ par $\varpi_v^{-n_v}\psi_v$, on obtient le modèle de  
Whittaker $\mathcal{W}(\pi_v, \varpi_v^{-n_v}\psi_v)$ par rapport à $\varpi_v^{-n_v}\psi_v$.

 Pour $J=(j_{1},\ldots,j_{n})\in\mathbb{Z}^{n}$, soit $$\varpi_v^{J}=\begin{pmatrix}\varpi_{v}^{j_{1}}&&\\
&\ddots&\\
&&\varpi_{v}^{j_{n}}
\end{pmatrix}
.$$ 
Soit $$J_v=\biggr(-(n-1)n_{v},-(n-2)n_{v},\ldots, 0 \biggr).$$
Alors on a un isomorphisme de $\mathcal{H}_{G,v}$-module
\begin{align*}\mathcal{W}({\pi_v, \psi_v}) &\cong \mathcal{W}(\pi_v, \varpi_v^{-n_v}\psi_v) \\
(x\mapsto W(x))&\mapsto (x\mapsto W(\varpi_v^{J_v}x ))
\end{align*}
Comme $\pi_v$ est non-ramifiée et $\varpi_v^{J_v}\psi_v$ est de conducteur $0$. On a d'après la formule de Shintani (cf. \cite[p.116]{Cogdell}) que pour toute fonction non-nulle $W'\in \mathcal{W}(\pi_v, \varpi_v^{-n_v}\psi_v)$ 
 $$W'(1)\neq 0.$$
Donc si $W\in \mathcal{W}(\pi_v, \psi_v)$ est non-nulle, cela implique que 
$$ W(\varpi_v^{J_v})\neq 0  $$

Soit $W_v^{o}\in \mathcal{W}(\pi_v, \psi_v)$ une fonction de Whittaker normalisée telle que $$W_v^{o}(\varpi_v^{J_v})=1.$$
Alors $W(x):=\prod_{v\in|X_1|} W_v^{o}(x)$ définit une fonction dans $\mathcal{W}(\psi)$ qui engendre un sous-$\mathcal{H}_G$-module isomorphe à $\pi$. Par l'unicité globale du modèle de Whittaker, c'est le sous-module unique de $\mathcal{W}(\psi)$ isomorphe à $\pi$. 
Soit $\varphi\in\pi$ non-nulle, on sait que la fonction \begin{equation}\label{whittaker}W_{\varphi}(x)=\int_{N_{B}(F)\backslash N_{B}(\mathbb{A})}\varphi(nx)\psi^{-1}(n)\d n \end{equation}
appartient à $\mathcal{W}( {\psi})$ et engendre aussi un sous-module isomorphe à $\pi$. Donc il existe une constante $C$ non-nulle, telle que
$$W_{\varphi}(x)=C \prod_{v\in|X_1|}W_{v}^o(x) \quad \forall\ x=(x_v)_{v\in |X_1|}\in G(\AAA).$$

Soit $x_{0}=(\varpi_{v}^{J_v})_{v\in|X_1|}\in G_{n}(\mathbb{A})$. Alors  $$\deg x_{0}=-\sum_{v\in|X_1|}\frac{n(n-1)}{2} n_{v}\deg v=-n(n-1)(g-1)$$
et $W_{\varphi}(x_{0})=C\neq 0$. Choisissons un élément $a\in \AAA^{\times}$ de degré $1$, vu comme une matrice diagonale, on a $$W_{\varphi}(x_{0}a^{(n-1)(g-1)})\neq 0, $$
et $\deg x_{0}a^{(n-1)(g-1)}=0$. Vu l'équation \eqref{whittaker}, on obtient le résultat voulu.  
\end{proof}

\begin{proof}[Démonstration de la proposition \ref{Fixer}]
Par la définition de $\mathcal{A}_{P,{\pi}}$, on est réduit aux formes automorphes dans une représentation discrète de sous-groupe de Levi $M_P$ de $G$, qui est un produit de groupes généraux linéaires. Donc il suffit de considérer les groupes ${G}_n$ ($n\geq 1$). 

Soit $\lambda_{ {\pi}}\in \Fix( {\pi})$, i.e. on a $\lambda_{ {\pi}} {\pi}= {\pi}$. Donc $\varphi$ et $\varphi\lambda_{\pi}$ sont toutes deux contenues dans $\pi$. Il existe une constante $c$ tel que $c \varphi= \varphi {\lambda_{ {\pi}}}$ puisque $\dim {\pi} =1$ (cf. le début de cette sous-section). 

Si $\pi$ est cuspidale, par le lemme \ref{tensor}, il existe $x_0\in G(\AAA)$ un élément de degré $0$  tel que $\varphi(x_0)$ soit non-nulle. 
Alors la relation $$c \varphi(x_0)= \varphi(x_0) {\lambda_{ {\pi}}}(x_0)$$ implique $c=1$, c'est-à-dire $\varphi=\varphi{\lambda_{ {\pi}}}$. 

Si $ {\pi}$ n'est pas cuspidale, d'après Langlands (voir \cite[V.3.13(iii)]{Wald-Moe}), $\forall\   \varphi\in  {\pi}$, il existe une paire discrète $(P', {\pi}')$ telle que tous les facteurs de ${\pi}'$ soient cuspidaux et il existe une fonction cuspidale $\varphi'\in \mathcal{A}_{P', {\pi}'}$ et un point $\lambda'\in X_{P'}^{G} $ tels que $$\varphi(g)=\mathrm{Res}_{\lambda'}\E(\varphi', \cdot)(g)\quad\forall\ g\in G(\mathbb{A}), $$
où $\mathrm{Res}$ est un opérateur de résidus et $\E(\varphi',\lambda)(g)$  pour $\lambda\in X_{P'}^G$, $g\in G(\AAA)$ est la série d'Eisenstein définie associée à $\varphi'$(cf. \rom{2}.1.5. \cite{Wald-Moe}). 
On a $\Fix({\pi})= \Fix( {\pi}')$ par l'inclusion $X_G^{G}\subseteq X_{P'}^{G}$ (la proposition \ref{res e}).
Donc $\forall\ \lambda_{0}\in \Fix( {\pi})$, on peut réduire le problème au cas cuspidal:
 \begin{align*}\varphi(g){\lambda_{0}}(g)&=  \mathrm{Res}_{\lambda'}\E(\varphi', \cdot)(g){\lambda_{0}}(g) \\
&=   \mathrm{Res}_{\lambda'}\E(\varphi'{\lambda_{0}}, \cdot)(g)\\
&=\mathrm{Res}_{\lambda'}\E(\varphi', \cdot)(g), \\
&= \varphi(g). 
\end{align*}
\end{proof}

\subsubsection{Calculs explicites d'opérateurs d'entrelacement. }\label{constop} 
On fixe une fois pour toutes une fonction $\varphi_\Pi$ non-nulle pour chaque représentation discrète partout non-ramifiée $\Pi$ de $G_k$ ($k\geq 1$). Soit $(P,{\pi})$ une bonne paire discrète partout non-ramifiée de $G$. Cela nous permet de fixer une fonction $\varphi_\pi$ sur $M_P(\AAA)$ dans $\pi$ d'une manière évidente. Pour tout $R\in \mathcal{P}(M)$, soit $\varphi_R$ la fonction sur $N_R(\AAA)M_P(F)\backslash G(\AAA)/K$ telle que $\varphi_R(nmk)=\rho_R(m)^{-1} \varphi_\pi(m)$ pour tout $n\in N_R(\AAA)$, $m\in M_P(\AAA)$ et $k\in K=G(\mathcal{O})$.

Soit  $L$ un sous-groupe de Levi semi-standard contenant $M$. Pour tous $P, Q\in \mathcal{P}(L)$ et $\lambda\in X_M^{G}$,  notons \begin{equation} \label{snPQ}n_{Q|P}({\pi},\lambda)=\prod_{\alpha\in\Phi_{Q}\cap\Phi_{\bar{P}}} \prod_{\{ \beta\in \Phi(Z_{M_P}, {G})|\ \beta|_{\ago_{L}^{G} }=\alpha  \} }       n_{\beta}({\pi},\lambda^{\beta^{\vee}})     ,\end{equation}
rappelons que $\Phi_P=\Phi(Z_P, N_P)$. 
%Donc si $\lambda\in X_L^{G}$, on a $$n_{Q|P}({\pi},\lambda)= \prod_{\alpha\in\Phi_{{Q}}\cap\Phi_{\bar{P}}} n_{\alpha}({\pi}, \lambda^{\alpha^{\vee}}).  $$ 

\begin{prop}\label{cons}
Avec les notations ci-dessus, soit $S, R\in \mathcal{P}(M_P)$, on a 
$$\M_{R|S}(\lambda)\varphi_S=n_{R|S}(\pi, \lambda)\varphi_R.$$
En particulier, soit $(w, 1)\in \stab(P,{\pi})$, alors $$\M(w,\lambda) \varphi_P = (\prod_{\substack{  \beta\in \Phi(Z_M, {G})  \\ \beta>0; w(\beta)<0} } n_{\beta}({\pi}, \lambda^{-\beta^{\vee}}) )  \varphi_P= n_{{\pi}}(w,\lambda)\varphi_P.$$
\end{prop}
\begin{proof}
Cette proposition est bien connue des experts. 

Par l'équation fonctionnelle (cf. \cite[5.2.2.(i)(1)]{Labe}), on peut supposer sans perte de généralité que $S=P$ est standard. Soit $w\in W_n/W^M$ tel que $wRw^{-1}$ est standard. On a
$$ \M_{R|P}(\lambda)=\M_{R|wRw^{-1}}(w^{-1}, w(\lambda))\circ \M(w, \lambda). $$ 
Il est clair que  (\cite[5.2.1]{Labe} où $T_0=0$ pour $G_n$)
$$\M_{R|wRw^{-1}}(w^{-1}, w(\lambda))\varphi_{wRw^{-1}}=\varphi_R. $$

Il suffit de considérer $\M(w,\lambda)$. L'existence d'une constante $c$ telle que $\M(w,\lambda)\varphi_P=c \varphi_{wRw^{-1}}$ vient du fait que $\mathcal{A}_{P, \pi}$ et $\mathcal{A}_{wRw^{-1}, \pi}$ sont de dimension $1$. Cette constante ne dépend que de la structure de $\mathcal{H}_G$-module de $\pi$ mais pas du modèle choisi. 
Par le prolongement analytique, il suffit de considérer l'identité dans la région du convergence de l'intégrale qui définit l'opérateur d'entrelacement $\lambda\mapsto M(w, \lambda)$. 
On peut utiliser l'expansion d'Euler de l'opérateur d'entrelacement (\cite[II.1.9]{Wald-Moe}) qui réduit la question à une formule locale, souvent appelée la formule de Gindikin-Karpelevich. 
En décomposant $w$ en une composition de symétries élémentaires par rapport aux racines simples, la formule de Gindikin-Karpelevich peut être ramené à un calcul local de $SL_2$ (cf. \cite[4.2.1, 4.2.2]{Shahidi}). On renvoie le lecteur au \cite[4.3.2]{Shahidi} pour ce calcul de $SL_2$. Notons que le facteur $q^{(g-1)n_in_j}$ dans la définition \ref{nbetapiz} de $n_{\beta}(\pi, \lambda^{\beta^{\vee}})$ vient du fait que la mesure de Haar locale est fixée telle que $\vol(N_P(\mathcal{O}_v))=1$ mais la mesure de Haar globale est fixée telle que $\vol(N_P(F)\backslash N_P(\AAA))=1$. 
\end{proof}

\subsubsection{Preuve de la proposition \ref{expr1}}\label{const}
Soit $(P,{\pi})$ une bonne paire discrète partout non-ramifiée. On revient au corollaire \ref{Trace}. 

Soit $L$ un sous-groupe de Levi contenant $M$. 
Soit $Q\in \mathcal{P}(L)$. On commence par calculer l'opérateur $\mathcal{M}_Q(\lambda, P; \mu)$. 

Rappelons que $\mathcal{P}^{L}(M)$ (resp. $\mathcal{P}^{Q}(M)$) est l'ensemble des sous-groupes paraboliques de $L$ (resp. de $G$ contenus dans $Q$) qui ont pour sous-groupe de Levi $M$.
On a une bijection entre $\mathcal{P}^{Q}(M)$ et $\mathcal{P}^{L}(M)$  donnée par l'application
$$R'  \mapsto R' \cap L$$ qui a pour inverse $R  \mapsto RN_{Q}$. 
Soit $S= (P\cap L)N_{Q_L}$. 
Soit $R=(P\cap L)N_Q$. Alors $S\in\mathcal{P}^{Q_L}(M)$ et $R\in \mathcal{P}^{Q}(M)$.
Par définition, on a 
$$\mathcal{M}_Q(\lambda, P; \mu )= \M_{R|P}(\lambda)^{-1}\circ  \M_{R|P}(\lambda/\mu), $$
donc (cf. \cite[5.2.2(i)(1)]{Labe})
$$\mathcal{M}_Q(\lambda, P; \mu )= \M_{S|P}(\lambda)^{-1}\circ   \M_{R|S}(\lambda)^{-1}\circ \M_{R|S}(\lambda/\mu)\circ  \M_{S|P}(\lambda/\mu).   $$
Notons qu'on a  
$$\Phi_R \cap \Phi_{\bar{S}} =   \Phi(Z_M, N_Q\cap N_{\bar{Q}_L}). $$
Donc $$ n_{R|S}(\pi, \lambda)= n_{Q|Q_L}(\pi, \lambda).  $$

On conclut par la considération ci-dessus, la proposition \ref{Fixer} et la proposition \ref{cons} que l'expression de Lafforgue dans le corollaire \ref{Trace} de $J_{e}$ (pour $e\in \mathbb{Z}$) est égale à la somme portant sur un ensemble de bons représentants des classes d'équivalence inertielle des paires discrètes partout non-ramifiées $(P,{\pi})$ et la somme sur les couples $(w,\lambda_{{\pi}}) \in \stab(P,{\pi})$, de l'intégrale suivante 
\begin{multline}\label{533}
\frac{1}{|\stab{(P,{\pi})|}}  \int_{\Im X_{L}^{{G}}}
\frac{1}{|w|| X_{L}^{L}|}    
\sum_{  \substack{      \lambda_{L}\in \Im X_{L}^{{G}}, \lambda^{L}\in (\Im X_{M_{P}}^{L})^{{\circ}}       \\        \lambda_{{{\pi}}}=\lambda^{L}\lambda_{L}     }   } \sum_{\substack{ \lambda_{w}\in \Im X_{M_P}^{L}   \\ \lambda_{w}/w^{-1}(\lambda_{w})=\lambda^{L}}          }              \\   n_{{{\pi}}}( w, w^{-1}(\lambda_{w})  )    
\frac{n_{S|P}(\pi, \frac{\lambda\lambda_w}{\lambda_L})}{n_{S|P}(\pi, \lambda\lambda_w)}
\lim_{\mu\rightarrow 1} 
 \sum_{Q\in \cal{P}(L)} \hat{\mathbbm{1}}_{Q}^{e}({\mu}{\lambda_{L}})  
\frac{ n_{Q|Q_{L}}({\pi},  \frac{\lambda \lambda_w }{ \mu \lambda_L}) }{n_{Q|Q_{L}}({\pi}, {\lambda}\lambda_{w}  )}   \d\lambda .  \end{multline}

Pour tout $\lambda\in \Im X_L^G$, la famille
$$\biggr(c_Q(\mu)= \frac{ n_{Q|Q_{L}}({\pi},  \frac{\lambda \lambda_w }{ \mu}) }{n_{Q|Q_{L}}({\pi}, {\lambda}\lambda_{w}  ) } \biggr)_{Q\in\mathcal{P}(L)} $$  est une $({G},L)$-famille sur un voisinage de $1$. 
Notons que $c_Q(\mu)$ est de type spécial comme dans le théorème \ref{GMM}:
Soit  $$c_\alpha(z) = \begin{cases}  \prod_{\{ \beta\in \Phi(Z_{M_P}, {G})|\ \beta|_{\ago_{L}^{G} }=\alpha  \}} \frac{n_\beta(\pi, {(\lambda\lambda_w)^{\beta^\vee}}/{z})}{n_{\beta}(\pi, (\lambda\lambda_w)^{\beta^\vee})}, \quad \text{ si $\alpha\in -\Phi_{Q_L}$}; \\    1, \quad \text{ sinon}.   \end{cases} $$
Alors $$c_Q(\mu)= \prod_{\alpha\in \Phi_Q} c_\alpha(\mu^{\alpha^{\vee}}) . $$

On sait \begin{equation}\label{5.3.3}c_{Q}(\mu\lambda_{L})=c_Q(\mu),\end{equation} pour tout $\mu\in X_{L}^{G}$. 
En effet, pour tout $(w,\lambda_{{\pi}}) \in \stab(P,{\pi})$, $\lambda_\pi\in \Fix(\pi)$,  on a $$w^{-1}(\frac{ \lambda_w}{ \lambda_{L}})=  \frac{\lambda_{w}}{\lambda_{{\pi}}}.$$ 
Notons que 
\begin{equation}\label{Lcons}n_{Q|Q_L}(\pi, \mu) =n_{Q|Q_L}(\pi, \mu  \lambda_{{\pi}}). \end{equation} 
L'égalité \eqref{5.3.3} résulte alors de la relation  \begin{equation*} n_{Q|Q_L}({\pi}, \lambda)= n_{Q|Q_L}({\pi}, w(\lambda)) , \quad \forall \lambda\in X_{M_P}^G,\end{equation*} puisque l'ensemble $\{ \beta\in \Phi(Z_{M_P}, {G})|\ \beta|_{\ago_{L}^{G} }=\alpha  \}$ est stable par tout $w$ qui est contenu dans $L$.  

Le lemme \ref{SPE}  (pour $e=-1$) et le lemme \ref{SPF} (pour tout $e\in \mathbb{Z}$) s'appliquent, et la limite 
$$\lim_{\mu\rightarrow 1} 
 \sum_{Q\in \cal{P}(L)} \hat{\mathbbm{1}}_{Q}^{e}({\mu}{\lambda_{L}})  
c_Q(\mu\lambda_L)$$
s'annule si $\lambda_{L}\in  X_{L}^{{G}}-X_{G}^{{G}}$. Quand $\lambda_L\in X_G^G$, on a clairement 
$$\frac{n_{S|P}(\pi, \frac{\lambda\lambda_w}{\lambda_L})}{n_{S|P}(\pi, \lambda\lambda_w)}=1.$$
On a aussi 
$$\lim_{\mu\rightarrow 1} 
 \sum_{Q\in \cal{P}(L)} \hat{\mathbbm{1}}_{Q}^{e}({\mu}{\lambda_{L}}) c_Q(\mu\lambda_L)= \lambda_{L1}^{e} \lim_{\mu\rightarrow 1}  \sum_{Q\in \cal{P}(L)}\hat{\mathbbm{1}}_{Q}^{e}(\mu)c_Q(\mu). $$

Par \eqref{533}, on conclut que pour tout $e\in \mathbb{Z}$, $J_{e}$ est égal à la somme portant sur un ensemble de bons représentants des classes d'équivalence inertielle des paires discrètes partout non-ramifiées $(P,{\pi})$ et la somme sur les couples $(w,\lambda_{{\pi}}) \in \stab(P,{\pi})$, de l'intégrale suivante 
\begin{multline}\label{aaaz}
\frac{1}{|\stab{(P,{\pi})|}}  
\frac{1}{|w|| X_{L}^{L}|}    
\sum_{  \substack{      \lambda_G \in \Im X_{G}^{{G}}, \lambda^{L}\in (\Im X_{M_{P}}^{L})^{{\circ}}       \\        \lambda_{{{\pi}}}=\lambda^{L}\lambda_G }   } \sum_{\substack{ \lambda_{w}\in \Im X_{M_P}^{L}   \\ \lambda_{w}/w^{-1}(\lambda_{w})=\lambda^{L}}          }              \\   n_{{{\pi}}}( w, w^{-1}(\lambda_{w})  )    
\lambda_{G1}^{e} \int_{\Im X_{L}^{{G}}}    \lim_{\mu\rightarrow 1} 
 \sum_{Q\in \cal{P}(L)} \hat{\mathbbm{1}}_{Q}^{e}({\mu})  
\frac{ n_{Q|Q_{L}}({\pi},  \frac{\lambda \lambda_w }{ \mu}) }{n_{Q|Q_{L}}({\pi}, {\lambda}\lambda_{w}  )}   \d\lambda,  \end{multline}
où $\lambda_{G1}$ est le première coordonnée de $\lambda_G$ dans $X_L^G$, qui est égal à $\lambda_G^{\deg \det x}$ pour n'importe quelle matrice $x\in G(\AAA)$ de degré $1$ par définition. 

Notons que (Proposition \ref{1Qexplicit}), 
\begin{equation}
\hat{\mathbbm{1}}^{-1}_{Q}(\mu)=\left((-1)^{\dim \mathfrak{a}_{Q}^{{G}}} \prod_{i=1}^{r}\mu_{i}\right)\theta_{Q}(\mu)^{-1} .\end{equation}
En appliquant le corollaire \ref{INT} pour $e=-1$, on obtient 
\begin{equation}\label{pppo} \int_{\Im X_L^G} \lim_{\mu \rightarrow 1}\sum_{Q\in \cal{P}(L)} \hat{\mathbbm{1}}_{Q}^{e}({\mu})   \frac{ n_{Q|Q_{L}}({\pi},  \frac{\lambda \lambda_w }{ \mu}) }{n_{Q|Q_{L}}({\pi}, {\lambda}\lambda_{w}  ) } \d \lambda = \sum_{F} \prod_{\alpha\in F} (\mathrm{N}(n_\alpha(\pi, \cdot ) ) -   \mathrm{P}(n_\alpha(\pi, \cdot )) ),  \end{equation}
où la somme sur $F$ est prise sur les parties de $-\Phi_{Q_L}$ telles que $F$ forme une base de $\ago_{L}^{G,*}$, car $|\lambda_w^{\beta^{\vee}}|=1$ pour tout $\beta\in \Phi(Z_M, G)$. Le lemme \ref{A3} implique alors que \eqref{pppo} est vraie pour $(e,n)=1$.

%En insérant l'égalité  \eqref{pppo} dans l'expression  \eqref{aaaz}, on obtient la proposition \ref{expr1}. 

Nous insérons l'égalité \eqref{pppo} dans l'expression \eqref{aaaz}. Notons que le côté droit de \eqref{pppo} ne dépend pas de $\lambda_\pi$, $\lambda_G$, $\lambda^L$ ou $\lambda_w$. On obtient donc la proposition \ref{expr1}.

\section{Preuve du théorème \ref{b}}\label{technique}
\subsection{Les zéros et les pôles des fonctions L}
Soit $(\Pi_1, \Pi_2)$ un couple de représentations automorphes discrètes. Soit $\mathrm{L}(\Pi_1\times \Pi_2, z)$ la fonction L associée à la paire $(\Pi_1, \Pi_2)$ et $\varepsilon(\Pi_1\times \Pi_2)$ le facteur epsilon (pour la définition cf. par exemple \cite[Appendice B]{Lafforgue}). On a besoin d'informations sur les zéros et les pôles de $\mathrm{L}(\Pi_1\times \Pi_2, z)$. On commence par le cas où $(\Pi_1, \Pi_2)$ est un couple de représentations cuspidales.

\begin{prop}\label{fL}
Soit $(\pi_{1},\pi_{2})$ un couple   de représentations cuspidales partout non-ramifiées unitaires de $G_{n_{1}}$ et $G_{n_{2}}$. Soit $\mathrm{L}(\pi_{1}\times\pi_{2}^{\vee}, z)$ la fonction $\mathrm{L}$ associée à la paire $(\pi_{1},\pi_{2})$.  C'est une fonction rationnelle en $z$ que nous pouvons écrire $\mathrm{L}(\pi_{1}\times\pi_{2}^{\vee}, z)=\frac{P(z)}{Q(z)}$ avec $P(z), Q(z)$ deux polynômes premiers entre eux. \\
Quand $n_{1}\neq n_{2}$ ou quand $n_{1}=n_{2}$ et qu'il n'existe pas de $\lambda\in X_{G_{n_2}}$ tel que $\pi_{1}\cong\pi_{2}\otimes\lambda$, on a $$\deg P(z)=(2g-2)n_{1}n_{2},$$ et $$Q(z)=1.$$ 
Pour le cas restant, supposons sans perte de généralité que $\pi_{1}=\pi_{2}$, alors $$Q(z)=(z^{|\Fix(\pi_{1})|}-1)((qz)^{|\Fix(\pi_{1})|}-1),$$
et $$ \deg P(z)= (2g-2)n_{1}^{2}+2|\Fix(\pi_{1})|.$$
On sait aussi $\varepsilon(\pi_1\times \pi_1^{\vee})=q^{(g-1)n_1^{2}}$. 
 \end{prop}
 \begin{proof} 
Soit $\mathscr{F}_1$ (resp. $\mathscr{F}_2$) un faisceau de Weil  lisse irréductible (ou un système local $\ell$-adique si on suppose sans perte de généralité que le caractère central de $\pi_1$ est d'ordre fini) sur $X_1$ correspondant à $\pi_1$ (resp. $\pi_2$) par la correspondance de Langlands (\cite[Théorème VI.9]{Lafforgue}, voir aussi Section IV.3.5 de \cite{Hen-Le}). On est ramené à traiter les fonctions L de $\mathscr{F}_1\otimes \mathscr{F}_2^{\vee}$. %(pour définitions, cf. \textit{loc. cit.}).

Tout d'abord, on remarque que par le théorème de pureté de Deligne et la conjecture de Ramanujan-Petersson démontrée par Lafforgue (\cite[Théorème VI.10]{Lafforgue}), les polynômes $$\det(1-\ z\mathrm{F}_q|H^{i}_{c}(X, \mathscr{F}_1\otimes \mathscr{F}_2^{\vee}  ))\quad  i=0,1,2$$ sont premiers entre eux car leurs racines ont des valeurs absolues distinctes. Donc d'après l'interprétation cohomologique des fonctions L de Grothendieck (cf. par exemple \cite[Théorème VI.1]{Lafforgue}) $$P(z)=\det(1-\ z\mathrm{F}_q|H^{1}_{c}(X, \mathscr{F}_1\otimes \mathscr{F}_2^{\vee}  )),$$ et $$Q(z)=\det(1-\ z\mathrm{F}_q|H^{0}_{c}(X, \mathscr{F}_1\otimes \mathscr{F}_2^{\vee}  ))\det(1-\ z\mathrm{F}_q|H^{2}_{c}(X, \mathscr{F}_1\otimes \mathscr{F}_2^{\vee}  )).$$ 
 
\sloppy
On a
\begin{equation}\label{H0}
H^{0}_c(X, \mathscr{F}_1\otimes \mathscr{F}_2^{\vee})\cong H^{0}(X, \mathscr{F}_1\otimes \mathscr{F}_2^{\vee})\cong \Hom_X(\mathscr{F}_1, \mathscr{F}_2), \end{equation}
et par la dualité de Poincaré, 
 \begin{equation}\label{H2}
H^{2}_c(X, \mathscr{F}_1\otimes \mathscr{F}_2^{\vee})\cong H^{0}(X, \mathscr{F}_1^{\vee}\otimes \mathscr{F}_2(1))^{\vee}\cong \Hom_X(\mathscr{F}_2, \mathscr{F}_1)^{\vee}(-1). \end{equation}
Donc lorsque $\mathscr{F}_1|_X$ et $\mathscr{F}_2|_X$ n'ont pas de facteur commun, i.e. $\pi_1\not\cong \pi_2\otimes\lambda $, pour tout $\lambda\in X_{G_{n_2}}$, $H^{0}_c(X, \mathscr{F}_1\otimes \mathscr{F}_2^{\vee})$ et $H^{2}_c(X, \mathscr{F}_1\otimes \mathscr{F}_2^{\vee})$ sont triviaux. 

\sloppy
Supposons maintenant $\mathscr{F}_1=\mathscr{F}_2$. Soit $|\Fix(\pi_1)|=d$. Alors  $\mathscr{F}_1|_{\pi_1^{\text{ét}}(X, o)}=\oplus_{j=0}^{d-1} \Fr_X^{*j}\mathscr{F}$ et par le lemme de Schur, on a
$$H^{0}_c(X, \mathscr{F}_1\otimes \mathscr{F}_1^{\vee})\cong \bigoplus_{j=0}^{d-1}\Hom_X(\Fr_{X}^{j*}\mathscr{F}, \Fr_{X}^{j*}\mathscr{F})\cong \bar{\mathbb{Q}}_{\ell}^{d}. $$
L'action de $\Fr_{X}^{*}$ coïncide avec $\mathrm{F}_q$ par le lemme 1.3 de \cite{Deligne-Flicker}, donc
$\mathrm{F}_q$ agit sur $\bar{\mathbb{Q}}_{\ell}^{d}$ par permutation cyclique sur les coordonnées. On en déduit que $\det(1-\ z\mathrm{F}_q|H^{0}_{c}(X, \mathscr{F}_1\otimes \mathscr{F}_1^{\vee}  )) = 1-z^d $. 
De même, par (\ref{H2}), on a $\det(1-\ z\mathrm{F}_q| H^{2}_{c}(X, \mathscr{F}_1\otimes \mathscr{F}_1^{\vee}  )) = 1-(qz)^d $. Comme le cup-produit est anti-symétrique sur $H^1$, on déduit par la dualité de Poincaré que  les valeurs propres de $\mathrm{F}_q$ sur   $H^{1}_{c}(X, \mathscr{F}_1\otimes \mathscr{F}_1^{\vee})$ peuvent être    appariées de manière à ce que chaque paire ait un produit $q$. La valeur du facteur $\varepsilon$ peut être ainsi obtenue. 

Le nombre de zéros de fonction $\mathrm{L}$ s'en déduit par l'équation fonctionnelle (cf. par exemple \cite[Théorème VI.6.]{Lafforgue}) \begin{equation}\label{equafon}\mathrm{L}(\mathscr{F}_{1}\otimes\mathscr{F}_{2}^{\vee},z)= \varepsilon(\mathscr{F}_{1}\otimes\mathscr{F}_{2}^{\vee}) z^{-\chi(\mathscr{F}_{1}\otimes\mathscr{F}_{2}^{\vee})}   \mathrm{L}(\mathscr{F}_{1}^{\vee}\otimes\mathscr{F}_{2},\frac{1}{qz}), \end{equation}
où $\varepsilon(\mathscr{F}_{1}\otimes\mathscr{F}_{2}^{\vee})\in \bar{\mathbb{Q}}_{\ell}^{\times}$ est une constante, et le facteur $\chi(\mathscr{F}_{1}\otimes\mathscr{F}_{2}^{\vee})$ est la caractéristique d'Euler de $\mathscr{F}_{1}\otimes\mathscr{F}_{2}^{\vee}$ qui est très simple dans notre cas (voir  \cite[Théorème 1, p.133]{Raynaud}):
$$  \chi(\mathscr{F}_{1}\otimes\mathscr{F}_{2}^{\vee})={(2-2g)n_1 n_2}.  $$
 \end{proof}

L'hypothèse de Riemann généralisée est démontrée par L. Lafforgue (Théorème VI.10. de \cite{Lafforgue}),  c'est-à-dire les zéros de $P(z)$ satisfont $|z|=q^{-1/2}$.

\begin{coro}\label{nombre}
Soit $ \Pi_1=\pi_1\boxtimes\nu_1$ et $ \Pi_{2}=\pi_2\boxtimes\nu_2$  deux représentations discrètes partout non-ramifiées unitaires de $G_{n_1\nu_1}$ et $G_{n_2\nu_2}$ respectivement (cf.  théorème \ref{MWres} pour la notation). Soit $\nu_{1}\geq \nu_{2}$. Pour toute fonction méromorphe $f$ sur  $\mathbb{C}$, on note $\gls{Nf}$  le nombre des zéros de $f$ dans le région $|z|<1$ et $\gls{Pf}$ pour celui des pôles, on a 
\begin{equation} \begin{split} 
&\mathrm{N}(\frac{ \mathrm{L}( \Pi_{1}\times  \Pi_{2}^{\vee}, z)}{ \mathrm{L}( \Pi_{1}\times  \Pi_{2}^{\vee}, q^{-1}z) })-\mathrm{P}( \frac{ \mathrm{L}( \Pi_{1}\times  \Pi_{2}^{\vee}, z)}{\mathrm{L}( \Pi_{1}\times  \Pi_{2}^{\vee}, q^{-1}z) } ) 
 \\   = &  \begin{cases}  \nu_{2} (   2g-2  ) n_1  n_2    \quad\text{si  $ \Pi_{1}\not\sim \Pi_{2}$   }  ;  \\       
\nu_{2} (   2g-2  ) n_1 n_2 +|\Fix({\pi}_{1})|  \quad\text{si  $ \Pi_{1} =  \Pi_{2}$    }    .  \end{cases} \end{split}
\end {equation}

\end{coro}

\begin{proof}

Rappelons le lemme bien connu ci-dessous:
\begin{lemm}\label{deco}
Soit $\Pi=\pi\boxtimes\nu$  une représentation discrète irréductible. Alors pour toute représentation discrète $\Pi'$, on a
 $$ \mathrm{L}( \Pi\times {\Pi'}^{\vee}, z) =\prod_{i=1}^{\nu} \mathrm{L}(\pi\times{\Pi'}^{\vee},   q^{\frac{\nu+1}{2}-i }z). $$
\end{lemm}
\begin{proof}[Preuve du lemme]
Elle résulte du  théorème \ref{MWres} et  du produit eulérien des fonctions $\mathrm{L}$. 
\end{proof}
Par ce lemme, on a
\[ \frac{ \mathrm{L}( \Pi_1\times  \Pi_2^{\vee}, z)}{ \mathrm{L}( \Pi_1 \times { \Pi_2}^{\vee}, q^{-1}z) } = \frac{  \prod_{j=1}^{\nu_2}\prod_{i=1}^{\nu_1} L(\pi_1\times \pi_2^{\vee}, q^{\frac{\nu_1+\nu_2}{2} +1 -i-j} z) }{  \prod_{j=1}^{\nu_2}\prod_{i=1}^{\nu_1} L(\pi_1\times \pi_2^{\vee}, q^{\frac{\nu_1+\nu_2}{2} -i-j}  z )   } .\]
Soit $\nu_{1}\geq \nu_{2}$, après simplification, on a alors 
\begin{equation}\label{Lfu} \frac{ \mathrm{L}( \Pi_1\times  \Pi_2^{\vee}, z)}{ \mathrm{L}( \Pi_1 \times { \Pi_2}^{\vee}, q^{-1}z) } = \frac{\prod_{i=1}^{\nu_{2}} \mathrm{L}({\pi}_{1}\times{\pi}_{2}^{\vee}, q^{\frac{\nu_{1}+\nu_{2}}{2}-i }z)    }{      \prod_{ i=1  }^{\nu_{2}}         \mathrm{L}({\pi}_{1}\times{\pi}_{2}^{\vee}, q^{-\frac{\nu_{1}+\nu_{2}}{2}+i-1 }z)       }    . \end{equation}
On conclut respectivement pour les zéros et les pôles du numérateur et du dénominateur: \\
$-$ Par l'hypothèse de Riemann généralisée, tous les zéros de numérateur sont dans la région $|z|< 1$, et aucun zéro de dénominateur n'est dans  cette région. \\
$-$ Utilisons  la proposition \ref{fL}. Si $\pi_1\not\sim\pi_2$, il n'y a pas de pôles pour les fonctions $\mathrm{L}$. Si $\pi_1=\pi_2$  on sait que quand $\nu_{1}\neq \nu_{2}$ tous les pôles  du numérateur sont dans la région $|z|< 1$ et aucun pôle du dénominateur n'est dans  cette région; mais quand $\nu_{1}=\nu_{2}$ les pôles sur $|z|=1$ se compensent. Notons que $\Pi_{1}\sim\Pi_{2}$ si et seulement si $\nu_{1}=\nu_{2}$ et ${\pi}_{1}\sim {\pi}_{2}$, le corollaire en résulte. \end{proof}

\subsection{Calcul du facteur (b.) dans la proposition \ref{expr1} } \label{4.5}
 % \begin{equation}\label{sum}
%\frac{\mu(l)   }{|\stab(P,{\pi})| |w|} 
% {(-1)^{\sum_{i}({m_i}/\xi_{i}-\alpha_i  ) }}       
% (\prod_{i} |\Fix(\Pi_{i})|^{m_i-\alpha_i} )
 %  \sum_{F}     \prod_{\beta\in F }\biggr(\mathrm{N}(n_{\beta}({\pi}, \cdot))- \mathrm{P}(n_{\beta}({\pi}, \cdot))\biggr) , \end{equation} 
\subsubsection{Résumé du résultat.} \label{bonrep} 
Soit $(P,{\pi})$ une paire discrète partout non-ramifiée, et $(w,1)\in \stab(P, \pi)$. Soit $L=L^w$. 

Soit $L\cong G_{l_{1}d_{1}}\times\cdots \times G_{l_{r}d_{r}}$.
Après une réindexation, on identifie $M_{P}$ avec \begin{equation}\label{6.2.4}
\underbrace{G_{d_{1}}\times\cdots\times G_{d_{1}}}_{l_1}\times\underbrace{ G_{d_{2}}\times\cdots\times G_{d_2} }_{l_2}\times \cdots\times \underbrace{G_{d_{r}}\times\cdots\times G_{d_{r}} }_{l_{r}},\end{equation} tel que $w$ permute cycliquement les facteurs de $M_P$ contenus  dans un même facteur de $L$. Soit $ \Pi_{i}=\pi_i\boxtimes \nu_i$ (notation du théorème \ref{MWres}) une représentation de $G_{d_{i}}$ qui est un des facteurs de ${\pi}$.

On désigne par $\gls{ne}$ le plus grand commun diviseur des deux entiers positifs $n$ et $e$, et on désigne par $\gls{p.g.c.d.}(a_j)$ le plus grand commun diviseur de la famille des entiers positifs $(a_j)_{j}$. 
\begin{prop}\label{ABCD}
Quand $(e,n)=1$, le facteur (b.) dans la proposition \ref{expr1} est égal à:
%$$ |X_{L}^{L}|  \sum_{ l \mid \mathrm{p.g.c.d.}( l_{i}|\Fix(\pi_{i})|) }\mu(l) (-1)^{ \sum_{j}(\frac{l_{j}}{l/(l, |\Fix(\pi_{j})|)}  -1 )} \prod_{j=1}^{ r } |\Fix(\pi_{j})|^{l_j-1} . $$
\begin{equation} |X_{L}^{L}|  \sum_{ l \mid \mathrm{p.g.c.d.}( l_{i}|\Fix(\pi_{i})|) }\mu(l)\prod_{j=1}^{ r } (-1)^{\frac{l_{j}}{l/(l, |\Fix(\pi_{j})|)}  -1 } |\Fix(\pi_{j})|^{l_j-1} , \end{equation}
où $\mu(\cdot)$ est la fonction de Möbius. 
\end{prop}
%Voir aussi la remarque  \ref{re623}. 

\subsubsection{Calculs}
Nous revenons à la proposition \ref{expr1}. 
Soit $\Delta({\pi},w)$ le facteur $(b.)$ dans la proposition \ref{expr1}. 
Par un changement d'ordre, $\Delta({\pi},w)$ est égal à
\begin{equation}
\label{weyl} 
\sum_{ {\lambda}\in X_{G}^{G} } {{\lambda }_{1}}^{-e}   \sum_{\{\lambda_{{\pi}}\in \Fix({\pi})|\ \lambda_{{\pi}} / \lambda  \in (\Im X_{M_{P}}^{L})^{{\circ}}   \}}    \sum_{\{\lambda_{w}\in \Im X_{M_P}^{L}   |\ \lambda_{w}/w^{-1}(\lambda_{w})=  \lambda_{{\pi}}/ { \lambda }            \}} n_{{\pi}}(w,w^{-1}(\lambda_{w}))
\end{equation}
où $L=L^w$ et ${\lambda}_1$ est la première coordonnée de ${\lambda}\in X_{{L}}^{G}$. 
Le lemme suivant concerne l'ensemble sur lequel une somme dans \eqref{weyl} est prise. 
 \begin{lemm}\label{LambdaL} 
Soit $c$ le plus grand commun diviseur des ${|\Fix( \Pi_{i})|l_{i}}={|\Fix( \pi_{i})|l_{i}}$ pour $i=1,\ldots, r$. Nécessairement c'est un diviseur de $n$. Soit ${\Lambda}_G$ le sous-groupe d'ordre $c$ de $X_{G}^{G}$. Soit ${\lambda}\in X_{G}^{G}$.  

Si ${\lambda}\not\in {\Lambda}_G$, l'ensemble $$\{\lambda_{{\pi}}\in \Fix({\pi})|\ \lambda_{{\pi}} { \lambda}^{-1} \in (\Im X_{M_{P}}^{L})^{{\circ}}   \}$$ est vide. 
Si ${\lambda}\in {\Lambda}_G$, l'ensemble $$\{\lambda_{{\pi}}\in \Fix({\pi})|\ \lambda_{{\pi}} { \lambda}^{-1} \in (\Im X_{M_{P}}^{L})^{{\circ}}   \}$$ est un espace principal homogène sous l'action de $\Fix({\pi}) \cap(\Im X_{M_{P}}^{L})^{{\circ}} $.% qui est d'ordre . 
\end{lemm}
\begin{proof}
L'ensemble considéré, s'il est non-vide,  porte une action transitive et fidèle de $\Fix({\pi}) \cap(\Im X_{M_{P}}^{L})^{{\circ}} $, donc il suffit de prouver que l'ensemble $\{\lambda_{{\pi}}\in \Fix({\pi})|\ \lambda_{{\pi}} { \lambda}^{-1} \in (\Im X_{M_{P}}^{L})^{{\circ}}   \}$ est non-vide si et seulement si ${\lambda}\in {\Lambda}_G$. 

Soit ${\lambda}=({\lambda}_{1}, \ldots, \lambda_{1})$ en utilisant l'inclusion $X_{G}^{G}\subseteq X_L^{G}$. 

S'il existe $\lambda_{{\pi}}\in \Fix({\pi})\subseteq X_M^M$ tel que \(\lambda_{{\pi}} { \lambda}^{-1} \in (\Im X_{M_{P}}^{L})^{{\circ}}\), où on sait que 
$$(\Im X_{M_{P}}^{L})^{{\circ}} \cong\{(z_{1,1},\ldots, z_{1,l_1}, z_{2,1}, \ldots, z_{r, l_{r}} ) \in ({\mathbb{C}^\times})^r |\ z_{j,1}z_{j,2}\cdots z_{j, l_j} =1, \forall j   \} .  $$
Nécessairement $$  (\lambda_{1}^{l_j})^{|\Fix( \Pi_j)|}=1; \quad j=1,\ldots, r   , $$  
donc $${\lambda}_{1}^{  \mathrm{p.g.c.d.}( l_j |\Fix( \Pi_j)|)}= {\lambda}_{1}^c=   1  . $$
C'est-à-dire ${\lambda}\in {\Lambda}_G$.

Inversement, supposons que ${\lambda}\in {\Lambda}_G$.  On a $${\lambda}_1^{l_j}\in\Fix( \Pi_j)  \quad j=1,\ldots, r,   $$
car le groupe $\Fix( \Pi_j)$ est  cyclique et $\lambda_1^{l_j|\Fix( \Pi_j)|}=1  $. 
Dans ce cas, on pose $z_{j,1}=\lambda_1^{l_j} $ et $z_{j,i}=1$ pour $i\neq 1$ et $j=1,\ldots, r$, alors
$$\lambda_{{\pi}}= (z_{1,1},\ldots, z_{1,l_1}, z_{2,1}, \ldots, z_{r, l_{r}} )\in \Fix({\pi})\subseteq X_M^M,$$
et $$\lambda_{{\pi}}  \lambda^{-1} \in (\Im X_{M_{P}}^{L})^{{\circ}}.$$\end{proof}

D'après ce lemme, $\Delta({\pi},w)$ est égal à
\begin{multline*}\sum_{ {\lambda}\in {\Lambda}_G } { \lambda_{1}}^{-e} \sum_{\{\lambda_{{\pi}}\in \Fix({\pi})|\ \lambda_{{\pi}} { \lambda}^{-1} \in (\Im X_{M_{P}}^{L})^{{\circ}}   \}} \\ 
 \sum_{\{\lambda_{w}\in \Im X_{M_P}^{L}   |\ \lambda_{w}/w^{-1}(\lambda_{w})=  \lambda_{{\pi}}  \lambda^{-1}            \}} n_{{\pi}}(w,w^{-1}(\lambda_{w})).
\end{multline*}

On peut faire le calcul suivant chaque facteur de $L$. Par simplification (de notations), supposons d'abord  $L={G}_m$, alors $$M_P\cong G_{d}\times \cdots \times G_{d},$$ et $w$ permute les facteurs de $M_P$ cycliquement. Soit $l=m/d$. 
Dans ce cas, les facteurs de la paire discrète $(P,{\pi})$ sont égaux. Soit $${\pi}=\Pi_0\otimes\cdots\otimes\Pi_0,  $$
et par le théorème de \ref{MWres}, on peut supposer 
$$\Pi_0= \pi_0\boxtimes\nu, $$
avec $\pi_0$ une représentation automorphe cuspidale et $\nu\geq 1$.

%soit $${\pi}=\pi_0\boxtimes\nu\otimes\cdots\otimes\pi_0\boxtimes\nu $$.
Notons qu'on a $\Fix({\pi}_0\boxtimes\nu)\cong\Fix(\pi_0)$ (proposition \ref{res e}).
 Après une réindexation, on peut supposer simplement que $w$ est la permutation cyclique  $(1,2,\ldots, l)$. 

Soit ${\lambda}\in {\Lambda}_{G}$.
Soit $\lambda_{{\pi}}=(\mu_1,\cdots, \mu_l )\in \Fix({\pi})$ tel que $\lambda_{{\pi}}{\lambda}^{-1}  \in (\Im X_{M}^{{G}})^{{\circ}}$. 
Par l'inclusion $\Lambda \subseteq \Im X_{M}^{{G}}$, on a $ {\lambda}=( {\lambda}_1, \ldots, \lambda_1) $. Comme $\Fix(\pi)\subseteq \Im X_M^G $, la condition $\lambda_{{\pi}}{\lambda}^{-1}  \in (\Im X_{M}^{{G}})^{{\circ}}$
se traduit
$$1=\frac{\mu_{1} \cdots  \mu_l     }{  \lambda^l_1    }.$$
L'ensemble des $\lambda_{w}\in \Im X_{M}^{{G}}$ tels que $\lambda_{w}/w^{-1}(\lambda_{w})=\lambda_{{\pi}} \lambda^{-1}$ est 
\begin{equation}\label{lambdaw}
\{  \lambda_{w}=\zeta^{i} z
(\frac{\mu_{1} \cdots  \mu_l     }{  \lambda^l_1 }, \frac{\mu_{2} \cdots  \mu_l   }{\lambda_{1}^{l-1} },\ldots, \frac{\mu_{l}}{\lambda_{1} })| \  i =1,\ldots, m   \}, \end{equation}
où $z\in\mathbb{C}^\times$ est un nombre tel que $$ z
(\frac{\mu_{1} \cdots  \mu_l     }{  \lambda_1^l    }, \frac{\mu_{2} \cdots  \mu_l   }{\lambda_{1}^{l-1} },\ldots, \frac{\mu_{l}}{\lambda_{1} }):=
(z\frac{\mu_{1} \cdots  \mu_l  }{  \lambda_1^l    }, z\frac{\mu_{2} \cdots  \mu_l   }{\lambda_{1}^{l-1} },\ldots, z\frac{\mu_{l}}{\lambda_{1} })$$ soit contenu dans $\Im X_M^{G}$ et $\zeta$ une racine $m$$^{i\text{è}me}$ primitive de l'unité. 

Maintenant, on peut faire le calcul:

Soit $\glslink{rpi}{\mathrm{r}(\pi_0)}$ le rang de $\pi_0$ (i.e. $\mathrm{r}(\pi_0)\nu = d$). 
Soit
$$L(\pi_0\times\pi_0^{\vee}, z) =\frac{P(z)}{Q(z)}, $$ pour $P$ et $Q$ deux polynômes premiers entre eux. On a d'après l'équation fonctionnelle de la fonction L et la proposition \ref{fL}:
\begin{equation}\label{Pz}
P(z)= q^{(g-1) {\mathrm{r}(\pi_0)}^{2}+ |\Fix(\pi_0)| } z^{(2g-2){\mathrm{r}(\pi_0)}^{2}+ 2|\Fix(\pi_0)|} P(\frac{1}{qz}). \end{equation}
Par \eqref{Lfu}, soit $\beta\in \Phi(Z_{M_P}, {G})$, on a
$$  n_{\beta}({\pi}, z) = q^{(1-g)d^{2}} \prod_{i=1}^{\nu}  \frac{ P(q^{i-1} z )  }{  P(q^{-i}z )  } \prod_{i=1}^{\nu} \frac{ 1-(q^{-i}z)^{ |\Fix(\pi_0)|  }   }{    1-(q^{i}z)^{ |\Fix(\pi_0)|  }    }   \prod_{i=1}^{\nu-1} \frac{ 1-(q^{-i}z)^{ |\Fix(\pi_0)|  }   }{    1-(q^{i}z)^{ |\Fix(\pi_0)|  }  } .      $$
On insère l'équation (\ref{Pz}) dans l'expression ci-dessus de $n_{\beta}(\pi, z)$,  on obtient
\begin{multline}\label{nbetp}  n_{\beta}({\pi}, z) = - z^{ ((2g-2)\mathrm{r}(\pi_0)^{2}+4|\Fix(\pi_0)| )\nu -|\Fix(\pi_0)|   } \\
\prod_{i=1}^{\nu}  \frac{ P(q^{-i} \frac{1}{z}  )  }{  P(q^{-i} {z} )  } \prod_{i=1}^{\nu} \frac{ 1-(q^{i}\frac{1}{z} )^{ |\Fix(\pi_0)|  }   }{    1-(q^{i}z)^{ |\Fix(\pi_0)|  }    }   \prod_{i=1}^{\nu-1} \frac{ 1-(q^{i} \frac{1}{z})^{ |\Fix(\pi_0)|  }   }{    1-(q^{i}z)^{ |\Fix(\pi_0)|  }  }       .   \end{multline}

La somme interne dans $\Delta(\pi, w)$ est égale à
\begin{align}\label{sooo}
\sum_{\substack{\lambda_{w} \in \Im X_{M_P}^{G}   \\ \lambda_{w}/w^{-1}(\lambda_{w})=\lambda_{{\pi}}{\lambda_{G}}^{-1}  }} n_{{\pi}}(w, w^{-1}(\lambda_{w}))     &=\sum_{i =1}^{m} \prod_{s=2}^{l} n_{\beta}({\pi},  {  {\lambda}_1^{s-1}  })  \\
&= m\prod_{s=2}^{l}   n_{\beta}({\pi}, {\lambda_1^{s-1 }  }  ) .    \nonumber
\end{align}

Notons que pour tout $\mu \in \Fix(\pi_0)$, on a $P(z\mu)=P(z)$, i.e. $P$ est un polynôme en $z^{|\Fix(\pi_0)|}$. Par le lemme \ref{LambdaL}, puisque $\lambda\in\Lambda_G$, le multi-ensemble $\{\lambda_1^{|\Fix(\pi_0)|}, \lambda_1^{2| \Fix(\pi_0)|}, \dots, \lambda_1^{(l-1)|\Fix(\pi_0)|}\}$ coïncide avec le multi-ensemble \[ \{\lambda_1^{-|\Fix(\pi_0)|}, \lambda_1^{-2| \Fix(\pi_0)|}, \dots, \lambda_1^{-(l-1)|\Fix(\pi_0)|}\}. \] Donc par l'expression  (\ref{nbetp}) de $n_{\beta}$, on a $$\prod_{s=2}^{l}   n_{\beta}({\pi}, {\lambda_1^{s-1 }  }  )= \prod_{s=2}^{l} (\lambda_1^{s-1})^{((2g-2)\mathrm{r}(\pi_0)^{2}+4|\Fix(\pi_0)| )\nu -|\Fix(\pi_0)| }. $$
C'est-à-dire
la somme (\ref{sooo}) est égale à
$$   (-1)^{l-1} m  \lambda_{1}^{-|\Fix(\pi_0)|\frac{l(l-1)}{2}}  \lambda_{1} ^{\nu\left((2g-2)\mathrm{r}(\pi_0)^{2}+4|\Fix(\pi_0)| \right)\frac{l(l-1)}{2} }.$$
Puisque $\lambda_1^{l|\Fix(\pi_0)|}=1$, on peut conclure que:
\begin{equation}\sum_{\substack{\lambda_{w} \in \Im X_{M_P}^{G} \\ \lambda_{w}/w^{-1}(\lambda_{w})=\lambda_{{\pi}}^{{G}}}}    n_{{\pi}}(w, w^{-1}(\lambda_{w})) = (-1)^{l-1}m \lambda_{1}^{- |\Fix(\pi_0)| \frac{l(l-1)}{2}  }. \end{equation}

Notons que le résultat ne dépend pas de $\lambda_{{\pi}}$, et que le groupe $\Fix({\pi_0}) \cap(\Im X_{M_{P}}^{{G}})^{{\circ}} $ a pour cardinal $|\Fix(\pi_0)|^{r-1} $. Donc par le lemme \ref{LambdaL}  \begin{multline}\label{jusqu'a} \sum_{\{\lambda_{{\pi}}\in \Fix({\pi_0})|\ \lambda_{{\pi}} { \lambda_{{G}}}^{-1} \in (\Im X_{M_{P}}^{{G}})^{{\circ}}   \}}\sum_{\{\lambda_{w}\in \Im X_{M_P}^{G}|\ \lambda_{w}/w^{-1}(\lambda_{w}) =\lambda_{{\pi}}  {\lambda_{{G}}}_{1}^{-1}    \}} n_{{\pi}}(w,w^{-1}(\lambda_{w}))\\ = m(-1)^{l-1} |\Fix(\pi_0)|^{ l-1}  \lambda_{1}^{- |\Fix(\pi_0)| \frac{l(l-1)}{2} }.\end{multline}

Pour le cas général, avec la notation de \eqref{6.2.4}, soit ${\lambda} \in {\Lambda}_G$ on obtient
\begin{equation}
\Delta(\pi,w)=\sum_{\lambda\in \Lambda_G}\lambda_1^{-e} |X_{L}^{L}|  \prod_{i=1}^{ r }\biggr( (-1)^{l_i-1} |\Fix(\pi_{i})|^{l_i-1}      {\lambda}_{1}^{- |\Fix(\pi_{i})| \frac{l_i(l_i-1)}{2} }\biggr).
\end{equation}

%Soit $l(w)$ le longueur de $w$.  Alors  la parité  de ${l(w)}$ est égale à celle de nombre de facteurs de $L$ qui sont de rang pair par le fait que $w\in W^{L}(M_{P})_{reg}$.
Soit $$\delta_e(\pi, L):= \sum_{{\lambda}\in {\Lambda_{G}}} {{\lambda_{ 1}}}^{-(\sum_i \frac{l_i(l_i-1)}{2} |\Fix(\pi_{i})|)- e }. $$
On a \begin{equation} \label{Deltapi}
\Delta({\pi}, w)= |X_{L}^{L}|  \prod_{j=1}^{ r }\left((-1)^{l_j-1} |\Fix(\pi_{j})|^{l_j-1} \right) \delta_e({\pi}, L)    .
\end{equation}
\begin{remark}\label{re623}
La somme $\delta_{e}({\pi},{L})$ est égale à $|{\Lambda}_{G}|$  ou $0$ selon si le caractère $$ {{\lambda}_{1}}\mapsto  {{\lambda}_{1}}^{-(\sum \frac{l_i(l_i-1)}{2} |\Fix(\pi_{i})|)-e}  $$ est trivial ou non. Mais il y a une autre expression de $\delta_e({\pi}, L)$ dans la proposition \ref{Deltapi1}.  Malgré l'apparence compliquée du côté droit de l'identité, %cela simplifiera beaucoup les choses dans les section suivantes. 
cela facilitera grandement les calculs dans les sections suivantes.
\end{remark}
\begin{prop}\label{Deltapi1}
Soit $e$ premier avec $n$, on a
$$\delta_{e}({\pi}, L)=\sum_{ l \mid \mathrm{p.g.c.d.}( l_{i}|\Fix(\pi_{i})|) }\mu(l) (-1)^{ \sum_{j}(l_{j}+\frac{l_{j}}{l/(l, |\Fix(\pi_{j})|)}  )}, $$
où $(l, |\Fix(\pi_{j})|) $ désigne le plus grand commun diviseur de $l$ et $|\Fix(\pi_{j})|$. 
\end{prop}
\begin{proof}
On vérifie le cas $e=-1$, et laisser le cas $(e,n)$ dans l'appendice (cf. Théorème \ref{Jeind}). 
On voit que $$\chi: {\lambda}\mapsto  {{\lambda_{1}}}^{-(\sum \frac{l_{i}(l_{i}-1)}{2} |\Fix(\pi_{i})|)+1 }$$ est un caractère de ${\Lambda}_G$ vers $\mathbb{C}^\times$ que nous noterons par $\chi$, il s'ensuit que $\delta_{-1}({\pi}, L)$ s'annule sauf si ce caractère est trivial. 

Cependant, on voit que $$\chi^{2}({\lambda})=\lambda_{1}^{2},$$  puisque l'ordre de ${\Lambda}_G$ divise tous $l_{i}|\Fix(\pi_{i})|$. 
Donc si $\chi$ est trivial, nécessairement le plus grand commun diviseur des $l_{i} |\Fix(\pi_{i})|$ est égal à $2$ ou $1$ et respectivement ${\Lambda}_G=\{\pm 1\}$ ou ${\Lambda}_G=\{1\}$.
Si le plus grand commun diviseur des $l_{i} |\Fix(\pi_{i})|$ est égal à $1$, alors $\delta_{-1}({\pi}, L)=1$, et la somme à droite dans le lemme est égale à $1$ aussi. 

On considère donc le cas où le plus grand commun diviseur des $l_{i} |\Fix(\pi_{i})|$ est plus grand que $1$. Dans ce cas $\delta_{-1}({\pi}, L)$ est égal à soit $2$ soit $0$.

Pour que $\delta_{-1}({\pi}, L)$ soit égal à 2, on a besoin que le plus grand commun diviseur des $l_{i} |\Fix(\pi_{i})|$ soit égal à $2$ et $-{\sum \frac{l_{i}(l_{i}-1)}{2} |\Fix(\pi_{i})|-1 }$ soit un nombre pair. 
Regardons $\frac{l_{i}(l_{i}-1)}{2} |\Fix(\pi_{i})|$ suivant la valeur de $l_{i}$ $mod$ $4$. La seule possibilité que la valeur $2\nmid \frac{l_{i}(l_{i}-1)}{2} |\Fix(\pi_{i})|$ et $2\mid l_{i} |\Fix(\pi_{i})|$  est quand $l_{i}\equiv 2$ $mod$ $4$ et $|\Fix(\pi_{i})|$ est un nombre impair. Il faut et suffit qu'on ait un nombre impair des indices  $i $ qui vérifient  la condition précédente  pour que $\delta_{-1}({\pi}, L)$ soit égale à 2. Dans tous les autres cas, $\delta_{-1}({\pi}, L)$ est nul.

Maintenant, on précise   la somme à droite de ce lemme.

Si $\mathrm{p.g.c.d.}( l_{i}|\Fix(\pi_{i})|) $ est un nombre impair, alors les $l$ apparaissant dans la somme sont tous impairs. %Notons que $l/(l, |\Fix(\pi_{j})|)$ est un entier qui divise $l_{j}$ pour tout $j$. 
Donc 
$ (-1)^{ \sum_{j}(l_{j}+\frac{l_{j}}{l/(l, |\Fix(\pi_{j})|)}  )}=1$, et 
$$\sum_{l\mid \mathrm{p.g.c.d.}( l_{i}|\Fix(\pi_{i})|) }\mu(l) (-1)^{ \sum_{j}(l_{j}+\frac{l_{j}}{l/(l, |\Fix(\pi_{j})|)}  )}= \sum_{l\mid \mathrm{p.g.c.d.}( l_{i}|\Fix(\pi_{i})|)  } \mu(l) =0, $$
comme nous sommes dans le cas où $\mathrm{p.g.c.d.}( l_{i}|\Fix(\pi_{i})|) >1$. Cela est égale à $\delta_{-1}(\pi, w)$. 

Si $\mathrm{p.g.c.d.}( l_{i}|\Fix(\pi_{i})|) $ est un nombre pair, alors soit $4\mid \mathrm{p.g.c.d.}( l_{i}|\Fix(\pi_{i})|)  $,  soit $ 4 \nmid  \mathrm{p.g.c.d.}( l_{i}|\Fix(\pi_{i})|) $ et $2\mid  \mathrm{p.g.c.d.}( l_{i}|\Fix(\pi_{i})|) $. 

Dans le premier cas. Considérons les $l$ apparaissant dans la somme tels que $\mu(l)\neq 0$. Par la définition de la fonction M\"obius, on a $4 \nmid l$. Il y a trois sous-cas pour les nombres $l_{j}$ et $|\Fix(\pi_{j})|$: \\
1. $2 \nmid\  l_{j}$, on a $2\mid l_{j}+\frac{l_{j}}{l/(l, |\Fix(\pi_{j})|)}$ car ``$1+1=2$''.    \\
2. $4\mid l_{j} $, on a $2\mid \frac{l_{j}}{l/(l, |\Fix(\pi_{j})|)}$ donc $2\mid l_{j}+\frac{l_{j}}{l/(l, |\Fix(\pi_{j})|)}$.  \\
3.  $2\ ||\ l_{j}$ (i.e. $2\mid l_{j}$ mais $4 \nmid l_{j}$ ), on a $2\mid l_{j}$ et $2\mid |\Fix(\pi_{j})|$, alors $2 \nmid \frac{l}{(l,|\Fix(\pi_{j})|)} $ et donc $2\mid l_{j}+\frac{l_{j}}{l/(l, |\Fix(\pi_{j})|)}$. \\
En tout cas $$\sum_{l\mid \mathrm{p.g.c.d.}( l_{i}|\Fix(\pi_{i})|) }\mu(l) (-1)^{ \sum_{j}(l_{j}+\frac{l_{j}}{l/(l, |\Fix(\pi_{j})|)}  )}= \sum_{l\mid \mathrm{p.g.c.d.}( l_{i}|\Fix(\pi_{i})|)  } \mu(l) =0.$$

Dans l'autre cas. Si $2\mid |\Fix(\pi_{j})|$ pour un $j$, on a $2 \nmid \frac{l}{(l,|\Fix(\pi_{j})|)} $ et donc $2\mid l_{j}+\frac{l_{j}}{l/(l, |\Fix(\pi_{j})|)}$  parce que $ l_{j}$ et $\frac{l_{j}}{l/(l, |\Fix(\pi_{j})|)}$ sont de même parité. Si $2 \nmid\ |\Fix(\pi_{j})| $, forcément $2\mid l_{j}$. Si $4\mid l_{j}$, on a $2\mid l_{j}+\frac{l_{j}}{l/(l, |\Fix(\pi_{j})|)}$. Si $2|| l_{j}$, alors $$(-1)^{  l_{j}+\frac{l_{j}}{l/(l, |\Fix(\pi_{j})|)}}=\begin{cases} 1, \quad \text{si  $2 \nmid l$ } \\ -1, \quad \text{si  $2| l$ }\end{cases}.
$$
Donc s'il y a un nombre pair de ${j}$ pour lesquels $2\ ||\ l_{j}$ et $|\Fix(\pi_{j})|$ sont impairs, on a 
$$\sum_{l\mid \mathrm{p.g.c.d.}( l_{i}|\Fix(\pi_{i})|) }\mu(l) (-1)^{ \sum_{j}(l_{j}+\frac{l_{j}}{l/(l, |\Fix(\pi_{j})|)}  )}= \sum_{l\mid \mathrm{p.g.c.d.}( l_{i}|\Fix(\pi_{i})|)  } \mu(l) =0.$$
S'il y a un nombre impair de ${j}$ pour lesquels $2\ ||\ l_{j}$ et $|\Fix(\pi_{j})|$ sont impairs, on a
\begin{align*}
&\sum_{l\mid \mathrm{p.g.c.d.}( l_{i}|\Fix(\pi_{i})|) }\mu(l) (-1)^{ \sum_{j}(l_{j}+\frac{l_{j}}{l/(l, |\Fix(\pi_{j})|)}  )}\\
=& \sum_{l\mid \mathrm{p.g.c.d.}( l_{i}|\Fix(\pi_{i})|);\ 2\nmid l  } \mu(l) - \sum_{l\mid \mathrm{p.g.c.d.}( l_{i}|\Fix(\pi_{i})|);\ 2| l  } \mu(l)       \\
=&2 \sum_{l|\frac{\mathrm{p.g.c.d.}( l_{i}|\Fix(\pi_{i})|)}{2}} \mu(l)\\
=&\begin{cases} 2, \quad \text{si  $\mathrm{p.g.c.d.}( l_{i}|\Fix(\pi_{i})|)=2$ } \\ 0, \quad \text{si  $\mathrm{p.g.c.d.}( l_{i}|\Fix(\pi_{i})|)>2$ }\end{cases}.
 \end{align*}
Cela est exactement $\delta_{-1}({\pi},L )$!
\end{proof}

\subsection{Petite préparation de la théorie des graphes}\label{4.6}
La préparation de cette section a pour but de traiter la somme $\sum_{F}$ sur des parties d'ensemble des racines dans le facteur (c.) de la proposition \ref{expr1}. Les résultats principaux de cette sous-section sont le lemme \ref{kappa} et le théorème \ref{Mat}. 

\subsubsection{}
Les graphes considérés dans cet article sont des graphes simples, finis et non-orientés. %: c'est-à-dire, les graphes  tels que  ne contenant qu'un nombre fini de sommets, deux sommets sont reliés par au plus un arête non-orienté. 
Un arbre est un graphe tel que deux sommets sont reliés par exactement un chemin.

Soit $T$ un graphe. Introduisons des
 variables $x_{ij}=x_{ji}$  correspondant  aux arêtes, et associons à  chaque graphe un monôme:
 $$ x^{T}=\prod_{\{i,j\}\in A(T)} x_{ij},$$
 où $A(T)$ est l'ensemble des arêtes de $T$.  Rappelons que pour une matrice de taille $n\times n$, le $i$$^{i\text{è}me}$ cofacteur principal est le déterminant de la matrice qui s'en déduit par suppression de  la $i$$^{i\text{è}me}$ ligne et la $i$$^{i\text{è}me}$ colonne.

\begin{theorem}[Un cas du théorème de ``Matrix-Tree"/ Kirchhoff]\label{Tree}
Le polynôme générateur  (appelé aussi le polynôme de Kirchhoff)  $$\sum_{T} x^{T}$$ de tous les arbres de sommets  $\{1,\ldots, n\}$, est égal au $i$$^{i\text{è}me}$ ($1\leq i\leq n$) cofacteur principal de la matrice (appelé la matrice de Kirchhoff) suivante (en particulier, ces cofacteurs principaux sont égaux entre eux.)
$$\begin{pmatrix} x_{12}+\cdots+x_{1n}&-x_{12}& \cdots  & -x_{1n} \\-x_{21}&  x_{21}+x_{23}+\cdots+x_{2n}   &\cdots &  -x_{2n}  \\ \cdots&\cdots& \ddots &\vdots   \\-x_{n1}&  \cdots &\cdots  &  x_{n1}+x_{n2}+\cdots+x_{n,n-1}
\end{pmatrix}.
$$
\end{theorem}
\begin{proof}
cf.  théorème \rom{6}.29 \cite{Tutte}.
\end{proof}

\subsubsection{}\label{V11}
Notons qu'une caractéristique importante de la matrice de Kirchhoff est que les sommes des coefficients de chaque ligne et  chaque colonne sont nulles. Le lemme suivant nous permet de calculer les cofacteurs d'une telle matrice d'une autre fa\c con:

Pour une matrice $A$ complexe de taille $n\times n$ de rang $<n$, on voit que $\det(A+\lambda\Id)$ est un  polynôme de  terme constant $0$, et on définit
\begin{equation} \gls{kappaA}=\frac{1}{n} \frac{\det(A+\lambda\Id)}{\lambda}|_{\lambda=0}. \end{equation}
Notons que $\kappa(A)$ est invariant par conjugaison. 

\begin{lemm}\label{kappa}
Soit $A$ une matrice carrée complexe de taille $n\times n$ de rang $<n$,  les valeurs suivantes sont égales: \\
$a.$ $\kappa(A)$\\
$b.$ le produit des valeurs propres non-nulles, prises avec leur multiplicité, divisé par $n$ si la matrice $A$ est de rang $n-1$, zéro si la matrice $A$ est de rang $<n-1$\\
$c.$ la moyenne des cofacteurs principaux de $A$ 

Si la somme des coefficients de chaque ligne de la matrice $A$ est nulle, alors $\kappa(A)$ est  égal à \\
$d.$ $$\frac{1}{n\sum_{i=1}^{n}v_{i} } \det(A+J_{n} \mathrm{diag}(v_{1},\ldots, v_{n}) ),$$ où  $J_{n}$ est une matrice de taille $n\times n$ dont tous les éléments sont $1$ et $\sum_{i=1}^{n}v_{i}\neq0$. 

Si de plus,  la somme des coefficients de chaque colonne de $A$ est nulle aussi, les cofacteurs principaux de $A$ sont égaux, et $\kappa(A)$ est égal à la valeur suivante:\\
$e.$ $$\frac{1}{(\sum_{i=1}^{n}u_{i})(\sum_{j=1}^{n}v_{j}) } \det(A+ (u_{i}v_{j})_{i,j} ),$$ où $\sum_{i=1}^{n}u_{i}\neq0$ et $\sum_{j=1}^{n}v_{j}\neq0$. 

\end{lemm}
\begin{proof}
On voit que l'égalité entre les quantités dans $(a.)$ et $(b.)$ est claire. Pour montrer que $\kappa(A)$ est égal à la moyenne des cofacteurs principaux, il suffit de regarder le terme de degré $1$ de $\det(A+\lambda \mathrm{Id})$.

Maintenant on vérifie $\kappa(A)$ est égal à la valeur dans $(d.)$ quand la somme des coefficients de chaque la ligne de la matrice $A$ est nulle. Soit 
$$P=\begin{pmatrix}
1&0&0&\cdots&0\\
1&1&0&\cdots&0\\
1&0&1&\cdots&0\\
\vdots& & & \cdots&0\\
1&0&0&\cdots&1\\ 
\end{pmatrix}_{n\times n} \  \text{ d'inverse }
\quad P^{-1}=\begin{pmatrix}
1&0&0&\cdots&0\\
-1&1&0&\cdots&0\\
-1&0&1&\cdots&0\\
\vdots& & & \cdots&0\\
-1&0&0&\cdots&1\\ 
\end{pmatrix}_{n\times n}.
$$
On voit que la première colonne de la matrice $P^{-1}AP$ est nulle donc tout le cofacteur principal est nul sauf le premier. Si on note $(P^{-1}AP)_{1}$ le premier cofacteur principal de $P^{-1}AP$, on a 
$$\kappa(A)=\kappa(P^{-1}AP)=\frac{1}{n}(P^{-1}AP)_{1}.$$
Cependant, on vérifie aussi que 
$$P^{-1}J_{n}\mathrm{diag}(v_{1},\ldots, v_{n})P=\begin{pmatrix}
\sum_{i=1}^{n}v_{i}& v_{2}&\cdots&v_{n}\\
0&0&\cdots&0\\
\vdots&  & \cdots&\\
0&0&\cdots&0\\ 
\end{pmatrix},$$
 est une matrice concentrée sur la première ligne dont le premier élément est $\sum_{i=1}^{n}v_{i}$, donc 
$$\det(A+J_{n}\mathrm{diag}(v_{1},\ldots, v_{n}))=(\sum_{i=1}^{n}v_{i})(P^{-1}AP)_{1}=n\kappa(A)\sum_{i=1}^n v_i.$$

Quand la somme des coefficients de chaque colonne de la matrice $A$ est nulle aussi, soit $A=(a_{ij})$. On peut calculer $\det(A+ (u_{i}v_{j})_{i,j} )$ comme cela: ajouter chaque ligne de $A+ (u_{i}v_{j})_{i,j} $ à la première ligne, ensuite ajouter chaque colonne à la première colonne. On voit que $A+ (u_{i}v_{j})_{i,j} $ est transformée en une matrice qui est de la forme $$\begin{pmatrix}
(\sum_{i=1}^{n}u_{i})(\sum_{j=1}^{n}v_{j})&v_{2}(\sum_{i=1}^{n}u_{i} )&v_{3}(\sum_{i=1}^{n}u_{i})&\cdots&v_{n}(\sum_{i=1}^{n}u_{i})\\
u_{2}(\sum_{j=1}^{n}v_{j})&a_{22}+u_{2}v_{2}&a_{23}+u_{2}v_{3}&\vdots&a_{2n}+u_{2}v_{n}\\
u_{3}(\sum_{j=1}^{n}v_{j})&a_{32}+u_{3}v_{2}&a_{33}+u_{3}v_{3}&\vdots &a_{3n}+u_{3}v_{n}\\
\vdots&\cdots &\cdots &\ddots &\vdots \\
u_{n}(\sum_{j=1}^{n}v_{j})&a_{n2}+u_{n}v_{2}&a_{n3}+u_{n}v_{3}&\cdots&a_{nn}+u_{n}v_{n}\\
\end{pmatrix}.$$
On peut soustraire à la $i^{i\text{è}me}$ ligne $\frac{u_{i}}{ \sum_{i=1}^{n}u_{i}}$ fois la première ligne (pour tout $i\geq2$). D'après ces opérations, il se trouve que $$\det(A+ (u_{i}v_{j})_{i,j} )= (\sum_{i=1}^{n}u_{i})(\sum_{j=1}^{n}v_{j})A_{1},$$
où $A_{1}$ est le premier cofacteur principal de $A$. Comme l'indice $1$ n'a pas de rôle particulier dans la preuve, on en déduit que les cofacteurs de $A$ sont égaux et cela finit la démonstration du lemme.
\end{proof}

\subsubsection{}

Soit $L\cong G_{n_{1}}\times\cdots\times G_{n_{r}}$ un sous-groupe de Levi semi-standard. 
Soit $Q$ le sous-groupe parabolique dans $\mathcal{P}(L)$ tel que $\alpha(\kappa)>0$ pour tout $\alpha\in\Delta_{Q}$ (rappelons que $\kappa$ est un vecteur fixé dans $\ago_B$, cf. \ref{1theta}). 
Considérons  l'ensemble des sommets $\{\frac{1}{n_{1}}\det_{L,1}, \ldots, \frac{1}{n_{r}}\det_{L,r}\}$ (cf.  \ref{syst}), on notera $e_i= \frac{1}{n_{i}}\det_{L,i}$ dans la suite.

\begin{lemm}\label{bijection}
Sous la bijection entre l'ensemble de toutes les arêtes de sommets $\{e_1, \ldots, e_r\}$ et l'ensemble $-\Phi_{{Q}}$ qui envoie l'arête $\{ e_i, e_j \}$ $(i>j)$ sur la racine $e_i-e_j$ dans $-\Phi_{{Q}}$, les arbres de sommets $\{ e_1, \ldots, e_r \}$ et les bases de  $\mathfrak{a}_{L}^{{G},*}$ formées de racines dans $-\Phi_{{Q}}$  sont en bijection.
\end{lemm}

\begin{proof}
Notons que le nombre d'arêtes d'un arbre et le nombre de vecteurs d'une base sont égaux à $r-1$.  

Si $F\subseteq -\Phi_{{Q_L}}$ est une base, le graphe $T$ associé n'aura pas de cycle. Car un cycle $e_{i_{1}}\rightarrow e_{i_{2}} \rightarrow \cdots \rightarrow e_{i_{k}}=e_{i_{1}}$ dans $T$ donnera une relation linéaire avec coefficients $\pm 1$ pour les vecteurs de $F$: $$\sum_{j=1}^{k-1}(e_{i_{j}} - e_{i_{j+1}})=0 .$$
Un graphe à $n$ sommets est un arbre si et seulement s'il est connexe et s'il y a exactement $n-1$ arêtes, donc $T$ est un arbre.

Réciproquement si $T$ est un arbre de sommets $\{ e_1, \ldots, e_r \}$, l'ensemble des arêtes donne un ensemble $F$ qui  engendre $\mathfrak{a}_{L}^{{G},*}$. Car un chemin $e_{i}=e_{i_{1}}\rightarrow e_{i_{2}} \rightarrow \cdots \rightarrow e_{i_{k}}=e_{i+1} $ dans $T$ donnera une expression de vecteur $e_{i}-e_{i+1}$  comme combinaison  linéaire des vecteurs de $F$:
$$ \sum_{j=1}^{k-1}(e_{i_{j}} - e_{i_{j+1}})= e_{i}-e_{i+1} , $$
et $\mathfrak{a}_{L}^{{G},*}$ est engendré par les $e_i-e_{i+1}$. Donc $F$ est une base. 
\end{proof}

D'après ce lemme et le  théorème \ref{Tree}, on a

\begin{theorem}\label{Mat}
Soit $(y_{\beta})_{\beta\in -\Phi_{{Q}}}$ une famille de nombres complexes, et $A$ la matrice symétrique dont l'élément indexé par $(i,j)$ $i< j$ est égal à $-y_{\beta}\in\mathbb{C}$, pour $\{i,j\}$ correspondant à $\beta$ au sens de la bijection du lemme précédent; on prend les éléments diagonaux de $A$ de sorte que la somme de chaque ligne soit nulle. Alors
$$ \sum_{F} \prod_{\beta\in F } y_{\beta}  =  \kappa(A),  $$
où la somme porte sur toute partie $F$ de $-\Phi_{Q}$ qui forme une base de $\ago_L^{{G},*}$.\end{theorem}

\subsection{Calcul du facteur (c.) dans la proposition \ref{expr1}}\label{G-A}\label{4.7}
\subsubsection{}\label{notations}
On utilisera $(P, \tilde{\pi})$ pour une paire discrète et on conservera $\pi$ pour une représentation cuspidale. 
Soit $(P, {\tilde{\pi}})$ une paire discrète partout non-ramifiée. Rappelons qu'on a choisi un bon représentant (cf. \ref{bonrep}), c'est-à-dire 
$$\tilde{\pi}=\underbrace{ \Pi_1\otimes\cdots\otimes  \Pi_{1}}_{m_{1}}\otimes \underbrace{  \Pi_{2}\otimes\cdots\otimes  \Pi_2 }_{m_{2}}\otimes\cdots \otimes\underbrace{ \Pi_{k}\otimes\cdots\otimes \Pi_{k}}_{m_{k}} ,$$
et $\Pi_i$ sont irréductibles deux-à-deux non inertiellement équivalentes. 
Soit $I=\{ \Pi_1, \ldots,   \Pi_k\}$; on a $| I |=k$. %, c'est-à-dire ces sont ceux que $m_{ \Pi}\geq1$.%, on va note $I$ au lieu de $I_{\tilde{\pi}}$ quand ${\pi}$ est claire dans le contexte.

Pour une représentation automorphe discrète $\Pi_i\in I$, avec la classification des représentations discrètes automorphes de Moeglin-Waldspurger (le théorème \ref{MWres}), on peut supposer que $\Pi_i=\pi_i\boxtimes \nu_i$ avec $\pi_i$ cuspidale. %, on définit deux applications par 
% On définit deux applications de $I_{\tilde{\pi}}$ vers l'ensemble des classes d'isomorphie des représentations automorphes cuspidales et vers $\mathbb{N}^{*}$ en envoyant une représentation discrète $ \Pi=\pi\boxtimes\nu$ respectivement sur $\pi$ et  sur $\nu$: 
    %      $$p_1(\pi\boxtimes\nu)= \pi ,$$
      %    $$ p_2(\pi\boxtimes\nu) = \nu . $$ 
Pour une représentation automorphe cuspidale irréductible $\pi$ et un entier $\nu\in \mathbb{N}^{*}$, soit
$$I_{\pi}=\{\Pi_i \in I \mid \pi_i =\pi       \},$$
$$I_{\nu} =\{\Pi_i \in I  \mid \nu_i =\nu     \}. $$
%on note   $I_{\pi}$ pour la fibre de la première application $p_{1}$ en $\pi$ et $I_{\nu}$ celle de $p_{2}$ en $\nu$.
 %Alors les $I_{\pi}$ et $I_{\nu}$ sont sous-ensembles de $I_{\tilde{\pi}}$. 
 On dispose aussi d'un ensemble
              $${{{N}}}:=\{\nu\in\mathbb{N}^{*}|\  I_{\nu}\neq \varnothing    \}.  $$

Soit $(w,1)\in \stab(P, \tilde{\pi})$.  Nécessairement $w$ laisse chaque $ \Pi_i^{m_{i}}$ fixé et permute  ses facteurs donc il s'identifie à un élément de $$\prod_{ i=1}^{k}\mathfrak{S}_{m_{i}}.$$ 
Soit $w=w_1\times w_2\times\cdots\times w_k$ la décomposition associée.  Soit $$\alpha_{i}$$ le nombre de cycles de la permutation $w_i$ et $$(l_{i, j})_{j=1,\ldots, \alpha_{ i}}$$ les longueurs de ses cycles. % Pour tout $\Pi_i=\pi_i\boxtimes \nu_i$, soit $d_{\Pi_i}$ l'entier tel que $\pi_i$ soit une représentation automorphe cuspidale de $G_{d_{\Pi_i}}(\AAA)$. 
Notons que pour tout $\Pi_i \in I$, on a \begin{equation}\sum_{j=1}^{\alpha_i } l_{i, j}=m_{i}.  \end{equation}

Pour tout $\nu\in \mathbb{N}^{*}$, soit \begin{equation}  a_{\nu}:=\sum_{ \Pi_i \in I_{\nu}} m_{i} d_i  , \end{equation} où $d_i$ est le rang de $\pi_i$. 
Donc on a \begin{equation} \label{parN}n=\sum_{\nu\in {{{N}}}  }   a_{\nu}\nu. \end{equation}

\subsubsection{Le résultat et la preuve} 
\begin{theorem}\label{Matr}
 Avec les notations précédentes, pour le facteur (c.) de la proposition \ref{expr1}, on a 
\begin{multline}\label{SSSS}   \sum_{F }     \prod_{\beta\in F }\biggr(\mathrm{N}(n_{\beta}(\tilde{\pi}, \cdot))- \mathrm{P}(n_{\beta}(\tilde{\pi}, \cdot))\biggr)  =
%{ \sum_{F\subseteq\Phi_{\bar{P}}   } \prod_{\beta\in F }( \bf{N}    (n_{\beta}({\pi}, \cdot))- \bf{{P}}({n_{\beta}}({\pi}, \cdot))) }
 \\
\frac{ |w| (\prod_{ \Pi_i \in I}d_{i})   (2g-2)^{|I|-1} \prod_{\mu\in {{{N}}}}  (  \sum_{\nu\in {{{N}}}} a_{\nu} \mathrm{min}\{\mu,\nu\}    )^{|I_{\mu}|}  }{ n  \sum_{\nu\in N_{ }} a_{\nu}   }       \\
\prod_{ \Pi_i  \in I  }    ( |\Fix( \Pi_i)|m_{i}+ (2g-2)d_{i} \sum_{ \nu\in {{{N}}}}a_{\nu}\mathrm{min}\{\nu_i,\nu\}  )^{\alpha_{i}-1}  , 
\end{multline}
où la première somme $\sum_{F}$ porte sur toutes les parties $F$ de $-\Phi_{Q_{L}}$  telles que $F$ forme une base de $\ago_{L}^{G,*}$. Rappelons que $|w|$ est le produit des longueurs des orbits de $w$;  les opérateurs $\mathrm{N}$ et $\mathrm{P}$ donnent respectivement le nombre de zéros et de pôles d'une fonction dans la région $|z|<1$; et $g$ est le genre de la courbe $X_1$.  
\end{theorem}

\begin{proof}
On va utiliser le théorème \ref{Mat}  en  posant $y_{\beta}=\mathrm{N}(n_{\beta}(\tilde{\pi}, \cdot))- \mathrm{P}(n_{\beta}(\tilde{\pi}, \cdot))$ 
pour construire une matrice $\mathrm{M}$ telle que $\kappa(\mathrm{M})$ (cf. \ref{V11} pour la notation) soit égale au membre à gauche de l'expression  (\ref{SSSS}). 
%La matrice $\mathrm{M}_{{\pi}, w}$ comme suit: 

Par le corollaire \ref{nombre} et l'équation (\ref{Zeros}) de \ref{constop}, le nombre de zéros et de pôles de $n_{\beta}(\tilde{\pi}, \cdot)$ dépend de la condition d'égalité, ou non, des deux facteurs de $\Pi$ correspondant à $\beta$ sont égaux ou non. Il faut décrire $\mathrm{M}$ comme une matrice par blocs.
On utilise $I$ pour indexer les blocs de $\mathrm{M}$. Pour chaque bloc associé à $(\Pi_i,\Pi_j )\in I\times I$, on utilise $1\leq s\leq m_{i}$ et $1\leq t\leq m_{j}$ pour indexer les éléments dans le bloc. 
Afin de réduire l'utilisation des indices,  
on désigne par  $({x_{i, j}}({s, t})_{s, t})$ la matrice par blocs $(A_{\Pi_i, \Pi_j})_{\Pi_i, \Pi_j \in I}$ où chaque bloc est donné par  $${A_{\Pi_i, \Pi_j}}=\biggr({x_{i, j}}(s,t)\biggr)_{\substack{1\leq s\leq m_{i}\\ 1\leq  t\leq m_{j} }}.$$

La matrice $\mathrm{M}$ peut alors s'écrire  
           $$\mathrm{M}= \mathrm{diag}( y_{i}(s)_{s})    -  (x_{{i}, {j}}  (     s   ,    t    )_{s,t}  )_{  {{i}} ,    { {j}}      }, $$ 
où on a par la définition \eqref{Zeros} de $n_\beta$ avec $\beta\in \Phi(Z_L, G)$ et le corollaire \ref{nombre}:   
$$x_{ {i}, {j}}(     s ,    t    ) = 
 \begin{cases} -l_{i, s} l_{ j, t}  (2g-2)d_{ {i}}d_{ {j}} \min\{\nu_i, \nu_j\}, \quad \quad \text{si $ \Pi_{i}\neq \Pi_{j}$;} \\
\\ 
-l_{i,s} l_{j,t}  \biggr((2g-2)d_{ {i}}d_{ {j}} \min\{\nu_i, \nu_j \} +|\Fix( \Pi_{i})|\biggr), \quad  \quad \text{si $ \Pi_{i}= \Pi_{j}$;} 
\end{cases}
$$
et $$y_{i}(s)=  l_{i, s }  |\Fix( \Pi_i)| m_{i}
 +   (2g-2)l_{i, s} d_{i }      \sum_{  \Pi_j    \in I } d_{j}    m_{j} \mathrm{min}\{ \nu_j,  \nu_i \}  ,$$
 est donné de manière à ce que la somme de chaque ligne de $\mathrm{M}$ soit nulle. 

Soit
 $$ y_{i}=|\Fix( \Pi_i)|m_{i}+ (2g-2)d_{i} \sum_{ \nu\in {{{N}}}}a_{\nu}\mathrm{min}\{\nu_i ,\nu\}, $$ 
alors on a $$y_{i}(s)=y_{i}l_{i,s} .$$
Soit $$x_{{i}, {j}} = (2g-2)d_{{i}}d_{{j}} \min\{\nu_i, \nu_j \} + \delta_{ {i}, {j}}|\Fix( \Pi_{i})|, $$ où $\delta_{ {i}, {j}}$ est le symbole de Kronecker, on a donc $$x_{{i}, {j}}(s, t)=l_{i, s}l_{j,t} x_{i, j} .  $$

On a besoin d'un lemme qui nous permet de réduire le rang de la matrice pour calculer le déterminant:

\begin{lemm}\label{det}
Soit $(a_{ij})_{1\leq i, j\leq k}$ une famille de nombres complexes. 
On se donne pour tout $1\leq i\leq k$ un vecteur complexe $U_{i}={^{t}(u^{i}_{1},\ldots, u^{i}_{n_{i}})}$ de taille $n_i\times 1$. Soient $N=\sum_{i=1}^{k}n_{i}$, et $J_{n_{i}\times n_{j}}$ la matrice de taille $n_{i}\times n_{j}$ dont tous les éléments sont $1$, alors on a l'égalité:
 \begin{multline*}\det(\Id_{N}- (a_{ij }J_{n_{i}\times n_{j}})_{1\leq i,j\leq k}\mathrm{diag}(u^{1}_{1},\ldots, u^{1}_{n_{1}}, u^{2}_{1},\ldots, u^{k}_{n_{k}} )) \\ =\det(\Id_{k}-(a_{ij}\sum_{l=1}^{n_{j}}u^{j}_{l})_{1\leq i, j\leq k}).  \end{multline*}
\end{lemm}
\begin{proof}
Supposons tout d'abord que les valeurs propres de la matrice $(a_{ij}\sum_{l=1}^{n_{j}}u^{j}_{l})_{1\leq i,j\leq k}$ sont deux-à-deux distinctes et non-nulles, en particulier la matrice est diagonalisable. 
%est diagonalisable et que ses valeurs propres sont non-nulles.  

Soient $(V_{i})_{1\leq i\leq k}$ des vecteurs de taille $n_{i}\times 1$, tels que ${^{t}U_{i}}V_{i}=0$ pour tout $1\leq i\leq k$, alors le vecteur %$^{t}(^{t}V_{1}, ^{t}V_{2},\ldots, ^{t}V_{k})$
 $\begin{pmatrix} V_1\\V_2\\ \vdots\\V_k  \end{pmatrix}$
est un vecteur propre pour la valeur propre $0$ de $(a_{ij }J_{n_{i}\times n_{j}})_{N\times N}\mathrm{diag}(u^{1}_{1},\ldots, u^{1}_{n_{1}}, u^{2}_{1},\ldots, u^{k}_{n_{k}}). $ Les vecteurs de ce type engendrent  un sous-espace de l'espace propre pour la valeur propre $0$, donc l'espace propre pour la valeur propre $0$ est de dimension$\geq \sum_{i=1}^{k}(n_{i}-1)=N-k$. Maintenant, soit $V=(v_{1},\ldots,v_{k})$ un vecteur propre pour la valeur propre $\lambda\neq 0$ de la matrice $(a_{ij}\sum_{l=1}^{n_{j}}u^{j}_{l})_{1\leq i, j\leq k}$, on vérifie que le vecteur 
$$^{t}(\underbrace{v_{1}, \ldots, v_{1}}_{n_1 \text{ fois}},\ldots, \underbrace{v_k,\ldots, v_{k}}_{n_k \text{ fois}})$$
est un vecteur propre de $(a_{ij }J_{n_{i}\times n_{j}})_{1\leq i, j\leq k}\mathrm{diag}(u^{1}_{1},\ldots, u^{1}_{n_{1}}, u^{2}_{1},\ldots, u^{k}_{n_{k}} )$ de valeur propre $\lambda$ aussi. Par notre hypothèse, ces vecteurs engendrent un sous-espace de dimension $k$. Donc l'égalité considérée est vraie d'après la comparaison des valeurs propres des deux côtés.  

Notons que l'ensemble des éléments $((a_{ij})_{1\leq i,j\leq k}, U_1, \ldots, U_k)\in Mat_{k}(\mathbb{C})\times \mathbb{C}^{\sum_{j=1}^k n_j}$ tels que les valeurs propres de la matrice $(a_{ij}\sum_{l=1}^{n_{j}}u^{j}_{l})_{1\leq i,j\leq k}$ soient deux-à-deux distinctes et non-nulles est un sous-ensemble ouvert non-vide de $Mat_{k}(\mathbb{C})\times \mathbb{C}^{\sum_{j=1}^k n_j}$. %, donc Zariski dense.  
L'égalité étant vraie pour un sous-ensemble ouvert est donc vraie pour tout $((a_{ij})_{1\leq i,j\leq k}, U_1, \ldots, U_k)\in Mat_{k}(\mathbb{C})\times \mathbb{C}^{\sum_{j=1}^k n_j}$. 
\end{proof}

Comme la somme des coefficients de chaque ligne et chaque colonne de  la matrice $\mathrm{M}$ est nulle, 
par le lemme \ref{kappa} (l'égalité entre les quantités $a.$ et $e.$), on a 
$$\kappa(\mathrm{M})=\frac{1}{(\sum_{\nu\in N_{ } }     a_{\nu}  )^{2}    }  \det( \mathrm{M}+  ((l_{i, s} l_{j, t} d_{{i}}d_{{j}} )_{s, t }) ) .$$
Donc par le lemme \ref{det},  cette expression est égale à
\begin{multline*}
 \frac{1}{(\sum_{\nu\in {{{N}}} }     a_{\nu}  )^{2}    }     \det\left(  \mathrm{diag}( y_{i}( s)_s) -(x_{i , j}  (  s   ,    t)_{   s   ,t})     +          ((l_{i,s} l_{j, t}  d_{i}d_{ j } )_{s,t} ) \right )\\
      =  \frac{\prod_{ \Pi_i \in I}  \biggr(y_{i}^{\alpha_{i}-1}     \prod_{s=1}^{\alpha_{i}}l_{ i, s} \biggr) }{(\sum_{\nu\in {{{N}}} }     a_{\nu}  )^{2}    }   \det\left(   \mathrm{diag}( y_{i})_{i} -(x_{ i, j}  m_{ {j}})_{ {i}, {j}}+ ( d_{ {i}}d_{{j}} m_{ {j}})_{ {i}, {j}}     \right)  .
                    \end{multline*} 
Soient $$z_{\mu} =  (2g-2)\sum_{ \nu\in {{{N}}}}a_{\nu}\mathrm{min}\{\mu,\nu\} ,$$ et $$ \omega_{\mu,\nu}=(2g-2)\mathrm{min}\{ \mu,\nu       \} , $$ pour $\mu,\nu\in {{{N}}}$.  En utilisant encore une fois le lemme \ref{det}, on voit que $   \kappa(\mathrm{M}) $ est égal à
\begin{multline*}
        \frac{\prod_{ \Pi_i  \in I}  \biggr(y_{i}^{\alpha_{i}-1}     \prod_{s=1}^{\alpha_{i}}l_{i, s} \biggr) }{(\sum_{\nu\in {{{N}}} }     a_{\nu}  )^{2}    }\det( \mathrm{diag}(z_{\nu_i} d_{i} )_{i} -(\omega_{\nu_i,  \nu_j} m_{{j}}d_{{i}}d_{{j}})_{{i}, {j}}+  (d_{ {i}}d_{ {j}}m_{ {j}})_{{i},{{j}}}    ) = \\        
                                                  \frac{(\prod_{\mu{\in {{{N}}}  }}   z_{\mu} ^{ |I_{\mu}| -1}  )  \prod_{ \Pi_i \in I}  \biggr(d_{i}y_{i}^{\alpha_{i}-1}     \prod_{s=1}^{\alpha_{i}}  l_{i,s} \biggr) }{(\sum_{\nu\in {{{N}}} }     a_{\nu}  )^{2}    } \det (   \mathrm{diag}(z_{\mu})_{\mu} - (\omega_{\mu,\nu} a_{\nu} )_{\mu,\nu} +  ( a_{\nu})_{\mu,\nu} )   .
                   \end{multline*}  
              De plus, comme la somme des coefficients de chaque ligne de la matrice $\mathrm{diag}(z_{\mu})_{\mu} - (\omega_{\mu,\nu} a_{\nu} )_{\mu,\nu} $ est nulle, par l'égalité entre les valeurs $a.$ et $d.$ du lemme \ref{kappa},  l'expression ci-dessus peut s'écrire comme \begin{equation}\label{kM1}  |{{{N}}}|   \frac{(\prod_{\mu{\in {{{N}}}  }}   z_{\mu} ^{ |I_{\mu}| -1}  )  \prod_{ \Pi_i \in I}  \biggr(d_{i}y_{i}^{\alpha_{i}-1}     \prod_{s=1}^{\alpha_{i}}l_{i,s} \biggr) }{  \sum_{\nu\in {{{N}}} }     a_{\nu}   } \kappa (   \mathrm{diag}(z_{\mu})_{\mu} - (\omega_{\mu,\nu} a_{\nu} )_{\mu,\nu}  ) .\end{equation}

\sloppy
Pour calculer $\kappa (   \mathrm{diag}(z_{\mu})_{\mu} - (\omega_{\mu,\nu} a_{\nu} )_{\mu,\nu}  )$, on va transformer la matrice $  \mathrm{diag}(z_{\mu})_{\mu} - (\omega_{\mu,\nu} a_{\nu} )_{\mu,\nu} $ en une matrice triangulaire supérieure. 
On ordonne l'ensemble ${{{N}}}$ de telle sorte que $${{{N}}}=\{\nu_{1},\ldots,\nu_{|{{{N}}}|}\},$$ et $$\nu_{1}<\nu_{2}<\cdots<\nu_{|{{{N}}}|} . $$   
Soit $P$ la matrice de taille $ |{{{N}}}| \times |{{{N}}}| $, qui est triangulaire inférieure dont tous les éléments sous-diagonaux sont $1$ :
$$P=\begin{pmatrix}
1 &0 &0 &\cdots &0   \\
 1 & 1 &0 &   \cdots    &0 \\
  1 & 1 &1 &   \cdots    &0 \\
   &     & \ddots & &  \\
   1&   1   &    \cdots        &  & 1
    \end{pmatrix}    \   
\text{d'inverse} \ P^{-1}= \begin{pmatrix}
1 &0 &0 &\cdots &0   \\
 -1 & 1 &0 &   \cdots    &0 \\
  0 & -1 &1 &   \cdots    &0 \\
   &     & \ddots & &  \\
   0&   0  &    \cdots        &  -1& 1
    \end{pmatrix}   . 
    $$
On considère la matrice: 
$$P^{-1} \left( \mathrm{diag}(z_{\mu})_{\mu} - (\omega_{\mu,\nu} a_{\nu} )_{\mu,\nu} \right)P .$$
Après un calcul direct, on trouve que c'est une matrice triangulaire supérieure telle que les éléments diagonaux sont $(2g-2)$  multiplié par $$0,\  \nu_{1}(\sum_{i=1}^{{|{{{N}}}|}}a_{\nu_{i}}),\  \nu_{1}a_{\nu_{1}}+\nu_{2}(\sum_{i=2}^{{|{{{N}}}|}}a_{\nu_{i}}),\ \ldots,\  (\sum_{i=1}^{|{{{N}}}|-1}\nu_{i}a_{\nu_{i}})+\nu_{{{|{{{N}}}|-1}}} a_{\nu_{{{|{{{N}}}|}}}} . $$
En utilisant encore une fois le lemme \ref{kappa} (l'égalité entre $(a.)$ et $(b.)$), on a
\begin{equation} \kappa (   \mathrm{diag}(z_{\mu})_{\mu} - (\omega_{\mu,\nu} a_{\nu} )_{\mu,\nu}  )=\frac{(2g-2)^{|{{{N}}}|-1 } \prod_{\mu\in {{{N}}}}  (\sum_{ \nu\in {{{N}}}} a_{\nu} \mathrm{min}\{\mu,\nu\} ) }{  n  |{{{N}}}   |  } , \end{equation}
où on a utilisé le fait que $n=\sum_{\nu\in N} a_\nu \nu$ pour simplifier l'expression. 

On conclut que la valeur cherchée (l'expression (\ref{kM1})) est égale à
\begin{multline}\label{CON}
\frac{ |w| (\prod_{ \Pi_i \in I}d_{i})   (2g-2)^{|I|-1} \prod_{\mu\in {{{N}}}}  (  \sum_{\nu\in {{{N}}}} a_{\nu} \mathrm{min}\{\mu,\nu\}    )^{|I_{\mu}|}  }{ n  \sum_{\nu \in {{{N}}}} a_{\nu}   }       \\
\prod_{ \Pi_i \in I }    ( |\Fix( \Pi_i )|m_{i}+ (2g-2)d_{i} \sum_{ \nu\in {{{N}}}}a_{\nu}\mathrm{min}\{\nu_i,\nu\}  )^{\alpha_{i}-1}, 
\end{multline}
où, par definition, $|w|= \prod_{\Pi_i \in I} \prod_{j=1}^{\alpha_{i}}l_{i,j} $. 
\end{proof}

\subsection{Calcul de la somme $\sum_{(w,1)}$ dans la proposition \ref{expr1}}
\begin{prop}\label{ee}
Avec les notations de \ref{notations}, quand $(e, n)=1$, on a que $J_e$ est égal à
\begin{multline} \sum_{l\mid n}  \sum_{(P, \tilde{\pi})}     \mu(l)  \frac{1}{ |\stab(P,\tilde{\pi})|(2g-2) n  \sum_{\nu\in {{{N}}}} a_{\nu}   }    \\   \prod_{ \Pi_i \in I} \begin{pmatrix}  S_{i}/\xi_{i}   \\  m_{i}/\xi_{i} \end{pmatrix} { |\Fix( \Pi_i )|^{m_{i}} (-1)^{\frac{m_{i}}{ \xi_{i}   }  } } m_{i}! ,  \end{multline} 
où la somme $\sum_{l\mid n}$ porte sur les diviseurs $l$ de $n$, la somme 
$\sum_{(P, \tilde{\pi})}$ porte sur un ensemble de bons représentants $(P, \tilde{\pi})$  des classes d'équivalence inertielle de paires discrètes telles que $\xi_i= \frac{l}{(l, |\Fix( \Pi_i)|)} \mid m_{i}$ pour tout $ \Pi_i\in I$ et $S_{i}=- \frac{ (2g-2)d_{i} \sum_{ \nu\in {{{N}}}}a_{\nu}\mathrm{min}\{\nu_i,\nu\} }{|\Fix( \Pi_i)|} $. 
\end{prop}
\begin{proof}
Par l'expression de $J_e$ donnée dans la  proposition \ref{expr1}, avec les résultats des calculs dans la proposition \ref{ABCD} et le théorème \ref{Matr},  le nombre $J^{ }_{e}$ est la somme $$\sum_{l\mid n}$$
 portant sur les diviseurs $l$  de $ n$, la somme $$\sum_{(P,\tilde{\pi})}$$
 portant sur un ensemble de bons représentants $(P, \tilde{\pi})$ des classes d'équivalence inertielle des paires discrètes partout  non-ramifiées, et la somme  $$\sum_{(w,1)\in \stab(P, \tilde{\pi})}$$ 
 portant sur les $(w,1)\in \stab(P, \tilde{\pi})$ tels que les longueurs des cycles $l_{i,j}$ de $w$  satisfaisants $l\mid  l_{i, j}|\Fix(\Pi_i)|$ pour tout $ \Pi_i \in I, $et tout $j$ (cf. \ref{notations}), de l'expression:
\begin{multline}\label{Better}
 \frac{\mu(l)}{|\stab(P, \tilde{\pi})|}    \frac{  (\prod_{ \Pi_i\in I}d_{i})    (2g-2)^{|I|-1} \prod_{\mu\in {{{N}}}}  (  \sum_{\nu\in {{{N}}}} a_{\nu} \mathrm{min}\{\mu,\nu\}    )^{|I_{\mu}|}  }{ n  \sum_{\nu\in {{{N}}}} a_{\nu}   } \\  \prod_{\Pi_i \in I } \biggr( |\Fix(\Pi_i)|^{m_i-1 }  (-1)^{{m_i}/{\xi_i}}  \biggr)    \\
\prod_{ \Pi_i \in I} ( - \frac{|\Fix( \Pi_i)|m_{i}+ (2g-2)d_{i} \sum_{ \nu\in {{{N}}}}a_{\nu}\mathrm{min}\{\nu_i,\nu\}}{|\Fix( \Pi_i)|}  )^{\alpha_{i}-1} ,
  \end{multline} 
où $\xi_{i}= \frac{l}{(l, |\Fix( \Pi_i)|)}$ et on a utilisé les identités
$$\prod_{j=1}^{\alpha_{i}}|\Fix(\Pi_i)|^{{l}_{i, j}-1}=|\Fix( \Pi_i)|^{m_{i}-1}|\Fix( \Pi_i )|^{1-\alpha_{i}}  \quad \forall \Pi_i \in I, $$
et 
$$(-1)^{ \sum_{ \Pi_i \in I} \sum_{j}({l_{i, j} }/{  \xi_{i} }-1)      }=\prod_{ \Pi_i \in I}(-1)^{({m_{i}}/{ \xi_{i} }-1)-(\alpha_{i}-1)}.$$

On observe que la première ligne de $(\ref{Better})$ ne dépend pas de $w$ et la deuxième ligne de $(\ref{Better})$ ne dépend que des longueurs des cycles des facteurs de $w$ dans chaque $\mathfrak{S}_{m_{i}}$ pour $\Pi_i\in I$. Soient $\xi, m\in \mathbb{N}^{*}$.
On utilise la notation $$(j^{c_{j}})\gls{vdashxi} m$$ si $(j^{c_{j}})=(1, \ldots, 1, 2, \ldots, 2, \ldots)$ (chaque $j$ est répété $c_j$ fois)   est une partition non-ordonnée de $m$ telle que $\xi\mid j$ dès que $c_j\neq 0$.  Notons que pour une partition $(j^{c_{j}})$ de $m_{i}$ le nombre des permutations dans $\mathfrak{S}_{m_i}$ dont les cycles ont pour longueurs  $(j^{c_j})$ est  $ \frac{m_{i}!}{\prod_{j}c_{j}!  \prod_{j} j^{c_{j}}}$.
De plus, pour tout $\Pi_i \in I$, avec $\xi_i= \frac{l}{(l, |\Fix( \Pi_i)|)}$, 
une partition $(l_{i, 1},\ldots, l_{i, \alpha_i})$ de $m_{i}$  satisfait $$l\mid l_{i ,j} |\Fix( \Pi_i )|, \quad  \forall\ j=1, \ldots, \alpha_i, $$ si et seulement si  $$\xi_{i}\mid l_{i, j} ,\quad   \forall j=1, \ldots, \alpha_i.$$

Considerer la somme  $$  \sum_{(w,1)\in \stab(P, \tilde{\pi})} \prod_{ \Pi_i \in I} ( - \frac{|\Fix( \Pi_i)|m_{i}+ (2g-2)d_{i} \sum_{ \nu\in {{{N}}}}a_{\nu}\mathrm{min}\{\nu_i,\nu\}}{|\Fix( \Pi_i)|}  )^{\alpha_{i}-1} ,  $$
qui est la somme portant sur les  $w$  de l'expression dans la deuxième ligne d'equation (\ref{Better}), elle est égale à
 \begin{multline*}\sum_{\{(j^{c_{i, j}  } )\vdash_{\xi_i} m_i|\ \Pi_i \in I \}}  \frac{m_{i}!}{\prod_{j}c_{i, j}^{}!  \prod_{j} j^{c_{i, j}  }}  \\ 
   \prod_{ \Pi_i \in I} ( - \frac{|\Fix( \Pi_i )|m_{i}+ (2g-2)d_{i} \sum_{ \nu\in {{{N}}}}a_{\nu}\mathrm{min}\{\nu_i,\nu\}}{|\Fix( \Pi_i )|}  )^{\sum_{j}c_{i, j} -1} .
\end{multline*}
La condition nécessaire pour que la somme ne soit pas vide est $\xi_i\mid m_{i}$.

\begin{lemma}
Soit $S\in \mathbb{C}^\times$. Soient $\xi, m$ deux nombres naturels tels que $\xi\mid  m$. On a l'identité:
\begin{equation} m! \sum_{(j^{c_{j}})\vdash_{\xi} m}  \frac{1}{\prod_{j}c_{j}!  \prod_{j} j^{c_{j}}}S^{\sum_{j}c_{j}-1}= {(m-1)!} \binom{\frac{S}{\xi}+\frac{m}{\xi}-1}{ \frac{m}{\xi}-1}.
\end{equation}
\end{lemma}
\begin{proof}[Preuve du lemme]
Considerons la fonction génératrice suivante: 
$$ \sum_{\{n\in\mathbb{N}| \ \xi | n\}} \sum_{(j^{c_j}) \vdash_{\xi} n}  \frac{1}{\prod_{j\geq 1}{c_j ! j^{c_j}} }S^{\sum_{j\geq 1}c_j } z^{n} ,$$
elle est égale à $$\exp(S\sum_{\{v\in \mathbb{N}^{*}| \xi \mid v\}}\frac{z^{v}}{v})=\exp(-\frac{S}{\xi}\ln(1-z^{\xi}))=\frac{1}{(1-z^{\xi})^{S/\xi}}    .$$
On obtient l'identité voulue  par une comparaison de coefficients. 
\end{proof}

Donc si $\xi_i\mid m_{i}$, $\forall\  \Pi_i \in I $,  la somme portant sur les $w$ de l'expression (\ref{Better}) est égale à 
\begin{multline*}
\frac{ (\prod_{ \Pi_i \in I}d_{i})  (2g-2)^{|I|-1}  \prod_{\mu\in {{{N}}}}  (  \sum_{\nu\in {{{N}}}} a_{\nu} \mathrm{min}\{\mu,\nu\}    )^{|I_{\mu}|}  }{ n  \sum_{\nu\in {{{N}}}} a_{\nu}   }       \\
  \prod_{ \Pi_i \in I  }(-1)^{\frac{m_{i}}{ \xi_i }  -1} \begin{pmatrix}-1+ \frac{  S_{ i}}{\xi_{ i}} \\ \frac{m_{ i}}{\xi_{ i}} -1  \end{pmatrix}   |\Fix({ \Pi_i })|^{m_{ i}-1 }{(m_{ i}-1)!},
\end{multline*}
où $S_{i}=- \frac{ (2g-2)d_{ i} \sum_{ \nu\in {{{N}}}}a_{\nu}\mathrm{min}\{\nu_i ,\nu\} }{|\Fix( \Pi_i )|} $.  Cette expression permet la simplification suivante:
$$\frac{1}{ (2g-2) n  \sum_{\nu\in {{{N}}}} a_{\nu}   }\prod_{ \Pi_i \in I }\begin{pmatrix}  S_{i}/\xi_{i}   \\ m_{i}/\xi_{i} \end{pmatrix} { |\Fix( \Pi_i )|^{m_{i}} (-1)^{\frac{m_{i}}{ \xi_{ i}   }  } }m_{ i}! . $$
Sinon, la somme  est égale à zéro.
\end{proof}

\subsection{Fin de la démonstration du théorème \ref{b}}\label{4.8}
\label{MoWa}
Dans cette sous-section, on achève la démonstration du théorème \ref{b}. Dans la proposition \ref{ee}, on a obtenu une expression pour $J_e$ qui est une somme indexée par les classes d'équivalence
inertielle des paires discrètes $(P, \pi)$ partout non-ramifiées. 
%Il faut trouver une expression  pour la somme portant sur toutes  classes d'équivalence inertielle des paires discrètes $(P,{\pi})$ partout  non-ramifiées dans le corollaire \ref{ee}. 
La stratégie est d'ensuite utiliser la description du spectre résiduel due à Moeglin-Waldspurger pour transformer cette somme en une somme sur les paires cuspidales. On obtient alors l'expression \ref{Paren}. Le problème restant est alors essentiellement combinatoire et nous utilisons le lemme \ref{THEO}  pour exprimer le nombre de représentations automorphes cuspidales en fonction de celui de représentations automorphes absolument cuspidales.

\subsubsection{Passer de la paire discrète à la paire cuspidale}

On rappelle les symboles et les nombres  (de \ref{notations}) qu'on continuera à utiliser (avec une petite modification).
Soit $(P, \tilde{\pi})$ un bon représentant d'une paire discrète partout non-ramifiée, c'est-à-dire:
$$\tilde{\pi}=\underbrace{ \Pi_1\otimes\cdots\otimes  \Pi_{1}}_{m_{1}}\otimes \underbrace{  \Pi_{2}\otimes\cdots\otimes  \Pi_2 }_{m_{2}}\otimes\cdots \otimes\underbrace{ \Pi_{k}\otimes\cdots\otimes \Pi_{k}}_{m_{k}} ,$$
et $\Pi_i=\pi_i\boxtimes \nu_i$ sont deux-à-deux non inertiellement équivalentes. 
Soit $$I({\tilde{\pi}})=\{ \Pi_1, \ldots,   \Pi_k\},$$
$$I_{j} =\{   \Pi_i \in I({\tilde{\pi}) } \mid \nu_i=j  \} \quad \forall j\in \mathbb{N}^{*}.  $$
Pour tout $j\in \mathbb{N}^{*}$,  soit 
\begin{equation}  a_{j}=\sum_{ \Pi_i \in I_{j}} m_{i} d_i, \end{equation} où $d_i $ est le rang de $\pi_i $. 
Notons qu'on a \begin{equation}\label{njaj} n=\sum_{j\geq 1 }   ja_{j}, \end{equation}
i.e. $(j^{a_j})$, les entiers naturels  $j$ répétés $a_j$ fois,  est une partition de $n$. %Soit $\lambda=(j^{a_j})=(1,\ldots, 1, 2, \ldots, 2, \ldots)$ (chaque $j$ est répété $a_j$ fois) cette partition de $ n $.  
On va regrouper la somme de la proposition \ref{ee} sur les paires discrètes partout non-ramifiées $(P, \tilde{\pi})$ suivant cette partition. 

%pour lesquelles on a pour tout $j\geq 1$,   \begin{equation}\sum_{  \Pi \in I_{j} } {\mathrm{r}(\Pi)m_\Pi}=j a_{j} \quad \text{ (i.e. $\sum_{  \pi\boxtimes j \in I_{j} } {\mathrm{r}(\pi)m_{\pi\boxtimes j} }= a_{j}$ ) }. \end{equation}
%où $\mathrm{r}(\Pi)$ est le rang de $\Pi$. 

Soit $(P, \tilde{\pi})$ un bon représentant d'une classe d'équivalence inertielle des paires discrètes de $G$.   On lui associe les pairs discrètes $(P_j, \tilde{\pi}_j)$ de $G_{ja_j}$ pour $j\geq 1$, de façon que $$M_{P_j}=\prod_{\Pi_i \in I_j }(G_{jd_i})^{m_{i}},$$ et $$\tilde{\pi}_j= \bigotimes_{\Pi_i \in I_j} \Pi_i^{\otimes m_{i}}.$$ C'est-à-dire, $\tilde{\pi}$ est un produit tensoriel de $\tilde{\pi}_j$ et les facteurs de $\tilde{\pi}_j$ sont des représentations automorphes discrètes de la forme $\pi\boxtimes j$ avec $\pi$ automorphe cuspidale. 
La classe d'équivalence inertielle de $(P_j, \tilde{\pi}_j)$ est uniquement déterminée par la classe d'équivalence inertielle de $(P, \tilde{\pi})$. De plus on a $$\frac{1}{|\stab(P, \tilde{\pi})|}=\prod_{j\geq 1}\frac{1}{|\stab(P_j, \tilde{\pi}_j)|} .$$

Par le théorème de la classification du spectre automorphe résiduel  de Moeglin-Waldspurger (le théorème \ref{MWres}) et la proposition  \ref{res e}, l'ensemble des classes d'équivalence inertielle des représentations discrètes $\pi\boxtimes j$ pour $\pi$ cuspidale et   $j$ fixé est en bijection avec l'ensemble des classes d'équivalence inertielle des représentations cuspidales $\pi$.  De plus,  on a $|\Fix(\pi\boxtimes j)|=|\Fix(\pi)|$. Soit $$\glslink{Silambda}{S_i(\lambda)}:=\sum_{j\geq 1} a_j \min\{i, j\}, $$ 
alors le nombre $S_i$ défini dans la proposition \ref{ee} satisfait (rappelons que $\nu_i=j$)
$$S_i=- \frac{ (2g-2)d_{i} \sum_{ \nu\in {{{N}}}}a_{\nu}\mathrm{min}\{ \nu_i ,\nu\} }{|\Fix( \Pi_i)|} =-\frac{(2g-2)d_i S_{j}(\lambda)}{|\Fix(\pi_i)|}. $$
On note $(P,\tilde{\pi})<_{cusp}G_k$ si $(P, \tilde{\pi})$ est une classe d'équivalence inertielle de paire discrète partout-non-ramifiée de $G_k$ avec $\tilde{\pi}$ une représentation automorphe cuspidale de $M_P(\AAA)$. 
Quand $(n,e)=1$, par la proposition \ref{ee},  $J^{ }_e$ est égal à 
\begin{multline}\label{Paren} \sum_{l\mid n}\sum_{ \lambda=(j^{a_j})\vdash n}  \frac{\mu(l)}{(2g-2) n \sum_{j\geq 1}a_{j}      } \prod_{j\geq 1} \\
\biggr(   \sum_{ \substack{ (P, \tilde{\pi})<_{cusp}G_{a_{j}}\\ \xi_{i}\mid m_i  , \forall \pi_i \in I({\tilde{\pi}})    }} \frac{1}{|\stab(P,\tilde{\pi})|}\prod_{ \pi_i\in I({\tilde{\pi}})}  \binom{ - \frac{ (2g-2)d_i  S_j(\lambda)      }{ \xi_{i}  |\Fix( \pi_i)|} }{ m_{ i}/\xi_{ i}} { |\Fix( \pi_i )|^{m_{i}} (-1)^{\frac{m_{ i}}{ \xi_{ i}   }  } }m_{ i}!      \biggr) .   \end{multline}

\subsubsection{}
On considère la somme entre parenthèses de la deuxième ligne de l'expression (\ref{Paren}), à savoir pour chaque entier $l$ divisant $n$, chaque partition $\lambda$ de $n$, et $j=j_0$:
\begin{equation}\label{EA} \begin{split}
 & \sum_{ \substack{ (P, \tilde{\pi})<_{cusp}G_{a_{j_0}}\\ \xi_{i}\mid m_i  , \forall \pi_i \in I({\tilde{\pi}})    }}  \frac{1}{|\stab(P, \tilde{\pi})|} \\  & \prod_{ \pi_i\in I({\tilde{\pi}})}  \binom{ - { (2g-2)d_i   S_{j_0}(\lambda)      }/{ \xi_{ i}  |\Fix( \pi_i )|} }{ m_{ i}/\xi_{ i}} { |\Fix( \pi_i )|^{m_{ i}} (-1)^{\frac{m_{ i}}{ \xi_{ i}   }  } }m_{ i}!    , \end{split}\end{equation}
où la somme est prise sur un ensemble des bons représentants $(P, \tilde{\pi})$ des classes d'équivalence inertielle des paires discrètes de $G_{a_{j_0}}$ tels que $\xi_i\mid m_i$ pour tout $\pi_i\in I(\tilde{\pi})$. 

On va regrouper cette somme suivant le groupe $\stab(P, \tilde{\pi})$. On veut expliciter l'expression \begin{equation}\label{produip}\prod_{ \pi\in I({\tilde{\pi}})}  \binom{ - { (2g-2) d_i   S_{j_0}(\lambda)      }/{ \xi_{ i}  |\Fix( \pi_i )|} }{ m_{ i}/\xi_{i}} { |\Fix( \pi_i )|^{m_{ i}} (-1)^{\frac{m_{ i}}{ \xi_{ i}   }  } }m_{ i}!  \end{equation}
et le nombre $$|\stab(P, \tilde{\pi})|. $$
%on a besoin de  dépendent du rang, du cardinal du fixateur et de la multiplicité de tout facteur de $\tilde{\pi}$.
Soit $(P, \tilde{\pi})<_{cusp}G_{a_{j_0}}$, tel que $M_P=\prod_{j\geq 1}G_j^{c_j}$. Cela fournit une partition $(j^{c_j})$ de ${a_{j_0}}$. 
\begin{center}
$(\ast)$ 
Soit $k_{d}^{j}$ le nombre de facteurs simples (avec multiplicités) $\pi$ de $\tilde{\pi}$ dont le rang est égal à $j$ et $|\Fix(\pi)|=d$. 
\end{center}
Nécessairement $d\mid j$ (autrement dit $k_d^j=0$ si $d\nmid j$) et $$\sum_{d\mid j}   k_d^j= c_j, \quad \forall\ j\geq 1.$$ 
\begin{center}
$(\ast\ast)$ 
Pour $i\geq 1$, soit $b_{d,i}^{j}$ le nombre de facteurs simples distincts (non équivalents inertiellement) $\pi$ de $\tilde{\pi}$ dont la multiplicité dans $\tilde{\pi}$ est $i$, le rang est $j$ et le cardinal du fixateur $|\Fix(\pi)|=d$. 
\end{center}
Pour chaque $d,j$, cela donne une partition $(i^{b_{d,i}^j})\vdash k_{d}^j$. Avec ces notations, on peut écrire \eqref{produip} comme
\begin{equation}
 \prod_{j\geq 1}\prod_{d\mid j} \prod_{i \geq 1}  \left(    \binom{-(2g-2)j S_{j_0}(\lambda)/d\frac{l}{(l,d)} }{i/\frac{l}{(l,d)}  }^{b_{d,i}^j}   d^{i b_{d,i}^j }(-1)^{ib_{d,i}^j/\frac{l}{(l,d)}}(i!)^{b_{d,i}^j}\right), 
\end{equation}
et on a 
 \begin{equation}\frac{1}{|\stab(P, \tilde{\pi})|}=\prod_{j\geq 1}\prod_{d\mid j}\frac{1}{ d^{k_{d}^{j}} \prod_{i\geq 1} (i!)^{b_{i,j}^d}  }. \end{equation}
Donc leur produit, après une simplification, est égal à
\begin{equation}
\prod_{j\geq 1}\prod_{d\mid j}\left(      (-1)^{k_d^j/ \frac{l}{(l,d)}}    \prod_{i\geq 1}  \binom{-(2g-2)j S_{j_0}(\lambda)/d\frac{l}{(l,d)} }{i/\frac{l}{(l,d)}  }^{b_{d,i}^j} \right). 
\end{equation}

Soit  $\gls{Dnd}$ le nombre des classes d'équivalence inertielle des représentations automorphes cuspidales $\pi$ de $G_n(\AAA)$ telles que $|\Fix(\pi)|=d$. 
Si on se donne une partition $(j^{c_j})\vdash {a_{j_0}}$, des entiers non-négatifs $k^j_d$ tels que $\sum_{d\mid j}k_d^j= c_j$ et des partitions $(i^{b_{d,i}^j})\vdash k^{j}_d$, le nombre de classes d'équivalence inertielle  des paires discrètes  $(P, \tilde{\pi})$ de $G_{a_{j_0}}$  avec $\tilde{\pi}$ cuspidales correspondantes à ces données, c'est-à-dire $(\ast)$ et $(\ast\ast)$ ci-dessus sont satisfaites, est égal à
  \begin{equation} \begin{split}  \prod_{j\geq 1}\prod_{d\mid j}\frac{D_{j}(d)(D_{j}(d)-1)\cdots (D_{j}(d)+1-\sum_{i}b_{d,i}^j  )}{  \prod_{i\geq 1}b_{d,i}^j! } \\ 
  =   \prod_{j\geq 1}\prod_{d\mid j} \binom{D_{j}(d)}{\sum_{i}b_{d,i}^j}\binom{\sum_{i}b_{d,i}^j}{b_{d,1}^j,b_{d,2}^j,\cdots}.  \end{split} \end{equation}

La somme (\ref{EA}) est égale à la somme portant sur toutes les partitions $(j^{c_j})$ de $a_{j_0}$, la somme portant sur tous les entiers non-négatifs $(k_d^j)_{d\mid j}$ tels que $ \sum_{d\mid j}k_{d}^{j} =c_{j}$ et $\frac{l}{(l,d)}\mid k_{d}^{j}$, et la somme portant sur toutes les partitions $(i^{b_{d,i}^j})$ de $k_d^j$ telles que  $b_{d,i}^j=0$ si  $ \frac{l}{(l,d)} \nmid i$ ($i\geq 1$), de l'expression
   \begin{equation} \label{EqA} \prod_{j\geq 1}\prod_{d|j}    (-1)^{k_d^j/ \frac{l}{(l,d)}}  \binom{D_{j}(d)}{\sum_{i}b_{d,i}^j}     \binom{\sum_{i}b_{d,i}^j}{b_{d,1}^j,b_{d,2}^j,\cdots}
  \prod_{i\geq 1}  \binom{-(2g-2)j S_{j_0}(\lambda)/d\frac{l}{(l,d)} }{i/\frac{l}{(l,d)}  }^{b_{d,i}^j}
  \end{equation} 
 
 Par un changement de l'ordre de sommation, pour chaque $(j^{c_j})$ et $(k_d^j)_{d\mid j}$,  la somme portant sur toutes les partitions $(i^{b_{d,i}^j})$ de $k_d^j$ telles que  $b_{d,i}^j=0$ si  $ \frac{l}{(l,d)} \nmid i$ ($i\geq 1$), de l'expression \eqref{EqA} est égale à 
  \begin{multline} \label{EB}
 \prod_{j\geq 1}\prod_{d|j}       (-1)^{k_d^j/ \frac{l}{(l,d)}} \\
   \Biggl(    \sum_{     (i^{b_{d,i}^j     }     )\vdash_{\frac{l}{(l,d)}} k_{d}^j }  \binom{D_{j}(d)}{\sum_{i}b_{d,i}^j}\binom{\sum_{i}b_{d,i}^j}{b_{d,1}^j,b_{d,2}^j,\cdots}
  \prod_{i\geq 1}  \binom{-(2g-2)j S_{j_0}(\lambda)/d\frac{l}{(l,d)} }{i/\frac{l}{(l,d)}  }^{b_{d,i}^j}\Biggl), 
  \end{multline}
rappelons que l'on a noté $(i^{b_{i}})\gls{vdashxi} k$ si $(i^{b_{i}})$ (chaque $i$ est répété $b_i$ fois)   est une partition non-ordonnée de $k$ telle que $b_i= 0$ si $\xi \nmid i$.  
Par le lemme \ref{colemm} ci-dessous, l'expression (\ref{EB}) peut-être simplifiée comme: 
\begin{equation} \prod_{j\geq 1}\prod_{d|j} (-1)^{\frac{k_{d}^{j}}{l/(l,d)} }       \binom{ -(2g-2){S_{j_0}(\lambda)} D_{j}(d) {j}/{d \frac{l}{(l,d)}   }     }{   k_{d}^{j}/  \frac{l}{(l,d)}}. 
\end{equation}

\begin{lemm}\label{colemm}
Soient $S,D\in \mathbb{C}$, et $\xi, k\in \mathbb{N}$ tels que $\xi \mid k $. On a l'identité:
$$\sum_{(i^{b_{i}})\vdash_{\xi} k}  \binom{D}{\sum_{i}b_{i}}\binom{\sum_{i}b_{i}}{b_{1},b_{2},\cdots}   \prod_{i} \begin{pmatrix} S\\
i/\xi
\end{pmatrix}^{b_{i}}= \begin{pmatrix}DS\\ k/\xi
\end{pmatrix} .$$ 
\end{lemm}
\begin{proof}
Tout d'abord, on note que $\binom{\sum_{i}b_{i}}{b_{1},b_{2},\cdots} $ est le nombre des partitions ordonnées dont la partition non-ordonnée associée est égale à  $(i^{b_{i}})$. Donc la somme à gauche est égale à 
$$\sum_{s\geq 1} \sum_{\{(\lambda_{1}, \ldots, \lambda_s )\in (\xi \mathbb{N}^*)^s|\ \lambda_{1}+\cdots+\lambda_{s}=k\} } \begin{pmatrix}D\\ s
\end{pmatrix}\prod_{i=1}^{s} \begin{pmatrix}S\\ \lambda_{i} /\xi
\end{pmatrix}. 
 $$
Soit $$\tilde{S}(z)=\sum_{\{i\geq 1|\ \xi\mid i      \}}\begin{pmatrix}S\\ i/\xi
\end{pmatrix} z^{i}= (1+z^{\xi})^{S}-1 . $$
Alors la somme considérée  est égale au coefficient de $z^{k }$ de la série formelle  
$$\sum_{i\geq 0}\begin{pmatrix}D\\ i
\end{pmatrix} \tilde{S}(z)^{i}= (1+\tilde{S}(z))^{D}=(1+z^{\xi})^{SD}. $$
Ce coefficient est égal à $\binom{DS}{ k/\xi
}$. 
\end{proof}

\subsubsection{Calcul par la fonction génératrice}
 Considérons la fonction génératrice $T(z)$ définie par 
$$  \prod_{j\geq 1}\Biggl( \sum_{\substack{c\geq 0 \\ \frac{l}{(l,d)}\mid c     }} \sum_{  \sum_{d\mid j} k_{d}^{j}=\frac{c}{l/(l,d)} }\biggr(\prod_{d\mid j} (-1)^{k_{d}^{j}} \binom{ -(2g-2){S_{j_0}(\lambda)} D_{j}(d) {j}/{d \frac{l}{(l,d)}   }     }{   k_{d}^{j}   }   \biggr)  z^{jc}\Biggl).
 $$ 
\textit{La somme (\ref{EA}) est égale au coefficient de degré ${{a_{j_0}}}$ de la série formelle  $T(z)$} (notons que le terme constant dans le produit est égal à $1$, donc le produit infini a un sens).  On fait un changement de variables $c\mapsto i\frac{l}{(l,d)}$, il s'ensuit que
\begin{equation}T(z)=\prod_{j\geq 1}\prod_{d|j}\left(  \sum_{i\geq 0} (-1)^{i} \binom{ -(2g-2){S_{j_0}(\lambda)} D_{j}(d)  {j}/{d \frac{l}{(l,d)}   }     }{    i }    z^{ij\frac{l}{ (l,d)}}
 \right).  \end{equation}
 Utilisons les fonctions formelles $\exp$ et $\log$, on a 
 \begin{align} T(z)&=\exp\biggr(\sum_{j \geq 1}\sum_{  d|j }    -(2g-2){S_{j_0}(\lambda)} D_{j}(d)   \frac{j}{d\frac{l}{ (l,d)}} \log(1-z^{j\frac{l}{ (l,d)}}) \biggr)  \\ 
 &= \exp\biggr((2g-2){S_{j_0}(\lambda)} \sum_{j \geq 1}\sum_{  d|j }     D_{j}(d)\frac{j}{d\frac{l}{ (l,d)}} \sum_{k\geq 1}\frac{  z^{jk\frac{l}{ (l,d)}}}{k}\biggr) \nonumber .
 \end{align}

Il faut faire un lien entre $D_n(d)$ et les nombres $C_n(X_k)$, les nombres de classes d'équivalence inertielle de représentations automorphes irréductibles, absolument cuspidales (i.e. celles qui satisfont $|\Fix(\pi)|=1$) et partout non-ramifiées de $G_n(\AAA_k)$. Par définition $C_n(X_1)=D_n(1)$.  
Par la correspondance de Langlands (cf. \ref{PASSS}), $D_n(d)$ est égal au nombre de classes d'équivalence inertielle des représentations $\ell$-adiques $\sigma$ du groupe de Weil $W(X_1, o)$ de rang $n$ telles que $|\Fix(\sigma)|=d$. 
On a le lemme suivant. 
\begin{lemm}\label{THEO}
%Soit $\sigma$ une représentation $\ell$-adique de $\pi_{1}(X\otimes\mathbb{F}_{q},\eta)$ de dimension $n$, soit $|\Fix(\sigma)|=d$, alors $\sigma|_{\pi_{1}(\bar{X},\bar{\eta})}$ est une somme directe de $d$ systèmes locaux $\ell$-adiques différentes (non-isomorphes) qui sont fixées par $\Fr^{d}$. 
On a 
$$D_{n}(d)= \frac{1}{d} \sum_{m\mid d}\mu(m) C_{n/d}(X_{{d/m}}) , $$
où $\gls{mu}$ est la fonction de M\"obius. 
%Similairement, pour les résiduels:  $$R_{i}^{r}(d)=\sum_{l|d}\mu(l) C_{i/dr}(\mathbb{F}_{q^{d/l}})  \frac{1}{d} $$
 \end{lemm}
\begin{proof}
On utilise la correspondance de Langlands et on analyse directement  le côté galoisien.  
Le groupe $(\Fr_{X}^{*})^{\mathbb{Z}}$ agit sur $E_n^{(\ell)}$, soit $O_{n}(d)$ le nombre des orbites de cardinal $d$. 

 L'ensemble des   $\sigma$, représentations $\ell$-adiques irréductibles de rang $n$  de $W(X_{1}, o)$ tels que $|\Fix(\sigma)|=d$ est  en bijection avec l'ensemble des orbites de cardinal $d$ dans $E_{n/d}^{(\ell)}$ sous l'action de $(\Fr_{X}^{*})^{\mathbb{Z}}$. En effet, un système local $\ell$-adique de rang $n/d$ dans une orbite de cardinal $d$, s'étend en une représentation  absolument irréductible de $W(X_d,o)$ par le lemme \ref{Descente}. Son induite $\sigma$ est une représentation irréductible de 
$W(X_1, o)$ telle que $|\Fix(\sigma)|=1$ par la direction (2.$\implies$1.) de la proposition \ref{GEO}. Toutes les extensions des systèmes locaux dans une orbite donnent des représentations inertiellement équivalentes par le lemme \ref{Descente}. Deux orbites disjointes donnent deux classes non-inertiellement équivalentes. 
 Réciproquement, une représentation irréductible $\sigma$ de $W(X_1, o)$ de rang $n$ telle que $|\Fix(\sigma)|=d$ est obtenue de telle façon par un système local $\ell$-adique de rang $n/d$ dans une orbite de cardinal $d$ par la direction (1.$\implies$2.) de la proposition \ref{GEO}. 
On a donc $$O_{n/d}(d) = D_{n}(d).$$

On a aussi que 
$$C_{r}(X_{d})=|(E_r^{(\ell)})^{\Fr_X^{*d}}|=\sum_{m\mid d}m O_{r}(m).$$
Donc par la formule d'inversion de Möbius des fonctions arithmétiques, on a 
$$D_{n}(d)=O_{n/d}(d)=\frac{1}{d}\sum_{m\mid d}\mu(m) C_{n/d}(X_{{d/m}}) .  $$
\end{proof}

Après ce lemme:
$$\log T(z)=(2g-2){S_{j_0}(\lambda)}\sum_{j\geq 1}\sum_{d| j} \sum_{m|d} \mu(m)C_{j/d }(X_{{d/m}}  ) \frac{j}{d^{2}\frac{l}{(l,d)}}  \sum_{k\geq 1} \frac{  z^{jk\frac{ l}{( l,d)}} }{k}  .  $$
On pose $j/d=s$ et $d/m=t$, alors:
$$\log T(z)=(2g-2){S_{j_0}(\lambda)}\sum_{s\geq 1}\sum_{t\geq 1} \sum_{m\geq 1}  \sum_{k\geq 1} \frac{  z^{smtk\frac{ l}{( l,mt)}} }{k} \mu(m)C_{s }(X_t ) \frac{smt}{(mt)^{2}\frac{ l}{( l,mt)}}  . $$
Il s'ensuit d'après une simplification que   
\begin{align*}
&\quad\frac{1}{(2g-2){S_{j_0}(\lambda)}}\log T(z)   \\ =&   \sum_{s\geq 1}\sum_{t\geq 1} \sum_{m\geq 1}  \sum_{k\geq 1} \frac{  z^{smtk\frac{ l}{( l,mt)}} }{k} \mu(m)C_{s }(X_{{t}}  ) \frac{s}{mt\frac{ l}{( l,mt)}}  \\
=& \sum_{N\geq 1} \Biggl( \sum_{\{s, t, k\geq 1|\  stk \mid N  \}  }  \sum_{\{m\geq 1|  \frac{ ml}{( l,mt)}= \frac{N}{stk} \} }   \frac{  C_{s }(X_{t}  ) s^{2}    }{  N}  \mu(m)  \Biggl) z^{N}\\
=& \sum_{N\geq 1} \Biggl( \sum_{ \{s, t \geq 1|\   st \mid N   \} }   \frac{  C_{s }(X_{t}  ) s^{2}    }{  N} \sum_{\{m\geq 1| \frac{ ml}{( l,mt)}| \frac{N}{st} \} }\mu(m)  \Biggl) z^{N} , 
\end{align*}  
où dans la seconde ligne, on a regroupé la somme suivante $N=\frac{ stkml}{( l,mt)} $, et dans la dernière ligne on a enlevé  $k$ de la somme (puisque $\sum_{N}\sum_{\{k,a\in \mathbb{N}|  ka=N\}}f(a)=\sum_{N}\sum_{a\mid N}f(a)$).

\begin{lemm}
Soit $t, l, L\in \mathbb{N}^{*}$ on a l'identité suivante: 
$$ \sum_{\{m\in \mathbb{N}^{*}|\  \frac{ ml}{( l,mt)}\mid L  \}}\mu(m) =\begin{cases}
1\quad \text{si $L=1$ et $l\mid t$}  ; \\
0\quad \text{sinon} .
\end{cases}
$$
\end{lemm}
\begin{proof}
Soient $(p_{s})_{s\in I}$ les nombres premiers qui apparaissent comme des facteurs premiers de $t, l$ ou $L$. 
Supposons $t=\prod_{s\in I} p_{s}^{\alpha_{s}}$, $l=\prod_{s\in I} p_{s}^{\beta_{s}}$ et $L=\prod_{s\in I} p_{s}^{\gamma_{s}}$ avec $\alpha_{s}, \beta_{s}, \gamma_{s}\geq 0$. 
Si $m\in \mathbb{N}^{*}$ est tel que $\frac{ ml}{( l,mt)}|L$, alors on peut supposer $m=\prod_{s\in I} p_{s}^{ \eta_{s} }$. 
On a $$(l,mt)=\prod_{s\in I}p_{s}^{ \text{min}\{\beta_{s}, \alpha_{s}+\eta_{s}   \} },$$ et la condition $\frac{ ml}{( l,mt)}\mid L$ se traduit par:
$$\eta_{s}+\beta_{s}-  \text{min}\{\beta_{s}, \alpha_{s}+\eta_{s}\}  \leq  \gamma_{s} , $$
notons que cela équivaut à $\gamma_{s}\geq \eta_{s}+\beta_{s} - \alpha_{s}-\eta_{s}=\beta_{s}-\alpha_{s}$ et $\gamma_s\geq \eta_s$.

Donc la somme est non-vide si et seulement si $l\mid tL$, et dans ce cas la somme se ramène à la somme $\sum_{m\mid L} \mu(m)$. 
\end{proof}

D'après ce lemme, on peut changer la variable $N$ en $Nl$ et on conclut que 
\begin{equation}\begin{split}
\log T(z) &= \frac{(2g-2){S_{j_0}(\lambda)}  }{l}\sum_{N\geq 1}\sum_{s|N} \frac{C_{s }(X_{{\frac{Nl}{s}}} ) s^{2}  }{N} z^{Nl}  \\ &=\frac{(2g-2){S_{j_0}(\lambda)} }{l}\sum_{s\geq 1}\sum_{k\geq 1} \frac{C_{s }(X_{kl} ) s  }{k} z^{skl}.   
\end{split}
\end{equation}
Le théorème \ref{b} maintenant est un corollaire de l'expression (\ref{Paren}) car pour chaque entier $l$ divisant $n$, chaque partition $\lambda$ de $n$, et $j=j_0$, la somme entre parenthèse de (\ref{Paren}) est le coefficient de degré $a_{j_0}$ de $T(z)$.

\section{Analyse du nombre de systèmes locaux $\ell$-adiques} \label{Last}

\subsection{L'anneau $\gls{ZSg}$}
\label{641}

Soit  $\mathbb{Z}[t^{\pm1}, z_{1}^{\pm1}, \ldots, z_{g}^{\pm1}]^{\mathcal{S}_g\ltimes(\mathcal{S}_2)^{g}}$ l'anneau des polynômes de Laurent qui sont symétriques en les $z_i$ et invariants par les substitutions $z_i\mapsto tz_i^{-1}$ $(1\leq i\leq g)$. Le sous-anneau $$\gls{ZSg}$$ ne consiste que des polynômes $P\in \mathbb{Z}[t^{\pm1}, z_{1}^{\pm1}, \ldots, z_{g}^{\pm1}]^{\mathcal{S}_g\ltimes(\mathcal{S}_2)^{g}}$ tels que chaque monôme $t^{m}z_1^{n_1}\cdots z_g^{n_g}$ qui  apparaît dans $P$ vérifie \begin{equation}\label{O_o} m+\sum_{i=1}^g \min\{n_i, 0\}\geq 0 . \end{equation}
Cette dernière condition est suffisante pour obtenir des entiers algébriques pour les $q$-nombres de Weil dans le théorème \ref{Final}. Elle est aussi nécessaire vu l'existence d'une courbe projective, lisse et géométriquement connexe définie sur un corps fini dont la variété jacobienne est ordinaire (qui m'a été indiqué par Deligne, voir la preuve de la proposition \ref{ordinary}).

La remarque suivante m'a été suggérée par Deligne. 
\begin{remark}\textup{
L'anneau $\mathbb{Z}[z_1,\cdots, z_g, tz_1^{-1},\ldots, tz_g^{-1}]^{\mathcal{S}_g\ltimes(\mathcal{S}_2)^{g}}$ admet une description conceptuelle: 
Soit $GSp_{2g}$ le groupe des similitudes symplectiques défini sur $\mathbb{Q}$: pour chaque $\mathbb{Q}$-algèbre $R$, 
$$GSp_{2g}(R):=\{h\in G_{2g}(R)\ |\ \exists\ z\in R^{\times}\ \text{tel que}\  h\begin{pmatrix}0& I_g\\ -I_g& 0
\end{pmatrix} {^{t}h}= z \begin{pmatrix}0& I_g\\ -I_g& 0
\end{pmatrix} 
  \},$$ où $I_g$ est la matrice identité de taille $g\times g$. On a  la représentation tautologique  $GSp_{2g}\rightarrow G_{2g}$ et le caractère de similitude $t:  GSp_{2g}\rightarrow \mathbb{G}_m$. %Soit $T_g$ un tore maximal déployé de $GSp_{2g}$ et $W_g\cong \mathcal{S}_g\rtimes(\mathcal{S}_2)^{g}$ le groupe de Weyl de $GSp_{2g}$. 
Soit $\mathcal{R}(GSp_{2g})$ l'anneau de Grothendieck des représentations de $GSp_{2g}$.  Soit  \[ \mathbb{Z}[t^{\pm1}, z_{1}^{\pm1}, \ldots, z_{g}^{\pm1}]^{\mathcal{S}_g\ltimes(\mathcal{S}_2)^{g}} \]  le sous-anneau de $\mathbb{Z}[t^{\pm1}, z_{1}^{\pm1}, \ldots, z_{g}^{\pm1}]$ des polynômes de Laurent qui sont  symétriques en les $z_i$ et invariants par les substitutions $z_i\mapsto tz_i^{-1}$ $(1\leq i\leq g)$. On a un isomorphisme $\mathcal{R}(GSp_{2g})\rightarrow \mathbb{Z}[t^{\pm1}, z_{1}^{\pm1}, \ldots, z_{g}^{\pm1}]^{\mathcal{S}_g\ltimes(\mathcal{S}_2)^{g}}$ qui envoie une représentation $V$ sur 
$$\mathrm{Trace}\left( \mathrm{diag}(z_1, \ldots, z_g, tz_1^{-1}, \ldots, tz_g^{-1})\ \middle \vert\ V  
\right)\in \mathbb{Z}[t^{\pm1}, z_{1}^{\pm1}, \ldots, z_{g}^{\pm1}]^{\mathcal{S}_g\ltimes(\mathcal{S}_2)^{g}}.$$
Par la théorie du plus haut poids, on sait aussi que l'anneau $\mathcal{R}(GSp_{2g})$ est l'anneau des polynômes en les $t^{\pm 1}$ et $\wedge^{k}V_{2g}- t\wedge^{k-2}V_{2g}$ (avec $\wedge^{k}V_{2g}=0$ si $k<0$), pour $1\leq k\leq g$, où $V_{2g}$ désigne la représentation tautologique. Alors le sous-anneau de $\mathcal{R}(GSp_{2g})$ engendré par les représentations qui figurent dans une puissance tensorielle de la représentation tautologique est $\mathbb{Z}[z_1,\cdots, z_g, tz_1^{-1},\ldots, tz_g^{-1}]^{\mathcal{S}_g\ltimes(\mathcal{S}_2)^{g}}$. En effet, c'est le sous-anneau engendré par $t$ et $\wedge^{k}V_{2g}- t\wedge^{k-2}V_{2g}$, pour $1\leq k\leq g$. }
\end{remark}

\subsection{Preuve de la première partie du théorème \ref{P}}
Dans cette section, on démontre la première partie du théorème \ref{P}. On le fait par récurrence, ce qui est expliqué dans   \ref{recu}.  La seule difficulté est l'intégralité des coefficients des polynômes du théorème, ce qui est fait dans le reste de cette sous-section.  

\subsubsection{}\label{recu}
Le polynôme dans le théorème \ref{P}.1 est nécessairement unique par un théorème de densité des $q$-entiers de Weil, cf. l'appendice B de  \cite{Schiffmann}.

Pour l'existence, on raisonne par récurrence sur $n$.  Il suffit qu'on démontre que pour tous entiers $g\geq 2$ et $n\geq 1$, il existe un élément $$P_{g, n}(t, z_{1},\ldots, z_{g})\in \mathbb{Z}[z_{1}, \ldots, z_{g}, tz_1^{-1}, \ldots, tz_g^{-1}]^{\mathcal{S}_g\ltimes(\mathcal{S}_2)^{g}},$$  tel que  pour tout corps fini $\mathbb{F}_q$ de cardinal $q$, toute courbe $X_1$ projective lisse et géométriquement connexe sur $\mathbb{F}_q$ de genre $g$, on ait
 $$C_n(X_1)=P_{g, n}(q,\sigma_{1}, \ldots, \sigma_{g}),$$
 où $\sigma_{1},\ldots, \sigma_{2g }$ les $q$-entiers de Weil de la courbe $X_{1}$, c'est-à-dire les valeurs propres de $\Fr_X$ sur $H^1_c(X_1, \bar{\mathbb{Q}}_{\ell})$, indexant de telle façon que $\sigma_{i}\sigma_{i+g}=q$ pour tout $g \geq i\geq 1$.
Quand $n=1$, on a $$C_{1}(X_{1})=|\text{Pic}_{X_1}^{0}(\mathbb{F}_{q})| . $$
Notons que $|\text{Pic}_{X}^{0}(\mathbb{F}_{q})|=\prod_{i=1}^{g}(1-\sigma_i)(1-q\sigma_i^{-1})$, donc le résultat est valide dans ce cas.

Soit $$\gls{etas}=\exp( \sum_{k\geq 1}\frac{C_{s}(X_{kl})}{k}z^{k }  ).$$ 
On suppose que pour tout $s\leq n-1$ le théorème est vrai.
Cela équivaut à dire que pour $1\leq s\leq n-1$ et $l\in \mathbb{N}^{*}$
$${\eta_{{s}, X_l}(z)}=\frac{\prod_{j=1}^{l_{s}  }(1-g_{j,s}(q, \sigma_{1}, \ldots, \sigma_{g }  )^lz)}{\prod_{j=1}^{m_{s}}(1-f_{j,s}(q, \sigma_{1}, \ldots, \sigma_{g })^lz)  },  $$
pour des monômes $f_{j,s}$ et $g_{j,s}$ satisfaisant les symétries désirées et l'inégalité (\ref{O_o}), et les nombres $l_s, m_s\in\mathbb{N}^{*}$ qui ne dépendent  que de $g$ et $s$.

Combinant l'expression spectrale de $J_e$ (le théorème \ref{b}), l'expression géométrique de $J_e$ (le théorème \ref{Maing}) et le théorème \ref{Mel} de Schiffmann  et Mellit, il existe un élément $$A_{g,n}\in  \mathbb{Z}[z_{1}, \ldots, z_{g}, tz_1^{-1}, \ldots, tz_g^{-1}]^{\mathcal{S}_g\ltimes(\mathcal{S}_2)^{g}},$$ tel que pour tout corps fini $\mathbb{F}_q$ de cardinal $q$, toute courbe $X_1$ projective lisse et géométriquement connexe sur $\mathbb{F}_q$ de genre $g$, on ait
 \begin{multline}\label{001}
 A_{g, n}(q,\sigma_{1}, \ldots, \sigma_{g})  = \\  
  \sum_{{\lambda=(i^{a_i})\vdash n  }}\frac{1}{n\sum_{i\geq 1}a_{i}(2g-2) }\sum_{ l \mid n}\mu(l)  \prod_{i\geq 1 } [z^{a_{i}}]   \biggr( \prod_{s\geq 1}\eta_{s, X_{l}}(z^{sl})^s     \biggr) ^{\frac{(2g-2)S_{i}(\lambda)}{l}} . 
 \end{multline}

On examine la relation (\ref{001}), elle montre que le nombre $A_{g, n}(q,\sigma_{1}, \ldots, \sigma_{g})$ peut s'exprimer comme un polynôme en $C_{k}(X_{d})$ pour $kd\leq n$. 
Quand $s>n$, $\eta_{s, X_{l}}(z^{sl})$ ne contribue pas dans la formule (\ref{001}), donc $C_n(X_1)$ n'apparait  que dans le terme de sommation où $l=1$, $\lambda=(1^n)$ et dans ce cas $S_1(\lambda)=n$. 
%On a %$$H_{g, n}(q,\sigma_{1}, \ldots, \sigma_{g})=C_{n}(X_{1}) + \cdots ,$$
%où les termes omis est un polynôme en $C_{k}(X_{t})$ avec $k<n$ car seulement les $C_{k}(X_{t})$ tels que $kt\leq n$ contribuent.  
D'après ces remarques, on a %montra dans la suite que pour toute partition $\lambda=(i^{a_{i}})$ de $n$, l'expression
\begin{multline}\label{002}
C_n(X_1)=A_{g, n}(q,\sigma_{1}, \ldots, \sigma_{g}) -  \\
 \sum_{{\lambda=(i^{a_i})\vdash n  }}\frac{1}{n\sum_{i\geq 1}a_{i}(2g-2) }\sum_{ l \mid n}\mu(l)  \prod_{i\geq 1 } [z^{a_{i}}]   \biggr( \prod_{s=1}^{n-1}\eta_{s, X_{l}}(z^{sl})^s     \biggr) ^{\frac{(2g-2)S_{i}(\lambda)}{l}} . 
 \end{multline}
Par l'hypothèse de récurrence, il suffit de montrer que pour tout $\lambda=(i^{a_i})\vdash n$, l'expression
\begin{equation}\label{topr} 
\frac{1}{n\sum_{i\geq 1}a_{i}(2g-2) }\sum_{ l \mid n}\mu(l)  \prod_{i\geq 1 } [z^{a_{i}}]   \biggr( \prod_{s=1}^{n-1}\eta_{s, X_{l}}(z^{sl})^s     \biggr) ^{\frac{(2g-2)S_{i}(\lambda)}{l}} 
\end{equation}
est une combinaison linéaire à coefficients entiers des nombres de Weil qui sont des monômes  en ceux de la courbe $X_{1}$, et les coefficients de cette écriture ne dépend que du $g$. Notons que tous ces énoncés sont clairs sauf l'intégralité  des coefficients des nombres de Weil, on se concentre sur sa preuve dans la suite.

\subsubsection{Preuve de l'intégralité }\sloppy
Pour tout $s<n$, soient $\alpha_{j,s}=f_{j,s}(q, \sigma_{1}, \ldots, \sigma_{g })$ et $\beta_{j,s}=g_{j,s}(q,  \sigma_{1}, \ldots, \sigma_{g })$ qui sont $q$-entiers de Weil. On a 
$${\eta_{{s}, X_l}(z)}=\exp( \sum_{k\geq 1}\frac{C_{s}(X_{kl})}{k}z^{k }  )=\frac{\prod_{j=1}^{l_{s}  }(1-\beta_{j,s}^lz)}{\prod_{j=1}^{m_{s}}(1-\alpha_{j,s}^lz)  }  . $$
Donc pour tout $m\in \mathbb{Z}$ (en particulier pour $m=(2g-2)S_i(\lambda)$), on sait que $\eta_{s, X_{l}}(z^{sl})^{\frac{sm}{l}}$ est alors
$$\left(\prod_{j=1}^{l_{s}}(\sum_{k\geq 0}(-1)^{k}\binom{-\frac{ms}{l}}{k}(\alpha_{j,s})^{kl}z^{ksl} )\right) \left(\prod_{j=1}^{m_{s}}(\sum_{k\geq 0}(-1)^{k}\binom{\frac{ms}{l}}{k}(\beta_{j,s})^{kl}z^{ksl} )\right) . $$
Soit  $\lambda=(\nu^{a_{\nu}})$ une partition de $n$, on va regarder les $\alpha_{j,s}$ et $\beta_{j,s}$ comme des variables; on pose pour tous $i, l\in\mathbb{N}^{*}$:
$$\widetilde{g_{i,l}^{\lambda}}(z^{l})=\prod_{s=1}^{n-1} \left( \prod_{j=1}^{l_{s}+m_{s}}(\sum_{k\geq 0}(-1)^{k}\binom{\frac{\varepsilon_{j,s} (2g-2)sS_{i}(\lambda)}{l} }{k } ({X^{i}_{j,s}})^{lk} z^{skl} )   \right), $$
où $$\varepsilon_{j,s}=\begin{cases}- 1, \quad 1\leq j\leq l_s \\
1, \quad l_s+1\leq j\leq l_s+m_s
\end{cases}
$$
 et $X_{j,s}^{i}$ sont des indéterminées à trois indices.  Si on évalue $X_{j,s}^i$ en les $\alpha_{j,s}$ ou $\beta_{j,s}$ de manière appropriée (pour tout $i$, $X_{j,s}^{i}=\alpha_{j,s}$ si $1\leq j\leq l_s$ et $X_{j,s}^{i}=\beta_{j,s}$ si $l_s+1\leq j\leq l_s+m_s$), on retrouvera $$ \prod_{s=1}^{n-1}(\eta_{s,X_l}(z^{sl})^{s})^{\frac{(2g-2)S_i(\lambda)}{l}}. $$   
On va montrer que 
\begin{equation}\label{60} \frac{1}{n\sum_{i\geq 1}a_{i}(2g-2) }\sum_{ l \mid n}\mu(l)  \prod_{i\geq 1} [z^{a_{i}}] \widetilde{g_{i,l}^{\lambda}}(z^{l})\ \in \ \mathbb{Z}[X_{j,s}^{i}]_{i,j,s} .   \end{equation}
En effet, comparant avec l'expression (\ref{topr}), cela est suffisant pour l'intégralité  désiré.
\begin{lemm}\label{combnumbers}
Pour un monôme $\prod_{i,j,s}(X_{j,s}^{i})^{k_{j,s}^{i}}$, si $\sum_{j,s}sk_{j,s}^{i}=a_{i}$, $\forall\ i\geq 1$, son coefficient dans le polynôme (\ref{60}) est  égal à $$\frac{1}{n\sum_{i}a_{i}(2g-2) }   \sum_{l\mid \mathrm{p.g.c.d.}(k_{j,s}^{i})} \mu(l)(-1)^{\frac{\sum_{ j,s }sk_{j,s}^{i}}{l} } \prod_{i,j,s}  \binom{{\varepsilon_{j,s}(2g-2)sS_{i}(\lambda)}/{l}}{k_{j,s}^{i}/l}\ ; $$
sinon le coefficient est zéro. 
\end{lemm}
\begin{proof}
C'est parce qu'on a $\forall\  i\geq 1$
$$[z^{a_{i}}]\widetilde{g_{i,l}^{\lambda}}(z)=\sum_{ \{ k_{j,s}\in\mathbb{N}^{*}| \sum_{j,s}sk_{j,s}=a_{i};\ l\mid k_{j,s}\} } \prod_{ j,s} \binom{\varepsilon_{j,s}(2g-2)sS_{i}(\lambda)/l }{k_{j,s}/l } ({X^{i}_{j,s}})^{ k_{j,s}  }  .  $$
De plus si $\sum_{j,s}sk_{j,s}^{i}=a_{i}$, $\forall\ i\geq 1$, alors $\mathrm{p.g.c.d.}(k_{j,s}^i)\mid n$ puisque $(i^{a_i})$ est une partition de $n$. 
\end{proof}

Nous allons prouver que les coefficients dans ce lemme sont des entiers dans le dernier théorème; avant de l'énoncer, nous indiquons le fait suivant. 
\begin{lemm}\label{small}
Pour tous nombres entiers non-nuls $n,m$ avec $m\geq1$, on a
$$\frac{n}{(n,m)}\mid \binom{n}{m} .$$
\end{lemm}
\begin{proof}
C'est parce que $$\frac{m}{n}\binom{n}{m}=\binom{n-1}{m-1}\in\mathbb{Z}$$
donc $$\frac{m}{(m,n)}\binom{n}{m}\in \frac{n}{(n,m)}\mathbb{Z} .$$
Comme $(\frac{m}{(m,n)}, \frac{n}{(m,n)})=1$, on a $$\binom{n}{m}\in \frac{n}{(n,m)}\mathbb{Z}  .$$
\end{proof}

Par simplicité, on note  $\chi=2g-2$ dans la suite; on va utiliser le fait que $\chi$ est un nombre pair. %Dans la suite, si on ne précise pas, le $\mathrm{p.g.c.d.}$ est porté sur tous les indices. 
\begin{theorem}\label{wala} Soit $m\geq 1$ un entier. 
Soient $a_{1},\ldots, a_{m}$ des nombres naturels non-nuls, et $(k_{j,s}^{i})_{\substack{1\leq i\leq m\\
 j\in J_i\\
s\in K_i
}
}$ des nombres naturels tels que $$\sum_{j\in J_i, s\in K_i}sk_{j,s}^{i}=a_{i}, $$
où pour tout $i$, $J_i$ et $K_i$ sont des sous-ensembles finis de $\mathbb{N}$. 
Soient $\nu_{1},\cdots,\nu_m$ des nombres entiers.  Soit $S_{i}=\sum_{j<i}\nu_{j} a_{j}+ \nu_{i} \sum_{j\geq i}a_{j}$, $i=1,\cdots, m$.
Soit $\varepsilon_{i,j,s}\in\{\pm 1\}$ pour $1\leq i\leq m$, $j\in J_i$ et $s\in K_i$. 
Le nombre suivant 
$$\sum_{l\mid \mathrm{p.g.c.d.}(k^{i}_{j,s})} \mu(l)(-1)^{\frac{\sum_{i,j,s} k^{i}_{j,s}}{l}} \prod_{\substack{1\leq i\leq m\\   j\in J_i, s\in K_i}}  \binom{\varepsilon_{i,j,s}\chi{sS_{i}}/{l} }{{k^{i}_{j,s}}/{l}}$$ 
est divisible par $\chi S_m\sum_{j}a_{j}$, où  $\chi\in 2\mathbb{N}$. 
\end{theorem}
\begin{proof}
Soit $\gls{vp}$ la valuation $p$-adique, normalisée telle que $v_p(p)=1$. 
\begin{lemm}\label{pdelicat}
Avec les notations de l'énoncé du théorème \ref{wala}, si \[ v_{p}(\mathrm{p.g.c.d.}(k^{i}_{j,s}))=0,\] on a
 \begin{equation}  \label{I} 
  v_{p}(\prod_{\substack{1\leq i\leq m\\   j\in J_i, s\in K_i}}\frac{\chi {sS_{i}}  }{    ( \chi {sS_{i}}  , {k^{i}_{j,s}}    ) })
 \geq v_{p}(\chi S_m\sum_{j}a_{j} ) \end{equation}
\end{lemm}
\begin{proof}[Preuve du lemme]
% v_{p}(\prod_{i,j,s}  \binom{\varepsilon_{i,j,s}\chi {sS_{i}}  }{{k^{i}_{j,s}}})

Soit $$A=v_{p}(\sum_{j}a_{j})\geq 0 .$$
Soient $$\alpha_{i}=v_{p}(\chi S_{i})\geq 0 \quad\forall\ i\geq 1 ,$$
$$\beta_{i}=v_{p}(\underset{j\in J_i,s\in K_i}{\mathrm{p.g.c.d}}\{k_{j,s}^{i}\})\geq 0 \quad\forall\ i\geq 1 ,$$
où pour chaque $i$, $\underset{j\in J_i,s\in K_i}{\mathrm{p.g.c.d}}\{k_{j,s}^{i}\}$ désigne le plus grand commun diviseur des nombres $k_{j,s}^{i}$ pour $j\in J_i$ et $s\in K_i$, et
$$\gamma_{i}=v_{p}(a_{i})\geq 0\quad\forall\ i\geq 1 . $$

Il faut montrer qu'on a $$ v_{p}(\prod_{\substack{1\leq i\leq m\\   j\in J_i, s\in K_i}}\frac{\chi {sS_{i}}  }{    ( \chi {sS_{i}}  , {k^{i}_{j,s}}    ) })
 \geq A+\alpha_m.$$

L'hypothèse $v_{p}(\mathrm{p.g.c.d.}(k^{i}_{j,s}))=0$ est équivalente  à la condition $\min_{1\leq i\leq m}\{\beta_i\}=0$. 
C'est toujours vrai que $ \beta_{i}\leq\gamma_{i}$ car pour tout $1\leq i\leq m$,
\begin{align*}
\gamma_i=v_p(a_i)&=v_p(\sum_{j\in J_i,s\in K_i}sk_{j,s}^{i}   )\\
&\geq \min_{j\in J_i,s\in K_i}\{v_p(s)+v_p(k_{j,s}^{i})\} \\
&\geq v_p(\underset{j\in J_i,s\in K_i}{\mathrm{p.g.c.d.}}(k_{j,s}^i) )=\beta_i  . 
\end{align*}

Notons que pour chaque $i$, il existe au moins un $(j_{0},s_{0})\in J_i\times K_i$ tel que $$v_{p}(k_{j_{0},s_{0}}^{i})=\beta_{i} .$$
On a alors
 $$ v_{p}(\frac{\chi {s_0 S_{i}}  }{    ( \chi {s_0 S_{i}}  , {k^{i}_{j_0,s_0}}    ) })\geq \max\{v_p(\frac{\chi s_{0}S_{i}}{k_{j_{0},s_{0}}^{i}}), 0 \}     \geq \max\{\alpha_{i}-\beta_{i}+v_{p}(s_{0}), 0 \}.$$
Soit $1\leq i\leq m$, on part de
$$\sum_{j\in J_i,s\in K_i}sk_{j,s}^{i}=a_{i} .$$
De deux choses l'une: 
1) Soit $v_{p}(s_{0})+\beta_{i}< \gamma_{i}$, i.e. $v_{p}(s_0 k_{j_0, s_0}^{i})<v_p(a_i)$,  alors il existe un autre $(j_{1},s_{1})\neq (j_0, s_0)\in J_i\times K_i$ tel que $$v_{p}(s_{1}k_{j_{1},s_{1}}^{i})\leq  v_{p}(s_0 k_{j_0, s_0}^{i})= v_{p}(s_{0})+\beta_{i}. $$
Comme $s_1$ est un entier, $v_p(s_1)\geq 0$. Donc on a  $$-v_{p}(k_{j_{1},s_{1}}^{i})\geq -  v_{p}(s_{0})-\beta_{i}, $$ et $$ v_{p}(\frac{\chi {s_1S_{i}}  }{    ( \chi {s_1S_{i}}  , {k^{i}_{j_1,s_1}}    ) })\geq \max\{\alpha_{i}-\beta_{i}-v_{p}(s_{0}), 0 \} . $$
Il s'ensuit que $$ v_{p}(\prod_{j\in J_i,s\in K_i}\frac{\chi {sS_{i}}  }{    ( \chi {sS_{i}}  , {k^{i}_{j,s}}    ) })   \geq \max\{2\alpha_{i}-2\beta_{i}, 0\} .$$
2) Soit $v_{p}(s_{0})+\beta_{i}\geq \gamma_{i}$, alors $$v_{p}(\prod_{j\in J_i,s\in K_i}\frac{\chi {sS_{i}}  }{    ( \chi {sS_{i}}  , {k^{i}_{j,s}}    ) })\geq \max\{\alpha_{i}+\gamma_{i}-2\beta_{i}, 0 \}. $$
Donc en tout cas pour tout $1\leq i\leq m$
\begin{align} \label{i}
v_{p}(\prod_{j\in J_i,s\in K_i}\frac{\chi {sS_{i}}  }{    ( \chi {sS_{i}}  , {k^{i}_{j,s}}    ) } )&\geq\min\{ \max\{2\alpha_{i}-2\beta_{i}, 0 \},  \max\{\alpha_{i}+\gamma_i-2\beta_{i}, 0\}   \} \\
&= \max\{ \min\{2\alpha_{i}-2\beta_{i}, \alpha_{i}+\gamma_i-2\beta_{i}\},  0\}  \nonumber  .
\end{align}

Si $m=1$, on a $S_{1}=\nu_1a_1$ donc $\gamma_1\leq \alpha_1$. L'hypothèse donne $\beta_1=0$, et l'inégalité (\ref{i}) implique que 
$$  v_{p}(\prod_{j\in J_i,s\in K_i} \frac{\chi {sS_{1}}  }{    ( \chi {sS_{1}}  , {k^{1}_{j,s}}    ) })\geq \alpha_1+\gamma_1 = v_{p}(\chi a_{1} S_1). $$

Si $m>1$.  L'inégalité (\ref{i}) implique 
\begin{equation}v_{p}(\prod_{j\in J_i,s\in K_i}\frac{\chi {sS_{i}}  }{    ( \chi {sS_{i}}  , {k^{i}_{j,s}}    ) } )\geq \max\{ \alpha_i-\beta_i, 0  \}    . \end{equation}

Pour tout $1\leq k\leq m$, on va prouver 
\begin{equation}\label{la}\sum_{1\leq i\leq k}  \max\{\alpha_{i}-\beta_{i}, 0\} \geq A- \min\{A, \beta_{1}, \ldots, \beta_{k}\} .   \end{equation}
On raisonne par récurrence. 
Quand $k=1$, c'est vrai car $\alpha_1\geq A$. Supposons que l'inégalité soit vraie pour $k-1$. Si $\beta_{k}>\min\{A, \beta_{1}, \ldots, \beta_{k-1}\}$
alors 
\begin{align*}
\sum_{i\leq k}  \max\{\alpha_{i}-\beta_{i}, 0\}& \geq \sum_{i\leq k-1}  \max\{\alpha_{i}-\beta_{i}, 0\}\\
& \geq  A- \min\{A, \beta_{1}, \ldots, \beta_{k-1}\} \\
&  = A- \min\{A, \beta_{1}, \ldots, \beta_{k}\}.\end{align*}
Si $\beta_{k}\leq \min\{A, \beta_{1}, \ldots, \beta_{k-1}\}$, comme $\chi S_{k}$ est une combinaison linéaire à coefficients entiers de $\chi \sum_{j}a_{j}$ et $a_{1},\ldots, a_{k-1}$, on a  
\begin{equation}\label{inequality}\alpha_{k}=v_p(\chi S_k)\geq\min\{A, \gamma_{1}, \ldots, \gamma_{k-1}\}\geq \min\{A, \beta_{1}, \ldots, \beta_{k-1}\} .\end{equation}
Donc 
\begin{align*}
\sum_{i\leq k}  \max\{\alpha_{i}-\beta_{i}, 0\}&\geq \alpha_{k}- \beta_{k}+ A- \min\{A, \beta_{1}, \ldots, \beta_{k-1}\}\\
& \geq A-\beta_k \\
& = A- \min\{A, \beta_{1}, \ldots, \beta_{k}\} . 
\end{align*}
Par récurrence, l'inégalité est vraie pour tout $k\geq 1$. \\ 

De même, on a
\begin{equation}\label{da}\sum_{i\geq k}  \max\{\alpha_{i}-\beta_{i}, 0\} \geq \alpha_m- \min\{\alpha_m, \beta_{k}, \ldots, \beta_{m}\},   \quad\forall 1\leq k\leq m . \end{equation}
En effet, quand $k=m$, l'inégalité (\ref{da}) est une égalité. Supposons que l'inégalité (\ref{da})  soit vraie pour $k+1$. Alors si $\beta_{k}>\min\{\alpha_m, \beta_{k+1}, \ldots, \beta_{m}\}$, 
on a
\begin{align*}
\sum_{i\geq k}  \max\{\alpha_{i}-\beta_{i}, 0\}& \geq \sum_{i\geq k+1}  \max\{\alpha_{i}-\beta_{i}, 0\}\\
& \geq  \alpha_m- \min\{\alpha_m, \beta_{k+1}, \ldots, \beta_{m}\} \\
&  = \alpha_m- \min\{\alpha_m, \beta_{k}, \ldots, \beta_{m}\}. \end{align*}
Si $\beta_{k}\leq \min\{\alpha_m, \beta_{k+1}, \ldots, \beta_{m}\}$. 
Comme $\chi S_k$ est une combinaison linéaire à coefficients entiers de $\chi S_m$ et $a_{m},\ldots,$ $a_{k+1}$, on a  
\begin{equation}\label{72}\alpha_{k}=v_p(\chi S_k)\geq\min\{\alpha_m, \gamma_{k+1}, \ldots, \gamma_{m}\}\geq \min\{\alpha_m, \beta_{k+1}, \ldots, \beta_{m}\} .\end{equation}
Donc 
\begin{align*}
\sum_{i\geq k}  \max\{\alpha_{i}-\beta_{i}, 0\}&\geq \alpha_{k}- \beta_{k}+ \alpha_m- \min\{\alpha_m, \beta_{k}, \ldots, \beta_{m}\}\\
& \geq \alpha_m-\beta_k \\
& \geq \alpha_m- \min\{\alpha_m, \beta_{k}, \ldots, \beta_{m}\} .
\end{align*}
Par récurrence l'inégalité est vraie pour tout $k\geq 1$.

Soit $i_{0}$ le plus petit indice tel que $\beta_{i_0}=0$ et $i_{1}$ le plus grand indice tel que $\beta_{i_1}=0$. Alors par l'inégalité (\ref{la}), on prend $k=i_0$,
\begin{equation}\sum_{i\leq i_{0}}\max\{\alpha_{i}-\beta_{i}, 0\} \geq A-\min\{ A, \beta_1, \ldots, \beta_{i_0} \}=A .
\end{equation}
De même, on prend $k=i_1$ dans l'inégalité (\ref{da}), on obtient:
 \begin{equation}\sum_{i\geq i_{1}}\max\{\alpha_{i}-\beta_{i}, 0\} \geq \alpha_{m}-\min\{ \alpha_m, \beta_{i_1}, \cdots, \beta_m \}=\alpha_m . 
\end{equation}

Donc si  $i_{0}\neq i_{1}$, on obtient $$v_{p}(\prod_{i,j,s}\frac{\chi {sS_{i}}  }{    ( \chi {sS_{i}}  , {k^{i}_{j,s}}    ) })\geq A+\alpha_m= v_p(\chi S_m\sum_j a_j ). $$

Si $i_{0}=i_{1}$, c'est-à-dire que $i_{0}$ est le seul indice tel que $\beta_{i}=0$. 
Comme $a_{i_{0}}=\sum_{j}a_{j}-\sum_{j\neq i_{0}}a_{j}$, on a $$\gamma_{i_{0}}\geq \min_{i\neq i_{0}}\{A,\gamma_{i}\}\geq \min_{i\neq i_{0}}\{A,\gamma_{i}, \alpha_{m}\}\geq \min_{i\neq i_{0}}\{A,\beta_{i}, \alpha_{m}\} .$$ 
Par l'inégalité (\ref{inequality}) et l'inégalité (\ref{72}),  on a alors  $$\alpha_{i_{0}}\geq \max\{    \min_{i< i_{0}}\{A, \beta_{i}\}, \min_{i>i_{0}} \{\alpha_{m},\beta_{i}\} \}   .$$
Donc
$$\alpha_{i_{0}}+\gamma_{i_{0}}\geq \min_{i< i_{0}}\{A, \beta_{i}\}+\min_{i>i_{0}} \{\alpha_{m},\beta_{i}\} ,$$
et clairement
$$2\alpha_{i_{0}}\geq \min_{i< i_{0}}\{A, \beta_{i}\}+\min_{i>i_{0}} \{\alpha_{m},\beta_{i}\}\  .$$
Combinant avec l'inégalité (\ref{i}), on obtient 
\begin{align*}
& v_{p}(\prod_{\substack{1\leq i\leq m\\ j\in J_i, s\in K_i} }\frac{\chi {sS_{i}}  }{    ( \chi {sS_{i}}  , {k^{i}_{j,s}}    ) } )   \\
\geq &  \max \big\{ \min\{2\alpha_{i_0}-2\beta_{i_0}, \alpha_{i_0}+\gamma_{i_0}-2\beta_{i_0}\}, 0   \big\}
  + \sum_{i\neq i_0} \max\{\alpha_{i}-\beta_{i}, 0\}    \\ 
   \geq & (\min_{i< i_{0}}\{A, \beta_{i}\}+\min_{i>i_{0}} \{\alpha_{m},\beta_{i}\})+  (A-\min_{i< i_{0}}\{A, \beta_{i}\})+(\alpha_m- \min_{i>i_{0}} \{\alpha_{m},\beta_{i}\}) \\
  = & A+\alpha_m
\end{align*}

Le lemme est prouvé.
\end{proof}

En utilisant ce lemme en rempla\c cant $k_{j,s}^{i}$ par $k_{j,s}^{i}/p^{v_{p}(\mathrm{p.g.c.d.}( k_{j,s}^{i}  ))}$ et le lemme \ref{small}, on a pour tout nombre $l\mid \mathrm{p.g.c.d.}(k_{j,s}^{i})$:  %\begin{equation} \label{II} v_{p}(\prod_{i,j,s} \binom{\varepsilon_{i,j,s}\chi {sS_{i}}  }{{k^{i}_{j,s}}})\geq v_{p}(\chi \sum_{j}a_{j} S_m)-2v_{p}(p.g.c.d\{k_{j,s}^{i} \}) \end{equation}
 \begin{equation} \label{II} v_{p}(\prod_{\substack{1\leq i\leq m\\ j\in J_i, s\in K_i}} \binom{\varepsilon_{i,j,s}\chi {sS_{i}}/l }{{k^{i}_{j,s}}/l} )\geq v_{p}(\chi \sum_{j}a_{j} S_m)-2v_{p}(\mathrm{p.g.c.d.}(k_{j,s}^{i}) ) . \end{equation}
Retournons au théorème, il faut montrer que pour tout nombre premier $p$
\begin{equation}\label{butt}v_p\left(\sum_{l\mid \mathrm{p.g.c.d.} (k^{i}_{j,s}) } \mu(l)(-1)^{\frac{\sum_{i,j, s} k^{i}_{j,s}}{l}} \prod_{i,j,s}  \binom{\varepsilon_{i,j,s}\chi {sS_{i}}/{l} }{{k^{i}_{j,s}}/{l}}\right)   \geq v_{p}(\chi \sum_{j}a_{j}S_m) . \end{equation}

Si  $p$ est un nombre premier tel que  $p\nmid \mathrm{p.g.c.d.} (k_{j,s}^{i}) $, l'inégalité (\ref{II}) déjà implique le but (\ref{butt}).

%Maintenant, soit $p$ un nombre premier qui divise $\mathrm{p.g.c.d.} (k_{j,s}^{i})$, on a \begin{equation} \begin{split}  &\sum_{l\mid \mathrm{p.g.c.d.} (k^{i}_{j,s})} \mu(l)(-1)^{\frac{\sum_{i,j,s} k^{i}_{j,s}}{l}} \prod_{i,j,s}  \binom{\varepsilon_{i,j,s}\chi {sS_{i}}/{l} }{{k^{i}_{j,s}}/{l}}     \\ =&\sum_{l\mid \mathrm{p.g.c.d.}(k^{i}_{j,s});\ p\nmid l } \mu(l)(-1)^{\frac{\sum_{i,j,s} k^{i}_{j,s}}{l}} \prod_{i,j,s}  \binom{\varepsilon_{i,j,s}\chi {sS_{i}}/{l} }{{k^{i}_{j,s}}/{l}}   + \sum_{l\mid \mathrm{p.g.c.d.} (k^{i}_{j,s});\ p\mid l} \mu(l)(-1)^{\frac{\sum_{i,j,s} k^{i}_{j,s}}{l}} \prod_{i,j,s}  \binom{\varepsilon_{i,j,s}\chi {sS_{i}}/{l} }{{k^{i}_{j,s}}/{l}}            \\ =&\label{Möbius} \sum_{l\mid \mathrm{p.g.c.d.}(k^{i}_{j,s} );\  p\nmid l }\mu(l) \left( (-1)^{ \frac{\sum_{i,j,s}k_{j,s}^{i}}{l} }    \prod_{i,j,s}  \binom{\varepsilon_{i,j,s}\chi {sS_{i}}/{l} }{{k^{i}_{j,s}}/{l}} -  (-1)^{ \frac{\sum_{i,j,s}k_{j,s}^{i}}{lp}}     \prod_{i,j,s}  \binom{\varepsilon_{i,j,s}\chi {sS_{i}}/{lp} }{{k^{i}_{j,s}}/{lp}}    \right) .  \end{split} \end{equation}

Maintenant, soit $p$ un nombre premier qui divise $\mathrm{p.g.c.d.} (k_{j,s}^{i})$. Considérons  \[ \sum_{l\mid \mathrm{p.g.c.d.} (k^{i}_{j,s})} \mu(l)(-1)^{\frac{\sum_{i,j,s} k^{i}_{j,s}}{l}} \prod_{i,j,s}  \binom{\varepsilon_{i,j,s}\chi {sS_{i}}/{l} }{{k^{i}_{j,s}}/{l}} , \] on peut regrouper la somme suivant $p\mid l$ ou $p\nmid l$. Après simplification, elle est égale à 
\begin{multline} \label{Möbius}  
 \sum_{l\mid \mathrm{p.g.c.d.}(k^{i}_{j,s} );\  p\nmid l }  \mu(l)  \\
   \left( (-1)^{ \frac{\sum_{i,j,s}k_{j,s}^{i}}{l} }    \prod_{i,j,s}  \binom{\varepsilon_{i,j,s}\chi {sS_{i}}/{l} }{{k^{i}_{j,s}}/{l}} -  (-1)^{ \frac{\sum_{i,j,s}k_{j,s}^{i}}{lp}}     \prod_{i,j,s}  \binom{\varepsilon_{i,j,s}\chi {sS_{i}}/{lp} }{{k^{i}_{j,s}}/{lp}}    \right). 
\end{multline}

Pour tout nombre premier $p$ et tout nombre entier $n\geq 1$, soit
  $$f_{p}(n):=\prod_{i=1, p \nmid i}^{n}i = \frac{n!}{p^{[n/p]}  [n/p]! },$$ 
où $[n/p]$ est la partie entière de $n/p$. 
On va invoquer le résultat suivant qui est le lemme 3.1 de \cite{WZ}: \\
\\ \label{star}
\textit{  
$(*)$ Si $p$ est un nombre premier impair et $\alpha\geq 1$ un entier ou $p=2$ et $\alpha\geq 2$, on a  $$p^{2\alpha}\mid f_{p}(p^{\alpha}n)-f_{p}(p^{\alpha})^{n} . $$ 
Quand $p=2$, on a $$f_{2}(2n)\equiv (-1)^{[\frac{n}{2}]}\quad (mod\ 4)  . $$}

Comme  $$\binom{np}{mp}=\binom{n}{m}\frac{f_{p}(np)}{f_{p}(mp)f_{p}((n-m)p)}, $$
  et $$\binom{-np}{mp}=(-1)^{m(p-1)}\binom{-n}{m}\frac{f_{p}(np+mp)}{f_{p}(mp)f_{p}(np)}. $$

On discute suivant deux cas: \\
$\bullet\ 1^{er}$-cas. Si $p\neq 2$, ou $p=2$ et $v_{p}( \mathrm{p.g.c.d.}(k_{j,s}^{i}))\geq 2$,  la différence entre parenthèses dans l'expression (\ref{Möbius}) est égale à
\begin{equation}\label{po} (-1)^{ \frac{\sum_{i,j,s}k_{j,s}^{i}}{l}}  \left(  \prod_{\substack{1\leq i\leq m\\ j\in J_i, s\in K_i}}g_{i,j,s}  -1 \right) \prod_{\substack{1\leq i\leq m\\ j\in J_i, s\in K_i}}  \binom{\varepsilon_{i,j,s} \chi {sS_{i}}/{lp} }{{k^{i}_{j,s}}/{lp}}, \end{equation}
où 
\begin{equation*}g_{i,j,s}=\frac{ f_{p}(\chi sS_{i}/l)   }{  f_{p}(\chi sS_{i}/l-k_{j,s}^{i}/l)f_{p}(k_{j,s}^{i}/l)  } \end{equation*}
quand $\varepsilon_{i,j,s}=1$, et 
\begin{equation*} g_{i,j,s}=\frac{ f_{p}(\chi sS_{i}/l+k_{j,s}^{i}/l)   }{  f_{p}(\chi sS_{i}/l)f_{p}(k_{j,s}^{i}/l)  } \end{equation*}
quand $\varepsilon_{i,j,s}=-1$. D'après l'inégalité (\ref{II}), pour montrer (\ref{butt}), il suffit de montrer que l'expression entre parenthèses dans (\ref{po})  est divisible par $p^{2\mathrm{p.g.c.d.}(k_{j,s}^{i})}$.  

On identifie le corps des nombres rationnels $\mathbb{Q}$ à un sous-corps de $\mathbb{Q}_p$ des nombres $p$-adiques. Dans le cas où $\alpha=0$, il n'y a rien à prouver. On considère le cas où $\alpha=v_p(\mathrm{p.g.c.d.}( k_{j,s}^{i} ))>0$.  
Notons que pour tous $1\leq i\leq m$, $j\in J_i$ et $s\in K_i$, on a $$v_{p}(k_{j,s}^{i}) \geq \alpha, $$
et comme $S_i$ est une combinaison linéaire à coefficients entirs en $k_{j,s}^{i}$, on  a aussi
$$v_{p}(\chi sS_{i})\geq \alpha. $$
De plus, par définition, $f_{p}(p^{\alpha})$ est une unité dans $\mathbb{Z}_{p}^{\times}$, le groupe des éléments de valuation $0$ dans $\mathbb{Q}_p$.
Soit $u_{i,j,s}=f_{p}(p^{\alpha})^{k_{j,s}^{i}/lp^{\alpha} }$. Quand $\varepsilon_{i,j,s}=1$ (resp. quand $\varepsilon_{i,j,s}=-1$), soit $v_{i,j,s}=f_{p}(p^{\alpha})^{\chi s S_i/lp^{\alpha}- k_{j,s}^{i}/lp^{\alpha}  }$ (resp.  $v_{i,j,s}=f_{p}(p^{\alpha})^{\chi s S_i/lp^{\alpha}}$). Par le fait $(*)$, il existe $x_{i,j,s}, y_{i,j,s}, z_{i,j,s}\in \mathbb{Z}$  tels que 
%$$   f_{p}(\chi sS_{i}/l- k_{j,s}^{i}/l) = u_{i,j,s}-x_{i,j,s}p^{2\alpha}             $$
%$$  f_{p}(k_{j,s}^{i}/l)= v_{i,j,s}-y_{i,j,s}p^{2\alpha}        $$
%$$ f_{p}(\chi sS_{i}/l) = u_{i,j,s}v_{i,j,s}-z_{i,j,s}p^{2\alpha} $$
$$ g_{i,j,s}=\frac{u_{i,j,s}v_{i,j,s}- z_{i,j,s}p^{2\alpha}}{(u_{i,j,s}-x_{i,j,s}p^{2\alpha})(v_{i,j,s}-y_{i,j,s} p^{2\alpha} )}  . $$
Utilisons l'identité:
$$\frac{1}{u-xp^{2\alpha}}=u^{-1}+\sum_{k\geq 1} \frac{x^{k}}{u^{k+1}}p^{2k\alpha} ,    $$
pour $x=x_{i,j,s} \in \mathbb{Z}_{p}$ et $u=u_{i,j,s}\in \mathbb{Z}_{p}^{\times}$  (resp. $x=y_{i,j,s}\in \mathbb{Z}_{p}$ et $u=v_{i,j,s}\in \mathbb{Z}_{p}^{\times}$), on obtient
$$ v_p (\prod_{i,j,s} \frac{u_{i,j,s}v_{i,j,s}- z_{i,j,s}p^{2\alpha}}{(u_{i,j,s}-x_{i,j,s}p^{2\alpha})(v_{i,j,s}-y_{i,j,s} p^{2\alpha} )} -1 )\geq 2\alpha  . $$
Donc on conclut que dans ce cas (\ref{butt}) est vraie. 
\\
$\bullet\ 2^{nd}$-cas.  Il ne reste qu'à prouver le cas où $p=2$ et $v_{2}(\mathrm{p.g.c.d.}(k_{j,s}^{i}))=1$. La différence entre parenthèses dans l'expression (\ref{Möbius}) est égale à
$$ \prod_{i,j,s}   \binom{\varepsilon_{i,j,s} \chi {sS_{i}}/{2l} }{{k^{i}_{j,s}}/{2l}}  \left(  \prod_{i,j,s}g_{i,j,s} -(-1)^{ \frac{\sum_{i,j,s}k_{j,s}^{i}}{2l}  } \right) ,$$
où 
\begin{equation*}g_{i,j,s}= \prod_{i,j,s} \frac{ f_{2}(\chi sS_{i}/l)   }{  f_{2}(\chi sS_{i}/l-k_{j,s}^{i}/l)f_{2}(k_{j,s}^{i}/l)  } ,\end{equation*}
quand $\varepsilon_{i,j,s}=1$, et 
\begin{equation*} g_{i,j,s}= (-1)^{k_{j,s}^{i}/2l}\prod_{i,j,s} \frac{ f_{2}(\chi sS_{i}/l+k_{j,s}^{i}/l)   }{  f_{2}(\chi sS_{i}/l)f_{2}(k_{j,s}^{i}/l)  }, \end{equation*}
quand $\varepsilon_{i,j,s}=-1$.
Dans ce cas, par (\ref{II}) on a
$$v_{2}(\prod_{i,j,s}   \binom{\varepsilon_{i,j,s} \chi {sS_{i}}/{2l} }{{k^{i}_{j,s}}/{2l}}) \geq v_{2}(  \chi \sum_{j}a_{j}S_m )-2  . $$
Donc il suffit de montrer que \begin{equation}\label{final}v_2(\prod_{i,j,s} g_{i,j,s} -(-1)^{ \frac{\sum_{i,j,s}k_{j,s}^{i}}{2l}  }  )\geq 2 . \end{equation}

Par $(*)$, on a $$ g_{i,j,s}\equiv  \begin{cases} (-1)^{([\frac{\chi sS_{i}}{4l}] - [\frac{k_{j,s}^{i}}{4l}]-  [\frac{\chi sS_{i}-k_{j,s}^{i}}{4l}] )} ,\quad   \text{si $\varepsilon_{i,j,s}=1 $}   \\ (-1)^{k_{j,s}^{i}/2l}
(-1)^{([\frac{\chi sS_{i}}{4l}+\frac{k_{j,s}^{i}}{4l}] - [\frac{k_{j,s}^{i}}{4l}]-  [\frac{\chi sS_{i}}{4l}] )}, \quad \text{si $\varepsilon_{i,j,s}=-1 $} 
\end{cases}   \quad(mod\ 4).
  $$
Cependant  $\chi$ est un nombre pair et $v_2(S_i)\geq v_2(\mathrm{p.g.c.d.}(k_{j,s}^{i}))=1$, donc $\frac{\chi sS_{i}}{4l}$ est un entier. On a $$[\frac{\chi sS_{i}}{4l}] - [\frac{k_{j,s}^{i}}{4l}]-  [\frac{\chi sS_{i}-k_{j,s}^{i}}{4l}] = - [\frac{k_{j,s}^{i}}{4l}]-  [\frac{-k_{j,s}^{i}}{4l}] \equiv \frac{k_{j,s}^{i}}{2l} \quad(mod\ 2)  .$$
De même
$${[\frac{\chi sS_{i}}{4l}+\frac{k_{j,s}^{i}}{4l}] - [\frac{k_{j,s}^{i}}{4l}]-  [\frac{\chi sS_{i}}{4l}] } \equiv 0 \quad(mod\ 2). $$
Donc c'est toujours vrai que $$g_{i,j,s}\equiv (-1)^{k_{j,s}^{i}/2l}   \quad (mod\ 4) .$$
Cela  implique (\ref{final}). 

La preuve est finalement complète !
\end{proof}

\subsection{Preuve de la deuxième  partie du théorème \ref{P}}\label{Char E}
\subsubsection{}
Dans l'article \cite{Deligne}, Deligne a conjecturé aussi que le nombre $C_n({X_k})$ est divisible par $|\text{Pic}_{X_k}^{0}(\mathbb{F}_{q^k})|$,  et il a donné des conjectures plus précises (la conjecture 6.3 et la conjecture 6.7 \cite{Deligne}). Dans cette section, on montre la conjecture 6.3 de $loc.\ cit.$ dans le cas partout non-ramifié  et calcule  ``la caractéristique d'Euler''. Rappelons que pour toute variété $V$  de dimension $m$, sa caractéristique d'Euler est définie par
$$\chi(V):= \sum_{i=1}^{2m} (-1)^{i}\dim_{\bar{\mathbb{Q}}_{\ell}} \mathrm{H}^{i}(V\otimes\mathbb{F}, \bar{\mathbb{Q}}_{\ell}) .$$
Par un théorème de Laumon (\cite{Cara}), on peut remplacer dans la définition les cohomologies $\ell$-adiques par celles à support compact. Donc par la formule des traces de Grothendieck-Lefschetz, si $$|V(\mathbb{F}_{q^k})|=\sum_{i}m_i\alpha_i^k, \quad \forall\ k\geq 1.$$ alors $\sum_{i}m_i$ est la caractéristique d'Euler de $V$.

Pour le polynôme $A_{g,n}$ du théorème \ref{Mel}, on a un résultat plus fin:

\begin{lemm}\label{zuihou}
Il existe $I_{g,n}(t, z_1, \cdots, z_g)\in \mathbb{Z}[t, z_1^{\pm 1}, \cdots, z_g^{\pm 1 }]$ tel que 
$$A_{g,n}(t, z_1, \ldots, z_g)=I_{g,n}(t,z_1,\ldots, z_g)\prod_{i=1}^{g}\left((1-z_i)(1-tz_i^{-1})\right)  .$$
Chaque monôme $t^{m}z_1^{n_1}\cdots z_g^{n_g}$ qui  apparaît dans $I_{g,n}$ vérifie \begin{equation*} m+\sum_{i=1}^g \min\{n_i, 0\}\geq 0 . \end{equation*}
De plus, on a $$I_{g,n}(1,1,\ldots,1)=\mu(n)n^{2g-3}.$$
\end{lemm}
\begin{proof}
On montre tout d'abord l'existence d'un tel polynôme. L'argument ci-dessous m'a été communiqué par A. Mellit. 

On utilise le théorème 1.1 de \cite{Mellit} et ses notations, qui a été rappelé après le théorème \ref{Mel}.

On désigne $\prod_{i=1}^{g}\left((z-z_i)(1-qz_i^{-1})\right)\in \mathbb{Q}[q,z][ z_1^{\pm 1}, \ldots, z_{g}^{\pm 1}]$ par $J_g$. 
Observons que $$\Omega_g\in 1+  T.J_g.  \mathbb{Q}(q,z)[T, z_1^{\pm 1}, \ldots, z_{g}^{\pm 1}].$$  On a  $$\mathrm{Log}\ \Omega_g\in T.J_g.  \mathbb{Q}(q,z)[T, z_1^{\pm 1}, \ldots, z_{g}^{\pm 1}],$$ et alors $ H_{g,n}\in J_g.  \mathbb{Q}(q,z)[z_1^{\pm 1}, \ldots, z_{g}^{\pm 1}]  $ pour tous $g, n\geq 1$. Donc il existe  $$I_{g,n}(q, z, z_1, \cdots, z_g)\in \mathbb{Q}(q,z)[z_1^{\pm 1}, \cdots, z_g^{\pm 1}],$$ tel que 
$$I_{g,n}(q, z, z_1, \cdots, z_g) J_g(q, z, z_1, \cdots, z_g)= H_{g,n}(q, z, z_1, \cdots, z_g) .$$ 
Comme \[ H_{g,n}(q,1,z_1,\ldots,z_g), J_g(q,1,z_1,\ldots,z_g)\in \mathbb{Z}[q][z_1^{\pm 1}, \cdots, z_g^{\pm 1}],\] et $J_g(q,1,z_1,\ldots,z_g)$ a pour terme constant $1$,  on déduit par lemme de Gauss que $I_{g,n}(q,1,z_1,\ldots,z_g)\in  \mathbb{Z}[q][z_1^{\pm 1}, \cdots, z_g^{\pm 1}]$.% En prenant $z=1$, on obtient la divisibilité désirée.

Chaque monôme $q^{m}z_1^{n_1}\cdots z_g^{n_g}$ qui  apparaît dans  $$A_{g,n}(q,z_1, \cdots, z_g)$$ vérifie \begin{equation*} m+\sum_{i=1}^g \min\{n_i, 0\}\geq 0 . \end{equation*}
S'il y avait un monôme dans $I_{g,n}(q,1,z_1, \cdots, z_g) $ qui ne satisfaisait  pas cette condition, on obtiendrait une contradiction en regardant les monômes   $q^{m}z_1^{n_1}\cdots z_g^{n_g}$  apparaissant   dans le produit de $I_{g,n}(q,1,z_1, \cdots, z_g)$ et $J_g(q,1,z_1, \cdots, z_g)$ tels que $m+\sum_{i=1}^g \min\{n_i, 0\}$ prenne la plus petite valeur.

Enfin, on calcule $I_{g,n}(1,1,\ldots, 1)$. 
Soit toujours $(e,n)=1$.  Soit $X_{\mathbb{C}}$ une surface de Riemann de genre $g$. On sait que la cohomologie de $\mathrm{Higgs}_{n,e}^{st}(X_{\mathbb{C}})$ est pure (cf. \cite{Hausel}). Donc son polynôme de Poincaré 
$$   \sum_{i} (-1)^{i} \dim (H^{i}_c (\mathrm{Higgs}_{n,e}^{st}(X_\mathbb{C}), \mathbb{C})) t^{i}$$
 est donné par 
$ t^{2(1+n^{2}(g-1))}A_{g,n}(t^{2}, t, \ldots, t)$ (cf. corollaire 1.3 de \cite{Schiffmann}).  
Par la théorie de Hodge non-abélienne (cf. \cite{Simpson}), $t^{2(1+n^{2}(g-1))}A_{g,n}(t^{2}, t, \ldots, t) $ est aussi le polynôme de Poincaré de la variété de caractères tordus de $G$. On sait alors $I_{g,n}(1, 1\cdots, 1)$ est égal la caractéristique d'Euler de la variété de caractères tordus de $PGL_n$ que Hausel et Rodriguez-Villegas ont calculée être égale à $\mu(n)n^{2g-3}$ (cf. corollaire 1.1.1 \cite{HR}).
\end{proof}

Par la formule (\ref{002}), le nombre $ A_{g,n}(q, \sigma_1, \ldots, \sigma_g)  -   C_n(X_1)$ est un polynôme à coefficients rationnels en $C_s(X_k)$ pour $s<n$. 
Donc par récurrence, le lemme précédent et le lemme de Gauss, l'existence de $Q_{g,n}$ est clair.

Maintenant on calcule $Q_{g,n}(1,1,\cdots,1)$.  Soit $\chi_n=Q_{g,n}(1,1,\cdots,1)$. 
Encore $A_{g,n}(q, \sigma_1, \ldots, \sigma_g)$ est un polynôme à coefficients rationnels en $C_s(X_k)$ pour $sk\leq n$. Par le lemme précédent $$P_{g,n}(1,1,\ldots, 1)=0, $$ il suffit de trouver les termes qui sont linéaires en $C_s(X_k)$ de ce polynôme. Par  la formule (\ref{001}), c'est 
$$ \sum_{d\mid n} \sum_{l\mid d} \frac{\mu(l)}{n(2g-2) d }[z^{d}] \biggr(\frac{(2g-2)n }{l}\sum_{s\geq 1}\sum_{k\geq 1} \frac{C_{s }(X_{kl} ) s  }{k} z^{skl} \biggr),$$
qui est égal à 
$$  \sum_{d\mid n} \sum_{l\mid d}\sum_{s\mid  \frac{d}{l}}    \frac{\mu(l)s^2}{d^2} C_s(X_{{d}/{s}}) = \sum_{d\mid n}\sum_{s\mid d} C_s(X_{d/s})( \sum_{l\mid \frac{d}{s}}\mu(l))=\sum_{d\mid n} C_d(X_1).$$
Par la formule (\ref{001}) et le lemme \ref{zuihou}, on obtient alors 
$$\mu(n)n^{2g-3}= \sum_{d\mid n} \chi_{d} . $$
Par l'inversion de Möbius, on a $$\chi_n= \sum_{l\mid n}\mu(l)\mu(n/l) l^{2g-3}.$$
%\end{proof}

%\addcontentsline{toc}{section}{Appendices}

\appendix
\section{Indépendance du degré}
Dans cet appendice on montre que $J^{ }_{e}$ ne dépend que de l'ordre de $e$ dans $\mathbb{Z}/n\mathbb{Z}$ et cela implique que le nombre des classes d'isomorphie des fibrés vectoriels géométriquement indécomposables  de degré $e$ de rang $n$ ne dépend pas du degré $e$ quand $(e, n)=1$. 

Soit $M$ un sous-groupe de Levi standard à $r$ facteurs: $M\cong  G_{n_1}\times\cdots\times G_{n_r}$.  Soit $W(\ago_M)$ l'ensemble des $w\in W_n/W^M$ tel que $w(M)$ soit standard. On identifie $W(\ago_M)$ avec $\mathfrak{S}_r$ (cf. les discussions au dessus de la proposition \ref{1Qexplicit}). Soit $s\in W(\ago_M)$, on va noter $Q_s$ le sous-groupe parabolique dans $ \mathcal{P}(M)$ qui lui est associé.

Rappelons qu'avec les notations de la proposition \ref{1Qexplicit}, 
on a 
\begin{align}\hat{\mathbbm{1}}^{e}_{Q_s}(\lambda)&=\lambda^{{\HE}_{Q_s}^{e}}\prod_{\alpha\in\Delta_{Q_s}}\frac{1}{1-\lambda^{\alpha^{\vee}}} ,
\end{align} où  $$\glslink{tHQe}{\mathrm{H}^e_{Q_s}}=s^{-1}([\frac{e}{n}   r_{s}^{0}]-[\frac{e}{n}r_{s}^{1} ],\ldots, [\frac{e}{n}r_{s}^{r-1}]-[\frac{e}{n}r_{s}^{r} ])$$
et $r_{s}^{i}=n_{s^{-1}(1)}+\cdots+n_{s^{-1}(i)}$.
Avec la fonction $\theta_Q$ d'Arthur, on a 
\begin{equation}\label{1thetae}
\hat{\mathbbm{1}}^{e}_{Q_s}(\lambda) = (-1)^{\dim \ago_{L}^{{G}}}\lambda^{I_{Q_s}^{e}}\theta_{Q_s}(\lambda)^{-1}, 
 \end{equation}
 où $I_{Q_s}^e={\HE}_{Q_s}^{e}+s^{-1}(0,1,1\ldots,1) $.

\begin{lem}\label{Appendix}
Pour tout $e\in\mathbb{Z}$, la famille des fonctions 
$(\lambda\mapsto\lambda^{I_{Q}^{e}})_{Q\in \mathcal{P}(M)}$ définie dans \ref{a_L} (en particulier (\ref{1thetae})) est une $({G},M)$-famille sur $X_M^{G}$. 
\end{lem}
\begin{proof} Cela a été vérifié par L. Lafforgue dans \cite[Lemme 5 (ii), p.301]{Laff}.

\end{proof}

\begin{lem}\label{SPF}
Soit $\mu_0\in X_M^{G}$.
Soit $(c_{Q})_{Q\in\mathcal{P}(M)}$  une $({G},M)$-famille sur un domaine contenant $\mu_0^{\mathbb{Z}}$ telle que pour chaque $L\in \mathcal{L}(M)$, $c_{M}^{Q}$ (définie dans \ref{GLfa}) soit indépendante de $Q\in \cal{P}(L)$, qu'on note $c_{M}^{L}$. Supposons que $c_{Q}(\lambda\mu_{0})=c_{Q}(\lambda)$ pour tout $\lambda\in X_{M}^{{G}}$ où $c_Q$ est définie. Alors 
$$\lim_{\lambda\rightarrow 1} \sum_{Q\in \mathcal{P}(M)} \hat{\mathbbm{1}}^{e}_Q(\lambda\mu_{0}){c_{Q}( \lambda\mu_{0})} , $$ 
s'annule sauf si $\mu_{0}\in X_{{G}}^{{G}}$. 
\end{lem}
\begin{proof}
C'est une variante du lemme \ref{SPE}. La preuve est similaire. Soit $$d_Q(\lambda)= \hat{\mathbbm{1}}^{e}_Q(\lambda\mu_{0})\theta_Q(\lambda)^{-1}.$$ 
On vérifie comme dans le lemme \ref{SPE} (ou par le lemme \ref{Appendix}) que $(d_Q(\lambda))_{Q\in \mathcal{P}(M)}$ est une $(G,M)$-famille. Il s'ensuit par le même argument dans le lemme \ref{SPE} que pour tout sous-groupe de Levi propre $L$ de $G$, on a $d_L=0$. 
\end{proof}

\begin{lem}\label{A3}
Supposons que $\{c_{Q} \}_{Q \in\mathcal{P}(M)}$ est une $({G},M)$-famille sur un voisinage de $1$ telle que pour chaque $L\in \mathcal{L}(M)$,  $c_{M}^{R}$ soit indépendante de $R\in \mathcal{P}(L)$ (qu'on désigne par $c_{M}^{L}$) pour tout sous-groupe de Levi $L$  défini sur $F$ contenant $M$. Alors la limite
$$\lim_{\mu\rightarrow 1} \sum_{Q\in \mathcal{P}(M)} \hat{\mathbbm{1}}_{Q}^{e}(\mu) c_{Q}(\mu)$$ ne dépend que de l'ordre de $e$ dans $\mathbb{Z}/n\mathbb{Z}$. 
\end{lem}
\begin{proof}
%Rappelons que $ \hat{\mathbbm{1}}_{Q}^{e}(\mu)= \theta_{Q}(\mu)^{-1} \mu^{I^{e}_{Q}}$ (l'identité  (\ref{1thetae}) dans \ref{a_L}).
Soit $d^{e}_{Q}(\lambda) =\lambda^{I^{e}_{Q}}$, $\forall \ Q\in\mathcal{P}(M)$. Alors par le lemme \ref{Appendix} cela définit une $({G},M)$-famille sur $X_M^{G}$.   Par la proposition \ref{SCIN},  il suffit donc de prouver que $d^{e}_{L}$ ne dépend que de l'ordre de $e$ pour tout $L\in \mathcal{L}(M)$.

%On a $I_{Q_s}^{e}={\HE}_{Q_s}^{e}+s^{-1}(0,1,1,\ldots,1) $ pour $s\in W(\ago_M)$ et $(0, 1, \cdots, 1)\in \ago_{s(M)}$.% où  \begin{equation}  {\HE}_{Q_s}^{e}=s^{-1}([\frac{e}{n}r_{s}^{0}]-[\frac{e}{n}r_{s}^{1}   )],\ldots, [\frac{e}{n}r_{s}^{r-1}]-[\frac{e}{n}r_{s}^{r}   )]) ,\end{equation}  pour $$ r_{s}^{i}= n_{s^{-1}(1)}+n_{s^{-1}(2)}+\cdots+n_{s^{-1}(i)}  .$$

Soit $L\in\mathcal{L}(M)$, un sous-groupe de Levi défini sur $F$ contenant $M$. Sans perte de généralité, on peut supposer que $M$ et $L$ sont standards. Par l'inclusion $X_{L}^{{G}}\subseteq X_{M}^{{G}}$, on suppose que $X_{L}^{{G}}$ est
$$
\{(\lambda_{1},\ldots, \lambda_{r})\in X_{M}^{{G}} \ |\ \lambda_{1}=\cdots=\lambda_{i_{1}}, \lambda_{i_{1}+1}=\cdots=\lambda_{i_{2}}, \ldots, \lambda_{i_{k-1}+1}=\cdots=\lambda_{i_{k}}          \} ,$$ 
pour $1\leq i_{1}\leq \cdots\leq i_{k}=r$. 
%Sous l'identification $W(\ago_M)\cong \mathfrak{S}_{r}$, l'éléments qui transforment l'intervalle $[i_{1}+\cdots+i_{j-1}+1, i_{1}+\cdots+i_{j}]\cap \{1, \cdots, r\}$ en un intervalle sont identifiés avec les éléments dans $W(\ago_L)$. 

%Soit $s\in W(\ago_M)$ tel que le sous-groupe parabolique associé soit dans un sous-groupe parabolique de $\mathcal{P}(L)$. Alors sous l'identification $W(\ago_M)\cong \mathfrak{S}_{r}$, l'élément  $s$ transforme l'intervalle $[i_{1}+\cdots+i_{j-1}+1, i_{1}+\cdots+i_{j}]\cap \{1, \cdots, r\}$ en un intervalle. 

Soit $s\in W(\ago_M)$ et $t\in W(\ago_L)$ tel que  $Q_s \subseteq Q_t$. 
Rappelons que $s\in W(\ago_{M})$ définit un morphisme 
$$s: X_{M}^{{G}}\rightarrow X_{s(M)}^{{G}}, $$
tel que pour $H\in \ago_{s(M)}$ et $\lambda\in X_{M}^{{G}}$ on ait
$\lambda^{s^{-1}(H)}=s(\lambda)^{H}$. 
Alors $\forall \ \lambda\in X_{L}^{{G}}$,  soit $i_0=0$, on a
\[ \begin{split} \lambda^{I^{e}_{Q_s}} &=s(\prod_{l=1}^{r} \lambda_{l}^{[ r_{s}^{l-1}\frac{e}{n} ]-[ r_{s}^{l}\frac{e}{n} ]    }) s(\prod_{l=2}^{r}\lambda_{l} ) \\ &=  \lambda_{i_{1}}^{i_{1}-i_0-1} \cdots \lambda_{i_{k}}^{i_{k}-i_{k-1}-1}   s(\prod_{j=1}^{k} \lambda_{i_j}^{[ r_{s}^{i_{j-1}}\frac{e}{n} ]-[ r_{s}^{i_j}\frac{e}{n} ]    }     ) s(\prod_{j=2}^{k}\lambda_{i_{j}} ) ,  \end{split} \]
qui est égal à $(\lambda_{i_{1}}^{i_{1}-i_0-1} \cdots \lambda_{i_{k}}^{i_{k}-i_{k-1}-1}    )\lambda^{I^{e}_{Q_t}}$.  
Notons que le premier facteur ne dépend pas de $t$, donc par définition de $d_{L}^{e}$, on a  $$d_{L}^{e}= \lim_{\lambda\rightarrow 1}\sum_{Q\in \mathcal{P}(L)} \lambda^{I^{e}_{Q}}\theta_{Q}(\lambda)^{-1} (\lambda_{i_{1}}^{i_{1}-i_0-1} \cdots \lambda_{i_{j}}^{i_{k}-i_{k-1}-1}    ) =\lim_{\lambda\rightarrow 1}\sum_{Q\in \mathcal{P}(L)} \lambda^{I^{e}_{Q}}\theta_{Q}(\lambda)^{-1} .$$

Soit  $\{x\}=x-[x]$, $\forall x\in\mathbb{R}$, alors   $$[x]-[y]=\begin{cases} [x-y], \quad \text{si $\{x\}\geq \{y\} $ };          \\ [x-y]+1 , \quad \text{si $\{x\}< \{y\} $ }  .
\end{cases} 
$$
Soient $L\cong G_{m_{1}}\times\cdots G_{m_{k}}$ et $t\in W(\ago_L)$.  On sait que 
$$[\frac{e}{n}r_{t}^{i-1}]-[\frac{e}{n}r_{t}^{i}   ]= -[\frac{e}{n} m_{t^{-1}(i)} ]-\varepsilon_{i}^{t, e},$$
où $$\varepsilon_{i}^{t,e} =\begin{cases}0, \quad \text{si $\{\frac{e}{n}r_{t}^{i}   \}\geq \{\frac{e}{n}r_{t}^{i-1}\} $ }   ;       \\ 1 , \quad \text{si $\{\frac{e}{n}r_{t}^{i}\}< \{\frac{e}{n}r_{t}^{i-1}\} $ }  .
\end{cases}
 $$
En particulier, $\varepsilon_{1}^{t,e}=0$.  Il s'ensuit que $$\lambda^{I_{Q_t}^{e}} = (\prod_{j=1}^{k}\lambda_{i_j}^{-[\frac{e}{n}m_{j}]} )t(\prod_{j=2}^{k}\lambda_{i_j}^{1-\varepsilon_{j}^{t,e}}). $$
Notons que le premier facteur ne dépend pas de $t\in W(\ago_L)$. 
La fonction $t(\prod_{j=2}^{k}\lambda_{i_j}^{1-\varepsilon_{i}^{t,e}})$ est un polynôme, et 
$$\deg t(\prod_{j=2}^{k}\lambda_{i_j}^{1-\varepsilon_{i_j}^{t,e}})  =\dim\ago_{L}^{{G}}- \sum_{i=2}^{k} \varepsilon_{i}^{t,e}.   $$

On calcule $d_L^{e}$. Puisqu'on sait déjà que la limite 
$$d_{L}^{e}=\lim_{\lambda\rightarrow 1}\sum_{Q\in \mathcal{P}(L)} \lambda^{I^{e}_{Q}}\theta_{Q}(\lambda)^{-1}$$
existe. Pour l'évaluer, il suffit de prendre $\lambda_i=1+i\epsilon$ pour tout $i$ et laisser $\epsilon$ tendre vers $0$. Par l'expansion de Taylor et le fait que la limite existe, $d_{L}^{e}=0$ sauf si $\varepsilon_{i}^{t,e}=0$ pour tout $i$ et $t$. Ce dernier équivaut à que $$\{\frac{e}{n}r_{ t(Q) }^{i} \}=0 \quad \forall\ i\geq 1 ,$$
i.e. $n\mid e\sum_{l=1}^{i}m_{t^{-1}(l)}$ pour tout $i\geq 1$ et $t\in \mathfrak{S}_{k}$, et donc cela équivaut  à
$$n\mid e m_{l}\quad \forall\ l\geq 1 .$$
Notons que si $e, e'$ ont le même ordre dans $\mathbb{Z}/n\mathbb{Z}$ alors $$(n\mid e m_{l}\quad \forall\ l\geq 1)\Longleftrightarrow (n\mid e' m_{l}\quad \forall\ l\geq 1) .$$ Donc $d^{e}_{L}$ ne dépend que de l'ordre de $e$. L'assertion maintenant est un corollaire de la formule suivante (Proposition \ref{SCIN})
\begin{equation} \label{....}
\lim_{\mu\rightarrow 1} \sum_{Q\in \mathcal{P}(M)} \hat{\mathbbm{1}}_{Q}^{e}(\mu) c_{Q}(\mu)=\sum_{L\in\mathcal{L}(M)}d_L^{e}c_M^L. 
\end{equation}
\end{proof}

\begin{thm}\label{Jeind}
Le nombre $J^{ }_{e}$ ne dépend que de l'ordre de $e$ dans $\mathbb{Z}/n\mathbb{Z}$. 
\end{thm}
\begin{proof}
Par l'expression \ref{aaaz}, il suffit de prouver que (1) l'intégrale 
\begin{equation}\label{bbbz}
\int_{\Im X_L^G} \lim_{\mu\rightarrow 1}\sum_{Q\in \mathcal{P}(L)} \hat{\mathbbm{1}}_{Q}^{e}(\mu) \frac{ n_{Q|Q_{L}}({\pi},  \frac{\lambda \lambda_w }{ \mu}) }{n_{Q|Q_{L}}({\pi}, {\lambda}\lambda_{w}  )}   \d\lambda,   \end{equation}
ne dépend pas de $\lambda_w$ et ne dépend que de l'ordre de $e$ dans $\mathbb{Z}/n\mathbb{Z}$, de plus (2) le facteur (b.) de la proposition \ref{expr1} ne dépend que de  l'ordre de $e$ dans $\mathbb{Z}/n\mathbb{Z}$. 

On démontre la première indépendance. Soit $$c_Q(\mu)=\frac{ n_{Q|Q_{L}}({\pi},  \frac{\lambda \lambda_w }{ \mu}) }{n_{Q|Q_{L}}({\pi}, {\lambda}\lambda_{w}  )}  .$$
Par lemme \ref{A3}, la limite dans \ref{bbbz} ne dépend que de l'ordre de $e$ dans $\mathbb{Z}/n\mathbb{Z}$. 
En appliquant le corollaire \ref{INT} (notons que $\Im X_L^G=\Im X_L^O\Im X_O^G$) aux $(O,L)$-familles $c_L^O$ pour tout $O\in \mathcal{L}(L)$,  par la formule \eqref{....}, l'intégrale \ref{bbbz} ne dépend pas de $\lambda_w$ comme $|\lambda_w^{\beta^{\vee}}|=1$ pour tout $\beta\in \Phi(Z_{M_P}, L)$.

On démontre la seconde indépendance. 
Par le calcul de \ref{bonrep} (auquel on renvoie pour les notations), il suffit de montrer (cf. \eqref{Deltapi}) que la valeur $$    \delta_{e}({\pi},{L})= \sum_{{\lambda}\in {\Lambda_{ {G}}}} {{\lambda} _{1}}^{-(\sum \frac{l_i(l_i-1)}{2} |\Fix(\pi_{i})|)-e}  $$
ne dépend que de l'ordre de $e$ dans $\mathbb{Z}/n\mathbb{Z}$. Rappelons que ${\Lambda_{{G}}}$ est le sous-groupe de $X_{{G}}^{{G}}$  de cardinal $\mathrm{p.g.c.d.} (l_{i}|\Fix(\pi_{i})|) $.  La somme $\delta_{e}({\pi},{L})$ est égale à $|{\Lambda}_{G}|$  ou $0$ selon si le caractère $  {{\lambda}_{1}}\mapsto  {{\lambda}_{1}}^{-(\sum \frac{l_i(l_i-1)}{2} |\Fix(\pi_{i})|)-e}  $ est trivial ou non. 
Soient $e$ et $e'$ deux entiers avec le même ordre dans $\mathbb{Z}/n\mathbb{Z}$. Alors ils ont le même ordre dans $\mathbb{Z}/|{\Lambda}_G| \mathbb{Z}$. 
L'ordre de $e+ \sum \frac{l_i(l_i-1)}{2} |\Fix(\pi_{i})|$ dans $\mathbb{Z}/|{\Lambda_G}|\mathbb{Z} $, est égal au plus petit commun multiple de celui de $e$ et celui de $\sum \frac{l_i(l_i-1)}{2} |\Fix(\pi_{i})|$. Comme $|{\Lambda}_G|$ divise $n$, il est égal à l'ordre de $e'+ \sum \frac{l_i(l_i-1)}{2} |\Fix(\pi_{i})|$ dans $\mathbb{Z}/|{\Lambda}_G|\mathbb{Z} $.
On conclut que
$$\delta_{e}({\pi}, {L})=\delta_{e'}({\pi}, {L}) .$$
\end{proof}

Par le théorème \ref{Maing}, on a le corollaire suivant:
\begin{cor}
Le nombre $\mathcal{P}_n^{e}(X_1)$ des classes d'isomorphie des fibrés isoclines sur $X_1$ ne dépend que de l'ordre de $e$ dans $\mathbb{Z}/n\mathbb{Z}$. \end{cor}

\section{Lien entre les fibrés vectoriels indécomposables et les fibrés de Higgs}
\sectionmark{Fibrés indécomposables et fibrés de Higgs}
Dans cet appendice, on explique comment extraire de nos résultats et ceux de \cite{Chau}  une preuve 
 du résultat suivant:
$$\mathcal{P}_{n}^{e}(X_1)=q^{-n^2(g-1)-1}| \mathrm{Higgs}_{n,e}^{st}(X_1)(\mathbb{F}_q)| \quad \text{quand $(n,e)=1$},$$
(cf.  \ref{Higgs}) qui n'utilise que  la formule des traces. 

%Pour tout sous-groupe parabolique  standard $P$ de $G$, soit $\mathfrak{m}_P$ (resp. $\mathfrak{n}_P$) l'algèbre de Lie de $M_P$ (resp. de $N_P$). 

Soit $\mathfrak{g}$ algèbre de Lie de $G$. %On normalise les mesures tels que $\vol( \mathfrak{n}_P(F) \backslash \mathfrak{n}_P(\AAA))=1$. 

Pour tout $\gamma \in G(F)$ (resp. $X\in \mathfrak{g}(F)$) soit $\chi_\gamma$ (resp. $\chi_X$) son polynôme caractéristique. %Pour une partie $R$ de $G(\AAA)$ ou $\mathfrak{g}(\AAA)$, et $p\in \AAA[X]$ un polynôme  unitaire de degré $n$ on note $$ R_p=\{ x\in R|\ \chi_p=p \} .$$
Si $x\in G(\AAA)$, et $\gamma\in G(F)$ ( resp. $X \mathfrak{g}(F)$), on a 
$$x^{-1}\gamma x\in G(\mathcal{O})\implies \chi_{\gamma}\in \mathbb{F}_q[X] .$$
$$(\text{resp.}\quad   x^{-1}X x\in \mathfrak{g}(\mathcal{O}) \implies  \chi_{X}\in \mathbb{F}_q[X].) $$

Soit $p\in \mathbb{F}_q[X]$ un polynôme unitaire de degré $n$. Pour tout $x\in G(\AAA)$ et $T\in\ago_B$  soient 
$$k_p^T(x,x)= F^{G}(x, T)\sum_{\gamma\in G(F), \chi_\gamma=p} \mathbbm{1}_K(x^{-1}\gamma x),$$ 
et 
$$\tilde{k}_p^{T}(x,x) = F^{G}(x, T)\sum_{X\in \mathfrak{g}(F), \chi_X=p} \mathbbm{1}_{ \mathfrak{g}({\mathcal{O}}) }(x^{-1}X x).   $$
%$$ k_p^T(g,g) = \sum_{P\in \mathcal{P}(B)}(-1)^{\dim \ago_P^{G}}\sum_{\delta\in P(F)\backslash G(F)} \hat{\tau}_P(H_P(\delta g)-T)\sum_{\gamma\in M_P(F)_p} \int_{N_P(\AAA)} \mathbbm{1}_{K}(g^{-1}\delta^{-1} \gamma n \delta g)\d n , $$ et $$\tilde{k}_p^{T}(g,g)= \sum_{P\in \mathcal{P}(B)}(-1)^{\dim \ago_P^{G}}\sum_{\delta\in P(F)\backslash G(F)} \hat{\tau}_P(H_P(\delta g)-T)\sum_{X\in \mathfrak{m}_P(F)_p} \int_{\mathfrak{n}_P(\AAA)} \mathbbm{1}_{\mathfrak{g}(\mathcal{O})}(g^{-1}(X+U)g)\d U .$$
où $\mathbbm{1}_{\mathfrak{g}(\mathcal{O})}$ est la fonction sur $\mathfrak{g}(\AAA)$ qui est la fonction caractéristique de $\mathfrak{g}(\mathcal{O})$. Quand $d(T):=\min_{\alpha\in \Delta_B}\alpha(T)$ est assez grand, soient $$J_{p,e}^T= \int_{G(F)\backslash G(\AAA)^e} {k}_p^T(x,x)\d g, $$
et 
$$\tilde{J}^{T}_{p,e}=\int_{G(F)\backslash G(\AAA)^e} \tilde{k}_p^T(x,x)\d g .$$
Donc par le dictionnaire des fibrés-adèles, on a (cf.  preuve de \ref{Maing})
$$ J_{p,e}^T=\sum_{\cal{E} \  T\text{-semi-stable de degré} \ e}\frac{|\{\gamma\in Aut(\mathcal{E})|\chi_\gamma=p \}  |}{|Aut(\mathcal{E})|}
$$
$$ \tilde{J}_{p,e}^T=\sum_{\cal{E} \  T\text{-semi-stable de degré} \ e}\frac{|\{\gamma\in End(\mathcal{E})|\chi_\gamma=p \}  |}{|Aut(\mathcal{E})|}
$$
où la somme porte sur les classes d'isomorphie des fibrés vectoriels $T$-semi-stables. Les $T\mapsto J_{p,e}^T$ et $T\mapsto \tilde{J}_{p,e}^T$ sont quasi-polynomiaux quand $d(T)$ est assez grand (Théorème \ref{QP}  et \cite[Théorème 5.2.1 et Théorème 6.1.1.2]{Chau}), et on désigne encore par $J_{p,e}^T$ et $\tilde{J}_{p,e}^T$ pour ces quasi-polynômes pour n'importe quel $T$.

%\begin{lem}
%Soit $\mathcal{E}$ un fibré vectoriel de rang $n$. Un endomorphisme $\gamma$ est un automorphisme si et seulement si le terme constant de son polynôme caractéristique est non-nul. 
%\end{lem}
%\begin{proof}
%Prenons garde que pour un endomorphisme $\gamma$ de $\mathcal{E}$ $$(\ker(\gamma)=0)\Longleftrightarrow (\mathrm{coker}(\gamma)=0 )\Longleftrightarrow( \gamma \text{ est un automorphisme}) .$$
%En effet, on a une suite exacte:
% $$0\longrightarrow \ker(\gamma)\longrightarrow \mathcal{E}\xrightarrow{\ \ \gamma\ \ } \mathcal{E}\longrightarrow \mathrm{coker}(\gamma)\longrightarrow 0 .$$
%Donc on a $\deg( \ker(\gamma))= \deg( \mathrm{coker}(\gamma))$ et $\mathrm{rg}(\mathrm{ker}(\gamma)) =\mathrm{rg}(\mathrm{coker}(\gamma)) $.
% \end{proof}

Soit $\mathcal{E}$ un fibré vectoriel de rang $n$. Un endomorphisme $\gamma$ est un automorphisme si et seulement si son déterminant est inversible. Comme $H^{0}(X, \mathcal{O}_X)\cong \mathbb{F}_q$, cela équivaut à dire que le terme constant de son polynôme caractéristique est non-nul. 
On en déduit alors que pour tout $T$ tel que $d(T)$ %(cf.  sous-section \ref{syst}  et \ref{31}) 
soit assez grand, si le terme constant de $p$ est non-nul 
\begin{equation}\label{iio}{J}_{p,e}^T=\tilde{J}_{p,e}^T.\end{equation}
%et si le terme constant de $p$ est nul \begin{equation}\label{io}J_{p,e}^ . \end{equation}

Donc on a
$$J^{ }_{e} = \sum_{p\in \mathbb{F}_q[X]} J^{ }_{p,e} = \sum_{p\in \mathbb{F}_q[X], p(0)\neq 0} \tilde{J}^{ }_{p,e} .$$
Quand $p=X^n$, on note $\tilde{J}_{nilp, e}^{ } = \tilde{J}^{ }_{p,e}$. 
Par le théorème 6.2.1 de \cite{Chau}, comme il y a $q-1$ éléments dans $\mathbb{F}_q^{\times}$, on obtient quand $(n,e)=1$
$$J^{ }_{e}=(q-1)\tilde{J}_{nilp, e}^{ }= \frac{q-1}{q} \sum_{p\in \mathbb{F}_q[X] }\tilde{J}_{p,e}^{ } .$$
Par  le corollaire 5.2.2 et le corollaire 5.2.3 de \cite{Chau}, on a 
$$  J^{ }_{e} = \frac{q-1}{q} q^{-n^2(g-1)} \vol(\mathbf{Higgs}_{n,e}^{st}(X_1)(\mathbb{F}_q )), $$
où $\vol$ est la masse d'un groupoïde (la somme pondérée par l'inverse de l'ordre du groupe des automorphismes  des classes d'isomorphie des objets). Comme  tout  fibré de Higgs semi-stable de rang $n$ de degré $e$ pour $(n,e)=1$ est automatiquement stable, le cardinal de son groupe d'automorphisme est $q-1$. On obtient $$J^{ }_{e} = q^{-n^2(g-1)-1} |\mathrm{Higgs}_{n,e}^{st}(X_1)(\mathbb{F}_q )| . $$
Cela prouve l'assertion par le théorème \ref{Maing}.

\section{Un théorème sur les polynômes universels}
Dans cet appendice,  on démontre un théorème sur les polynômes universels.

\begin{thm}\label{ordinary}
Soit $p$ un nombre premier. 
Soit $P(t, z_1, \ldots, z_g)$ un élément dans $\mathbb{Z}[t^{\pm 1}, z_1^{\pm 1},\cdots, z_g^{\pm 1}]$ tel que $P(q,\sigma_{1}, \ldots, \sigma_{g})\in \mathbb{Z}$ pour toute courbe projective, lisse et géométriquement connexe $C$ de genre $g$ définie sur un corps fini $\mathbb{F}_q$ en caractéristique $p$ (où $\sigma_1, \ldots, \sigma_{2g}$ sont ses $q$-nombres de Weil indexés tels que $\sigma_i\sigma_{2g-i}=q$). 
Alors chaque monôme $t^{m}z_1^{n_1}\cdots z_g^{n_g}$ qui  apparaît dans $P$ vérifie \begin{equation} m+\sum_{i=1}^g \min\{n_i, 0\}\geq 0 . \end{equation}
\end{thm}
\begin{proof}
Le cas où $g=0$ est trivial.

Soit $g\geq 1$. D'abord, on a besoin d'établir l'existence d'une courbe projective, lisse et géométriquement connexe $C$ de genre $g$ définie sur un corps fini $\mathbb{F}_{p^d}$ de cardinal $p^d$ pour un certain $d$ telle que les conditions suivantes soient satisfaites: 
\textit{(1) Les valuations $p$-adiques (la valuation est normalisée de telle sorte que $v_p(p)=1$) des $p^d$-entiers de Weil $\sigma_i$ de $C$ sont soit $d$, soit $0$ pour un isomorphisme $\bar{\mathbb{Q}}_{\ell}\rightarrow \bar{\mathbb{Q}}_p$. 
(2) De plus, si $m_i$ et $m'_i$ ($0\leq i\leq g$) sont des entiers tels que
$$p^{dm_0}\sigma_1^{m_1}\cdots \sigma_g^{m_g}=p^{dm'_0}\sigma_1^{m'_1}\cdots \sigma_g^{m'_g},$$ alors $m_i=m'_i,$ pour tout $0\leq i\leq g$.} 

On utilise des arguments et des résultats dans \cite{Kowalski}. En effet, si la caractéristique $p$ est strictement plus grande que $2$, un exemple qui satisfait les conditions ci-dessus a déjà été donné par \cite[Proposition 4]{Kowalski}. La preuve ci-dessous marche pour tout $p$. En fait, elle démontre que la plupart (en un certain sens) des courbes satisfont ces deux propriétés.

Les conditions $(1)$ et $(2)$ sont satisfaites si 
\textit{$(1')$ la variété jacobienne de $C$ est une variété abélienne ordinaire et $(2')$ le numérateur de la fonction zêta de $C$ est un polynôme irréductible sur $\mathbb{Q}$ dont le groupe de Galois est isomorphe au groupe de Weyl $\mathcal{S}_g\ltimes(\mathcal{S}_2)^{g}$ de $Sp(2g)$.} 
En effet, c'est une conséquence de la preuve de \cite[Proposition 1]{Kowalski} (même si la proposition elle-même semble moins forte que la condition $(2)$ ci-dessus, sa preuve a établi cette assertion plus forte), où les hypothèses de cette proposition sont vérifiées dans le \cite[Lemma 1, Lemma 2]{Kowalski}.

Si $g=1$, l'existence d'une telle courbe elliptique est un corollaire du théorème de Honda-Tate, par exemple on peut appliquer ce théorème au polynôme $X^2+X+p=0$ pour tout nombre premier $p$. 

Soit $g\geq 2$. %Fixer un nombre premier $p$.  
Pour les détails des définitions qui suivent, on renvoie le lecteur à \cite[§10.6]{KS}. 
Soit $\mathcal{M}_{g,3K}$ le schéma de modules fins de courbes stables de genre $g$ avec une structure $3K$ définies sur ${\mathbb{F}}_p$. Soit $\mathcal{M}_{g,3K}^{\circ}\subseteq \mathcal{M}_{g,3K}$ l'ouvert correspondant aux courbes lisses. Soit $\pi: \mathcal{C}\rightarrow \mathcal{M}_{g,3K}$ la courbe universelle avec structure de $3K$. 
Comme l'ordinarité  des fibres est une condition ouverte sur la base, il existe un ouvert $V\subseteq \mathcal{M}_{g,3K}$ tel que les courbes $\mathcal{C}_s$ sont ordinaires pour tous points $s\in V$. L'ensemble $V$ est non-vide car un exemple d'une courbe stable (qui est singulière) est construit par Koblitz \cite[Theorem 5]{Kob}. On sait alors que $V^{\circ}:= V\cap  \mathcal{M}_{g,3K}^{\circ}$ est non-vide puisque, d'après Deligne-Mumford, $\mathcal{M}_{g,3K}$ est irréductible (voir \cite[Theorem 10.6.10]{KS}) et donc $V^{\circ}$ est ouvert et dense dans $\mathcal{M}_{g,3K}^{\circ}$. 

Pour tout $\ell \neq p$, le groupe de monodromie géométrique du système local $\ell$-adique $R^1\pi_{!}\bar{\mathbb{Q}}_{\ell}$ est $Sp(2g)$ tout entier (\cite[Theorem 10.6.11]{KS}). 
Comme a expliqué dans la preuve de la proposition $3$ de \cite{Kowalski} à laquelle on renvoie le lecteur pour des références additionnelles, par un théorème de Larsen, pour une densité $1$ de nombres premiers  $\ell$, le groupe de monodromie mod-$\ell$ est $Sp(2g, \mathbb{F}_{\ell})$. On peut alors utiliser un théorème de Chavdarov \cite[Theorem 2.1]{Chavdarov} qui implique que la plupart des points dans $\mathcal{M}_{g,3K}^{\circ}(\mathbb{F}_{p^k})$ satisfont  la condition $(2')$ quand $k\rightarrow \infty$. On fait juste remarquer que le théorème de Chavdarov est énoncé sous l'hypothèse que pour presque tout nombre premier $\ell$, le groupe de monodromie mod-$\ell$ est $Sp(2g, \mathbb{F}_{\ell})$, mais il suffit d'avoir un nombre infini de tels nombres premiers comme le montre  sa preuve.  

On conclut alors qu'en toute caractéristique $p$, il existe une courbe projective, lisse et géométriquement connexe définie sur le  corps fini $\mathbb{F}_{p^m}$ pour un certain $m$, telle que les conditions $(1)$ et $(2)$ soient satisfaites. Fixons une telle courbe et $m$.
Soient $\sigma_1, \cdots, \sigma_{2g}$ ses $p^m$-entiers de Weil.

Soit $$P(t, z_1, \ldots, z_g)=\sum_{(a,\underline{n})\in \mathbb{Z}\times \mathbb{Z}^{g}} c_{(a,\underline{n})}t^a \underline{z}^{\underline{n}}, $$
où $\underline{z}^{\underline{n}}=z_{1}^{n_1}\cdots z_g^{n_g}$ si $\underline{n}=(n_1, \ldots, z_g)$. 
Supposons par l'absurde qu'il existe un monôme $t^{a'}z_1^{n'_{1}}\cdots z_g^{n'_g}$ avec $c_{(a', n'_1, \ldots, n'_g)}\neq 0$  satisfaisant $$m'+\sum_{i=1}^g \min\{n'_i, 0\}< 0.$$
Soit $I\subseteq \{1, 2, \ldots, g\}$ l'ensemble des indices $i$ pour lesquels $n'_i<0$. 
Selon le signe de $m+\sum_{i\in I} n_i $ de chaque monôme $t^{m}z_1^{n_1}\cdots z_g^{n_g}$, on peut décomposer $P$ en une somme $P=P_{\geq 0}+P_{<0}$. En fixant un isomorphisme $\bar{\mathbb{Q}}_{\ell}\rightarrow \bar{\mathbb{Q}}_p$, par la condition $(1)$, on peut supposer en ré-indexant que $\sigma_i$ a pour valuation $p$-adique $d$ si $i\in I$ et $0$ si $i\in \{1, 2, \ldots, g\}-I$. Alors, comme $$P_{<0}=P-P_{\geq 0},$$ on a $P_{<0}(p^{mk}, \sigma_1^k, \sigma_2^k, \ldots, \sigma_g^k)\in \bar{\ZZZ}_p$ pour tout $k\geq 1$. Par la condition $(2)$, on a une contradiction par le lemme suivant.  
\begin{lemma}
Soient $c_1, \ldots, c_r\in \bar{\mathbb{Z}}_{p}$ non-nuls. Soient $\gamma_1, \ldots, \gamma_r \in \bar{\mathbb{Z}}_{p}$ deux à deux distincts de valuation p-adique $\val_p(\gamma_i)>0$, $i=1, \ldots, r$. Alors il existe $k\in \mathbb{N}^{*}$, tel que
$$ \sum_{i=1}^{r}c_i \gamma_i^{-k} \notin \bar{\mathbb{Z}}_{p}. $$
\end{lemma}
\begin{proof}[Preuve du lemme]
On raisonne par l'absurde. 
On suppose que $\val_p(\gamma_1)$ est le plus petit parmi les $\val_p(\gamma_i)$. Si pour tout $k\geq 1$,
$$\sum_{i=1}^{r}c_i \gamma_i^{-k}= \gamma_1^{-k}(c_1+ \sum_{i=2}^r c_i(\frac{\gamma_i}{\gamma_1})^{-k}    )\in \bar{\mathbb{Z}}_p, $$
alors la suite  $c_1+ \sum_{i=2}^r c_i (\frac{\gamma_i}{\gamma_1})^{-k}   $ converge vers $0$ lorsque $k$ tend vers $\infty$ et $\val_p(\frac{\gamma_i}{\gamma_1})\geq 0$. Donc la suite $$\sum_{i=2}^r c_i (\frac{\gamma_i}{\gamma_1}-1)(\frac{\gamma_i}{\gamma_1})^{-k}  $$ converge vers $0$  lorsque $k$ tend vers $\infty$ aussi.  

Notons que pour un $\gamma$ tel que $v_p(\gamma)\geq 0$, on a $$(\gamma^{-k}a_k \xrightarrow{k\rightarrow +\infty} 0)  \implies (a_k  \xrightarrow{k\rightarrow +\infty} 0).$$
En répétant la procédure ci-dessus, on obtiendra une contradiction. 
\end{proof}
\end{proof}

%\section{Une deuxième preuve courte la proposition \ref{Fixer}}

\printnoidxglossary

\end{document}